\documentclass[reqno]{amsart}
\usepackage{amsmath,amsrefs,color}
\usepackage{cleveref}
\usepackage{geometry}
\usepackage[toc,page]{appendix}
\usepackage{appendix}
\usepackage{amssymb}
\usepackage{graphicx,picture,
cleveref}
\usepackage[colorinlistoftodos]{todonotes}
\numberwithin{equation}{section}
\usepackage{multirow}
\usepackage{empheq}

  \usepackage{amsmath}
    
    \usepackage{graphicx}
    \usepackage{dcolumn}
    \usepackage{bm}
    \usepackage{amsmath}

\def\neweq#1{\begin{equation}\label{#1}}
\def\endeq{\end{equation}}

\newtheorem{theorem}{Theorem}[section]
\newtheorem{defi}{Definition}

\newtheorem{proposition}[theorem]{Proposition}
\newtheorem{lemma}[theorem]{Lemma}
\newtheorem{cor}[theorem]{Corollary}

\newtheorem{rem}[theorem]{Remark}
\begin{document}

\title[nonlinear elliptic equations with mixed reaction terms]{Classification and symmetry of global solutions for nonlinear elliptic equations with mixed reaction terms}

\author{Huyuan Chen}
\address{Huyuan Chen, School of Mathematics and Statistics, The University of Sydney, NSW 2006, Australia and }
\curraddr{Shanghai Institute for Mathematics and Interdisciplinary Sciences, Fudan University, Shanghai 200433, P.R. China} 
\email{chenhuyuan@yeah.net}

\author{Florica C. C\^irstea}

\address{Florica C. C\^irstea, School of Mathematics and Statistics, The University of Sydney, NSW 2006, Australia}
\email{florica.cirstea@sydney.edu.au}

\author{Aleksandar Miladinovic}

\address{Aleksandar Miladinovic, School of Mathematics and Statistics, The University of Sydney, NSW 2006, Australia}
\email{aleks.m678@gmail.com}

\date{}

\begin{abstract} 
In this paper, we describe the set of {\em all} positive distributional $C^1(\mathbb R^N\setminus \{0\})$-solutions 
of elliptic equations with mixed reaction terms of the form
\begin{equation} \label{unu1}\mathbb L_{\rho,\lambda,\tau}[u]:=  \Delta u-(N-2+2\rho) \frac{x\cdot \nabla u}{|x|^2} +\lambda \frac{u^\tau |\nabla u|^{1-\tau}}{|x|^{1+\tau}}=|x|^\theta u^q\quad \mbox{in }
\mathbb R^N\setminus \{0\}, 
\end{equation}
where $\rho,\lambda, \theta\in \mathbb R$ are arbitrary, $N\geq 2$, $q>1$ and $\tau\in [0,1)$.   
Defining  $\beta=(\theta+2)/(q-1)$ and 
$f_{\rho,\lambda,\tau}(t)=t\left(t+2\rho\right) +\lambda |t|^{1-\tau}$ for $t\in \mathbb R$, we show that \eqref{unu1} has positive solutions if and only if $f_{\rho,\lambda,\tau}(\beta)>0$. 
Under this condition, we provide existence and 
 the exact asymptotic behaviour near zero and at infinity for all positive solutions of \eqref{unu1}. We 
obtain that all such solutions are   
{\em radially symmetric}.  When $\theta<-2$ and $\rho,\lambda\in \mathbb R$, we also find the precise {\em local} behaviour near zero for {\em all} positive solutions of \eqref{unu1} in $\Omega\setminus \{0\}$, where $\Omega$ is an open set containing $0$.   
By introducing the second term in $\mathbb L_{\rho,\lambda,\tau}[\cdot]$ with $\rho\in \mathbb R$, we reduce the study of \eqref{unu1} 
to $\theta<-2$ via a modified Kelvin transform. 
We reveal new and surprising phenomena compared with the work of C\^{\i}rstea and F\u arc\u a\c seanu (2021), where $\rho=(2-N)/2$ and $\tau=1$. 
We illustrate this fact 
in 
Case $[N_2]$: $\lambda>0$, $\tau\in (0,1)$ and $\rho\geq \Upsilon :=[(\tau+1)/(2\tau) ]\left( \lambda \tau\right)^{1/(\tau+1)}$ 
when $f_{\rho,\lambda,\tau}(t)=0$ has two negative roots $\varpi_1(\rho)\leq \varpi_2(\rho)$. If 
$\beta\in (\varpi_2(\rho),0)$, then \eqref{unu1} yields two new families of {\em bounded} solutions $\{u_{\infty,c}\}_{c>0}$ and $\{u_{b,\infty,c}\}_{b,c>0}$
that have no analogue for $\tau=1$. 
\end{abstract}

\maketitle

\tableofcontents

\section{Introduction} \label{intro-m}

A fundamental problem in the study of partial differential equations is to understand the local 
behaviour of solutions. 
The classification of local and/or global solutions and their symmetry properties are central themes in the study of nonlinear elliptic and parabolic equations (see, e.g., \cite{Soup} and \cite{bov}). 
Even for deceptively ``simple-looking" semilinear elliptic problems (e.g., the Yamabe problem), the classification of singularities can be very hard to elucidate. Yet, research on isolated singularities has led to the development of new and significant approaches such as the moving plane method. This is
a fundamental tool for establishing symmetry in geometry and partial differential equations (see, for example, \cite{brs}). 

Serrin's celebrated papers \cites{serrin64,serrin65} for quasilinear elliptic problems of the form
$$ {\rm div}\, \mathbf A(x,u,\nabla u)=B(x,u,\nabla u)
$$ have generated a lot of interest. Much research has been devoted to finding suitable classification theorems or discovering  
analogues of Serrin's removability results applicable to broader contexts.

This paper falls within the scope of an extensive research program stimulated by V\'eron's question (see \cite{Veron1}) of addressing the singularity problem outside Serrin's framework in  \cites{serrin64,serrin65}. This means considering nonlinear divergence-form problems with an isolated singularity in which the growth of the lower order term $B$ 
dominates that of $\mathbf A$. Now, the structure of the lower order term $B$ and the sign of its constituents play a fundamental role in the classification of the local behaviour of solutions.  One new difficulty is, for example, the formation of ``strong singularities". These singular solutions dominate near the singularity the ``fundamental solution" of the principal operator, which provides the asymptotic model for the so-called ``weak singularities". Weak and strong singularities arise even in the classification of non-negative solutions for Laplacian type equations with super-linear power nonlinearities (see the pioneering works by Brezis and V\'eron \cite{brv} and V\'eron \cites{ver1,ver2}). Early generalisations of these results to $p$-Laplacian-type equations $\Delta_p u=u^q$ in $\mathbb R^N\setminus \{0\}$  with  $1<p\leq N$ are due to Friedman and V\'eron \cite{FV} (for $p-1<q<N(p-1)/(N-p)$) and V\'azquez and V\'eron
\cite{vv} (for $q\geq N(p-1)/(N-p)$ and $1<p<N$). For related extensions to weighted $p$-Laplacian equations with more general nonlinearities, we refer to  \cites{bra,cc,CC2015,CD2010}. 

The last decades have brought significant progress in the study of isolated singularities, see V\'eron \cite{bov}. A key question is 
how the local and global classification of singularities is influenced by gradient-dependent lower order terms. Answering this question will also shed light on the conditions leading to Liouville type results, another important and related topic studied extensively. For recent contributions in these directions, see e.g., \cites{BiGV,BVGV,BGV0,BGV3,BGV1,BNV1,BVer1,CC2015,CC2,FPS,FPS2,SZ}.

The aim of this paper is to classify the local behaviour near zero (and at infinity if $\Omega=\mathbb R^N$) for the positive  
distributional $C^1(\Omega\setminus \{0\})$-solutions of
elliptic equations of the form
\begin{equation} \label{e1}
 \mathbb L_{\rho,\lambda,\tau}[u(x)]:=\Delta u-\left(N-2+2\rho\right) \frac{x\cdot \nabla u}{|x|^2} +
 \lambda \frac{u^\tau\,|\nabla u|^{1-\tau}}{|x|^{1+\tau}}
 =|x|^\theta \,u^q\quad \mbox{in  }\Omega\setminus \{0\},
\end{equation}
where $\Omega$ is a domain in $\mathbb R^N$ $(N\geq 2)$ such that $0\in \Omega$, while 
$\rho,\lambda,\theta$ are any real parameters. 

Unless otherwise stated, throughout this paper, we assume that 
\begin{equation} \label{e2} 
0\leq \tau<1 \quad \mbox{and} \quad q>1.
\end{equation}

In this setting, we develop a unified approach to fully classify, and prove existence of, all positive solutions of \eqref{e1} in $\mathbb R^N\setminus \{0\}$ under optimal 
assumptions (see Theorem~\ref{globo} and Corollary~\ref{glob22}). As a by-product, we obtain that all such solutions are {\em radially symmetric}. We make no use of the moving plane method. Instead, with delicate  
constructions of many families of barriers (sub-super-solutions) adapted to different ranges of our parameters, we classify 
all the behaviours of positive solutions near zero and at infinity. Then, we prove the existence of positive {\em radial} solutions $u(x)=u(|x|)=u(r)$ of \eqref{e1} in $\mathbb R^N\setminus \{0\}$ with {\em all} these viable behaviours and, in addition, $u'(r)\not=0$ for every $r\in (0,\infty)$.    
By the comparison principle, we conclude that there are no positive solutions of \eqref{e1} in $\mathbb R^N\setminus \{0\}$ that are non-radial.

Our results lead to a complete understanding of the set of all 
$C^1(\mathbb R^N\setminus \{0\})$ solutions for other nonlinear equations with gradient dependent terms via a change of variables (see Subsection~\ref{apz}).  

To our best knowledge, the analysis of \eqref{e1} is completely new for $\tau\in (0,1)$. 
Our findings  reveal surprising new phenomena compared with 
 $\tau=1$ and $\rho=(2-N)/2$ in \eqref{e1}, the latter being completely understood, see Subsection~\ref{conta}. 
Yet, prior to this work, 
nothing was known about the positive solutions of \eqref{e1} with $\tau\in (0,1)$,  
even for $\rho=(2-N)/2$. By introducing the second term in the operator $\mathbb L_{\rho,\lambda,\tau}[\cdot ]$, we generate novel techniques to contend with many additional difficulties. In particular, when $\lambda>0$, we unveil a {\em critical threshold for $\rho$}, namely, $\Upsilon=\frac{\tau+1}{2\tau} \left( \lambda \tau\right)^{\frac{1}{\tau+1}}$, which is the dividing line between Case $[N_0]$ (b) and Case $[N_2]$ in  \eqref{3cas}. 
 
Our techniques are amenable to generalisations to quasilinear equations and other nonlinearities in the right-hand side of \eqref{e1}. This could be done by combining our methods here with those in \cites{Cmem,CD2010,cc}. However, such a generalisation goes beyond the scope of this paper. 
 
In proving the existence of positive solutions in $\mathbb R^N\setminus \{0\}$, we rely on the invariance of \eqref{e1} under 
the scaling transformation $ T_\sigma$ (for $\sigma>0$) given by 
\begin{equation} \label{ssig} u_\sigma(x)=T_\sigma[u](x)= \sigma^{\beta} \, u(\sigma x),\quad \mbox{where }  \beta=\beta(\theta,q):=\frac{\theta+2}{q-1}.
\end{equation} 
Thus, if $u$ is a positive solution of \eqref{e1} in $\mathbb R^N\setminus \{0\}$, so is $u_\sigma$.

We show in Theorem~\ref{globo} that $f_\rho(\beta)>0$ is the {\em optimal}
condition for the existence of positive solutions of \eqref{e1} in $\mathbb R^N\setminus \{0\}$ (see \eqref{frol} for the definition of $f_\rho$). 
Let's assume $f_\rho(\beta)>0$. Remark that 
\eqref{e1} has a positive radial solution in $\mathbb R^N\setminus \{0\}$ given by
\begin{equation} \label{u0} U_{\rho,\beta}(x)=\Lambda\, |x|^{-\beta},\quad \mbox{where } 
\Lambda:=\left[f_{\rho}(\beta)\right]^{\frac{1}{q-1}}.\end{equation}
Moreover, $U_{\rho,\beta}$ is never the only positive solution of \eqref{e1} in $\mathbb R^N\setminus \{0\}$ (unlike the case $\tau=1$ treated in \cite{MF1} and recalled in Theorem~\ref{alun}).  
In \eqref{stor1} when $\theta<-2$, respectively, in \eqref{stor2} when $\theta>-2$, we give the structure of the set of all positive solutions of \eqref{e1} in $\mathbb R^N\setminus \{0\}$.  

The most difficult analysis, leading to totally new findings, arises in Case [$N_2$] of \eqref{3cas} when $f_{\rho}(t)=0$ has {\em two negative roots}, say $\varpi_1(\rho)$ and $\varpi_2(\rho)$ 
(with $\varpi_1(\rho)\leq \varpi_2(\rho)$).
If $\beta\in  (\varpi_2(\rho),0)$ in Case $[N_2]$, then $f_\rho(\beta)>0$ and the set of all positive solutions of \eqref{e1} in $\mathbb R^N\setminus \{0\}$ is 
 $$ \{U_{\rho,\beta}\}  \cup \{u_b\}_{b\in \mathbb R_+} \cup \{u_{\infty,c}\}_{c\in \mathbb R_+}\cup  \{u_{b,\infty,c}\}_{b,c\in \mathbb R_+}.$$  
The solutions 
$\{u_{\infty,c}\}_{c\in \mathbb R_+}$ and $ \{u_{b,\infty,c}\}_{b,c\in \mathbb R_+}$ have no analogue for $\tau=1$ (see Theorem~\ref{alun} for a comparison). 
For every $b,c\in \mathbb R_+$, we show that \eqref{e1} in $\mathbb R^N\setminus \{0\}$ has a unique positive radial solution 
$u_{\infty,c}$ (respectively, $u_b$) satisfying $(z_1)$ 
$\lim_{|x|\to 0} u(x)/U_{\rho,\beta}(x)= 1$ and $(i_1)$ $\lim_{|x|\to \infty} u(x)= c$ (respectively, $(z_2)$ $\lim_{|x|\to 0} u(x)/\Phi_{\varpi_2(\rho)}(|x|)=b$ and 
$(i_2)$ $\lim_{|x|\to \infty} u(x)/U_{\rho,\beta}(x)= 1$). For the definition of $\Phi_{\varpi_2(\rho)}(|x|)$, see \eqref{phi2}. 
We prove that \eqref{e1} in $\mathbb R^N\setminus \{0\}$, subject to $(z_2)$ and $(i_1)$, has a unique positive radial 
solution  $u_{b,\infty,c}$, where 
$b,c\in \mathbb R_+$ are arbitrary.

The existence of the positive solutions of \eqref{e1} in $\mathbb R^N\setminus \{0\}$ (see Sections~\ref{alt1}, \ref{alt2} and \ref{alt3}) requires first 
classifying their behaviour  near zero and at infinity (see Theorem~\ref{rod1c}). 
We develop an iterative method for constructing more and more refined 
local sub-super-solutions of \eqref{e1} both near zero and at infinity. Then, we inductively improve the behaviour of positive solutions by a local comparison with these sub-super-solutions until attaining the desired profile.  
This iterative scheme, which is very flexible and adaptable to many situations, is an innovation that appears here for the first time (see  Subsection~\ref{strag} for details).    
Our 
methods are susceptible to applications to a wide class of equations admitting comparison principles and the strong maximum principle (see \cites{GT1983,PS2004,PS2007} for such contexts).

In deciphering the behaviours near zero (or at infinity) for the positive solutions of \eqref{e1}, we need to account for the competition arising between 
the operator $\mathbb L_{\rho,\lambda,\tau}[u]$ and the right-hand side of \eqref{e1}. 
We point out that  it is not the Laplacian alone but the operator $\mathbb L_{\rho,\lambda,\tau}[\cdot]$ that will be crucial in providing (through the 
solution(s) of $\mathbb L_{\rho,\lambda,\tau}[\cdot]=0$ when they exist) the asymptotic model(s) of reference. We observe that,  like the right-hand side of \eqref{e1}, the second and third terms in the operator $\mathbb L_{\rho,\lambda,\tau}[u]$ don't fit into  
the settings of Serrin's papers  \cites{serrin64,serrin65}.  

Now, if $v_\eta(x)=|x|^{-\eta}$ for $|x|>0$ and $\eta\in \mathbb R$, then we see that $v=v_\eta$ solves 
\begin{equation} \label{juv} \mathbb L_{\rho,\lambda, \tau}[v]=0\quad \mbox{in } B_1(0)\setminus \{0\} \quad \mbox{(or in } \mathbb R^N\setminus \overline{B_1(0)} )
\end{equation}
if and only if 
$f_\rho(\eta)=0$, where
for every $\rho,\lambda\in \mathbb R$, we define 
 \begin{equation} \label{frol}  f_{\rho}(t):=f_{\rho,\lambda,\tau}(t)=t\left(t+2\rho\right) +\lambda \,|t|^{1-\tau}\quad \mbox{for all } t\in \mathbb R.\end{equation}
(To simplify notation, we will write $f_\rho$ instead of $f_{\rho,\lambda,\tau}$. 
The parameters $\lambda$ and $\tau$ won't change when applying the modified Kelvin transform, see Appendix~\ref{Kel}.) 

The assumption $0\leq \tau<1$ gives that $t=0$ is always a solution of $f_{\rho}(t)=0$. 
However, a crucial role will be played by the {\em non-zero roots} of $f_{\rho}(t)=0$ and their position in relation to $\beta:=(\theta+2)/(q-1)$.   
For $t<0$, the function $\widetilde f_{\rho}(t)=-t-2\rho +\lambda \left(-t\right)^{-\tau}$ has the same sign as $f_{\rho}(t)$.  
From the graph of $t\longmapsto \widetilde f_{\rho}(t)$ on $(-\infty,0)$, we distinguish three cases:
\begin{equation} \label{3cas}
\boxed{ \begin{aligned}
& \mbox{\bf Case }[N_0]:\   \mbox{(a)}\ \tau=0, \ \lambda\geq 2\rho \ \mbox{ or }
 \mbox{(b) } \lambda>0, \ \tau\in (0,1), \ \rho< \Upsilon :=\frac{\tau+1}{2\tau} \left( \lambda \tau\right)^{\frac{1}{\tau+1}},\\
& \mbox{\bf Case }[N_1]: \ \mbox{(a)}\ 
\tau=0, \ \lambda<2\rho \ 
		 \mbox{ or (b) } \lambda<0, \ \tau\in (0,1),\ \rho\in \mathbb R,\\
& \mbox{\bf Case }[N_2]:\  \lambda>0, \ \tau\in (0,1), \ \rho\geq  \Upsilon.  		 
\end{aligned} }
\end{equation}

$\bullet$ In Case $[N_0]$, we have $\widetilde f_\rho>0$ on $(-\infty,0)$. More precisely, 
$  \widetilde f_\rho(t)>\lambda-2\rho\geq 0 $  for all $ t<0$  in Case $[N_0] $ (a) and 
$ \widetilde f_\rho(t)\geq \widetilde f_\rho(-(\lambda \tau)^{1/(\tau+1)})=2\left(\Upsilon-\rho\right)>0$ for all $ t<0$  in Case 
$[N_0]$ (b). 

$\bullet$ In Case $[N_1]$, we see that $\widetilde f_\rho(t)=0$ has exactly one negative root, denoted by $\varpi_1(\rho)$. (If Case $[N_1]$ (a) holds, then $\varpi_1(\rho)=\lambda-2\rho$.) We have 
$\widetilde f_\rho>0 \ \mbox{on } (-\infty,\varpi_1(\rho))$ and $ \widetilde f_\rho<0 \ \mbox{on } (\varpi_1(\rho),0)$.  

$\bullet$ 
In Case $[N_2]$, we denote by $\varpi_1(\rho)$ and $\varpi_2(\rho)$ the negative roots of 
$\widetilde f_\rho(t)=0$, where 
$$ \varpi_1(\rho)< -(\lambda \tau)^{1/(\tau+1)}< \varpi_2(\rho)\ \mbox{if }
 \rho\not=\Upsilon \quad \mbox{and}\quad 
 \varpi_1(\rho)=\varpi_2(\rho)=-(\lambda \tau)^{\frac{1}{\tau+1}} \ \mbox{when } \rho=\Upsilon. $$
  
We have $ \widetilde f_\rho>0$ on $ (-\infty, \varpi_1(\rho))\cup (\varpi_2(\rho),0)$ and 
$ \widetilde f_\rho<0$ on $(\varpi_1(\rho),\varpi_2(\rho))$ when $\rho\not=\Upsilon$.

\vspace{0.1cm}
{\bf Notation.} 
The negative root(s) of $\widetilde f_\rho(t)=0$ depend on $\rho,\lambda$ and $\tau$, 
 but to simplify notation, we do not explicitly indicate their dependence on $\lambda$ and $\tau$. 
To unify our presentation, we also denote by $\varpi_1(\rho)$ the unique negative root of 
$\widetilde f_{\rho}(t)=0$ in Case $[N_1]$. From Section~\ref{sha1} onward, we will simply write $\varpi_1$ (respectively, $\varpi_2$) instead of $\varpi_1(\rho)$ (respectively, $\varpi_2(\rho)$).

We use $f(x)\sim g(x)$ as $|x|\to 0$ (respectively, as $|x|\to \infty$) 
to mean that $\lim_{|x|\to 0} f(x)/g(x)=1$ (respectively, 
$\lim_{|x|\to \infty} f(x)/g(x)=1$).
 We set $\mathbb R_+=(0,\infty)$. 

\vspace{0.2cm}
In Case $[N_2]$, we define $\Psi_{\varpi_1(\rho)}(t)$ for every $t>1$ and 
$\Phi_{\varpi_2(\rho)}(r)$ for every $r\in (0,1)$ as follows 
 \begin{eqnarray}  
 & \Psi_{\varpi_1(\rho)}(t)= t^{-\varpi_1(\rho)} \  \mbox{if } \rho>\Upsilon\  \mbox{and}\  
 \Psi_{\varpi_1(\rho)}(t)=t^{-\varpi_1(\rho)} \left( \log t \right)^{\frac{2}{1+\tau}} \  \mbox{if } \rho=\Upsilon, \label{psia} \\
& \Phi_{\varpi_2(\rho)}(r)=r^{-\varpi_2(\rho)}\ \mbox{if } \rho>\Upsilon\ \mbox{and}\ 
\Phi_{\varpi_2(\rho)}(r)=r^{-\varpi_2(\rho)} \left| \log r \right|^{\frac{2}{1+\tau}} \ \mbox{if } \rho=\Upsilon.
\label{phi2}
\end{eqnarray}

\subsection{Main results}
In Theorem~\ref{globo} (for $\theta<-2$) 
and Corollary~\ref{glob22} (for $\theta>-2$), we give the exact structure of the set of all positive solutions of \eqref{e1} in $\mathbb R^N\setminus \{0\}$, under the optimal condition $f_\rho(\beta)>0$. In particular,  when $\theta=-2$, there are no positive solutions for \eqref{e1} in $\mathbb R^N\setminus \{0\}$.

\begin{theorem}[Global existence and classification, $\theta<-2$] \label{globo}
	Let \eqref{e2} hold and $\rho,\lambda, \theta\in \mathbb R$.  
	
$\bullet$ Then, \eqref{e1} has positive solutions in $\mathbb R^N\setminus \{0\}$ if and only if  
$f_{\rho}( \beta)>0$.

$\bullet$ Assume that $\theta<-2$ and  $f_\rho(\beta)>0$. 
Then, every positive solution $u$ of \eqref{e1} in $\mathbb R^N\setminus \{0\}$ is 
radially symmetric ($u(x)=u(|x|)$ for all $x\in \mathbb R^N\setminus \{0\}$), $\lim_{r\to 0^+} u(r)=0$ and $u'>0$ on $(0,\infty)$.

\begin{itemize}
\item[(i)]  
For every $c\in \mathbb R_+$, there exists a unique positive 
solution $u_{\infty,c}$ of 
\eqref{e1} in $\mathbb R^N\setminus \{0\}$ satisfying $u(x)\sim U_{\rho,\beta}(x)$ as $|x|\to 0$ and one of the following as $|x|\to \infty$: 
\begin{equation}
u(x)\sim \left\{ 
\begin{aligned} 
& c \,|x|^{-\varpi_1(\rho)} && \mbox{in Case } [N_1]\ \mbox{with }\beta<\varpi_1(\rho),&\\
& c \,\Psi_{\varpi_1(\rho)}(|x|) &&  	\mbox{in Case }[N_2]\ \mbox{with }\beta<\varpi_1(\rho) ,&\\
& c \,\log |x| &&   \mbox{if } \lambda=2\rho\ \mbox{when } \tau=0, &\\
& c &&  \mbox{in Case } [N_0] 
  \ \mbox{with } \lambda\not=2\rho\ \mbox{if }\tau=0&\\
& &&  \mbox{or if }  \beta\in (\varpi_2(\rho),0)\ \mbox{in Case }[N_2].&
\end{aligned}
\right.
 \label{fbio} \end{equation}

\item[(ii)] If $\beta\in (\varpi_2(\rho),0)$ in Case $[N_2]$, then for every $b,c\in \mathbb R_+$, there exists a unique positive solution $u_{b}$ (respectively, $u_{b,\infty,c}$) of \eqref{e1} in $\mathbb R^N\setminus \{0\}$ satisfying  
\begin{equation} \label{kno1}
\begin{aligned} 
&   u(x)\sim b\, \Phi_{\varpi_2(\rho)}(|x|)\quad \mbox{as } |x|\to 0\ \mbox{and}\\
& u(x)\sim U_{\rho,\beta}(|x|) \quad (\mbox{respectively, } 
u(x)\sim c)\quad \mbox{as } |x|\to \infty.
\end{aligned}
\end{equation}

\item[(iii)]   
The set of all positive solutions of \eqref{e1} in $\mathbb R^N\setminus \{0\}$ is given by 
\begin{equation} \label{stor1} \left\{  
\begin{aligned}
 &  \{U_{\rho,\beta}\} \cup \{ u_{\infty,c} \}_{c>0} \cup 
 \{ u_b \}_{b>0} \cup \{ u_{b,\infty,c} \}_{b,c>0} \quad  \mbox{if } \beta\in (\varpi_2(\rho),0)\ \mbox{in Case }[N_2], \\
 &  \{U_{\rho,\beta} \}\cup \{u_{\infty,c} \}_{c>0} \quad \mbox{otherwise}.  
  \end{aligned}	
\right.
\end{equation} 
\end{itemize}
\end{theorem}
\begin{rem}
\label{obsr} 	
{\rm When $\beta\in (\varpi_2(\rho),0)$ in Case $[N_2]$ of Theorem~\ref{globo},  we gain two  new families of 
{\em bounded} positive (radial) solutions $\{u_{\infty,c}\}_{c>0}$ and $\{ u_{b,\infty,c} \}_{b,c>0}$ for \eqref{e1} in $\mathbb R^N\setminus \{0\}$, 
leading to 
the crossing of the graphs of $u_{\infty,c}$ and $u_b$.  
Let $b,c\in \mathbb R_+$ be arbitrary. 
We have $u_{b,\infty,c}\nearrow u_{\infty,c}$ and $u_b\nearrow U_{\rho,\beta}$ as $b\to \infty$, as well as 
$u_{b,\infty,c}\nearrow u_b$ and $u_{\infty,c}\nearrow U_{\rho,\beta}$ as $c\to \infty$.   
Since $\lim_{r\to 0^+} u_{\infty,c}(r)/u_b(r)=\infty$ and $\lim_{r\to \infty} u_{\infty,c}(r)/u_b(r)=0$, 
the graphs of $ u_{\infty,c}(r)$ and $u_b(r)$  for $r\in (0,\infty)$ cross each other only once 
(by the strong maximum principle in Lemma~\ref{adol}). Yet, zero is a  removable singularity for all positive solutions of 
\eqref{e1} if $\theta<-2$, see also Remark~\ref{micx}. }
\end{rem}

By Case $[N_j](-\rho)$ (for $j=0,1,2$), we mean that Case $[N_j]$ holds with $-\rho$ instead of $\rho$.

\begin{cor}[Global existence and classification, $\theta>-2$] 
\label{glob22} 	
Let \eqref{e2} hold and $\rho,\lambda\in \mathbb R$. Assume that $ \theta>-2$ and $f_\rho(\beta)>0$. Then, every positive solution $u$ of \eqref{e1} is 
radially symmetric, $\lim_{r\to \infty} u(r)=0$ and $ u'(r)<0$ for all $r=|x|>0$.  

\begin{itemize}
\item[(i)]  
For every $c\in \mathbb R_+$, there exists a unique positive 
solution $u_{c,0}$ of 
\eqref{e1} in $\mathbb R^N\setminus \{0\}$ satisfying 
$ u(x)\sim U_{\rho,\beta}(x)$  as $ |x|\to \infty$ and one of the following as $|x|\to 0$:
\begin{equation} 
u(x)\sim 
\left\{ \begin{aligned}
& c\, |x|^{\varpi_1(-\rho)} && 
\mbox{if } \beta>-\varpi_1(-\rho)\ \mbox{in Case }[N_1](-\rho),&\\
& c\,\Psi_{\varpi_1(-\rho)}(1/|x|) &&
\mbox{if } \beta>-\varpi_1(-\rho)\ \mbox{in Case }[N_2](-\rho), &\\ 
&  c \log \,(1/|x|) && \mbox{if } \lambda=-2\rho\ \mbox{and } \tau=0, &\\ 
& c && \mbox{in Case }[N_0](-\rho)\ \mbox{with }  \lambda\not=-2\rho  \mbox{ if }\tau=0, & \\
& && \mbox{or if } \beta\in (0,-\varpi_2(-\rho))\ 
 \mbox{in Case } [N_2](-\rho) . &\\
\end{aligned}
\right. \label{mixo}
\end{equation}

\item[(ii)] If 
$\beta\in (0, -\varpi_2(-\rho))$  in Case $[N_2](-\rho)$, then for every  
$b,c\in \mathbb R_+$, there exists a unique positive solution $u_{b,\infty}$ (respectively, $ u_{c,0,b}$) of \eqref{e1} 
 in $\mathbb R^N\setminus \{0\}$
satisfying \begin{equation} \label{mixx1}
\begin{aligned}
& u(x)\sim b\, \Phi_{\varpi_2(-\rho)}(1/|x|) \ \mbox{ as } |x|\to \infty\ \mbox{and}\\	
& u(x)\sim U_{\rho, \beta}(|x|)\ \  (\mbox{respectively, } \ \ 
 u(x)\sim c)\ \mbox{as } |x|\to 0.
\end{aligned}
\end{equation}

\item[(iii)] The set of all positive solutions of \eqref{e1} 
in $\mathbb R^N\setminus \{0\}$ is given by 
\begin{equation} \label{stor2} \left\{  
\begin{aligned}
 &  \{ U_{\rho, \beta} \}\cup \{  u_{c,0} \}_{c>0} \cup 
 \{ u_{b,\infty} \}_{b>0} \cup \{ u_{c,0,b} \}_{b,c>0} \  
 \mbox{if } \beta\in (0,-\varpi_2(-\rho))\ \mbox{in Case }[N_2](-\rho), \\
 & \{ U_{\rho, \beta} \} \cup \{ u_{c,0} \}_{c>0} \quad  \mbox{otherwise}.
  \end{aligned}	
\right.
\end{equation}
\end{itemize}
\end{cor}

We show how Corollary~\ref{glob22} follows from 
Theorem~\ref{globo} using the modified Kelvin transform $\widetilde u(x)=u(x/|x|^2)$ for $|x|\not=0$. Indeed, understanding the positive solutions of \eqref{e1} in $\mathbb R^N\setminus \{0\}$ with $\theta<-2$ is equivalent to deciphering the positive solutions $\widetilde u$ of the equation
\begin{equation} \label{eeq1}
\mathbb L_{-\rho,\lambda,\tau} [\widetilde u(x)]=|x|^{\widetilde \theta} \,[\widetilde u(x)]^q\quad \mbox{for } x\in \mathbb R^N\setminus \{0\},
\end{equation}
where $\rho,\lambda\in \mathbb R$ and $\widetilde \theta:=-\theta-4>-2$ (see Appendix~\ref{Kel}). Similar to $\beta$ in \eqref{ssig}, we define 
$\widetilde \beta=(\widetilde \theta+2)/(q-1)$. Note that 
$\widetilde \beta=-\beta$, $ f_\rho(\beta)=f_{-\rho}(\widetilde \beta)$ and $ 
U_{-\rho,\widetilde \beta} (r)=U_{\rho,\beta}(1/r)=[f_{\rho}(\beta)]^\frac{1}{q-1}r^{ \beta}$ for every $ r>0$.  
Hence, by restating Theorem~\ref{globo} in equivalent terms for \eqref{eeq1} and then relabelling $-\rho$ and $\widetilde \theta$, respectively, by $\rho$ and $\theta$, respectively, we arrive at Corollary~\ref{glob22}.

To prove Theorem~\ref{globo}, we first need to 
obtain the local behaviour near zero for all positive solutions of \eqref{e1} (under the assumption $f_\rho(\beta)>0$). In Theorem~\ref{thi} below, we completely resolve this matter for every $\theta<-2$ and $\rho,\lambda\in \mathbb R$, distinguishing between $f_\rho(\beta)>0$ and 
 $f_\rho(\beta)\leq 0$.

\begin{theorem} \label{thi} 
Let \eqref{e2} hold, $\rho,\lambda\in \mathbb R$ and $\theta<-2$. Let $u$ be any positive solution of \eqref{e1}.   
           
$\bullet$ Assume that $f_\rho(\beta)>0$. 
\begin{enumerate}
\item[($P_0$)] If $\beta\in (\varpi_2(\rho),0)$ in Case $[N_2]$, then exactly one of the following holds as $|x|\to 0$:	
\begin{eqnarray}   
 & {\rm (i)}\  u(x)\sim a\,|x|^{-\varpi_1(\rho)}\  \mbox{and}\  
x\cdot \nabla u(x)\sim -a \,\varpi_1(\rho) \,|x|^{-\varpi_1(\rho)}\ 
\mbox{for some } a\in \mathbb R_+;\label{mobb1}\\
& {\rm (ii)}\  u(x)\sim b\, \Phi_{\varpi_2(\rho)}(|x|)  \  \mbox{and} 
\  x\cdot \nabla u(x) \sim -b\, \varpi_2(\rho)\, \Phi_{\varpi_2(\rho)}(|x|)\ \mbox{ for some } b\in \mathbb R_+; \label{gapp}\\
& {\rm (iii)}\  
u(x)\sim U_{\rho,\beta}(|x|).  \label{tirv}
\end{eqnarray}
\item[$(P_1)$] In the remaining situations, $u$ satisfies \eqref{tirv} as $|x|\to 0$.   
\end{enumerate}

$\bullet$ Assume that $f_\rho(\beta)\leq 0$. 
\begin{enumerate}
\item[$(Z_1)$] 
If $\beta=\varpi_1(\rho)$ in Case $[N_j]$ $(j=1,2)$
 and, in addition, $\rho\not=\Upsilon$ if Case $[N_2]$ holds, then 
$$ u(x)\sim  \left(
\frac{ |\varpi_1(\rho)| \left(1-\lambda \tau |\varpi_1(\rho)|^{-\tau-1} \right)}{q-1}\right)^{\frac{1}{q-1}} \frac{|x|^{-\varpi_1(\rho)}}{ |\log |x||^{\frac{1}{q-1}}} \quad \mbox{as } |x|\to 0. 
$$
\item[($Z_2$)] If $\rho=\Upsilon$ in Case $[N_2]$ and $\beta=\varpi_1(\rho)=\varpi_2(\rho)=-(\lambda \tau)^{\frac{1}{\tau+1}} $, then 
\begin{equation} \label{abv}
u(x)\sim  \left(
\frac{2\left(q+\tau\right)}{(q-1)^2}\right)^{\frac{1}{q-1}} \frac{|x|^{-\varpi_1(\rho)}}{ |\log |x||^{\frac{2}{q-1}}} \quad \mbox{as } |x|\to 0.
\end{equation}
\item[$(N)$] In the remaining situations, $u$ satisfies
\eqref{mobb1} as $|x|\to 0$. 
	
\end{enumerate}

\end{theorem}

\begin{rem} \label{micx} {\rm When $\theta<-2$ and $\rho,\lambda\in \mathbb R$, we show that every positive solution $u$ of \eqref{e1} satisfies
 $\lim_{|x|\to 0} u(x)=0$ and $u\in H^1_{\rm loc}(\Omega)$ can be extended as a continuous non-negative solution of \eqref{e1} in $\mathcal D'(\Omega)$ (see Corollary~\ref{contt} in Appendix~\ref{sectiune2}). Proposition~\ref{aurr} in Appendix~\ref{sectiune2} gives that  
$(x\cdot \nabla u)/|x|^2$, $u^\tau |\nabla u|^{1-\tau}/|x|^{1+\tau}$ and $|x|^\theta u^q $ are locally integrable in $\Omega$.  
Moreover, if $\theta\leq -2$ and, in addition, $\Omega$ is bounded, then \eqref{e1} has 
no positive solutions satisfying $u=0$ on $\partial \Omega$ (see Corollary~\ref{non-exi}). }
\end{rem}

\begin{rem} \label{no-on9} {\rm If 
$\beta\in (\varpi_2(\rho),0)$ in Case $[N_2]$, then 
there are no positive solutions of \eqref{e1} in $\mathbb R^N\setminus \{0\}$ with the profile in alternative {\rm $(P_0)$ (i)} of Theorem~\ref{thi} (see Lemma~\ref{gold1}). 	}
\end{rem}

\subsection{Applications of our results} \label{apz} 
We present two nonlinear elliptic problems with gradient-dependent lower order terms for which our results can be applied. Let $\Omega$ be an open set in $\mathbb R^N$ ($N\geq 2$) with $0\in \Omega$. We consider the $C^1(\Omega\setminus \{0\})$ solutions of 
\begin{equation} \label{rigo1}
\Delta w+|\nabla w|^2 +\mu \frac{x\cdot \nabla w}{|x|^2} +\lambda \frac{|\nabla w|^{1-\tau}}{|x|^{1+\tau}} =|x|^\theta e^{(q-1)w}\quad \mbox{in } 
\Omega\setminus \{0\},
\end{equation}
where $\mu,\lambda,\theta\in \mathbb R$, $q>1$ and $\tau\in [0,1)$, respectively, the positive $C^1(\Omega\setminus \{0\})$ solutions $v$ of 
\begin{equation} \label{rigo2}
	\Delta (v^\alpha) +\mu \frac{x\cdot \nabla (v^\alpha)}{|x|^2} 
	+\eta \frac{|\nabla v|^{\alpha}}{|x|^{2-\alpha}} =|x|^\theta v^{K}\quad \mbox{in }\Omega\setminus \{0\},
\end{equation}
where $\mu,\eta,\theta\in \mathbb R$, while $\alpha\in (0,1)$ and $q=K/\alpha>1$. For both problems, we let $\rho=(2-N-\mu)/2$. 

When $\Omega=\mathbb R^N$,  under the optimal condition $f_\rho(\beta)>0$, 
the set of all solutions of \eqref{rigo1} (respectively, positive solutions of \eqref{rigo2}) is described 
by Theorem~\ref{globo} and Corollary~\ref{glob22}. These results with $u=e^w$
(respectively, with $u=v^\alpha$, $\lambda=\eta/\alpha^\alpha$ and $\tau=1-\alpha\in (0,1)$) give 
the precise asymptotic behaviour near zero and at infinity for such solutions.  
Furthermore, assuming $\theta<-2$ and no sign restriction on $f_\rho(\beta)$,  
the local behaviour near zero 
for an arbitrary solution $w$ of \eqref{rigo1} (respectively, positive solution $v$ of \eqref{rigo2}) can be derived from Theorem~\ref{thi}.

Problems of the form \eqref{rigo2} are the stationary versions of parabolic equations with fast diffusion ($\alpha\in (0,1)$), featuring absorption ($|x|^\theta v^K$) and a weighted gradient term ($\eta |x|^{\alpha-2} |\nabla v|^\alpha $). The fast diffusion equation, $u_t=\Delta (u^\alpha)$ with $\alpha\in (0,1)$, is an important model for singular nonlinear diffusion phenomena, including gas-kinetics, thin liquid film dynamics and diffusion in plasmas (see \cites{das,Vazq}).  
Degenerate and singular parabolic equations with absorption have been intensively studied by many authors. The main feature is the competition between
the diffusion $\Delta (v^\alpha)$ and the absorption, which leads to different dynamics and  large time behaviour of non-negative solutions for various regimes 
 of the exponents (see, e.g.,  \cites{BIL,IL,XF} and their references).  

The understanding of the structure of all positive solutions for \eqref{rigo2} opens up new avenues for studying a wide class of parabolic equations with fast diffusion and gradient dependent terms.

\subsection{Comparison with previous works for $\tau=1$} \label{conta}
We mention related works for the case 
$\tau=1$ in \eqref{e1}. This corresponds to equations with a Hardy-type potential, which have  been actively studied for $\rho=(2-N)/2$ and various choices of $\lambda$ and $\theta$ (see \cites{bra,Cmem,MF1,GuV,LD2017}). For example, when $\theta>-2$, we refer to 
\cite{Cmem} (for $\lambda\leq (N-2)^2/4$) and \cite{LD2017} (for $\lambda>(N-2)^2/4$).  

C\^{\i}rstea and F\u arc\u a\c seanu \cite{MF1} removed the restrictions on $\lambda$ and $\theta$ imposed in \cites{Cmem,LD2017} and 
fully classified all positive solutions for equations of the form  
\begin{equation} \label{subbb} \Delta u+\lambda \frac{u}{|x|^2}=|x|^\theta u^q\quad \mbox{in } \mathbb R^N\setminus \{0\},\ \ (N\geq 3), 
\end{equation} where $\lambda,\theta\in \mathbb R$ and $q>1$. Let $\beta=(\theta+2)/(q-1)$. If 
\begin{equation} \label{guf}  \lambda>\lambda^*=\beta\left(N-2-\beta\right) \ \mbox{(or, equivalently, } f_{(2-N)/2,\lambda,1}(\beta)>0),
	\end{equation}
	then \eqref{subbb} has a positive radial solution given by
\begin{equation} \label{rubo} U_{0}(x)= (\lambda-\lambda^*)^{1/(q-1)} |x|^{-\beta}\quad \mbox{for } x\in \mathbb R^N\setminus \{0\}.
\end{equation}

It was shown in \cite{MF1} that \eqref{guf} is the sharp  condition 
for the existence of positive solutions of \eqref{subbb}. 
If $\lambda>(N-2)^2/4$, then \eqref{guf} holds for every $\beta\in \mathbb R$. 

If $\lambda\leq (N-2)^2/4$, then \eqref{guf} is equivalent to 
$\beta\in (-\infty,p_1)\cup (p_2,\infty)$, where 
$$ p_1=p_1(\lambda)=\frac{N-2}{2} -\sqrt{\frac{(N-2)^2}{4}-\lambda}\quad \mbox{and} \quad 
p_2=p_2(\lambda)=\frac{N-2}{2} +\sqrt{\frac{(N-2)^2}{4}-\lambda}.
$$
In this case,  an important role in 
Theorem~\ref{alun} will be played by $\Psi_{p_1}$ and $\Phi_{p_2}$, which on their domain of definition satisfy $\mathbb L_{(2-N)/2,\lambda,1}[\cdot]=0$. More precisely, we define 
 $$\begin{aligned} 
 & \Psi_{p_1}(x)=  |x|^{-p_1} \  \mbox{if } \lambda<\frac{(N-2)^2}{4},\quad    \Psi_{p_1}(x)= |x|^{-\frac{N-2}{2}} \log |x| \ \mbox{if } \lambda=
 \frac{(N-2)^2}{4}\ \mbox{for } |x|>1,\\
 &  \Phi_{p_2}(x)=|x|^{-p_2} \  \mbox{if } \lambda<\frac{(N-2)^2}{4},\quad  \Phi_{p_2}(x)= |x|^{-\frac{N-2}{2}} \log \frac{1}{|x|} 
  \mbox{ if } \lambda=\frac{(N-2)^2}{4}\ \mbox{for } |x|\in (0,1).
 \end{aligned} $$

\begin{theorem}[See \cite{MF1}] \label{alun}
	Equation \eqref{subbb} has positive solutions if and only if  \eqref{guf} holds. In this case, 	
	every positive solution of \eqref{subbb} is radially symmetric.

	$\bullet$ If $\lambda>(N-2)^2/4$, then $U_0$ in \eqref{rubo} is the only positive solution of \eqref{subbb}. 
	
	$\bullet$ If $\lambda^*<\lambda\leq (N-2)^2/4$, then for every $\gamma\in \mathbb R_+$,  equation \eqref{subbb} has a unique positive solution 
	$U_\gamma$ (which is radially symmetric) satisfying     
	$$ 
	\begin{aligned} 
	& 	U_\gamma(x)\sim U_0(x) \ \mbox{as } |x|\to 0\ \mbox{and}\quad U_\gamma(x)\sim \gamma \,\Psi_{p_1}(x)\ \mbox{as } |x|\to \infty\ 
	\mbox{if  } \beta<p_1,\\
	&  U_\gamma(x)\sim U_0(x)\ \mbox{as } |x|\to \infty\ \mbox{and}\ 
	U_\gamma(x)\sim \gamma\, \Phi_{p_2}(x)\ \mbox{as } |x|\to 0 \ \mbox{if } \beta>p_2.
	\end{aligned} $$

	Moreover, the set of all positive solutions of \eqref{subbb} is given by 
	$\{U_0\}\cup \{U_{\gamma}\}_{\gamma\in \mathbb R_+}$. 
\end{theorem}

Clearly, the structure of the set of positive solutions of \eqref{e1} in $\mathbb R^N\setminus \{0\}$ with $\tau\in [0,1)$ revealed in Theorem~\ref{globo} and Corollary~\ref{glob22} is much richer than that for 
\eqref{subbb} in Theorem~\ref{alun}. 

Our approach is to completely understand the case  
$\theta<-2$ (based on which we also elucidate $\theta>-2$). Indeed, when $\theta<-2$ and $f_\rho(\beta)>0$,   
in passing from the local behaviour near zero for the positive solutions of \eqref{e1} (see Theorem~\ref{thi})
to the analysis of \eqref{e1} in $\mathbb R^N\setminus \{0\}$, all asymptotic profiles but one 
will be appearing near zero, see Theorem~\ref{globo} and Remark~\ref{no-on9}. 
 
We will further develop the ideas put forward in \cite{MF1} for the special case 
$\tau=1$, $\rho=(2-N)/2$. The method of sub-super-solutions will continue to play a crucial role in our local analysis (near zero and at infinity). 
However, due to our general framework, we need many additional and non-trivial constructions of local upper and lower barriers.     
We mention a few significant challenges that appear due to  the introduction in \eqref{e1} 
of the term $\lambda u^\tau |\nabla u|^{1-\tau}/|x|^{1+\tau}$ with $\tau\in [0,1)$:  

$\bullet$ We lose the linearity of the operator $\mathbb L_{\rho,\lambda,\tau}[\cdot]$. 
Unlike the case $\tau=1$, we cannot conclude anymore that the sum of two super-solutions of \eqref{e1} is necessarily 
a super-solution.   

$\bullet$ Changes of variables used for $\tau=1$ to reduce the study of positive radial solutions to previously understood ODEs (see, e.g., \cite{Cmem}) are not suitable when $\tau\in [0,1)$. Unlike the case $\tau=1$, we could not trace any known results for ODEs that are applicable to our setting.   

$\bullet$ The standard Kelvin transform used in \cite{MF1} for $\tau=1$ and $\rho=(2-N)/2$ to reduce the study of the asymptotic behaviour at infinity 
(e.g., for $\beta>p_2$) to that near zero (for $\beta<p_1$) is not effective anymore if $\tau\in [0,1)$. Instead, we use the modified Kelvin transform $\widetilde u(x)=u(x/|x|^2)$ for $|x|\not=0$, which brings the need to introduce the second term in the operator $\mathbb L_{\rho,\lambda,\tau}[\cdot]$. 

$\bullet$ The critical situation $\rho=\Upsilon$ n Case $[N_2]$ (when $\varpi_1(\rho)=\varpi_2(\rho)=-(\lambda \tau)^{1/(\tau+1)}$) adds a new layer of complexity specific to $\tau\in (0,1)$ and $\lambda>0$: for $\mathbb L_{\rho,\lambda,\tau}[\cdot]=0$ with $|x|>0$ small (or $|x|>1$ large), it seems difficult to find an {\em explicit} positive solution  that is linearly independent from the solution $|x|^{-\varpi_1(\rho)}$. We work around this issue by providing in \eqref{phi2} an explicit positive radial super-solution 
$\Phi_{\varpi_2(\rho)}(|x|)$ for $|x|>0$ small (see Lemma~\ref{fiqa}) respectively, an explicit radial super-solution 
$\widetilde \Psi(|x|)$ for $|x|>1$ large that is asymptotically equivalent at infinity to $\Psi_{\varpi_1(\rho)}(|x|)$ in \eqref{psia} (see Lemmas \ref{siww} and \ref{feed}). We refer to Section~\ref{ditto} for further details.  
 
 $\bullet$ 
When $\tau\in (0,1)$, by introducing the second term in the operator $\mathbb L_{\rho,\lambda,\tau}[\cdot]$ with $\rho\in \mathbb R$, we face additional obstacles in our construction of local 
sub/super-solutions compared with \cite{MF1}. For the latter, where $\rho=(2-N)/2$ and $\tau=1$, there was no need to construct more than two families of sub/super-solutions for a specific situation. 
However, in our general setting, new thresholds appear for $\rho$, which require designing several delicate iterative schemes. These will help us incrementally improve the upper and lower bounds of the solutions of \eqref{e1} near zero and at infinity, starting at a certain level and then proceeding by induction until reaching very close to the  desired asymptotic profile. Another family of sub/super-solutions will be needed to bridge the gap. We illustrate this point in Subsection~\ref{strag} when outlining the ideas in the proof of Theorem~\ref{pro2} in Section~\ref{wabb}.        

\vspace{0.1cm}
{\bf Notation and definitions.} For every $x\in \mathbb R^N\setminus \{0\}$, $z\in [0,\infty)$ and 
$\boldsymbol{\xi}\in \mathbb R^N$, we define 
$$  G_\tau(x,z,\boldsymbol{\xi}):=\frac{z^{\tau} \, |\boldsymbol{\xi}|^{1-\tau}}{|x|^{\tau+1}}\ \mbox{if } 0<\tau<1\quad \mbox{and}\quad 
G_\tau(x,z,\boldsymbol{\xi}):=\frac{|\boldsymbol{\xi}|}{|x|}\ \mbox{if } \tau=0.
$$ 
We denote by $C_c^1(\Omega\setminus\{0\})$ the space of $C^1(\Omega\setminus \{0\})$-functions with compact support in $\Omega\setminus \{0\}$. 
Let  $ \mathbb L_{\rho}[u]$ be the following linear operator
\begin{equation} \label{adel} \mathbb L_{\rho}[u]:=\Delta u-\left(N-2+2\rho\right) \frac{x\cdot \nabla u}{|x|^2}.
\end{equation}

\begin{defi} {\rm 
By a  {\em solution} of \eqref{e1} we mean any {\em non-negative} function $u\in C^1(\Omega\setminus\{0\})$ that satisfies \eqref{e1} in $\Omega\setminus \{0\}$ in the sense of distributions, that is, 
\begin{equation*}\label{def1}
\int_\Omega \nabla u\cdot \nabla \varphi\,dx +(N-2+2\rho) \int_\Omega \frac{x\cdot \nabla u}{|x|^2}\varphi\,dx
-\lambda\int_\Omega G_\tau(x,u,\nabla u)\,\varphi\,dx+\int_\Omega |x|^{\theta}u^q\varphi\,dx=0
	\end{equation*}
for all  $ \varphi \in C_c^1(\Omega\setminus \{0\})$. 
Similarly, we define a  sub-solution (respectively, super-solution) of \eqref{e1} by replacing the equal sign in the above identity by ``$\leq$" (respectively, ``$\geq$"), where the resulting inequality is to be applied to all  non-negative functions $ \varphi \in C_c^1(\Omega\setminus \{0\})$.}
\end{defi}

\subsection{Strategy for proving our main results} \label{strag}
We summarise the main ideas in the proofs of Theorems~\ref{thi} and \ref{globo}.

In the first part of this paper, we classify all the behaviours near zero for the positive solutions of \eqref{e1}, assuming $\theta<-2$ and $\rho,\lambda\in \mathbb R$. In Section~\ref{sha1}, we only assume that $f_\rho(\beta)>0$ and using suitable sub/super-solutions (Lemma~\ref{seco}) and the {\em a priori} estimates in Lemma~\ref{api} (see Appendix~\ref{sectiune2}), we prove in Theorem~\ref{pro1} that  every positive sub-solution $u$ of \eqref{e1} satisfies
\begin{equation} \label{su1}
\limsup_{|x|\to 0} |x|^\beta u(x)\leq \Lambda,\quad \mbox{where } \Lambda=\left[f_{\rho}(\beta)\right]^{\frac{1}{q-1}}.
\end{equation} 
In turn, given any positive super-solution $u$ of \eqref{e1} such that 
\begin{equation} \label{nopi}
\mbox{for every } \eta>0, \ \mbox{we have } \lim_{|x|\to 0}  |x|^{\beta-\eta} u(x)=\infty, \end{equation}
we establish the optimal lower bound 
\begin{equation} \label{sunn2}
 \liminf_{|x|\to 0} |x|^\beta u(x)\geq \Lambda. \end{equation}

In Section~\ref{wabb}, we show that when $\theta<-2$ and $f_\rho(\beta)>0$, excluding $\beta\in (\varpi_2(\rho),0)$ in Case $[N_2]$, we always have \eqref{nopi} for every positive super-solution $u$ of \eqref{e1} (see Theorem~\ref{pro2}). The assertion of $(P_1)$ in Theorem~\ref{thi} follows from 
Theorems~\ref{pro1} and \ref{pro2}. We split the proof of Theorem~\ref{pro2} into two parts. In the first part, we impose an additional assumption, $\beta\leq -\rho$,
when $\lambda>0$ and $\tau\in (0,1)$. For any $\delta>0$ small and $\alpha>0$, we construct a family $\{w_{\delta,\alpha}\}_{\delta>0}$ of positive radial sub-solutions of \eqref{e1} on annuli 
$\{\delta<|x|<r_1(\alpha)\}$ such that $\lim_{|x|\to 0} |x|^{\beta-\alpha} (\lim_{\delta\searrow 0} w_{\delta,\alpha}(|x|))=c>0$, 
 $w_{\delta,\alpha}(|x|)$ is zero for $|x|=\delta$ and is dominated by $u(x)$ for $|x|=r_1(\alpha)$ (see Lemma~\ref{lem2}). By the comparison principle, $u\geq w_{\delta,\alpha}$ in $\{\delta<|x|<r_1(\alpha) \}$. Letting $\delta\to 0$ and $|x|\to 0$, we get \eqref{nopi}. 
The additional restriction imposed in the first part will be removed in the second part, see Section~\ref{kvz}. But unlike the first part of the proof of Theorem~\ref{pro2}, it seems very difficult to obtain \eqref{nopi} with just one family of sub-solutions. This difficulty appears specifically to $\lambda>0$ and $\tau\in (0,1)$. We propose a new approach, which we adapt also in later sections to deal with similar challenges for other ranges of our parameters. 
We carefully design an iterative scheme, the starting point of which is to establish that (see Lemma~\ref{lema44})
$$ \mbox{for every } \eta>0\ \mbox{ small } \lim_{|x|\to 0} |x|^{-\rho+\eta} u(x)=\infty. $$  This first step requires a modification of the sub-solutions 
$w_{\delta,\alpha}$.  The purpose of the iterative scheme is to incrementally improve the lower bounds for the super-solution $u$, ensuring also that the process leads to \eqref{nopi} after a finite number of steps, say $k_*$. To determine $k_*$, we let $M_1=\rho$ and 
construct a suitable decreasing sequence $\{M_k\}_{k\geq 1}$ that converges to  $0$ if $\rho<\Upsilon$ and to $-\varpi_1(\rho)$ (which is less than $-\beta$) when $\rho\geq \Upsilon$ (in Case $[N_2]$). 
So, there exists $k_*\geq 1$ large such that 
$M_{k_{*}+1}\leq -\beta<M_{k_*}$.  
At each stage $j=1,\ldots, k_*$ of the iteration, we devise a new family of positive radial sub-solutions $\{v_{j,\delta}\}_{\delta>0}$ of \eqref{e1} 
for $0<|x|<r_*^{1/\alpha}$ such that $v_{j,\delta}(|x|)$ is dominated by $u(x)$ for $|x|=r_*^{1/\alpha}$, $\lim_{r\to 0^+} r^{-m_j} v_{j,\delta}(r)\in (0,\infty)$ and
$\lim_{r\to 0} r^{m_{j+1}+\nu} (\lim_{\delta\to 0} v_{j,\delta}(r)) =c>0$ for sufficiently small $\nu>0$.  
The definition of $v_{j,\delta}$ involves 
$m_j$ and $m_{j+1}$, which help us measure the jump from one step of the iteration to the next. We define $m_j$ to be very close to $M_j$ (for $1\leq j\leq k_*$) with $m_1=\rho-\varepsilon$ and $m_{k_{*}+1}=-\beta \left(1+\varepsilon^{1/k_*}\right)$ for small enough $\varepsilon>0$. 
By the choice of $m_1<\rho$ and the starting point (Lemma~\ref{lema44}), we have $\lim_{|x|\to 0} u(x)/v_{1,\delta}(|x|)=\infty$. By the comparison principle, 
$u(x)\geq v_{1,\delta}(x)$ for all $0<|x|\leq r_*^{1/\alpha}$. Letting $\delta\to 0$ and $|x|\to 0$, we improve the estimate to 
$\liminf_{|x|\to 0} |x|^{-m_2+\nu} u(x)\geq c>0$. This ensures that $u(x)$ dominates $v_{2,\delta}(|x|)$ for $|x|>0$ small and, by the comparison principle, for all 
$0\leq |x|\leq  r_*^{1/\alpha}$. By induction on $j=1,\ldots, k_*$ and the definition of $m_{k_*+1}$, we arrive at \eqref{nopi}. 

In Section~\ref{oaz}, we show that whenever $\beta\in (\varpi_1(\rho),0)$ in Case $[N_1]$ or Case $[N_2]$, then every positive super-solution of \eqref{e1} satisfies
\begin{equation} \label{nilo}  \liminf_{|x|\to 0} |x|^{\varpi_1(\rho)} u(x)>0.
\end{equation} 
We first prove a weaker version of \eqref{nilo}: for every $\eta>0$, we have $\lim_{|x|\to 0} |x|^{\varpi_1(\rho)-\eta} u(x)=\infty$. 
When $\Upsilon\leq \rho<\lambda^{1/(\tau+1)} $ in Case $[N_2]$, then we conclude the latter property by modifying the proof of Lemma~\ref{lema44} (see Proposition~\ref{gain}). 
In turn, in Case~$[N_1]$ or if $\rho\geq \lambda^{1/(\tau+1)}$ in Case~$[N_2]$, then we fix $\mu \in (0,\eta)$ and in \eqref{eqo1} we construct  
positive radial sub-solutions  $\{u_{\delta}\}_{\delta>0}$ of \eqref{e1} in 
$\{\delta<|x|<1/e\}$, which vanish on $|x|=\delta$ and are dominated by $u(x)$ on $|x|=r_0$ for $r_0\in (0,1/e)$, where $B_{r_0}(0)\subset \subset \Omega$. 
Since $\lim_{r\to 0^+} r^{\varpi_1(\rho)} |\log r|^{\mu}(\lim_{\delta\to 0} u_\delta (r))=c>0$, we obtain the desired weaker property via the comparison principle (see Proposition~\ref{nwp1}).  Finally, we fix $\varepsilon>0$ small and reach \eqref{nilo} by comparing $u$ with positive radial sub-solutions  $\{z_{\delta,\varepsilon}\}_{\delta\in (0,1)}$ of \eqref{e1} in $B_{r_\varepsilon}(0)\setminus \{0\}$ satisfying the following: 
$\lim_{r\to 0} r^{\varpi_1(\rho)-\varepsilon \delta} z_{\delta,\varepsilon}(r)=c>0$, $u(x)\geq z_{\delta,\varepsilon}(|x|)$ on $|x|=r_\varepsilon$ (for every $\delta\in (0,1)$) 
and $\lim_{r\to 0^+} r^{\varpi_1(\rho)} (\lim_{\delta\to 0} z_{\delta,\varepsilon}(r))=c>0$ (see Section~\ref{sect62}).  
  
In Section~\ref{ditto}, we give sufficient conditions for a positive radial super-solution $\Phi$ of $\mathbb L_{\rho,\lambda,\tau}[\Phi]=0$ for $|x|>0$ small to be an asymptotic model near zero for a given positive solution $u$ of \eqref{e1}. 

In Section~\ref{errv}, we prove the assertion (N) in Theorem~\ref{thi}: if $\beta\in (\varpi_1(\rho),0)$ in Case $[N_1]$ or
$\beta\in (\varpi_1(\rho),\varpi_2(\rho)]$ in Case $[N_2]$ with $\rho>\Upsilon$, then for every positive solution $u$ of \eqref{e1}, there exists 
$a\in \mathbb R_+$ so that  
$ \lim_{|x|\to 0} |x|^{\varpi_1(\rho)} u(x)=a$ (see  Theorem~\ref{axiv1}).  Since \eqref{nilo} holds, by Theorem~\ref{urz1}, it remains to prove 
$ \limsup_{|x|\to 0} |x|^{\varpi_1(\rho)} u(x)=a<\infty$.  We use $\lim_{|x|\to 0} |x|^\beta u(x)=0$ (Corollary~\ref{zerop}) and a comparison with suitable super-solutions (see Proposition~\ref{axv1}).   

In Section~\ref{sept}, we establish the assertions of ($Z_1$) and $(Z_2)$ in Theorem~\ref{thi}. 
We adapt the ideas in the proof of Theorem~\ref{zic1} to show that every positive super-solution $u$ of \eqref{e1} satisfies 
$\liminf_{|x|\to 0} |x|^{\varpi_1(\rho)} |\log |x||^p u(x)>0$, where $p=1/(q-1)$ (respectively, $p=2/(q-1)$) in the settings of $(Z_1)$ (respectively, $(Z_2)$) in Theorem~\ref{thi}.  
By comparison with adequate super-solutions, we derive that $\limsup_{|x|\to 0} |x|^{\varpi_1(\rho)} |\log |x||^p u(x)<\infty$ for every positive sub-solution $u$ of \eqref{e1}. We obtain the optimal constants by following the same strategy as in Theorem~\ref{pro1}, adapting the construction of sub-super-solutions to our case.   

Section~\ref{sect9} is devoted to the proof of the trichotomy stated in $(P_0)$ of Theorem~\ref{thi}. 
We compare a positive solution $u$ of \eqref{e1} with $\Phi_{\varpi_2(\rho)}$ via $b=\limsup_{|x|\to 0} u(x)/\Phi_{\varpi_2(\rho)}(|x|)$. 
If $b\in \mathbb R_+$, then the alternative (ii) in ($P_0$) follows from Theorem~\ref{urz1}. We show that 
the alternative (i) and (iii) in ($P_0$) corresponds to $b=0$ and $b=\infty$, respectively, see Theorem~\ref{ama1}.  
We need different arguments when $\rho=\Upsilon$ (see Sections~\ref{secla2} and \ref{secla4}) compared with $\rho>\Upsilon$ (see Sections~\ref{secla1} and \ref{secla3}) 
given the different expression of $\Phi_{\varpi_2(\rho)}$.
We illustrate the ideas for $\rho>\Upsilon$. If $b=0$, then we show that for every $\varepsilon>0$, we have $\lim_{|x|\to 0} |x|^{\varpi_2(\rho)-\varepsilon} u(x)=0$
by a comparison with suitable super-solutions of \eqref{e1}. This is the starting point for an iterative scheme leading to 
$\lim_{|x|\to 0} |x|^{\varpi_1(\rho)+\eta} u(x)=0$ for $\eta>0$ arbitrary. Finally, by another construction of radial super-solutions, we 
get $\limsup_{|x|\to 0} |x|^{\varpi_1(\rho)} u(x)=a<\infty$. We conclude the alternative (i) in ($P_0$) of Theorem~\ref{thi} from Theorem~\ref{urz1}.  
If $b=\infty$ (for $\rho>\Upsilon$), then by a comparison with suitable radial sub-solutions of \eqref{e1}, we show that for every $\varepsilon_0>0$ small, $\lim_{|x|\to 0} |x|^{\varpi_2(\rho) +\varepsilon_0} u(x)=\infty$. This is the first step in an iterative scheme to incrementally improve the lower bound to attain 
\eqref{nopi}. We set up the iterative scheme in the spirit of Proposition~\ref{lad}, although we need to alter the definition of $\{m_j\}$. 
From \eqref{su1} and \eqref{sunn2}, we obtain the alternative (iii) in ($P_0$) of Theorem~\ref{thi}. 

In the second part of the paper, we prove the optimal condition, $f_{\rho}(\beta)> 0$, for the existence of positive solutions of \eqref{e1} in $\mathbb R^N\setminus \{0\}$ and fully describe the set of all such solutions for $\theta<-2$.  

In Section~\ref{nine}, we show that if $f_\rho(\beta)\leq 0$, then there are no positive solutions of \eqref{e1} in $\mathbb R^N\setminus \{0\}$. 
Indeed, assuming $u$ to be such a solution, we prove that $\lim_{|x|\to 0} |x|^\beta u(x)=0$, which by the modified Kelvin transform, yields $\lim_{|x|\to \infty} |x|^\beta u(x)=0$. This leads to a contradiction.  

To complete the proof of Theorem~\ref{globo}, we need two major ingredients, apart from Theorem~\ref{thi}. First, we reveal all the  asymptotic profiles {\em at infinity} that can coexist with those at zero for the positive solutions of \eqref{e1} in $\mathbb R^N\setminus \{0\}$ other than $U_{\rho,\beta}$. 
This is done in Theorem~\ref{rod1c} in 
Section~\ref{amod}. 
The behaviours at infinity are many and split according to  various regimes of the parameters. The strategy for proving Theorem~\ref{rod1c} is summarised in Section~\ref{trap1}.  
Second, for each of the possibilities presented in Theorem~\ref{rod1c}, we show that there exists a positive (increasing) radial solution of \eqref{e1} in $\mathbb R^N\setminus \{0\}$ which, by the comparison principle, turns out to be unique. These existence and uniqueness results, also stated in (i) and (ii) of Theorem~\ref{globo}, are proved in  Theorems~\ref{exol1}, \ref{exol2} and \ref{exol3} in Sections~\ref{alt1}, \ref{alt2} and \ref{alt3}, respectively. 

In Appendix \ref{Kel}, we show the effect that the modified Kelvin transform has on a positive solution $u$  
of \eqref{e1} in $B_1(0)\setminus \{0\}$, leading to Corollary~\ref{glob22} being a consequence of Theorem~\ref{globo}. For the reader's convenience, we give 
two strong maximum principles (Lemmas~\ref{add} and \ref{adol}) and two comparison principles (Lemmas~\ref{co1} and \ref{co2}) often used in the paper. 
{\em By the comparison principle, we mean either Lemma~\ref{co1} or Lemma~\ref{co2} as explained in Remark~\ref{reka6} in Appendix~\ref{Kel}.} 

In Appendix~\ref{sectiune2} we collect a few results such as the {\em a priori} estimates in Lemma~\ref{api}, gradient estimates and a spherical Harnack-type inequality in Proposition~\ref{aurr}, the proofs of which are obtained by combining well-known techniques.  
Such results have previously appeared in related contexts (see, for example, \cite{CC2015,Cmem,CD2010,FV}), 
although we need to 
adjust their proofs because of the introduction of the term $\lambda u^\tau |\nabla u|^{1-\tau}/|x|^{1+\tau}$. We also include other auxiliary results invoked in the paper 
 (Corollaries~\ref{peree} amd \ref{pere}, Lemma~\ref{mar11}), which rely on ideas similar to those in \cite{Cmem}. 
 
 In Appendix~\ref{sec-sum}, we capture the classification in Theorem~\ref{thi} for each Case $[N_j]$ with $j=0,1,2$ (see Table~\ref{tabel}).
 The findings of our Theorem~\ref{globo} and Corollary~\ref{glob22} are summed up in Table~\ref{table11}. 

\vspace{0.1cm}
{\em Acknowledgements.} The research of the first two authors was supported by the Australian Research Council (ARC) under grant DP220101816. The second author was also partially supported by the ARC grant DP190102948. 

\part{Proof of Theorem~\ref{thi}}


\section{Sharp upper bounds when $f_{\rho}(\beta)>0$} \label{sha1}

Our main result in this section is as follows. 

\begin{theorem} \label{pro1} Let \eqref{e2} hold and 
$\rho,\lambda,\theta\in \mathbb R$ be such that
 $f_{\rho}(\beta)>0$. We obtain the following. 
\begin{itemize}
\item[(a)] Every positive sub-solution $u$ of \eqref{e1} satisfies \eqref{su1}. 
\item[(b)] If $u$ is a positive super-solution of \eqref{e1} such that \eqref{nopi} holds, 
then $u$ satisfies the optimal (lower) inequality \eqref{sunn2}.  
\item[(c)] For any positive solution $u$ of \eqref{e1}, we have the following dichotomy: 
\begin{equation} \label{dual}  \mbox{either } \lim_{|x|\to 0} |x|^\beta u(x)=0\quad \mbox{or } \lim_{|x|\to 0} \frac{u(x)}{U_{\rho,\beta}(x)}=1.
\end{equation}
\end{itemize}
\end{theorem}

\begin{proof} We prove the assertions of (a) and (b) simultaneously. 
Suppose that $u$ is a positive sub-solution of \eqref{e1}  in relation to (a) 
(respectively, a positive super-solution of \eqref{e1} satisfying \eqref{nopi} in relation to (b)). Let $r_0\in (0,1)$ be small such that $B_{r_0}(0)\subset\subset \Omega$. 

For (a), to obtain the optimal inequality \eqref{su1} (respectively, \eqref{sunn2} for (b)), we fix $C>\Lambda$ (respectively, $C\in (0,\Lambda)$) arbitrary and for $\eta_C>0$ chosen small (depending on $C$), we construct a one-parameter family $\{P^+_\eta\}_{\eta\in (0,\eta_C)}$ (respectively, 
$\{P^-_\eta\}_{\eta\in (0,\eta_C)}$) of positive radial super-solutions (respectively, sub-solutions) of \eqref{e1} in $B_1(0)\setminus \{0\}$
satisfying: 

\begin{enumerate}
\item[(i)] $(P^\pm_\eta)'(r)\not=0$ for every $r\in (0,1)$;

\item[(ii)] $u(x)\leq P^+_\eta(|x|)$ (respectively, $u(x)\geq P_\eta^-(|x|)$) for all $x\in \partial B_{r_0}(0)$;

\item[(iii)] $u(x)\leq P^+_\eta(|x|)$ (respectively, $u(x)\geq P^-_\eta(|x|)$) for every $|x|>0$ small enough; 

\item[(iv)] $\lim_{|x|\to 0} |x|^\beta (\lim_{\eta\searrow 0} P^\pm_\eta(|x|))=C$. 
\end{enumerate}

The properties (i)--(iii) will be satisfied for each $\eta\in (0,\eta_C)$. 
Then, the comparison principle (see Remark~\ref{reka6} in Appendix~\ref{Kel}) gives that
$ u(x) \leq P^+_{\eta}(|x|)$ (respectively, $ u(x) \geq P^-_{\eta}(|x|)$) for all $ x\in B_{r_0}(0)\setminus \{0\}$. Since $\eta\in (0,\eta_C)$ was arbitrary, by letting $\eta\to 0$, we arrive at 
$$ u(x)\leq \lim_{\eta\searrow 0} P^+_\eta(|x|)\quad \mbox{(respectively, } u(x)\geq \lim_{\eta\searrow 0} P^-_\eta(|x|))
\quad\mbox{for every } 0<|x|\leq r_0.$$ In view of (iv), we have
$\limsup_{|x|\to 0} |x|^\beta u(x)\leq C$ (respectively, $ \liminf_{|x|\to 0} |x|^\beta u(x)\geq C)$.   
By letting $C\searrow \Lambda$ (respectively, $C\nearrow \Lambda$), we conclude \eqref{su1} (respectively, \eqref{sunn2}). 
  
The crucial ingredient in the proof is the construction of $\{P^\pm_\eta\}_{\eta\in (0,\eta_C)}$ with the properties (i)--(iv). 
 For the motivation behind 
such a construction, see
 \cite{MF1}*{Section 5.1}, where $\tau=1$ and $\rho=(2-N)/2$. Here, we introduce a simpler family of sub-super-solutions than in \cite{MF1}. 

\vspace{0.2cm}
{\bf Construction of a family $\{P^\pm_\eta\}_{\eta\in (0,\eta_C)}$ with the properties (i)--(iv).}

We fix $\nu>0$ small (depending on $u$ and $r_0$) such that 
$$   \max_{|x|=r_0} u(x) \leq \Lambda r_0^{-\beta}(1+r_0/\nu)\quad 
\mbox{(respectively,}  \min_{|x|=r_0} u(x) \geq \Lambda r_0^{-\beta}(1+r_0/\nu)^{-1}).
$$ 

 Let $\eta$ and $\alpha$ be positive constants. We define 
\begin{equation} \label{pf} P^\pm_{\eta,\alpha,\nu}(r):=  C \,r^{-\beta\mp\eta} \left(1+\frac{r^\alpha}{\sqrt{\alpha}\,\nu}\right)^{\pm 1}\quad \mbox{for every } r\in (0, 1],
\end{equation} 
where $C>\Lambda$ for $P^+_{\eta,\alpha,\nu}(r)$ and $C\in (0,\Lambda)$ for $P^-_{\eta,\alpha,\nu}(r)$. 
Then, by Lemma~\ref{seco}, there exist $\eta_C,\alpha_C>0$ such that by fixing $\alpha\in (0,\min\{1,\alpha_C\})$ and letting $\eta\in (0,\eta_C)$ arbitrary, 
we have  
\begin{equation} \label{sumad}  \begin{aligned} 
&  \mathbb L_{\rho,\lambda,\tau}[P^+ _{\eta,\alpha,\nu}(|x|)] \leq |x|^\theta [P^+ _{\eta,\alpha,\nu}(|x|)]^{q} \ \  \mbox{in } B_1(0)\setminus \{0\},\\
&  \mathbb L_{\rho,\lambda,\tau}[P^- _{\eta,\alpha,\nu}(|x|)] \geq |x|^\theta [P^- _{\eta,\alpha,\nu}(|x|)]^{q} \ \  \mbox{in } B_1(0)\setminus \{0\}. 
 \end{aligned}
\end{equation}

Since $\alpha$ and $\nu$ are fixed, we simply write $P^\pm_\eta$ instead of $P^\pm_{\eta,\alpha,\nu}$ in \eqref{pf}. 
Then, $P^+_{\eta}$ (respectively, $P^-_{\eta}$) is a positive radial super-solution (respectively, sub-solution) of \eqref{e1} in $B_1(0)\setminus \{0\}$ satisfying properties (i), (ii) and (iv) (at the beginning of the proof of Theorem~\ref{pro1}). Finally, we derive (iii) from $ \lim_{|x|\to 0} u(x)/P^+_{\eta}(|x|)=0$, which follows from 
Lemma~\ref{api} in Appendix~\ref{sectiune2}. On the other hand, $P_\eta^-$ satisfies (iii) because of the assumption \eqref{nopi}, which yields 
$ \lim_{|x|\to 0} P^-_{\eta}(|x|)/u(x)=0$. 

\begin{lemma}[Sub-super-solutions] \label{seco} Let \eqref{e2} hold and 
$f_{\rho}(\beta)>0$. Set $\Lambda=[f_\rho(\beta)]^{1/(q-1)}$. 
Then, for  every $C>\Lambda$ for $P^+_{\eta,\alpha,\nu}(r)$ and every $C\in (0,\Lambda)$ for $P^-_{\eta,\alpha,\nu}(r)$,  there exist small $\eta_C,\alpha_C\in (0,|\beta|/3)$ such that for every $\eta\in (0,\eta_C)$,  $\alpha\in (0,\alpha_C)$ and $\nu>0$ arbitrary, we have 
$(P^\pm_{\eta,\alpha,\nu})'(r)\not= 0$ for all $r\in (0,1)$ and \eqref{sumad} holds. 
\end{lemma}

\begin{proof}
The assumption $f_{\rho}(\beta)>0$ implies that $\beta\not=0$. 
	We fix $C>\Lambda$ for $P^+_{\eta,\alpha,\nu}$ and $C\in (0,\Lambda)$ for $P^-_{\eta,\alpha,\nu}$, where  $\nu>0$ is arbitrary.	
	In what follows, we let $r\in (0,1)$, while $\alpha,\eta\in (0,|\beta|/3)$ are small enough (depending on $C$). 
	We define 
	$\psi_{\eta,\alpha,\nu}(r)$, $V_{\eta,\alpha,\nu}(r)$ and $Z_{\eta,\alpha,\nu}^\pm(r)$ by  
	$$\begin{aligned}
	& \psi_{\eta,\alpha,\nu}(r)=\left(1+\frac{r^\alpha}{\sqrt{\alpha} \,\nu} \right)^{-1}\in (0,1) ,\quad V_{\eta,\alpha,\nu}(r)=-\left[ \eta-\alpha+ \alpha\,
	\psi_{\eta,\alpha,\nu}(r)\right]/\beta\\
	&  Z_{\eta,\alpha,\nu}^\pm(r)=(\eta-\alpha) \left[\eta-\alpha \pm 2\left(\beta+\rho\right) \right] \pm 2\alpha
	\left[\beta+\rho\pm\left(\eta-\alpha\right) \right] \psi_{\eta,\alpha}(r)
	+\alpha^2  \psi^2_{\eta,\alpha}(r) 
	 \left( 1\pm \frac{r^\alpha}{\sqrt{\alpha} 
	\,\nu}\right). 
	\end{aligned} $$ 	
	For simplicity of notation, we drop the subscript $\nu$ from $P^\pm_{\eta,\alpha,\nu}$, $\psi_{\eta,\alpha,\nu}$, $V_{\eta,\alpha,\nu}$ and $Z_{\eta,\alpha,\nu}^\pm$. 
	Note that
	$  \lim_{(\eta,\alpha)\to(0,0)}Z_{\eta,\alpha}^\pm(r)=\lim_{(\eta,\alpha)\to(0,0)}V_{\eta,\alpha}(r)=0$ uniformly in $r\in (0,1)$. 	
	Using that 
	$$ \beta^2+2\rho\beta+\lambda |\beta|^{1-\tau}
= f_{\rho} (\beta),$$
we choose $\eta_C\in (0,|\beta|/3)$ and 
$\alpha_C\in (0,|\beta|/3)$ small to ensure that  $ |V_{\eta,\alpha}(r)|<1/2$ and 
\begin{equation} \label{laus} \begin{aligned}
&  \beta^2+2\rho\beta +Z^+_{\eta,\alpha}(r)+\lambda|\beta|^{1-\tau} |1+ V_{\eta,\alpha}(r)|^{1-\tau} \leq  C^{q-1} \ \mbox{in the case of } P^{+}_{\eta,\alpha}\\
&  \beta^2+2\rho\beta +Z^-_{\eta,\alpha}(r)+\lambda|\beta|^{1-\tau} |1- V_{\eta,\alpha}(r)|^{1-\tau}\geq C^{q-1} \ \mbox{in the case of } P^{-}_{\eta,\alpha}
\end{aligned}\end{equation}
for every $r\in (0,1)$, $\eta\in (0,\eta_C)$ and $\alpha\in (0,\alpha_C)$. 
	Using $\mathbb L_\rho[\cdot]$ introduced in \eqref{adel}, we arrive at 
	\begin{equation} \label{la}
	\begin{aligned} 
	& (P_{\eta,\alpha}^\pm)'(r)
	=-\beta \,\frac{P^\pm_{\eta,\alpha}(r)}{r} \left( 1\pm V_{\eta,\alpha}(r) \right)  \not=0,
	\\
	& \mathbb L_\rho[ P_{\eta,\alpha}^\pm(r)]= (P^\pm_{\eta,\alpha})''(r)+\left(1-2\rho\right) \frac{(P^\pm_{\eta,\alpha})'(r)}{r}=
	\frac{P_{\eta,\alpha}^\pm(r)}{r^2}\left[ \beta^2+2\rho\beta+Z^\pm_{\eta,\alpha}(r) \right]. 
	\end{aligned}
	\end{equation}  
	It follows that 
\begin{equation} \label{ford} \begin{aligned} 
 \mathbb L_{\rho,\lambda,\tau}[P^\pm_{\eta,\alpha}(r)] & =
  \mathbb L_\rho[ P_{\eta,\alpha}^\pm(r)]+\lambda \frac{(P^\pm_{\eta,\alpha}(r))^\tau |(P^\pm_{\eta,\alpha})'(r)|^{1-\tau} }
	{r^{1+\tau}}
  \\
 &=
\frac{P^\pm_{\eta,\alpha}(r)}{r^{2}} \left[ \beta^2+2\rho\beta +Z^\pm_{\eta,\alpha}(r)+\lambda|\beta|^{1-\tau} |1\pm V_{\eta,\alpha}(r)|^{1-\tau}
\right]. 
\end{aligned} \end{equation}
We conclude \eqref{sumad} from \eqref{ford} using 
\eqref{pf} and \eqref{laus}. The proof of Lemma~\ref{seco} is complete. 		
\end{proof}

(c) For any positive solution  $u$ of \eqref{e1}, using Lemma~\ref{api} in Appendix~\ref{sectiune2}, we have 
\begin{equation} \label{nuiz} \mbox{either } \lim_{|x|\to 0} |x|^\beta u(x)=0\quad \mbox{or} \quad  
\limsup_{|x|\to 0}|x|^\beta u(x)\in (0,\infty) .\end{equation}
Since $f_\rho(\beta)>0$, we get that, for every $c\in (0,\Lambda]$, the function $c |x|^{-\beta}$ is a positive radial sub-solution of \eqref{e1} in $\mathbb R^N\setminus \{0\}$. 
Hence, in the latter case of \eqref{nuiz}, by Corollary~\ref{peree} in Appendix~\ref{sectiune2}, 
we have $\liminf_{|x|\to 0} |x|^\beta u(x)>0$ so that 
 \eqref{nopi} holds. 
By (a) and (b) above, we find that $\lim_{|x|\to 0} |x|^\beta u(x)=\Lambda$. 
This ends the proof of Theorem~\ref{pro1}. 
\end{proof}

\begin{cor} \label{zerop} Let \eqref{e2} hold and $f_{\rho}(\beta)\leq 0$ with $\beta\not=0$. Then, every  
positive sub-solution $u$ of \eqref{e1} satisfies $\lim_{|x|\to 0} |x|^\beta u(x)=0$. 
\end{cor}

\begin{proof} Let $C>0$ be arbitrary. For $\eta,\alpha,\nu>0$, we define $P^+_{\eta,\alpha,\nu}$ as in \eqref{pf}. Then, since $f_{\rho}(\beta)\leq 0$ and $\beta\not=0$, the proof of Lemma~\ref{seco} gives the existence of 
small constants $\eta_C,\alpha_C>0$ such that \eqref{laus} holds for every $r\in (0,1)$, $\eta\in (0,\eta_C)$, $\alpha\in (0,\alpha_C)$ and $\nu>0$.
Hence, the function
$P^+_{\eta,\alpha,\nu}$ is a positive radial super-solution of \eqref{e1} in $B_1(0)\setminus \{0\}$. 
The proof of Theorem~\ref{pro1} (a) yields that $\limsup_{|x|\to 0} |x|^\beta u(x)\leq C$. By letting $C\searrow 0$, we conclude 
that $\lim_{|x|\to 0} |x|^\beta u(x)=0$. 
\end{proof}

\begin{rem} {\rm If $\beta=0$ in Corollary~\ref{zerop}, then we later prove that $\lim_{|x|\to 0} u(x)=0$ for every positive solution $u$ of \eqref{e1}, see Proposition~\ref{pzerop}.}
\end{rem}


\section{Proof of Assertion $(P_1)$ in Theorem~\ref{thi}}  \label{wabb}

From  Theorem~\ref{pro1} and Theorem~\ref{pro2} below, we conclude the claim of $(P_1)$ in Theorem~\ref{thi}. 

\begin{theorem} \label{pro2}
Let \eqref{e2} hold, $\beta<0$ and $f_{\rho}(\beta)>0$ in Case $[N_j]$ with $j\in \{0,1,2\}$. 
In addition, in Case $[N_2]$, we assume 
$\beta<\varpi_1$.  
Let $u$ be any positive super-solution of \eqref{e1}. Then, 
 \begin{equation} \label{lowe}
 \mbox{for every } \eta>0, \ \mbox{we have }
\lim_{|x|\to 0} |x|^{\beta-\eta} u(x)=\infty. 
\end{equation}
\end{theorem}

\subsection{Proof of Theorem~\ref{pro2}: Part I} \label{sect401}
Let $u$ be a positive super-solution of  \eqref{e1} in the framework of Theorem~\ref{pro2}. 
If Case $[N_0](b)$ or Case $[N_2]$ holds, then we 
assume, 
in addition, that 
\begin{equation} \label{extar} \beta\leq -\rho. \end{equation}

Let $\alpha>0$ be fixed and 
$r_*\in (0,\log 2)$ be small such that $B_{r_*^{1/\alpha}}(0)\subset \subset \Omega$.
To obtain \eqref{lowe}, we compare $u$ with a suitable family of sub-solutions $\{w_{\delta,\alpha}\}_{\delta\in (0,r_*^{1/\alpha})}$ of \eqref{e1} in $\delta<|x|<r_*^{1/\alpha}$. 

For $c\in (0,\Lambda]$ fixed and every $\delta\in (0,r_*^{1/\alpha})$, 
we define 
\begin{equation} \label{e8} w_{\delta,\alpha}(r)= c\, r^{-\beta} \left( e^{r^\alpha}-e^{\delta^\alpha}\right) \quad \mbox{for all } r\in [\delta, r_*^{1/\alpha}]. 
\end{equation}

By Lemma~\ref{lem2}, $w_{\delta,\alpha}$ in \eqref{e8} is a positive sub-solution of \eqref{e1} in $\delta<|x|<r_*^{1/\alpha}$, that is, 
\begin{equation} \label{e7}
 \mathbb L_{\rho,\lambda,\tau} [w_{\delta,\alpha}(r)] \geq r^\theta (w_{\delta,\alpha}(r))^q   \quad \mbox{for every } \delta<r<r_*^{1/\alpha}.
\end{equation}
Moreover, using that $\beta<0$, we find that $w_{\delta,\alpha}'(r)>0$ for every $r\in (\delta,r_*^{1/\alpha})$ since
\begin{equation} \label{gao}  w_{\delta,\alpha}'(r)=-\beta r^{-1} \,w_{\delta,\alpha}(r) +c\,\alpha\, r^{-\beta+\alpha-1} e^{r^\alpha}.
\end{equation}


We diminish $c=c(\alpha,u)>0$ in \eqref{e8}, for example, $c=\min\, \{\Lambda, \min_{|x|=r_*^{1/\alpha}} u(x)\}$ to ensure that $u\geq w_{\delta,\alpha}$ on  $\partial B_{r_*^{1/\alpha}}(0)$. Then, since $w_{\delta,\alpha}=0$ on $\partial B_{\delta}(0)$, by the comparison principle, we get 
$u(x)\geq w_{\delta,\alpha}(|x|)$ for every $ \delta<|x|<r_*^{1/\alpha}$. By letting $\delta\to 0$ and then $|x|\to 0$, we get 
\begin{equation} \label{inop} \liminf_{|x|\to 0} |x|^{\beta-\alpha} u(x)\geq c>0.\end{equation}
Since $\alpha>0$ is arbitrary, from \eqref{inop} we conclude  the proof of Part I of Theorem~\ref{pro2}.

\begin{lemma}
\label{lem2} 
For any $c\in (0,\Lambda]$, $\alpha>0$ and $\delta\in (0, r_*^{1/\alpha})$, 
letting $w_{\delta,\alpha}$ as in \eqref{e8}, we have  \eqref{e7}.  
\end{lemma}

\begin{proof}
Let $r\in (\delta,r_*^{1/\alpha})$ be arbitrary. We will write $w_\delta$ instead of $w_{\delta,\alpha}$. 
Since $r_*\in (0,\log 2)$ and $c\in (0,\Lambda]$, using that $q>1$ and $\beta=(\theta+2)/(q-1)$, we have
\begin{equation} \label{haoo} r^{\theta+2} w_\delta^{q-1}(r)=c^{q-1} \left(e^{r^\alpha}-e^{\delta^\alpha}\right)^{q-1}\leq \Lambda^{q-1}=f_{\rho}(\beta)=\beta\left(\beta+2\rho\right)+\lambda |\beta|^{1-\tau}.
\end{equation}
Hence, we conclude \eqref{e7} by showing that 
\begin{equation} \label{asa1} \mathbb L_{\rho,\lambda,\tau}[w_\delta(r)]\geq 
\Lambda^{q-1} \, w_\delta(r)/r^2. 
\end{equation}

Using \eqref{gao}, we obtain that 
\begin{equation} \label{cv} \begin{aligned}
& \mathbb L_\rho[w_\delta(r)]= \beta\left(\beta+2\rho\right) r^{-2} w_\delta(r) +\left(-2\rho-2\beta+\alpha\right) \alpha \,c r^{-\beta+\alpha-2} e^{r^\alpha}
+ c\,\alpha^2 r^{2\alpha-\beta-2} e^{r^\alpha},\\
& \lambda \frac{ w^\tau_\delta(r) |w_\delta'(r)|^{1-\tau}}{r^{1+\tau}}=\lambda \, |\beta|^{1-\tau}
\frac{w_\delta(r)}{r^2} 
\left( 1+ \frac{\alpha\,c\,r^{-\beta+\alpha}}{|\beta|\, w_\delta(r)} e^{r^\alpha}
\right) ^{1-\tau}.
\end{aligned}
\end{equation}

{\bf Proof of \eqref{asa1} in Case $[N_0](a)$ or Case $[N_1]$.}

We use that $(1+a)^{1-\tau}\leq 1+\left(1-\tau\right) a $ for every $ a>0$ with equality when $\tau=0$.  
Hence, since either $\tau=0$ or $\lambda\leq 0$, from \eqref{cv}, we infer that
\begin{equation} \label{ma} 
\lambda  \frac{ w^\tau_\delta(r) |w_\delta'(r)|^{1-\tau}}{r^{1+\tau}}
 \geq 
\lambda |\beta|^{1-\tau} r^{-2} \,w_\delta(r) +\lambda  |\beta|^{-\tau} \left(1-\tau\right) \alpha \,c \,r^{-\beta+\alpha-2} e^{r^\alpha}. 
 \end{equation}
 From the first identity in \eqref{cv} and \eqref{ma}, we have 
\begin{equation} \label{ci} 
\mathbb L_{\rho,\lambda,\tau}[w_\delta(r)]\geq \Lambda^{q-1} \,r^{-2} w_\delta(r)+\alpha\,c \,B_\alpha r^{-\beta+\alpha-2} e^{r^\alpha},
\end{equation}
where $B_\alpha:=-2\rho -2\beta +\lambda |\beta|^{-\tau}(1-\tau)+\alpha $. Remark that $\beta<0$ and $f_\rho(\beta)>0$ imply that 
$B_\alpha=f_\rho(\beta)/|\beta|
-\beta-\lambda \tau |\beta|^{-\tau}+\alpha>0$ (as $\tau=0$ or $\lambda\leq 0$).
 Thus, \eqref{ci} implies \eqref{asa1}. 

\vspace{0.2cm}
{\bf Proof of \eqref{asa1} in Case $[N_0](b)$ or Case $[N_2]$.} 
Here, we also assume \eqref{extar}. 
Note that \eqref{ma} does not apply since $\lambda>0$ and $\tau\in (0,1)$.  
From \eqref{cv} and the assumption $\beta\leq -\rho$, we find that 
$$ 
 \mathbb L_\rho[w_\delta(r)]\geq \beta\left(\beta+2\rho\right)  \,w_\delta(r)/r^2\quad \mbox{and}\quad 
\lambda \frac{ w^\tau_\delta(r) |w_\delta'(r)|^{1-\tau}}{r^{1+\tau}}
 \geq \lambda |\beta|^{1-\tau} \, w_\delta(r)/r^2. $$
By adding these inequalities and using \eqref{haoo}, 
we conclude \eqref{asa1}. 

The proof of Lemma~\ref{lem2} is now complete. \end{proof}

\begin{rem} {\rm 
Let \eqref{e2} hold, $\beta<0$ and $f_\rho(\beta)>0$.   If $\rho\leq 0$ in Case $[N_0](b)$, then \eqref{extar} always holds.  
In Case $[N_2]$, we have $\lambda^{1/(\tau+1)}>\Upsilon$. 
Moreover, when \begin{equation} \label{gres} \rho\geq \lambda^{1/(\tau+1)}\ \mbox{ in Case } [N_2],\end{equation} then $\widetilde f_\rho(-\rho)=f_\rho(-\rho)/\rho=\lambda \rho^{-\tau}-\rho\leq 0$ so that 
$\varpi_1\leq -\rho<-(\lambda \tau) ^{1/(\tau+1)}< \varpi_2$ (and $\varpi_1=-\rho$ when $\rho=\lambda^{1/(\tau+1)}$). This means that when 
\eqref{gres} holds, then 
\eqref{extar} is equivalent to $\beta< \varpi_1$. }
\end{rem}


\subsection{Proof of Theorem~\ref{pro2}: Part II} \label{kvz} 
Let $u$ be a positive super-solution of  \eqref{e1} in the framework of Theorem~\ref{pro2}.
In this part, we address the remaining ranges of $\beta$ in Case $[N_0](b)$ and Case $[N_2]$, namely, 
\begin{equation} \label{appi} 0<-\beta<\rho<\Upsilon\ \mbox{(in Case } [N_0](b)) \ \mbox{and } \beta\in (-\rho,\varpi_1)\  \mbox{for }\Upsilon\leq 
\rho<\lambda^{\frac{1}{\tau+1}} 
\mbox{ (in Case } [N_2]) \end{equation}
and show that $u$ satisfies \eqref{lowe}. We reach this aim in Proposition~\ref{lad} by a 
new iterative scheme, which after finitely many steps leads to \eqref{lowe}. The starting point of this scheme is the following. 

\begin{lemma} \label{lema44} Let $q>1$, $\lambda>0$ and $\tau\in (0,1)$. If $0<\rho<\lambda^{1/(\tau+1)}$ and $\beta>-\rho$, then 
every positive super-solution $u$ of  \eqref{e1} satisfies 
\begin{equation} \label{musi} \mbox{for every } \eta>0\ \mbox{small}, \ \lim_{|x|\to 0} |x|^{-\rho+\eta} u(x)=\infty.  
\end{equation}
\end{lemma}

\begin{rem} \label{sirr} {\rm In Lemma~\ref{lema44}, we don't need $\beta<0$. Also in Case $[N_2]$, we don't require 
$\beta<\varpi_1$ since we don't use  
$f_\rho(\beta)>0$. Reasoning as in Part I, by replacing $\beta$ by $-\rho$ in the definition of $w_{\delta,\alpha}$ in \eqref{e8} and using that
$f_\rho(-\rho)=-\rho^2+\lambda \rho^{1-\tau}>0$,  we get only that for every
$\eta>0$, 
$$  \lim_{|x|\to 0}  |x|^{-\rho-\eta} u(x)=\infty. 
$$ Hence, we need to construct another family of sub-solutions to reach \eqref{musi}.}
\end{rem}

\begin{proof} 
Since $0<\rho<\lambda^{1/(1+\tau)}$ and $\beta>-\rho$, we can choose $0<\alpha<\min\{\rho, \rho+\beta\}$ small such that 
\begin{equation} \label{liop}  c_*:= \rho^2 \left(\lambda \rho^{-1-\tau}-1\right)-2\alpha \left[2\alpha +\lambda \left(1-\tau\right) \rho^{-\tau} \right] e>0.  
\end{equation}
Let $t_*\in (0,1)$ be small such that 
$ (1-t)^{1-\tau}\geq 1-2 \left(1-\tau\right) t $ for every $t\in (0,t_*)$. We diminish $\alpha>0$ such that $ \alpha e<\rho\, t_*$. 
Observe that as $\varepsilon\to 0^+$, 
$$ f_\rho(-\left(1+\varepsilon\right)\rho) =\left(\varepsilon^2-1\right)\rho^2+ \lambda \left(1+\varepsilon\right)^{1-\tau} \rho^{1-\tau}
=\rho^2 \left( \lambda \rho^{-\tau-1} -1\right)+\lambda \left(1-\tau\right) \rho^{1-\tau} \varepsilon (1+o(1)).
$$
Hence, we fix $\varepsilon\in (0, \alpha/\rho)$ small such that 
\begin{equation} \label{framp}  f_\rho(-\left(1+\varepsilon\right)\rho) >\rho^2 \left(\lambda \rho^{-\tau-1} -1\right).
\end{equation}


\vspace{0.2cm}
{\bf Claim:} {\em For every $\eta\in (0, \alpha-\rho\varepsilon)$, we have 
\eqref{musi}.} 

\vspace{0.2cm}
{\bf Proof of the Claim.} 
We show that for every $c\in (0,c_*^{1/(q-1)})$ and $\delta>0$, the function $z_\delta$ is a positive radial sub-solution of \eqref{e1} in $B_1(0)\setminus \{0\}$, where we define  
\begin{equation} \label{zoi} z_\delta(r)= c\, r^{\left(1+\varepsilon\right)\rho} \left( e^{r^\alpha}-e^{-\delta}\right)^{-1} \quad \mbox{for all } r\in (0,1).
\end{equation}
Let $\delta>0$ be arbitrary. We need to prove that $z_\delta$ satisfies
\begin{equation} \label {gimm}  \mathbb L_{\rho,\lambda,\tau} [z_\delta(r)] \geq  r^\theta z^{q}_\delta(r)\quad \mbox{for all } r\in (0,1).
\end{equation}

{\bf Proof of \eqref{gimm}.}
Fix $r\in (0,1)$ arbitrary. Using \eqref{zoi}, $\lambda>0$ and $0<\alpha e<\rho t_*<\rho$, we get
\begin{eqnarray}  
&\displaystyle  z_\delta'(r)=\left(1+\varepsilon\right) \rho \frac{z_\delta(r)}{r} \left(1-\frac{\alpha r^\alpha e^{r^\alpha}}{ \left(1+\varepsilon\right)
\rho \left(e^{r^\alpha}-e^{-\delta}\right) } \right) >0, \label{hxo1}
\\
&\displaystyle  \lambda \frac{z_\delta^\tau(r) |z_\delta'(r)|^{1-\tau}}{r^{1+\tau}}\geq \lambda \left[\left(1+\varepsilon\right)\rho \right]^{1-\tau} \frac{z_\delta(r)}{r^2} 
\left[1- \frac{2\left(1-\tau\right)\alpha \,r^\alpha e^{r^\alpha}}{(1+\varepsilon) \rho \left(e^{r^\alpha}-e^{-\delta}\right)}  \right]. \label{hxo2}
\end{eqnarray} 

On the other hand, since $\varepsilon \rho<\alpha$ and $r\in (0,1)$, we find that 
\begin{equation} \label{hxo3} \begin{aligned}  \mathbb L_{\rho} [z_\delta(r)]& >  \frac{z_\delta(r)}{r^2} 
\left\{ 
\rho^2 (\varepsilon^2-1) -\frac{ \alpha\left[
2\varepsilon \rho +\alpha \left(1+r^\alpha\right) \right] r^\alpha e^{r^\alpha} }{e^{r^\alpha}-e^{-\delta}}
\right\} \\
&>\frac{z_\delta(r)}{r^2} 
\left[
\rho^2 (\varepsilon^2-1) -\frac{ 4\alpha^2 r^\alpha e^{r^\alpha} }{e^{r^\alpha}-e^{-\delta}}
\right]
\end{aligned} \end{equation}
By adding the inequalities in \eqref{hxo2} and \eqref{hxo3}, then using \eqref{liop} and \eqref{framp}, we arrive at 
$$  \mathbb L_{\rho,\lambda,\tau} [z_\delta(r)] >\frac{z_\delta(r)}{r^2} \left[ f_{\rho} (-\rho\left(1-\varepsilon\right))
-\frac{ 2\alpha\left[2\alpha+\lambda \left(1-\tau\right) \rho^{-\tau} \right] r^\alpha e^{r^\alpha}}{e^{r^\alpha}-e^{-\delta}}
 \right]>c_* \frac{z_\delta(r)}{r^2}.
$$ 
On the other hand, since $0<\alpha<\rho+\beta$, $q>1$ and $c\in (0,c_*^{1/(q-1)})$, we have 
$$ r^{\theta+2} z_\delta^{q-1}(r)=c^{q-1} \left(e^{r^\alpha} -e^{-\delta}\right)^{1-q} r^{(q-1) \left(\beta+\rho+\varepsilon \rho\right)}\leq c^{q-1} r^
{(q-1)(\beta+\rho+\varepsilon\rho-\alpha) }\leq c_*.
$$ 
This completes the proof of  \eqref{gimm}. \qed

\vspace{0.2cm}
Let $r_0\in (0,1)$ be small such that $B_{r_0}(0)\subset\subset \Omega$. We 
diminish $c>0$ in \eqref{zoi} such that $c< \min_{|x|=r_0} u(x)$. Thus, since $\alpha<\rho$, we have 
$ z_{\delta}(r_0)\leq c r_0^{\varepsilon \rho+\rho-\alpha}<c\leq u(x)$ for all $|x|=r_0$.  
By Remark~\ref{sirr}, we have $\lim_{|x|\to 0} u(x)/z_\delta(|x|)=\infty$ so that by the comparison principle, we get 
$u(x)\geq z_{\delta}(|x|)$ for all $0<|x|\leq r_0$ and $\delta>0$. By letting $\delta\to 0$ and then $|x|\to 0$, we obtain  
\begin{equation} \label{uop} \liminf_{|x|\to 0} |x|^{\alpha-\rho- \rho\varepsilon} u(x)\geq c>0.\end{equation}
Since  $\eta\in (0, \alpha-\rho\varepsilon)$, from \eqref{uop} we infer \eqref{musi}. This proves the Claim and Lemma~\ref{lema44}. 
\end{proof}

We are now ready to complete the proof of Theorem~\ref{pro2} by establishing the following result.  

\begin{proposition} \label{lad}
Let $q>1$,  
$\lambda>0$ and $\tau\in (0,1)$. If \eqref{appi} holds, then $u$ satisfies \eqref{lowe}. 
\end{proposition}

\begin{proof} 
We assume that either $0<-\beta<\rho<\Upsilon$ or $\beta\in (-\rho,\varpi_1)$ for $\Upsilon\leq \rho<\lambda^{1/(\tau+1)}$.  

Define $\{M_k\}_{k\geq 1}$ as follows: 
\begin{equation} \label{sem} M_1=\rho\quad \mbox{and}\quad  M_{k+1}=\left( \frac{M_k\left(2\rho-M_k\right)}{\lambda}\right)^{\frac{1}{1-\tau}}\quad \mbox{for every } k\geq 1.  
\end{equation}
Since $\rho<\lambda^{1/(\tau+1)}$, by induction, $M_k\in (0,\rho)$ for every $k\geq 2$ and $\{M_k\}_{k\geq 1}$ is {\em decreasing}.  
Indeed, $M_2=(\rho^2/\lambda)^{1/(1-\tau)}<\rho=M_1$. Assuming that $M_k<\rho$ for $k\geq 2$, then 
$$\rho^{1+\tau} M_{k+1}^{1-\tau} <\lambda M_{k+1}^{1-\tau}=M_k(2\rho-M_k)\leq \rho^2$$ so that $M_{k+1}<\rho$. Hence, $M_k\in (0,\rho)$ for every $k\geq 2$. 

Assuming that $M_k<M_{k-1}$ for $k\geq 2$, we find 
$$  \lambda \left(M_{k+1}^{1-\tau}-M_k^{1-\tau}\right)=\left(M_k-M_{k-1}\right) (2\rho-M_k-M_{k-1})<0,
$$
which implies that $M_{k+1}<M_k$. Hence, $\{M_k\}_{k\geq 1}$ is decreasing and $\lim_{k\to \infty} M_k=M_\infty\in [0,\rho)$.  
Letting $k\to \infty$ in \eqref{sem}, we get $f_\rho(-M_\infty)=0$. In fact,  
\begin{eqnarray} 
& \lim_{k\to \infty} M_k=0\quad \mbox{if }  \rho<\Upsilon\ (\mbox{in Case } [N_0](b)), \label{amzq1}\\
&  \lim_{k\to \infty} M_k=-\varpi_1<-\beta \quad \mbox{if }  \Upsilon\leq \rho<\lambda^{1/(\tau+1)}\  (\mbox{in Case } [N_2]). \label{amzq2}
\end{eqnarray} 
For \eqref{amzq1}, we use that $ f_\rho(t)=t^2+2\rho t+\lambda |t|^{1-\tau}>0$ for every $t<0$ so that $M_\infty=0$. 
For \eqref{amzq2}, using that $y\longmapsto y\left(2\rho-y\right)$ is increasing on $(0,\rho)$, $M_k\in (0,\rho)$ for $k\geq 2$ and 
$f_\rho(\varpi_1)=0$, by \eqref{sem}, $M_1=\rho>-\varpi_1$ and induction, we find that 
$M_k>-\varpi_1$ for all $k\geq 1$. As the negative roots of $f_\rho$ are $\varpi_1$ and $\varpi_2$ with 
$\varpi_1\leq \varpi_2$, we conclude \eqref{amzq2}. 

Hence, there exists $k_*\geq 1$ large such that 
\begin{equation} \label{fomm} M_{k_*+1}\leq   -\beta< M_{k_*}. \end{equation}

Fix $\eta>0$ arbitrary. In what follows, let $\varepsilon\in (0,1)$ be small as needed. In particular, 
\begin{equation} \label{figg} |\beta| \,\varepsilon^{\frac{1}{k_*}}<\eta. 
\end{equation}
We define $m_k$ for every $1\leq k\leq k_{*}+1$ as follows
\begin{equation} \label{fow} m_1=\rho-\varepsilon,\quad m_k=M_k\left(1+\varepsilon^{1/(k-1)}\right) \ \  \mbox{for  } 2\leq k\leq k_*\ \ 
\mbox{and } m_{k_*+1}=-\beta \left(1+\varepsilon^{1/k_*}\right). 
\end{equation}
We let $\varepsilon>0$ small enough to ensure that $m_{k+1}<m_{k}$ for every $1\leq k\leq  k_*$ and
\begin{equation} \label{posi} m_j^2-2\rho m_j+\lambda m_{j+1}^{1-\tau} >0\quad \mbox{for all } 1\leq j\leq k_*.  
\end{equation}
Indeed, we have $m_1^2-2\rho m_1+\lambda m_{2}^{1-\tau} >\varepsilon^2$.  
By the definition of $\{M_k\}_{k\geq 1}$ in \eqref{sem}, we find that 
$$ \begin{aligned} 
m_j^2-2\rho m_j+\lambda m_{j+1}^{1-\tau}& =M_j^2-2\rho M_j+\lambda M_{j+1}^{1-\tau} +\lambda M_{j+1}^{1-\tau}
\left(1-\tau\right) 
\varepsilon^{\frac{1}{j}} (1+o(1))\\
&=\lambda M_{j+1}^{1-\tau}
\left(1-\tau\right) 
\varepsilon^{\frac{1}{j}} (1+o(1))>0\quad \mbox{as } \varepsilon\to 0^+
\end{aligned} 
$$
for every  $2\leq j\leq k_*-1$. 
Similarly, when $j=k_*$ in \eqref{posi}, using \eqref{fomm}, we get
\begin{equation} \label{faq} \begin{aligned} 
m_{k_*}^2-2\rho m_{k_*}+\lambda m_{k_*+1}^{1-\tau}& =
M_{k_*}^2-2\rho M_{k_*} +\lambda |\beta|^{1-\tau} +\lambda |\beta|^{1-\tau} \left(1-\tau\right)\varepsilon^{\frac{1}{k_*}} (1+o(1))\\
& \geq  \lambda |\beta|^{1-\tau} \left(1-\tau\right)\varepsilon^{\frac{1}{k_*}} (1+o(1))\quad \mbox{as } \varepsilon\to 0^+. 
\end{aligned}
\end{equation}
Hence, for $\varepsilon>0$ fixed small enough, we have \eqref{posi}. 

From \eqref{fow} and the choice of $\varepsilon>0$, we have $|\beta| /2<m_{k_{*}+1} /2<m_k/2$ for all $1\leq k\leq k_*$. 
Let $\nu>0$ be sufficiently small such that 
\begin{equation} \label{fanm} 
\begin{aligned}
& 0<\nu<|\beta| \min\{ \varepsilon^{1/k_*},1/2\},\\
& m_{j}^2 -2\rho m_j +\lambda \left( m_{j+1} -2\nu\right)^{1-\tau}  \geq \frac{1}{2} (m_j^2-2\rho m_j +\lambda m_{j+1}^{1-\tau})
\quad \mbox{for all } 1\leq j\leq k_{*}. 
\end{aligned} \end{equation}

{\bf Construction of a family of sub-solutions.} 
Fix $\alpha\in (0,\varepsilon)$ and $c>0$ such that 
\begin{equation} 
\label{exp} c^{q-1}\leq \frac{1}{2}\min_{1\leq j\leq k_*} (m_j^2 -2\rho m_j+\lambda m_{j+1}^{1-\tau}).\end{equation}
For every $1\leq j\leq k_*$ and $\delta>0$, we define $v_{j,\delta}(r)$ by 
\begin{equation} \label{vij} v_{j,\delta}(r)=  c \,r^{m_{j}} \left( e^{r^\alpha} -e^{-\delta}\right)^\frac{m_{j+1}-m_j-\nu}{\alpha}\quad \mbox{for all } r\in (0,1].
\end{equation}

We show that there exists $r_*\in (0,1)$ small such that 
for every $1\leq j\leq k_*$, the function $v_{j,\delta}$ in \eqref{vij} is a 
positive radial sub-solution of \eqref{e1} on $0<|x|<r_*^{1/\alpha}$, namely, 
\begin{equation} \label{fass}
 \mathbb L_{\rho,\lambda,\tau} [v_{j,\delta}(r)] \geq  r^\theta v^{q}_{j,\delta}(r)\quad \mbox{for all } r\in (0,r_*^{1/\alpha}).
\end{equation}

{\bf Proof of \eqref{fass}.}
We fix $r_*\in (0,1)$ small such that $B_{r_*^{1/\alpha}}(0)\subset\subset \Omega$ and for all $t\in (0,r_*)$, 
\begin{equation} \label{coppi} \frac{t e^t}{e^t-1} <1+\frac{\nu}{m_{j} -m_{j+1}+\nu} \quad \mbox{ for every } 1\leq j\leq k_*. 
\end{equation}

Let $1\leq j\leq k_*$ be arbitrary. In what follows, we understand that $r\in (0,r_*^{1/\alpha})$ is arbitrary. 
Using \eqref{coppi} and $m_{j+1}<m_j$, we find that 
\begin{equation} \label{ere} \begin{aligned} 
 v_{j,\delta}'(r)& = \frac{ v_{j,\delta}(r)}{r}\left( m_j +\left(m_{j+1}-m_j -\nu\right) \frac{r^\alpha e^{r^\alpha}}{e^{r^\alpha}-e^{-\delta}} \right) \\
 &>
 \frac{ v_{j,\delta}(r)}{r}\left( m_j +\left(m_{j+1}-m_j-\nu\right) \frac{r^\alpha e^{r^\alpha}}{e^{r^\alpha}-1} \right)\\
 &>\left(m_{j+1}-2\nu\right) \frac{v_{j,\delta}(r)}{r}>0.
\end{aligned} \end{equation}
 
Moreover, we obtain that 
\begin{equation} \label{sos} \begin{aligned} 
\mathbb L_\rho[v_{j,\delta}(r) ]=&  \frac{ v_{j,\delta}(r)}{r^2} \left\{ 
m_{j}^2-2\rho m_j +\left[ 2\left(m_j-\rho\right)+\alpha+\alpha r^\alpha \right]\left(m_{j+1}-m_j-\nu\right) \frac{r^\alpha e^{r^\alpha}}{e^{r^\alpha}-e^{-\delta}}
\right. \\
& \left. 
 \ \quad \qquad + (m_j-m_{j+1}+\nu) \left(m_j-m_{j+1}+\alpha+\nu\right) \frac{r^{2\alpha} e^{2r^\alpha}}{(e^{r^\alpha}-e^{-\delta})^2}
\right\} .
\end{aligned} \end{equation}
Our choice of $\alpha\in (0,\varepsilon)$ yields $\alpha<\rho-m_1\leq \rho-m_j$. 
Since $m_{j+1}<m_j$, we see that in the expression of $\mathbb L_\rho[v_{j,\delta}(r) ]$, the coefficients of
$ r^\alpha e^{r^\alpha}/(e^{r^\alpha}-e^{-\delta})$ and $ r^{2\alpha} e^{2r^\alpha}/(e^{r^\alpha}-e^{-\delta})^2$  are positive. Thus, we obtain that 
$$ 
\mathbb L_\rho[v_{j,\delta}(r) ]
 \geq \frac{ v_{j,\delta}(r)}{r^2} \left( m_{j}^2-2\rho m_j\right). 
$$
This, jointly with \eqref{ere}, implies that 
\begin{equation} \label{ror1} \mathbb L_{\rho,\lambda,\tau} [v_{j,\delta}(r)] \geq  \frac{ v_{j,\delta}(r)}{r^2}\left[ m_{j}^2-2\rho m_j+\lambda \left(m_{j+1}-2\nu\right)^{1-\tau}
 \right].
\end{equation}
Hence, in view of \eqref{fanm} and \eqref{exp}, we get that 
\begin{equation} \label{sig}  \mathbb L_{\rho,\lambda,\tau} [v_{j,\delta}(r)] \geq   \frac{ v_{j,\delta}(r)}{r^2}  
\frac{1}{2}\min_{1\leq j\leq k_*} (m_j^2 -2\rho m_j+\lambda m_{j+1}^{1-\tau}) 
\geq c^{q-1}  \frac{ v_{j,\delta}(r)}{r^2}.  
\end{equation}
By our choice of $\nu>0$ and \eqref{fow}, we have 
\begin{equation} \label{mott}
\begin{aligned} 
& \nu<-\beta \,\varepsilon^{1/k_*}=m_{k_*+1} +\beta\leq m_{j+1}+\beta,\\
& \theta+2+\left(m_{j+1}-\nu\right) \left(q-1\right)> \theta+2-\beta\left(q-1\right)=0.
\end{aligned}
\end{equation}
Hence, we see that 
$$ \begin{aligned}
r^{\theta+2} v^{q-1}_{j,\delta}(r) &=c^{q-1} r^{m_j\left(q-1\right)+\theta+2} \left(e^{r^\alpha} -e^{-\delta} \right)^{\frac{\left(m_{j+1}-m_j-\nu\right)(q-1)}{\alpha}}\\
& \leq c^{q-1} r^{m_j \left(q-1\right)+\theta+2} \left(e^{r^\alpha} -1 \right)^{\frac{\left(m_{j+1}-m_j-\nu\right)(q-1)}{\alpha}}\\
&\leq c^{q-1} r^{\theta+2+\left(m_{j+1}-\nu\right)\left(q-1\right)}\leq c^{q-1}
 \quad \mbox{for all } r\in (0,r_*^{1/\alpha}).
\end{aligned}$$
This and \eqref{sig} prove that for every $1\leq j\leq k_*$, the function $v_{j,\delta}$ in \eqref{vij} 
satisfies \eqref{fass}. 

\vspace{0.2cm} 
{\bf Proof of Proposition~\ref{lad} concluded.} We let $c>0$ in \eqref{vij} small enough such that $c<\min_{|x|=r_*^{1/\alpha}} u(x)$. Hence, 
\begin{equation} \label{fvb} v_{j,\delta}(|x|)\leq u(x)\quad \mbox{ for every } |x|=r_*^{1/\alpha}\ \mbox{ and all } 1\leq j\leq k_*.\end{equation}

From Lemma~\ref{lema44}, we have $\lim_{|x|\to 0} |x|^{-\rho+\varepsilon} u(x)=\infty$ if $\varepsilon>0$ is small enough. 
Hence, recalling that $m_1=\rho-\varepsilon$, we find that 
\begin{equation} \label{emb} u(x)\geq v_{1,\delta}(|x|)\quad \mbox{ for } |x|>0\ \mbox{ close to zero}. \end{equation}
By \eqref{fvb} with $j=1$ and the comparison principle (see  \eqref{ere}), we obtain that 
$$ u(x)\geq v_{1,\delta}(|x|)\quad \mbox{for all } 0<|x|\leq r_*^{1/\alpha}.$$
By letting $\delta\to 0$ and then $|x|\to 0$, we arrive at 
$ \liminf_{|x|\to 0} |x|^{-m_2+\nu} u(x)\geq c$. This shows that 
$ u(x)\geq v_{2,\delta}(|x|)$  for $|x|>0$ close to zero. In view of \eqref{fass} with $j=2$ and \eqref{fvb}, 
we again can use the comparison principle to infer that 
$ u(x)\geq v_{2,\delta}(|x|)$  for all $0<|x|\leq r_*^{1/\alpha}$. 
Then, by induction on $j=1,\ldots, k_*$, we find that 
$$ u(x)\geq v_{k_*,\delta}(|x|)\quad \mbox{for all } 0<|x|\leq r_*^{1/\alpha}.$$
Passing to the limit $\delta\to 0$, we arrive at 
\begin{equation} \label{topp} \liminf_{|x|\to 0} |x|^{-m_{k_*+1}+\nu} u(x)\geq c,\end{equation}where 
$m_{k_*+1}=-\beta \left(1+\varepsilon^{1/k_*}\right)<\eta-\beta$ (see \eqref{figg} and  \eqref{fow}). Consequently, from 
\eqref{topp}, we get $\lim_{|x|\to 0} |x|^{\beta-\eta} u(x)=\infty$. Since $\eta>0$ was arbitrary, we conclude the proof of Proposition~\ref{lad}. 
\end{proof}
In view of Subsections~\ref{sect401} and \ref{kvz}, the proof of Theorem~\ref{pro2} is complete. \qed


\section{Sharp lower bounds when $\beta> \varpi_1$ in Case $[N_j]$ with $j\in \{1,2\}$}
\label{oaz}

\begin{theorem} \label{zic1} Let \eqref{e2} hold. Assume that $\beta> \varpi_1$ in either Case $[N_1]$ or Case $[N_2]$. Then, every positive super-solution $u$ of 
\eqref{e1} satisfies 
\begin{equation} \label{mb2}
\liminf_{|x|\to 0} |x|^{\varpi_1} u(x)>0. 
\end{equation} 
\end{theorem}

To establish \eqref{mb2} 
under the assumptions of Theorem~\ref{zic1}, we first show that
\begin{equation} \label{simo} \mbox{for every } \eta>0,\ \mbox{we have } \lim_{|x|\to 0} |x|^{\varpi_1-\eta} u(x)=\infty .\end{equation}

$\bullet$ For Case~$[N_1]$, as well as Case $[N_2]$ with $\varpi_1\leq -\rho$, we obtain \eqref{simo} from Proposition~\ref{nwp1}. 

$\bullet$ For Case $[N_2]$ with $\varpi_1>-\rho$, we conclude \eqref{simo} from Proposition~\ref{gain}. 

\vspace{0.2cm}
In Section~\ref{sect61}, we state and prove Propositions~\ref{nwp1} and \ref{gain}. Using these ingredients, we give the proof in Theorem~\ref{zic1} in Section~\ref{sect62}.

\subsection{Critical ingredients for proving Theorem~\ref{zic1}} \label{sect61}

\begin{proposition} \label{nwp1} Let \eqref{e2} hold and $\beta> \varpi_1$ in Case $[N_j]$ with $j\in \{1,2\}$. In Case $[N_2]$, we also assume
$\varpi_1\leq -\rho$. Let $u$ be any positive super-solution of \eqref{e1}. Then,  
\begin{equation} \label{fio}  \mbox{for every } \eta>0, \mbox{we have }  \lim_{|x|\to 0} |x|^{\varpi_1} \left| \log |x|   \right|^{\eta} u(x)=\infty.
\end{equation} 
\end{proposition}

\begin{proof} Let $\mu>0$ be arbitrary. 
We define $c_0=c_0(\mu)>0$ by 
	\begin{equation} \label{cmu} c_0(\mu):=
	\left\{ \begin{aligned}
	& \left[\mu |\varpi_1| \right]^{\frac{1}{q-1}} \inf_{r\in (0,1/e]} \left(r^{\varpi_1-\beta} |\log r|^{\mu-\frac{1}{1-q}} \right) && \mbox{in 
	Case $[N_1]$,}&\\
	& \left[\mu(\mu+1)\right]^{\frac{1}{q-1}}  \inf_{r\in (0,1/e]} \left(r^{\varpi_1-\beta} |\log r|^{\mu-\frac{2}{1-q}} \right)
	&& \mbox{in Case [$N_2$]}.&
	\end{aligned}\right.
	\end{equation}
	
Let $c\in (0,c_0)$ and $\delta\in (0,1/e)$ be arbitrary. For every $\delta\leq r\leq 1/e$, we define $u_\delta(r)$ by 
\begin{equation} \label{eqo1} u_\delta(r):= c\,r^{-\varpi_1} \left[|\log r|^{-\mu}-|\log \delta|^{-\mu}\right].
\end{equation}

\begin{lemma} \label{lema64} In the framework of Proposition~\ref{nwp1}, the function 
$u_\delta$ in \eqref{eqo1} is a positive sub-solution of \eqref{e1} on $\{ \delta<|x|<1/e\}$. 
\end{lemma}

\begin{proof}
	In what follows, we show that 
	\begin{equation}\label{eqo2}
\mathbb L_{\rho,\lambda,\tau}[u_{\delta}(r)]\geq r^\theta\,(u_{\delta}(r))^q \quad\mbox{for every } \delta< r<1/e.
\end{equation}
	Let $\delta<r<1/e$ be arbitrary.	We remark that 
	\begin{equation} \label{gr} u_\delta'(r)=|\varpi_1| \,r^{-1} u_\delta(r) +c\mu r^{-1-\varpi_1} | \log r|^{-\mu-1}>|\varpi_1| \, r^{-1} u_\delta(r)>0. 
	\end{equation}
	After a routine calculation, we see that 
	\begin{equation} \label{eqo5}  
	\mathbb L_\rho [u_\delta(r)]=\varpi_1 \left(\varpi_1+2\rho\right) \frac{u_\delta(r)}{r^2} +c\mu r^{-2-\varpi_1} | \log r|^{-\mu-1}\left[ -2\left(\rho+\varpi_1\right) +
	\frac{\mu+1}{\log  \left(1/r\right)} \right].
	\end{equation}

$\bullet$ If Case $[N_1]$ holds, then since $\lambda<0$ or $\tau=0$, by using \eqref{gr}, we find that 
\begin{equation*} \label{mb3} \lambda \,  \frac{u^\tau_\delta(r) |u'_\delta(r)|^{1-\tau}}{r^{1+\tau}} \geq \lambda \,|\varpi_1|^{-\tau} r^{-1} u_\delta'(r)= 
\lambda |\varpi_1|^{1-\tau} r^{-2} u_\delta(r) +\lambda c\mu |\varpi_1|^{-\tau} r^{-2-\varpi_1}| \log r|^{-\mu-1}.
\end{equation*}

$\bullet$ If Case $[N_2]$ holds, then using \eqref{gr}, we get
\begin{equation*} \label{mab3} \lambda \,  \frac{u^\tau_\delta(r) |u'_\delta(r)|^{1-\tau}}{r^{1+\tau}} \geq \lambda |\varpi_1|^{1-\tau} r^{-2} u_\delta(r).
\end{equation*}

Since  
$\varpi_1<0$ satisfies $f_{\rho}(\varpi_1)=0$ (and, in addition, $\rho+\varpi_1 \leq 0$ in Case $[N_2]$), we arrive at 
	\begin{equation} \label{rhh}   \mathbb L_{\rho,\lambda,\tau}[u_{\delta}(r)]\geq 
	\left\{ \begin{aligned} 
	&c\mu |\varpi_1| r^{-2-\varpi_1} |\log r|^{-\mu-1} && \mbox{if Case $[N_1]$ holds},&\\
	& c\mu \left(\mu+1\right) r^{-2-\varpi_1} |\log r|^{-\mu-2} && \mbox{if Case $[N_2]$ holds}.&
	\end{aligned} \right.
	\end{equation}
	
	On the other hand, the right-hand side of \eqref{eqo2} is bounded above by 
	\begin{equation} \label{eqo4}  c^{q} r^{\theta-q\varpi_1}|\log r|^{-\mu q}. 
	\end{equation} 
	Since $c\in (0,c_0(\mu))$ with $c_0(\mu)$ given in \eqref{cmu}, 
	we deduce that for each $r\in (\delta,1/e)$, the quantity in \eqref{eqo4} is less than or equal to the right-hand side of \eqref{rhh}.  
	This proves \eqref{eqo2}. 
\end{proof}

{\bf Proof of Proposition~\ref{nwp1} completed.} 
Let $u$ be any positive super-solution of \eqref{e1}.  
Fix $r_0\in (0,1/e)$ such that $B_{r_0}(0)\subset\subset \Omega$. 
We let $c=c_\mu \in (0,c_0(\mu))$ small such that 
$$c_\mu\leq r_0^{\varpi_1}\,\min_{\partial B_{r_0}(0)}u.$$ This implies that 
$ u\geq u_\delta $ on $ \partial B_{r_0}(0)$. Since $u_\delta=0$ on $\partial B_\delta(0)$ and \eqref{eqo2} holds, using \eqref{gr} and the comparison principle, 
 we get
$ u(x)\geq u_{\delta}(|x|)$ for all $\delta\leq |x|\leq r_0$. By letting 
$\delta\to 0$, we have
\begin{equation} \label{omina}
u(x)\geq c_\mu \,|x|^{-\varpi_1} [\log \,(1/|x|)]^{-\mu} \quad \mbox{for every } x\in B_{r_0}(0)\setminus \{0\}.
\end{equation}
Let $\eta>0$ be arbitrary. By choosing $\mu\in (0,\eta)$, from \eqref{omina}, we conclude the proof of \eqref{fio}.  
\end{proof}

\begin{rem}
{\rm In the framework of Proposition~\ref{nwp1}, we cannot conclude \eqref{mb2} from \eqref{omina} by letting $\mu\to 0$ 
since $c_\mu\in (0,c_0(\mu))$ and 
$c_0(\mu)\to 0$ as $\mu\to 0^+$, see \eqref{cmu}.}
\end{rem}

\begin{proposition} \label{gain} Let \eqref{e2} hold and $\beta\geq \varpi_1>-\rho$ in Case $[N_2]$.  Then, every positive super-solution $u$ of \eqref{e1} satisfies \eqref{simo}. 
\end{proposition}

\begin{proof} The assumption $\varpi_1>-\rho$ in Case $[N_2]$ implies that 
$\Upsilon\leq  \rho<\lambda^{1/(\tau+1)}$. Indeed, if we assume $\rho\geq \lambda^{1/(\tau+1)}$, then $\widetilde f_\rho(-\rho)=\lambda \rho^{-\tau}-\rho\leq 0$, leading to $\varpi_1\leq -\rho<\varpi_2$, a contradiction. 
Therefore, the assumptions in Lemma~\ref{lema44} are satisfied so that \eqref{musi} is valid (see also Remark~\ref{sirr}).  

Let $\eta>0$ be arbitrary.  We fix $\eta_0>0$  small such that 
\begin{equation} \label{muo} 0<\eta_0<\min\, \{\eta,\rho+\varpi_1\}. \end{equation}
We adapt the ideas in the proof of Proposition~\ref{lad}, replacing $\beta$ there by 
$\varpi_1-\eta_0$. We define $\{M_k\}_{k\geq 1}$ as in \eqref{sem}. As before $\{M_k\}_{k\geq 1}$ is decreasing and since we are in Case $[N_2]$, we have 
\begin{equation} \label{rioo} \lim_{k\to \infty} M_k=-\varpi_1<-\varpi_1+\eta_0. 
\end{equation}
Recall that $M_1=\rho>-\varpi_1+\eta_0$.  
In view of  \eqref{muo} and \eqref{rioo}, there exists $k_*\geq 1$ such that 
$$ M_{k_*+1}\leq -\varpi_1+\eta_0<M_{k_*}.  
$$
Let $\varepsilon\in (0,1)$ be small as needed. Instead of \eqref{figg}, we let $\varepsilon>0$ satisfy 
\begin{equation} \label{fiuu}  \left(-\varpi_1+\eta_0 \right) \varepsilon^{\frac{1}{k_*}} <\eta-\eta_0. 
\end{equation}
In \eqref{fow}, \eqref{faq}, \eqref{fanm} and \eqref{mott} we replace $\beta$ by  $\varpi_1-\eta_0$. 
Hence, using that 
$-\varpi_1+\eta_0>-\varpi_1\geq -\beta$,  
we recover that $\theta+2+\left(m_{j+1}-\nu\right) \left(q-1\right)>0$ for every $1\leq j\leq k_*$. 
All the arguments in the proof of Proposition~\ref{lad} hold, leading to \eqref{topp}, namely, 
\begin{equation} \label{omm} \liminf_{|x|\to 0} |x|^{-m_{k_*+1}+\nu} u(x)\geq c,\end{equation}where 
here $m_{k_*+1}=(-\varpi_1+\eta_0) \left(1+\varepsilon^{1/k_*}\right)$.  
From \eqref{fiuu} and \eqref{omm}, we conclude \eqref{simo}. 
\end{proof}

\subsection{Proof of Theorem~\ref{zic1}} \label{sect62}

We assume that $\beta>\varpi_1$ in either Case $[N_1]$ or Case $[N_2]$. 
Let $u$ be a  positive super-solution of \eqref{e1}. We show that $u$ satisfies \eqref{mb2}. 

From Propositions~\ref{nwp1} and \ref{gain}, we know that 
\begin{equation} \label{ino} 
\mbox{for every } \eta>0,\ \mbox{it holds }
\lim_{|x|\to 0} |x|^{\varpi_1-\eta} u(x) =\infty.
\end{equation}
We fix $\varepsilon>0$ small. Then, we let $r_0\in (0,1)$ be small as needed such that $B_{r_0}(0)\subset \subset \Omega$. 
To finish the proof of Theorem~\ref{zic1}, we need to construct a family of positive radial sub-solutions $\{z_{\delta,\varepsilon}\}_{\delta\in (0,1)}$ of \eqref{e1} in $B_{r_\varepsilon}(0)\setminus \{0\}$, where $r_\varepsilon\in (0,r_0)$ is independent of $\delta$, such that
\begin{itemize}
\item[(A)] For each $\delta\in (0,1)$, we have $\lim_{|x|\to 0^+} u(x)/ z_{\delta,\varepsilon}(|x|)=\infty$;
\item[(B)] For some constant $c>0$, we have $\lim_{r\to 0^+} r^{\varpi_1} (\lim_{\delta \to 0^+} z_{\delta,\varepsilon}(r))=c $.
\item[(C)] For every $r\in (0,r_\varepsilon)$ and each $\delta\in (0,1)$, we have $z_{\delta,\varepsilon}'(r)>0$. 
\end{itemize}

\vspace{0.2cm}
The construction of the sub-solutions $\{z_{\delta,\varepsilon}\}_{\delta\in (0,1)}$ is non-trivial and challenging because we need the property (B) to hold, whereas $|x|^{-\varpi_1}$ is a super-solution of \eqref{e1} since $\mathbb L_{\rho,\lambda,\tau} [|x|^{-\varpi_1}]=0$. 

\vspace{0.2cm}
{\em Construction of $\{z_{\delta,\varepsilon}\}_{\delta>0}$ satisfying the properties} (A)--(C).

Let $\varepsilon>0$ and $c>0$ be fixed. For every $\delta\in (0,1)$ and $r,t>0$, we define 
\begin{equation} \label{zis} 
\begin{aligned}
& z_{\delta,\varepsilon}(r):= c \, r^{\varepsilon-\varpi_1} \left(1 -e^{-r^\varepsilon}\right)^{\delta-1},\\
& h(t)= \frac{ t \,e^{-t} }{1-e^{-t}}\in (0,1) 
\quad \mbox{and}\quad 
 \varphi_{\delta}(t):=  1-\left(1-\delta\right) h(t) \in (0,1). 
\end{aligned}
\end{equation}
In view of \eqref{ino}, we readily see that $z_{\delta,\varepsilon}$ in \eqref{zis} satisfies the property (A). 
The property (B) is easily verified.  
By differentiating $z_{\delta,\varepsilon}(r)$ 
with respect to $r$, we get
 \begin{equation} \label{nam0} 
z_{\delta,\varepsilon}'(r)= |\varpi_1|\frac{z_{\delta,\varepsilon}(r)}{r} \left( 1+\frac{\varepsilon\, \varphi_\delta(r^\varepsilon)}{|\varpi_1|} 
 \right)>0. 
\end{equation}
Hence, the property (C) also holds. Remark that $1-e^{-t}> te^{-t}$ for every $t> 0$, which implies that $\varphi_\delta(t)>0$ for every $t>0$. Moreover, for every
$\delta\in (0,1)$, we have 
\begin{equation} \label{rumm} \varphi_\delta'(t)=\frac{\left(1-\delta\right) e^{-t}}{(1-e^{-t})^2} (t-1+e^{-t})>0\quad \mbox{for all } t>0.  
\end{equation}

In Lemma~\ref{l53} below, we show that there exist $r_\varepsilon\in (0,1)$ and $c_\varepsilon>0$ such that for all $c\in (0,c_\varepsilon)$ and $\delta\in (0,1)$, the function $z_{\delta,\varepsilon} $ in \eqref{zis} a sub-solution of \eqref{e1} in $B_{r_\varepsilon}(0)\setminus \{0\}$, namely, 
\begin{equation} \label{sun}    \mathbb L_{\rho,\lambda,\tau} [z_{\delta,\varepsilon}(r)]\geq r^{\theta} (z_{\delta,\varepsilon}(r))^q\quad \mbox{for every } 0<r< r_\varepsilon.
\end{equation}

Assuming \eqref{sun}, we complete the proof of Theorem~\ref{zic1} as follows. 
If necessary, we diminish $c\in (0,c_\varepsilon)$ such that $\min_{|x|=r_\varepsilon} u(x)\geq c\, (r_\varepsilon)^{\varepsilon-\varpi_1}(1-e^{-(r_\varepsilon)^{\varepsilon}})^{-1}$. This ensures that
$$ u(x)\geq z_{\delta,\varepsilon}(|x|)\quad \mbox{for all } x\in \partial B_{r_\varepsilon}(0)\ \mbox{ and every }\delta\in (0,1).$$ 
From \eqref{ino} and the property (A), we have $u(x)\geq z_{\delta,\varepsilon}(|x|)$ for every $|x|>0$ small, where $\delta\in (0,1)$ is arbitrary. 
Using (C) and the comparison principle (see Remark~\ref{reka6}), we find that  
$$  u(x) \geq z_{\delta,\varepsilon} (|x|) \quad \mbox{for every } 0<|x|\leq r_\varepsilon\ \mbox{and all } \delta\in (0,1). 
$$
Now, using the property (B), we get $\liminf_{|x|\to 0} |x|^{\varpi_1} u(x) \geq c>0$, which proves \eqref{mb2}. 

It remains to establish Lemma~\ref{l53}, which is stated next. 

\begin{lemma} \label{l53} Under the assumptions of Theorem~\ref{zic1}, for every $\varepsilon>0$ small, there exist 
constants $c_\varepsilon>0$ and $r_\varepsilon\in (0,r_0)$ such that for all 
$c\in (0,c_\varepsilon)$ and $\delta\in (0,1)$, 
 the function 
$z_{\delta,\varepsilon}$ in \eqref{zis} is a positive radial sub-solution of \eqref{e1} in $B_{r_\varepsilon}(0)\setminus \{0\}$. 
\end{lemma}

\begin{proof}
We fix $\varepsilon>0$ small as needed (see \eqref{ch}). To simplify the writing, we drop the subscript $\varepsilon$ when referring to  
$z_{\delta,\varepsilon}$.  Let $r\in (0,1)$ be arbitrary. Using \eqref{nam0} and \eqref{rumm}, we obtain that
\begin{equation} \label{revv} \begin{aligned}
\mathbb L_\rho[ z_\delta(r)]&= \frac{z_\delta(r)}{r^2}\left[ \varpi_1^2+2\rho \varpi_1-2\varepsilon\left( \varpi_1+\rho\right)\varphi_\delta(r^\varepsilon)+
\varepsilon^2 \varphi_\delta^2(r^\varepsilon) 
+\varepsilon^2 r^\varepsilon\varphi_\delta'(r^\varepsilon)\right] \\
& \geq  \frac{z_\delta(r)}{r^2} \left[ \varpi_1^2+2\rho \varpi_1-2\varepsilon\left( \varpi_1+\rho\right)\varphi_\delta(r^\varepsilon)+
\varepsilon^2 \varphi_\delta^2(r^\varepsilon) 
\right].
\end{aligned}
\end{equation}

Moreover, we have 
\begin{equation}  \label{riww}  \lambda \,G_\tau (z_\delta(r))=\lambda\, |\varpi_1|^{1-\tau} \frac{z_\delta(r)}{r^2}  \left( 1+\frac{\varepsilon \,\varphi_{\delta}(r^\varepsilon)}{|\varpi_1|}  \right)^{1-\tau}. 
\end{equation}

Since $f_{\rho}(\varpi_1)=\varpi_1^2+2\rho \varpi_1+\lambda |\varpi_1|^{1-\tau}=
0$ and either Case $[N_1]$ or Case $[N_2]$ holds, we have 
\begin{equation} \label{num} 
\lambda |\varpi_1|^{-\tau} (1-\tau)-2\varpi_1-2\rho =|\varpi_1| (1-\lambda \tau |\varpi_1|^{-\tau-1}). 
\end{equation}
In Case $[N_1]$, as well as in Case $[N_2]$ with $\rho>\Upsilon$, we have 
\begin{equation} \label{xio}  |\varpi_1| (1-\lambda \tau |\varpi_1|^{-\tau-1})=(q-1)\gamma^{q-1}>0.
\end{equation}

$\bullet$ Let Case $[N_1]$ hold. Then, we obtain that 
$$ \lambda \,G_\tau (z_\delta(r))\geq 
\lambda\,|\varpi_1|^{1-\tau}\, \frac{z_\delta(r)}{r^2}\left(1+ \frac{\varepsilon \left(1-\tau\right) \,\varphi_{\delta}(r^\varepsilon)}{|\varpi_1|} \right).
$$
Using 
\eqref{revv} and \eqref{num}, we see that
\begin{equation} \label{cx} 
\begin{aligned} 
 \mathbb L_{\rho,\lambda,\tau} [z_\delta(r)]
& \geq \frac{z_\delta(r)}{r^2} \varepsilon \varphi_{\delta}(r^\varepsilon)
 \left[  \lambda\, |\varpi_1|^{-\tau} \left(1-\tau\right) -2\varpi_1-2\rho\right]  
\\ 
&= \left(q-1\right)\gamma^{q-1}  \frac{z_\delta(r)}{r^2} \varepsilon \varphi_{\delta}(r^\varepsilon).
\end{aligned}
\end{equation}

\vspace{0.2cm}
$\bullet$ Let Case $[N_2]$ hold. 
We point that 
$\varpi_1=-(\lambda\tau)^{1/(\tau+1)}$ if $\rho=\Upsilon$. Thus,  we have
\begin{equation} \label{nimm} 
\lambda |\varpi_1|^{-\tau} (1-\tau)-2\varpi_1-2\rho =|\varpi_1| \left(1-\lambda \tau |\varpi_1|^{-\tau-1}\right)=0\ \mbox{if } 
\rho=\Upsilon.
\end{equation}

Fix $\nu\in (0,1)$. 
Let $t_\nu>0$ be small such that 
\begin{equation} \label{hart1} 
\left( 1+\frac{t}{|\varpi_1|}\right)^{1-\tau}\geq 
1+ \frac{\left(1-\tau\right)}{|\varpi_1|}t- \frac{\tau \left(1-\tau\right)\left(1+\nu\right) }{2 \varpi_1^2} \,t^2
\quad \mbox{for every } t\in (0,t_\nu).
\end{equation}
We let $\varepsilon \in (0,t_\nu)$. Since $\varphi_\delta(r^\varepsilon)\in (0,1)$, using \eqref{riww}, \eqref{hart1} and $\lambda>0$, 
we obtain that
\begin{equation} \label{wov2} \lambda \,G_\tau (z_\delta(r))
\geq \lambda\, |\varpi_1|^{1-\tau}\frac{z_\delta(r)}{r^2} \left(1+\frac{\left(1-\tau\right) 
}{|\varpi_1|}\,\varepsilon\varphi_\delta(r^\varepsilon)
-\frac{\tau \left(1-\tau\right)\left(1+\nu\right) }{2 \varpi_1^2} \,\varepsilon^2 \varphi^2_\delta(r^\varepsilon)
\right).
\end{equation}
From \eqref{revv}, \eqref{hart1}, \eqref{xio} (when $\rho>\Upsilon$) and \eqref{nimm} (when $\rho=\Upsilon$), it follows that
\begin{equation} \label{cxx2}
\begin{aligned}
& \mathbb L_{\rho,\lambda,\tau} [z_\delta(r)]\geq  \left(q-1\right)\gamma^{q-1}  \frac{z_\delta(r)}{r^2} \varepsilon \varphi_{\delta}(r^\varepsilon)
&& \mbox{if } \rho>\Upsilon,&\\
& \mathbb L_{\rho,\lambda,\tau} [z_\delta(r)]
 \geq \frac{z_\delta(r)}{r^2} \varepsilon^2 \varphi^2_{\delta}(r^\varepsilon)
  \left[ 1-\frac{(1-\tau)(1+\nu)}{2} \right]
&& \mbox{if } \rho=\Upsilon . &
\end{aligned}
 \end{equation}

{\bf Proof of Lemma~\ref{l53} concluded.} We define $a>0$ as follows
\begin{equation} \label{defa} 
\begin{aligned}
& a=\left(q-1\right)\gamma^{q-1}
\ \mbox{ in Case } [N_1],  \mbox{ as well as in Case }[N_2]\ \mbox{with } \rho>\Upsilon,\\
& a=1-\frac{\left(1-\tau\right)\left(1+\nu\right)}{2} \ \mbox{ in Case } [N_2]\ \mbox{when } \rho=\Upsilon.
\end{aligned}
\end{equation}
Since $q>1$, $\varpi_1<0$ and $\beta\in (\varpi_1,0)$, we can fix 
$\varepsilon>0$ small such that 
\begin{equation} \label{ch} 0<\varepsilon<\frac{\left(q-1\right) \,(\beta-\varpi_1)}{2}.
\end{equation}

Let $c\in (0,c_\varepsilon)$ be arbitrary, where $c_\varepsilon>0$ is defined by
\begin{equation} \label{cer} 
\begin{aligned}
& c_\varepsilon= \left( \frac{\varepsilon a}{3}\right)^{\frac{1}{q-1}}\ \mbox{in Case } [N_1]\ \mbox{as well as Case }[N_2]\ \mbox{with } \rho>\Upsilon,\\
& c_\varepsilon=  \left( \frac{\varepsilon^2 a}{5}\right)^{\frac{1}{q-1}}\ \mbox{in Case } [N_2]\ \mbox{with } \rho=\Upsilon.
\end{aligned}
\end{equation}

Observe that $\varphi_{\delta}(t)\geq  \lim_{\delta\to 0} \varphi_{\delta}(t):=\varphi_{0}(t) $ for each $t>0$ and every $\delta\in (0,1)$. 
In Case $[N_1]$ and also in Case $[N_2]$ with $\rho>\Upsilon$, we have proved that for every $r\in (0,1)$ and $\delta\in (0,1)$, 
\begin{equation} \label{coll1}
 \mathbb L_{\rho,\lambda,\tau} [z_\delta(r)]\geq 
 \varepsilon  a \varphi_{\delta}(r^\varepsilon) 
  \frac{z_\delta(r)}{r^2} 
  \geq \varepsilon a  \varphi_0(r^\varepsilon)  \, \frac{z_\delta(r)}{r^2},
\end{equation}
whereas in Case $[N_2]$ with $\rho=\Upsilon$, we have 
\begin{equation} \label{coll11}
 \mathbb L_{\rho,\lambda,\tau} [z_\delta(r)]\geq 
 \varepsilon^2  a \varphi^2_{\delta}(r^\varepsilon) 
  \frac{z_\delta(r)}{r^2} 
  \geq \varepsilon^2 a  \varphi^2_0(r^\varepsilon)  \, \frac{z_\delta(r)}{r^2}.
\end{equation}
On the other hand, the right-hand side of \eqref{sun} is given by 
$$ \begin{aligned} 
r^\theta (z_\delta(r))^q& =c^{q-1}  \,r^{\theta+2+(\varepsilon-\varpi_1 ) \left(q-1\right)} 
\left(1-e^{-r^\varepsilon}\right)^{(\delta-1)(q-1)} \frac{z_\delta(r)}{r^2} \\
& \leq c^{q-1}  \,r^{\theta+2+(\varepsilon-\varpi_1 ) \left(q-1\right)} 
\left(1-e^{-r^\varepsilon}\right)^{-(q-1)} \frac{z_\delta(r)}{r^2}.
\end{aligned} $$
In view of \eqref{ch} and $\beta=(\theta+2)/(q-1)$, we get
\begin{equation} \label{coll2} r^\theta (z_\delta(r))^q\leq c^{q-1} r^{\varepsilon (q+1)} \left(1-e^{-r^\varepsilon}\right)^{-(q-1)}\,\frac{z_\delta(r)}{r^2} .
\end{equation}

To compare the right-hand sides of \eqref{coll1} (respectively, \eqref{coll11}) and \eqref{coll2}, we remark that 
$$ \lim_{t\to 0^+} \frac{t^q (1-e^{-t})^{-(q-1)}}{1-h(t)}=2\quad \mbox{and} \quad
 \lim_{t\to 0^+} \frac{t^{q+1} (1-e^{-t})^{-(q-1)}}{\left(1-h(t) \right)^2 }=4,
$$ where $h(t)$ is defined in \eqref{zis}. 
Thus, there exists $t_*\in (0,1)$ small such that for all $ t\in (0,t_*)$, 
\begin{equation} \label{coll3}  t^q (1-e^{-t})^{-(q-1)}\leq 3 \left(1-h(t) \right)\  \mbox{and }\ 
t^{q+1} (1-e^{-t})^{-(q-1)}\leq 5 \left(1-h(t)\right)^2 . 
\end{equation}
Hence, from \eqref{cer}, \eqref{coll2} and \eqref{coll3}, we see that  for every $ r\in (0, t_*^{1/\varepsilon})$, 
\begin{equation} \label{coll4} 
r^\theta (z_\delta(r))^q\leq 3 \,c^{q-1} \varphi^2_0(r^\varepsilon) \,\frac{z_\delta(r)}{r^2} \leq 
\varepsilon a \varphi_0(r^\varepsilon) \frac{z_\delta(r)}{r^2}
\end{equation}
in Case $[N_1]$, as well as  in Case $ [N_2]$ with  $\rho>\Upsilon$. 

In turn, if Case $[N_2]$ holds with $\rho=\Upsilon$, then for every $ r\in (0, t_*^{1/\varepsilon})$, we have
\begin{equation} \label{coll41} 
r^\theta (z_\delta(r))^q\leq 5 \,c^{q-1} \varphi_0(r^\varepsilon) \,\frac{z_\delta(r)}{r^2} \leq 
\varepsilon^2 a \varphi^2_0(r^\varepsilon) \frac{z_\delta(r)}{r^2}.
\end{equation}

By combining \eqref{coll1} and \eqref{coll4} in Case $[N_1]$, as well as  in Case $ [N_2]$ with  $\rho>\Upsilon$,  
respectively, using \eqref{coll11} and \eqref{coll41} in Case $[N_2]$ with $\rho=\Upsilon$,  
we arrive at \eqref{sun}, where $r_\varepsilon=t_*^{1/\varepsilon}$. 

The proof of Lemma~\ref{l53} (and, hence, of Theorem~\ref{zic1}) is now complete. \end{proof}


\section{Asymptotic profiles via solutions (super-solutions) of $\mathbb L_{\rho,\lambda,\tau}[\cdot]=0$} \label{ditto}

\subsection{On the operator $\mathbb L_{\rho,\lambda,\tau}[\cdot]$} \label{appab}
We give useful properties of  $\Phi_{\varpi_2(\rho)}$ and $\Psi_{\varpi_1(\rho)}$ in Case $[N_2]$. 

$\bullet$ 
If $\rho>\Upsilon$ in Case $[N_2]$, the functions $\Psi_{\varpi_1(\rho)}(r)$ and $\Phi_{\varpi_2(\rho)}(r)$ defined for every $r>0$ by 
$$ \Psi_{\varpi_1(\rho)}(r)=r^{-\varpi_1}\quad   \mbox{and}\quad   \Phi_{\varpi_2(\rho)}(r)=r^{-\varpi_2},$$ are positive radial solutions of 
$\mathbb L_{\rho,\lambda,\tau}[\cdot] =0$
 in $\mathbb R^N\setminus \{0\}$. 

\vspace{0.2cm}
In the rest of this subsection, we assume Case $[N_2]$ with $\rho=\Upsilon$, that is,  
\begin{equation} \label{case2} \lambda>0,\  \ \tau\in (0,1)\ \mbox{ and  }  \rho=\Upsilon=\frac{\tau+1}{2\tau} (\lambda \tau)^{\frac{1}{1+\tau}}.
\end{equation} 
Then, $\Phi_{\varpi_2(\rho)}:(0,1)\to (0,\infty)$
and $\Psi_{\varpi_1(\rho)}:(1,\infty)\to (0,\infty)$ are given by (see \eqref{phi2} and \eqref{psia})
\begin{eqnarray}  
& \Phi_{\varpi_2(\rho)}(r)=r^{-\varpi_2} [\log\, (1/r)]^{\frac{2}{1+\tau}}
\label{kis} \quad \mbox{for every } r\in (0,1), \label{siga} \\
& \Psi_{\varpi_1(\rho)}(t) =t^{-\varpi_1} (\log t)^{\frac{2}{1+\tau}}\quad \mbox{for every } t>1. \label{fasd}
\end{eqnarray}

Our arguments  rely frequently on explicit positive radial {\em super-solutions} $\Phi$ of 
 \begin{equation} \label{fncx}
    \mathbb L_{\rho,\lambda,\tau}[ \Phi(x) ]= 0\quad \mbox{for } |x|>0\ \mbox{small} 
\end{equation}
or positive radial {\em super-solutions} $\Psi$ of 
\begin{equation} \label{modal} \mathbb L_{\rho,\lambda,\tau}[ \Psi(x) ]= 0\quad \mbox{for } |x|>1\ \mbox{large}. 
\end{equation}

 $\bullet$ 
The function $\Phi=\Phi_{\varpi_2(\rho)}$ in \eqref{siga} is a positive radial {\em super-solution} of \eqref{fncx} (see Lemma~\ref{fiqa}). 

$\bullet$ 
For every $\alpha\in (0, 2/(1+\tau))$, it can be checked that 
$\Psi=\Psi_\alpha(r)=r^{-\varpi_1} \left(\log r\right)^\alpha $ (defined for $r>1$) 
is a positive radial {\em super-solution} of \eqref{modal}. This result does {\em not} extend to $\alpha=2/(1+\tau)$. 

$\bullet$ 
If $\beta<\varpi_1$, then  $\Psi_{\varpi_1(\rho)}$ in \eqref{fasd} is a 
{\em sub-solution} of \eqref{e1} for $|x|>1$ large (see Lemma~\ref{siww}). 

$\bullet$ For \eqref{modal}, we provide a positive radial {\em super-solution} that is asymptotically equivalent to 
$\Psi_{\varpi_1(\rho)}(|x|)$ as $|x|\to \infty$, see Lemma~\ref{feed}. This fact plays a critical role in Corollary~\ref{blabla}. 
When $\beta<\varpi_1$, we will use Lemma~\ref{feed}  to rule out the existence of positive solutions $u$ of \eqref{e1} in $\mathbb R^N\setminus \{0\}$ satisfying $\lim_{|x|\to \infty} u(x)/\Psi_{\varpi_1(\rho)}(|x|)=0$ (see Proposition~\ref{gild}).
 
 \vspace{0.2cm}
 Note that under the assumptions in \eqref{case2}, for $j=1,2$, we have 
\begin{equation} \label{varo}
\begin{aligned}
& \varpi_1=\varpi_2=-(\lambda \tau)^{1/(\tau+1)},\quad 
f_{\rho}(\varpi_j)=\varpi_j^2+2\rho \varpi_j+\lambda |\varpi_j|^{1-\tau}=0 ,\\
& 
2\left(\varpi_j+\rho\right) =\lambda \left(1-\tau\right) |\varpi_j|^{-\tau}.
\end{aligned}
\end{equation}

\begin{lemma} \label{fiqa}
Let $\rho=\Upsilon$ in Case $[N_2]$. Then, there exists $r_*\in (0,1)$ such that $\Phi_{\varpi_2(\rho)}$ in \eqref{siga} is a positive radial super-solution of \eqref{fncx} for $0<|x|<r_*$ and $\Phi'_{\varpi_2(\rho)}(r)>0$ for all $r\in (0,r_*)$. 
	
\end{lemma}

\begin{proof}
Fix $r_*\in (0,1)$ such that 
$ \log\,(1/r_* )>2/[\left(1+\tau\right)|\varpi_2|] 
$. 	
Let $r\in (0,r_*)$ be arbitrary. Then,  
\begin{equation} \label{gol5} 
\Phi_{\varpi_2(\rho)}'(r)=\frac{\Phi_{\varpi_2(\rho)}(r)}{r} \left( -\varpi_2-\frac{2}{(1+\tau)} \frac{1}{\log \,(1/r)}\right)>0.  
\end{equation}
Moreover, it follows that 
$$ \begin{aligned}  \mathbb L_{\rho,\lambda,\tau} [\Phi_{\varpi_2(\rho)}(r)]=\frac{\Phi_{\varpi_2(\rho)}(r)}{r^2}&  \left[ 
\varpi_2^2 +2\rho \varpi_2 +\frac{4 \left(\varpi_2+\rho\right)}{(1+\tau)} \frac{1}{\log (1/r )} 
+\frac{2\left(1-\tau\right)}{(1+\tau)^2} \frac{1}{\log^2 (1/r)}\right.\\
& \left.  +\lambda |\varpi_2|^{1-\tau} \left(1-\frac{2}{(1+\tau) |\varpi_2| } \frac{1}{\log \,(1/r )} \right)^{1-\tau}\right]. 
\end{aligned}
$$

Using \eqref{varo} and $(1-\varepsilon)^{1-\tau}\leq 1-\left(1-\tau\right) \varepsilon-\tau \left(1-\tau\right) \varepsilon^2/2$ for every $\varepsilon\in (0,1)$, we get 
$$  \mathbb L_{\rho,\lambda,\tau} [\Phi_{\varpi_2(\rho)}(|x|)]\leq 0\quad \mbox{for every } |x|\in (0,r_*). $$
This finishes the proof of Lemma~\ref{fiqa}. 
\end{proof}

\begin{lemma} \label{siww}
Let \eqref{e2} hold and $\beta<\varpi_1$ in Case $[N_2]$ with 
$ \rho=\Upsilon$. 
Then,  
there exists $R_*>1$ such that  
the function $\Psi_{\varpi_1(\rho)}$ in \eqref{fasd} 
is a positive radial sub-solution of \eqref{e1} in $ \mathbb R^N\setminus \overline {B_{R_*}(0)}$. 
\end{lemma}

\begin{proof}
For the reader's convenience, we provide the calculations. 
For every $t>0$ small, we have
\begin{equation} \label{dazi} (1+ t)^{1-\tau} \geq 1+\left(1-\tau\right) t -\frac{\tau \left(1-\tau\right) t^2}{2} +\frac{\tau\left(1-\tau^2\right) t^3}{8}.
\end{equation} 

Since $\beta<\varpi_1$ and $q>1$, there exists $R_*>1$ large such that
\begin{equation} \label{enw1}
r^{\theta+2-(q-1) \varpi_1}  (\log r)^{3+\frac{2(q-1)}{1+\tau}}\leq \frac{1-\tau}{(1+\tau)^2 |\varpi_1|}\quad \mbox{for all } r\geq R_*.
\end{equation}
If necessary, we increase $R_*>1$ to ensure that \eqref{dazi} holds for every $0<t< 2/[(1+\tau)|\varpi_1| \log R_*]$. 

Fix $r>R_*$ arbitrary. 
By differentiating \eqref{fasd}, we have
\begin{equation}  \label{shh} 
\begin{aligned}
& \Psi'_{\varpi_1(\rho)}(r)=|\varpi_1|\frac{\Psi_{\varpi_1(\rho)}(r)}{r} \left[1+\frac{2}{(1+\tau) |\varpi_1|\,\log r}
\right]\\
&  \mathbb L_{\rho} [\Psi_{\varpi_1(\rho)}(r)]=\frac{\Psi_{\varpi_1(\rho)}(r)}{r^2} \left[ \varpi_1^2+2\rho\, \varpi_1 -\frac{4\left(\varpi_1+\rho\right)}{(1+\tau)\, \log r}
 +\frac{2\left(1-\tau\right)}{(1+\tau)^2 (\log r)^2}
 \right],\\
 & 
 \lambda \,G_{\tau} [\Psi_{\varpi_1(\rho)}(r)]=\lambda |\varpi_1|^{1-\tau} \frac{\Psi_{\varpi_1(\rho)}(r)}{r^2}  \left[1+\frac{2}{(1+\tau) |\varpi_1|\,\log r}
\right]^{1-\tau}.
 \end{aligned}
\end{equation}
By the choice of $R_*$, we find that
 $$ \left[1+\frac{2}{(1+\tau) |\varpi_1|\,\log r}
\right]^{1-\tau}\geq 1+\frac{2\left(1-\tau\right)}{(1+\tau) |\varpi_1|\,\log r}-\frac{2 \tau \left(1-\tau\right)}{(1+\tau)^2 \varpi_1^2 \,(\log r)^2}+ 
\frac{\tau \left(1-\tau\right)}{(1+\tau)^2 |\varpi_1|^3 (\log r)^3}. 
$$
Using this fact, jointly with \eqref{enw1} and \eqref{shh}, we arrive at
\begin{equation}
 \mathbb L_{\rho,\lambda,\tau} [\Psi_{\varpi_1(\rho)}(r)]\geq 
 \frac{\Psi_{\varpi_1(\rho)}(r)}{r^2} \frac{(1-\tau)}{(1+\tau)^2 |\varpi_1| (\log r)^3}\geq r^{\theta} \Psi^q_{\varpi_1(\rho)}(r)
\end{equation}
for every $r>R_*$. This completes the proof of Lemma~\ref{siww}. 
\end{proof}

\begin{lemma} \label{feed} 
Assume Case $[N_2]$ with $\rho=\Upsilon$. Then, there exists $R_1>1$ large such that  
\begin{equation}
\label{stay}
\widetilde \Psi(r) =\widetilde \Psi_{\varpi_1}(r)=r^{-\varpi_1} \left(\log r\right)^{\frac{2}{1+\tau}}	
+r^{-\varpi_1} \log r\quad \mbox{for every } r>1
\end{equation}
 is a positive radial super-solution	 of 
$ \mathbb L_{\rho,\lambda,\tau} [\widetilde \Psi]=0$ in $ \mathbb R^N\setminus \overline{B_{R_1}(0)}$.  
\end{lemma}

\begin{rem}
{\rm (a) In \eqref{stay}, we could take the coefficient of $r^{-\varpi_1}\log r$ to be any positive constant $C$, in which case the constant $R_1>1$ will also depend on $C$.  

(b) If in each term of $\widetilde \Psi$, we replace $\log r$ by $\log \,(rR_1)$, then the new $\widetilde \Psi$ is a positive super-solution of $\mathbb L_{\rho,\lambda,\tau}[\cdot]=0$ in $\mathbb R^N\setminus \overline{B_1(0)}$.}	
\end{rem}

\begin{proof}
Fix $r>1$ arbitrary. We define $\mathfrak P(r)$ as follows 
$$ \mathfrak P (r)= \left(1+\frac{1}{(\log r)^{\frac{1-\tau}{1+\tau}}}\right)^\tau
\left[  1+\frac{1}{(\log r)^{\frac{1-\tau}{1+\tau}}}
+\frac{2}{(1+\tau)\,|\varpi_1| \log r} +
\frac{1}{(\log r)^{\frac{2}{1+\tau}}}
\right]^{1-\tau}.
$$

By a routine calculation, we arrive at 
$$
\begin{aligned} 
&  \widetilde \Psi'(r)=
-\varpi_1\, r^{-\varpi_1-1} \left(\log r\right)^{\frac{2}{1+\tau}}  \left[
1+ 
\frac{1}{(\log r)^{\frac{1-\tau}{1+\tau}}}+
\frac{2}{\left(1+\tau\right) |\varpi_1| \log r} 
+ 
\frac{1}{(\log r)^{\frac{2}{1+\tau}}}
\right],\\
 & 	
\mathbb L_\rho[\widetilde \Psi(r)]=
  r^{-\varpi_1-2} \left(\log r\right)^{\frac{2}{1+\tau}}  \left[
  \varpi_1^2+2 \rho \,\varpi_1+\frac{\varpi^2_1+2\rho\varpi_1}{\left(\log r\right)^{\frac{1-\tau}{1+\tau}}} -\frac{4\left( \varpi_1+\rho\right)}{\left(1+\tau
  \right)\log r}-\frac{2\left( \varpi_1+\rho\right)}
  {\left(\log r\right)^{\frac{2}{1+\tau}}}
   \right],\\
   &  
   \lambda \,G_\tau[\widetilde \Psi(r)] =
\lambda \frac{\widetilde \Psi^\tau(r) \,|\widetilde \Psi'(r)|^{1-\tau}}{r^{1+\tau}}=
\lambda \,|\varpi_1|^{1-\tau} r^{-\varpi_1-2} \left( \log r\right)^{\frac{2}{1+\tau}}
 \mathfrak P(r).
   \end{aligned}
$$

We next use that $(1+b)^\tau\leq 1+\tau b-\tau \left(1-\tau\right) b^2/2$ and $(1+b)^{1-\tau}\leq 1+\left(1-\tau\right) b-\tau\left(1-\tau\right) b^2/2$ for every $b>0$. 
Hence, by taking $R_1>1$ large enough, we see that for every $r>R_1$, we have 
$$ \begin{aligned}
 \mathfrak P(r)& \leq  \left(
1+\frac{\tau}{(\log r)^{\frac{1-\tau}{1+\tau}}}
-\frac{\tau \left(1-\tau\right)}{2\left
(\log r\right)^{\frac{2\left(1-\tau\right)}{1+\tau}}}
\right)\\
& \quad \times 
\left(
1+\frac{1-\tau}{(\log r)^{\frac{1-\tau}{1+\tau}}}
+\frac{2\left(1-\tau\right)}{(1+\tau)\,|\varpi_1| \log r}-\frac{\tau\left(1-\tau\right)}{2 \left(\log r\right)^{\frac{2\left(1-\tau\right)}{1+\tau}}}
 +\frac{ \left(1-\tau\right) 
 \left( 1-\frac{2\tau}{(1+\tau) |\varpi_1| }\right) }{(\log r)^{\frac{2}{1+\tau}}}\right)\\
 &\leq \left( 1+\frac{1}{(\log r)^{\frac{1-\tau}{1+\tau}}}
+\frac{2\left(1-\tau\right)}{\left(1+\tau\right) |\varpi_1| \log r} 
+\frac{1-\tau}{(\log r)^{\frac{2}{1+\tau}}} 
-\frac{\tau\left(1-\tau\right)}{2 \left( \log r\right)^{
\frac{ 3\left(1-\tau\right)}{1+\tau} }}
\right). 
\end{aligned} $$

In light of \eqref{varo}, we arrive at 
$$
\mathbb L_{\rho,\lambda,\tau} [\widetilde \Psi(r)]\leq 
 -\lambda \,|\varpi_1|^{1-\tau} \frac{\tau\left(1-\tau\right)}{2} 
 r^{-\varpi_1-2} \left(\log r\right)^{\frac{3\tau-1}{1+\tau}}<0\quad \mbox{for every } r>R_1 .
$$
This completes the proof of Lemma~\ref{feed}.  \end{proof}

\subsection{Asymptotic models provided by $\mathbb L_{\rho,\lambda,\tau}[\cdot]= 0$}

\begin{theorem} \label{urz1}
Let \eqref{e2} hold and  $\rho,\lambda,\theta\in \mathbb R$. 
Let $u$ be any positive solution of \eqref{e1}. 
 
Let $\Phi:(0,2r_0)\to (0,\infty)$ be a 
positive $C^1$-function for some small $r_0>0$ such that  
\begin{eqnarray}
& \mathbb L_{\rho,\lambda,\tau}[\Phi(r)]\leq 0\quad \mbox{for all } r\in (0,r_0), \label{xinb} \\
& \displaystyle \lim_{r\to 0^+} r^{\theta+2} \Phi^{q-1}(r)=0,\quad \lim_{r\to 0^+} \frac{r\,\Phi'(r)}{\Phi(r)}=-\eta\in \mathbb R\setminus \{0\}, \label{wiop}  \\	
& \displaystyle \limsup_{|x|\to 0} \frac{u(x)}{\Phi(|x|)}=\mu \in (0,\infty). \label{pomii}
\end{eqnarray}
If $f_{\rho}(\eta)\leq 0$, then we have 
\begin{equation}
\label{pomi2i}
\lim_{|x|\to 0} \frac{u(x)}{\Phi(|x|)}=\mu \quad \mbox{and}\quad \lim_{|x|\to 0} \frac{ x\cdot \nabla u(x)}{\Phi(|x|)}=-\mu\, \eta. 	
\end{equation}
\end{theorem}

\begin{proof}
Without loss of generality, we can assume that $r_0>0$ is small such that $B_{4r_0}(0)\subset \subset \Omega$ and $\Phi'(r)\not=0$ for all $r\in (0,r_0)$ in view of \eqref{wiop}. 
	For $r\in (0,r_0)$ fixed, we define 
	\begin{equation} \label{niom} V_{(r)}(\boldsymbol{\xi})=\frac{u(r\boldsymbol{\xi})}{\Phi(r)}\quad \mbox{for every } \boldsymbol{\xi}\in \mathbb R^N\ \mbox{with } 0<|\boldsymbol{\xi}|<r_0/r,
	\end{equation}

	We show that for every $ \boldsymbol{\xi}\in \mathbb R^N\setminus\{0\}$, we have 
	\begin{equation} \label{rioti} 
	 \lim_{r\to 0^+} V_{(r)}(\boldsymbol{\xi})=\mu |\boldsymbol{\xi}|^{-\eta},\quad \lim_{r\to 0^+} \nabla V_{(r)} (\boldsymbol{\xi})=
	-\mu\, \eta \,|\boldsymbol{\xi}|^{-\eta-2} \boldsymbol{\xi}.
	\end{equation}
	
	{\em Proof of \eqref{rioti}.} 
By the comparison principle, using \eqref{xinb} 
 and $\Phi'(r)\not=0$ for all $r\in (0,r_0)$, 
we can find  a constant $d>0$, depending only on $r_0, \mu,N,q,\theta,\rho, \lambda, \tau$ and $\Phi(r_0)$, such that 
\begin{equation} \label{niou} u(x)\leq d \,\Phi(|x|) \quad \mbox{ for all } 0<|x|\leq r_0.\end{equation}
(We can take $d=\max\,\{2\mu,C_0\, r_0^{-\beta}/\Phi(r_0)\}$, where $C_0$ is given by \eqref{c0}.) 

We write $\Phi(r)=r^{-\eta} L(r)$ for $r\in (0,r_0)$, where $L\in C^1((0,r_0))$ is a normalised slowly varying function, that is, $\lim_{r\to 0^+} r L'(r)/L(r)=0$. It is known that 
	$$ \lim_{r\to 0^+} L(r s)/L(r)=1 \quad \mbox{ uniformly on each compact }s \mbox{-set in }(0,\infty).$$ 
Thus, from \eqref{niom} and \eqref{niou}, we get 
\begin{equation} \label{mot1i} 0< V_{(r)}(\boldsymbol{\xi})\leq d \,|\boldsymbol{\xi}|^{- \eta} \frac{L(r|\boldsymbol{\xi}|)}{L(r)}\quad \mbox{for all } 0<|\boldsymbol{\xi}|<r_0/r.
\end{equation}
In light of Proposition~\ref{aurr}, there exist positive constants $d_2$, $d_3$ and $\alpha\in (0,1)$, depending only on $N,q,\theta,\rho, \lambda$ and 
$\tau$ such that \eqref{noi2} and \eqref{noi3} hold, which imply that 
\begin{equation} \label{mot2i} |\nabla V_{(r)}(\boldsymbol{\xi})|\leq d_2 \frac{V_{(r)}(\boldsymbol{\xi})}{|\boldsymbol{\xi}|} \quad
\mbox{and}\quad | \nabla V_{(r)}(\boldsymbol{\xi})- \nabla V_{(r)}(\boldsymbol{\xi}') | \leq d_3 \frac{\left( 
V_{(r)}(\boldsymbol{\xi}) +V_{(r)}(\boldsymbol{\xi}') \right)  
|\boldsymbol{\xi} -\boldsymbol{\xi}'|^\alpha}{|\boldsymbol{\xi}|^{1+\alpha}}
\end{equation} for every $\boldsymbol{\xi},\boldsymbol{\xi}'$ in $\mathbb R^N$ such that $0<|\boldsymbol{\xi}|\leq |\boldsymbol{\xi}'| <r_0/r$. 
Observe that $V_{(r)}$ is a positive solution of 
	\begin{equation} \label{sof}
	\mathbb L_{\rho,\lambda,\tau} [V_{(r)}(\boldsymbol{\xi})]=r^{\theta+2} \Phi^{q-1}(r) |\boldsymbol{\xi}|^\theta V_{(r)}^q(\boldsymbol{\xi})\quad
	\mbox{for }  0<|\boldsymbol{\xi}|<r_0/r.
	\end{equation}
Since $\lim_{r\to 0^+} r^{\theta+2} \Phi^{q-1}(r)=0$ (from \eqref{wiop}), 
using \eqref{mot1i} and \eqref{mot2i}, for every sequence $\{\overline r_n\}_{n\geq 1}$ decreasing to $0$, there exists a subsequence $\{r_n\}_{n\geq 1}$ such that
\begin{equation} \label{fanci} V_{(r_n)}\to V\ \mbox{ in } C^1_{\rm loc}(\mathbb R^N\setminus \{0\}),\end{equation} where $V$ is a non-negative $C^1(\mathbb R^N\setminus \{0\})$ solution of 
\begin{equation} \label{soppi} \mathbb L_{\rho,\lambda,\tau} [V(\boldsymbol{\xi})]=0 \quad \mbox{in } \mathcal D'(\mathbb R^N\setminus \{0\}). 
\end{equation}
We next show that $V(\boldsymbol{\xi})=\mu \,|\boldsymbol{\xi}|^{- \eta}$ for all $\boldsymbol{\xi}\in \mathbb R^N\setminus \{0\}$. 
To this end, we define 
	\begin{equation} \label{viro}  v(r)=\sup_{|x|=r} \frac{u(x)}{\Phi(|x|)}\quad \mbox{for every } r\in (0,r_0).
	\end{equation}		
	By \eqref{pomii} and 
	Lemma~\ref{mar11} in Appendix~\ref{sectiune2}, we have $ \lim_{r\to 0^+} v(r)=\mu\in (0,\infty)$.
	For each $n\geq 1$, let $\boldsymbol{\xi}_{r_n}\in \mathbb S^{N-1}$ (the unit sphere in $\mathbb R^N$) be such that 
	$ v(r_n)= u(r_n \boldsymbol{\xi}_{r_n})/\Phi(r_n)$.  
Without loss of generality, by passing to a subsequence, relabelled $\{\boldsymbol{\xi}_{r_n}\}_{n\geq 1}$, we can assume that 
$\boldsymbol{\xi}_{r_n}\to \boldsymbol{\xi}_0\in \mathbb S^{N-1}$ as $n\to \infty$. 
Using the definitions of $V_{(r)}$ and $v(r)$, we see that for all $n\geq 1$, 
$$ \begin{aligned} 
&  |\boldsymbol{\xi}|^{\eta} V_{(r_n)}(\boldsymbol{\xi}) \leq v(r_n|\boldsymbol{\xi|}) \frac{L(r_n|\boldsymbol{\xi}|)}{L(r_n)} \quad \mbox{for all } 0<|\boldsymbol{\xi}| <r_0/r_n,\\
& |\boldsymbol{\xi}_{r_n}|^{ \eta} V_{(r_n)}(\boldsymbol{\xi}_{r_n}) =v(r_n).
\end{aligned} $$
By letting $n\to \infty$, we obtain that 
\begin{equation} \label{nusi} |\boldsymbol{\xi}|^{ \eta} V(\boldsymbol{\xi})\leq \mu\ 
\mbox{ for every } \boldsymbol{\xi}\in \mathbb R^N\setminus \{0\}\ \mbox{ and } V(\boldsymbol{\xi}_0)=\mu \in (0,\infty).
\end{equation}  
Then, by the strong maximum principle in Lemma~\ref{add} applied to the non-negative $C^1(\mathbb R^N\setminus \{0\})$ solution $V$ of \eqref{soppi}, we infer that $V>0$ in $\mathbb R^N\setminus \{0\}$. 
The assumption $f_{\rho}(\eta)\leq 0$ yields that
$$ \mathbb L_{\rho,\lambda,\tau}[\mu |\boldsymbol{\xi}|^{-\eta}]=\mu f_{\rho}(\eta) \,|\boldsymbol{\xi}|^{-\eta-2}\leq 0\quad 
\mbox{for every } |x|>0.$$ 

By Lemma~\ref{adol} (see Appendix~\ref{Kel})  applied to the positive $C^1(\mathbb  R^N\setminus\{0\})$-solutions $u_1(\boldsymbol{\xi})=V(\boldsymbol{\xi})$ 
and $u_2(\boldsymbol{\xi})=\mu \,|\boldsymbol{\xi}|^{-\eta}$ of \eqref{soppi} with $|\nabla u_2(\boldsymbol{\xi})|\not=0$ for all 
$\boldsymbol{\xi}\in \mathbb R^N\setminus \{0\}$, we conclude that 
$$V(\boldsymbol{\xi})=\mu \,|\boldsymbol{\xi}|^{- \eta}\ \mbox{ for all } \boldsymbol{\xi}\in \mathbb R^N\setminus \{0\}.$$ 
Thus, using \eqref{fanci}, for every $\boldsymbol{\xi}\in \mathbb R^N\setminus \{0\}$, we obtain that 
$$ \lim_{n\to \infty} V_{(r_n)} (\boldsymbol{\xi}) =\mu\, |\boldsymbol{\xi}|^{- \eta}\quad \mbox{and}
 \lim_{n\to \infty} \nabla V_{(r_n)} (\boldsymbol{\xi}) =-\mu\,  \eta |\boldsymbol{\xi}|^{- \eta-2}\boldsymbol{\xi}. 
$$
Since $\{\overline r_n\}_{n\geq 1}$ was an arbitrary sequence decreasing to zero, we conclude the proof of \eqref{rioti}. 
Letting $|\boldsymbol{\xi}|=1$ and $x=r\boldsymbol{\xi}$ in \eqref{rioti}, we reach \eqref{pomi2i}.
\end{proof}

\begin{rem} \label{rek47} {\rm 
If $\tau=0$ in Theorem~\ref{urz1}, then in \eqref{wiop}, we can take any $\eta\in \mathbb R$. Indeed, when $\eta=0$, then we use the comparison principle in Lemma~\ref{co2} and conclude that $V=\mu$ in $\mathbb R^N\setminus \{0\}$ as in the proof of Theorem~\ref{rrzi} under the assumptions of $(a_1)$. 
}	
\end{rem}

We next extend the conclusions
of Theorem~\ref{urz1} to cover the case
$\Phi=1$. 

\begin{theorem} \label{rrzi}
Let \eqref{e2} hold, $\rho,\lambda\in \mathbb R$ and $\theta>-2$. Assume 
any one of the following situations: 
\begin{itemize} 
\item[$(a_1)$] Let $\tau=0$.  
\item[$(a_2)$] Let $\lambda>0$ and $\tau\in (0,1)$.   
\end{itemize}

Let $u$ be a positive solution of \eqref{e1} satisfying 
$ \liminf_{|x|\to 0} u(x)=\ell\in (0,\infty)$  or, equivalently, $ \limsup_{|x|\to 0} u(x)=\mu\in (0,\infty)$. 
 Then, there exists $\lim_{|x|\to 0} u(x)=\ell=\mu$. 
	\end{theorem}

\begin{rem}
{\rm In Corollary~\ref{peree} (see Appendix~\ref{sectiune2}), we prove that $ \liminf_{|x|\to 0} u(x)\in (0,\infty)$ is equivalent to $ \limsup_{|x|\to 0} u(x)\in (0,\infty)$. If such a solution exists, then necessarily $\theta>-2$. Indeed, if $\theta\leq -2$, then $\lim_{|x|\to 0} u(x)=0$ for every positive solution $u$ of \eqref{e1} (see Proposition~\ref{pzerop} if $\theta=-2$ and Lemma~\ref{api} in Appendix~\ref{sectiune2} if $\theta<-2$).}
\end{rem}

	\begin{proof}
We adapt the proof of Theorem~\ref{urz1},  taking
$\Phi(r)=1$ for every $r>0$. Clearly, from the assumption $\theta>-2$, we see that  \eqref{wiop} holds with $\eta=0$. However, we need to modify some arguments since here  $\Phi'(r)=0$ for every $r>0$.

Let $r_0>0$ be small such that $B_{4r_0}(0)\subset \subset \Omega$ and 
$$  \ell/2\leq u(x)\leq 2\mu  \quad \mbox{for every } 0<|x|\leq r_0.$$
Let $r\in (0,r_0) $ be fixed. For every 
$0<|\boldsymbol{\xi}|<r_0/r$, we define 
$V_{(r)}(\boldsymbol{\xi})$ as in \eqref{niom}, that is,  
\begin{equation} \label{nimf} V_{(r)}(\boldsymbol{\xi})=u(r\boldsymbol{\xi})\in [\ell/2,2\mu]. 
	\end{equation}
As before,   
we find that for every sequence $\{\overline r_n\}_{n\geq 1}$ decreasing to $0$, there exists a subsequence $\{r_n\}_{n\geq 1}$ such that $V_{(r_n)}\to V$ in $C^1_{\rm loc}(\mathbb R^N\setminus \{0\})$, where 
$V$ is a positive $C^1(\mathbb R^N\setminus \{0\})$-solution of \eqref{soppi}. 
We treat each of the situations $(a_1)$ and $(a_2)$ separately. 
	
	\vspace{0.2cm}
	($a_1$) Let $\tau=0$. 
	Let $v$ be as in  \eqref{viro} with $\Phi=1$. As in the proof of Theorem~\ref{urz1}, 
	we recover that $\lim_{r\to 0^+} v(r)=\mu$ (see Remark~\ref{rek44}). From \eqref{nusi} with $\eta=0$, we deduce that $$ V(\boldsymbol{\xi})=\mu \quad \mbox{for every } \boldsymbol{\xi}\in \mathbb R^N\setminus \{0\}$$ by applying the strong maximum principle in \cite{PS2007}*{Theorem~2.5.1, p.~34} to 
	$Z=\mu-V$, which is a non-negative $C^1(\mathbb R^N\setminus \{0\})$-solution of 
	$$ \Delta Z(x)-(N-2+2\rho)\frac{x\cdot \nabla Z(x)}{|x|^2} -\lambda \frac{|\nabla Z(x)|}{|x|}=0\quad \mbox{for every }x\in \mathbb R^N\setminus \{0\}. 
	$$  
	The rest of the argument in the proof of Theorem~\ref{urz1} carries over.  
	
	\vspace{0.2cm} 	
	($a_2$) Let $\lambda>0$ and $\tau\in (0,1)$.   
	We define $z(r)=\min_{|x|=r} u(x)$ for every $r\in (0,r_0)$.  
Since $\liminf_{r\to 0^+} z(r)=\ell=\liminf_{|x|\to 0} u(x)$, we prove that $\lim_{r\to 0^+} z(r)=\ell$ 
by establishing that
		\begin{equation} \label{ahop} \limsup_{r\to 0^+} z(r)=\ell. 
	\end{equation}  
	
{\bf Proof of \eqref{ahop}.}	
We use our assumptions to construct suitable positive radial sub-solutions (respectively, super-solutions) of \eqref{e1} in $B_{r_*}(0)\setminus \{0\}$ for small $r_*\in (0,r_0)$. 

For every $c_1>0$, we define $c_2=c_2(c_1,\theta,\lambda, q,\tau)>0$ as follows
	\begin{equation}  c_2=\frac{1-\tau}{\theta+2} \left(\frac{c_1^{q-\tau}}{\lambda}\right)^{\frac{1}{1-\tau}} . 
	\label{zarr}
	\end{equation} By a simple calculation, we obtain that for every $c\in \mathbb R$ satisfying $|c|>c_2$ (respectively, 
	$|c|<c_2$),  
	there exists $r_{*}=r_*(c_1,c)\in (0,r_0)$ small such that  
	\begin{equation} \label{sc1c} S_{c_1,c}(r)=c_1+ c \,r^{\frac{\theta+2}{1-\tau}} >0\quad  \mbox{for every } r\in (0,r_{*}) \end{equation} and $S_{c_1,c}$ is a positive radial sub-solution (respectively, super-solution) of \eqref{e1} in $B_{r_{*}}(0)\setminus \{0\}$.

Let $\ell_0:= \limsup_{r\to 0^+} z(r)$. 
	Suppose by contradiction that $\ell_0>\ell$. 
	Let
	$\{r_n\}_{n\geq 1}$ be a sequence of positive numbers decreasing to $0$ as $n\to \infty$ with
	$ \lim_{n\to \infty} z(r_n)=\ell_0$. 
	
	Fix $\varepsilon>0$ small such that 
	$\ell+\varepsilon<\ell_0$.  
	In what follows, we need a positive radial sub-solution of \eqref{e1} in $B_{r_*}(0)\setminus \{0\}$; thus in the definition of $S_{c_1,c}(r)$, we take $c_1=\ell+\varepsilon$ and fix $c\in \mathbb R$ with $|c|>c_2$. 
	Since $\theta>-2$ and $\tau\in (0,1)$, we have 
	$ \lim_{r\to 0^+} S_{c_1,c}(r)=c_1=\ell+\varepsilon<\ell_0$.  

Let $n_0\geq 1$ be large so that $r_{n_0}<r_*$ and 
 $ z(r_n)\geq S_{c_1,c}(r_n)$ for all $n\geq n_0$.  
	Thus, for every $n\geq n_0$, we have
	$u(x)\geq z(r_n)\geq S_{c_1,c}(|x|)$  for every 
		$x\in \partial B_{r_n}(0)$. 
	Since $ S_{c_1,c}'(r)\not=0$ for every $r\in (0,r_*)$, by the comparison principle (see Remark~\ref{reka6}), for all $n>n_0$, we get
	$$ 
  u(x)\geq S_{c_1,c} (|x|) \quad \mbox{for every } r_n\leq |x|\leq r_{n_0}. 
 $$
	By letting $n\to \infty$, we get $\liminf_{|x|\to 0} u(x)\geq \ell+\varepsilon$, which is a contradiction. \qed

	\vspace{0.2cm} 
	To conclude that $\lim_{|x|\to 0} u(x)=\ell$,
	it is enough to show that 
	 \begin{equation} \label{ceva} V(\boldsymbol{\xi})=\ell \quad \mbox{for all }\  \boldsymbol{\xi}\in \mathbb R^N\setminus \{0\}.\end{equation}

	{\bf Proof of \eqref{ceva}.}
	For each $n\geq 1$, let $\boldsymbol{\xi}_{r_n}\in \mathbb S^{N-1}$ be such that $z(r_n)=u(r_n \boldsymbol{\xi}_{r_n})$. Hence, up to a subsequence, relabelled $\{\boldsymbol{\xi}_{r_n}\}_{n\geq 1}$, we have $\boldsymbol{\xi}_{r_n}\to \boldsymbol{\xi}_0\in \mathbb S^{N-1}$ as $n\to \infty$. Then, for each $n\geq 1$,  
	$$  
  V_{(r_n)}(\boldsymbol{\xi}) \geq z(r_n|\boldsymbol{\xi|}) \quad \mbox{for all } 0<|\boldsymbol{\xi}| <r_0/r_n
\quad \mbox{and}\quad 
  V_{(r_n)}(\boldsymbol{\xi}_{r_n}) =z(r_n).
 $$
	By letting $n\to \infty$, we arrive at $  V(\boldsymbol{\xi})\geq \ell $  
for every $\boldsymbol{\xi}\in \mathbb R^N\setminus \{0\}$  and $V(\boldsymbol{\xi}_0)=\ell\in (0,\infty)$. 
	
	We define $W(x)=V(x)-\ell\geq 0$ for every $x\in \mathbb R^N\setminus \{0\}$. 
	Then, 
	we get that $W\in C^1(\mathbb R^N\setminus \{0\})$ is a non-negative distribution solution of 
	$$ \Delta W-(N-2+2\rho) \frac{x\cdot \nabla W(x)}{|x|^2}+\lambda \frac{\left(W(x)+\ell\right)^\tau |\nabla W(x)|^{1-\tau}}{|x|^{1+\tau}}=0
	\quad \mbox{in } \mathbb R^N\setminus \{0\}. 
	$$
	Since $\lambda>0$, the strong maximum principle holds (see \cite{PS2007}*{Theorem~2.5.1, p.~34}). 
	Thus, since $W(\boldsymbol{\xi}_0)=0$ for $\boldsymbol{\xi}_0\in \mathbb S^{N-1}$, 
we obtain that $W(x)=0$ for every $x\in \mathbb R^N\setminus \{0\}$. Hence, for every $\boldsymbol{\xi}\in \mathbb R^N\setminus \{0\}$, we attain \eqref{ceva}, which leads to $\lim_{r\to 0^+} V_{(r)}(\boldsymbol{\xi})=\ell$. This proves that $\lim_{|x|\to 0} u(x)=\ell$.
The proof of Theorem~\ref{rrzi} is complete. 
	\end{proof}

\begin{rem} {\rm 
Under the assumptions of $(a_2)$ in Theorem~\ref{rrzi}, we could not extend the argument in part ($a_1$). Indeed, we see that $Z=\mu-V$ is a non-negative $C^1(\mathbb R^N\setminus \{0\})$-solution of 
 $$ \Delta Z(x)-(N-2+2\rho)\frac{x\cdot \nabla Z(x)}{|x|^2} -\lambda \frac{[Z(x)]^\tau |\nabla Z(x)|^{1-\tau}}{|x|^{1+\tau}}=0\quad \mbox{for every }x\in \mathbb R^N\setminus \{0\}. 
	$$ 
When $\lambda>0$ and $\tau\in (0,1)$, then 
the conditions for applying the strong maximum principle (\cite{PS2007}*{Theorem~2.5.1, p.~34}) would not have been satisfied.} 
\end{rem}

\begin{cor} \label{urz}
Let \eqref{e2} hold and $\rho,\lambda,\theta\in \mathbb R$.  
Let $u$ be an arbitrary positive solution of \eqref{e1}.
\begin{itemize} 
\item[(a)] If $\beta>\varpi_1$ in Case $[N_j]$ with $j=1,2$ and 
$\limsup_{|x|\to 0} |x|^{\varpi_1} u(x)=a\in (0,\infty)$, 
then 
\begin{equation}
\label{pomi2}
\lim_{|x|\to 0} |x|^{\varpi_1} u(x)=a\quad \mbox{and}\quad \lim_{|x|\to 0} |x|^{\varpi_1} x\cdot \nabla u(x)=-a\,\varpi_1. 	
\end{equation}
\item[(b)] If $\beta>\varpi_2$ in Case $[N_2]$ and 
$\limsup_{|x|\to 0} u(x)/\Phi_{\varpi_2(\rho)}(|x|)=b\in (0,\infty)$, 	
then 
$$ \lim_{|x|\to 0} \frac{u(x)}{ \Phi_{\varpi_2(\rho)}(|x|)}=b \quad \mbox{and}\quad
\lim_{|x|\to 0} \frac{x\cdot \nabla u(x)}{ \Phi_{\varpi_2(\rho)}(|x|)}=-b\, \varpi_2. 
 $$ (For the definition of $\Phi_{\varpi_2(\rho)}$, see \eqref{phi2}.) 
 \end{itemize}
	\end{cor}
	
\begin{proof}
We apply Theorem~\ref{urz1} with $\Phi(r)=r^{-\varpi_1}$ for (a) and $\Phi(r)=\Phi_{\varpi_2(\rho)}(r)$ for (b) (see also Lemma~\ref{fiqa} when $\rho=\Upsilon$ in Case $[N_2]$). 
\end{proof}

\begin{rem} {\rm By applying the modified Kelvin transform, we can restate Corollary~\ref{urz} in terms of the behaviour at infinity for the positive solutions of \eqref{e1} in a domain of $\mathbb R^N$ containing 
$\mathbb R^N\setminus \overline{B_R(0)}$ for some $R>0$. }
\end{rem}

\begin{cor} \label{blabla} Let \eqref{e2} hold. Assume that $\beta<\varpi_1$ in Case $[N_2]$. Let $u$ be a positive solution of \eqref{e1} in a domain of $\mathbb R^N$ containing 
$\mathbb R^N\setminus \overline{B_R(0)}$ for some $R>0$. If $u$ satisfies
\begin{equation} \label{aoll}
\limsup_{|x|\to \infty} \frac{u(x)}{\Psi_{\varpi_1(\rho)}(|x|)}:=c\in (0,\infty),
\end{equation} 
where $\Psi_{\varpi_1(\rho)}$ is given by \eqref{psia}, then we have 
\begin{equation} \label{gydd}
\lim_{|x|\to \infty} \frac{u(x)}{\Psi_{\varpi_1(\rho)}(|x|)}=c\quad 
\mbox{and}\quad \lim_{|x|\to \infty} \frac{x\cdot \nabla u(x)}{\Psi_{\varpi_1(\rho)}(|x|)}=-c\, \varpi_1.
\end{equation}
\end{cor}

\begin{proof} 
We apply the modified Kelvin transform (see Appendix~\ref{Kel}). We let 
$ v(x)=u(x/|x|^2)$ for $ 0<|x|<1/R$. 
Then, defining $\widetilde \theta=-\theta-4$ and $\widetilde \rho=-\rho$, we find that $v$ is a positive solution of 
\begin{equation} \label{muss} \mathbb L_{\widetilde \rho,\lambda,\tau} [v]=|x|^{\widetilde \theta} v^q(x)\quad \mbox{for every } 0<|x|<1/R. 
\end{equation}
For every $r\in (0,1)$, we define  
$$ \Phi(r)=
r^{\varpi_1} (\log \,(1/r))^{\frac{2}{1+\tau}} +r^{\varpi_1} \log \,(1/r)\quad \mbox{if } \rho=\Upsilon,\quad 
\Phi(r)=r^{\varpi_1}\ \mbox{if } \rho>\Upsilon. $$

$\bullet$  If $\rho=\Upsilon$, then Lemma~\ref{feed} gives $r_0\in (0,1)$ small such that 
$$ \mathbb L_{\widetilde \rho,\lambda,\tau} [\Phi(r)]\leq 0 \quad \mbox{for every } r\in (0,r_0).$$

$\bullet$ If $\rho>\Upsilon$, then $ \mathbb L_{\widetilde \rho,\lambda,\tau} [\Phi(r)]=0$ for every $r\in (0,1)$. 

In either case, the assumption \eqref{aoll} means that 
$ \limsup_{|x|\to 0} v(x)/\Phi(|x|)=c\in (0,\infty)$.  
Hence, we conclude \eqref{gydd} from Theorem~\ref{urz1} applied to $v$ satisfying \eqref{muss}.  
\end{proof}


\begin{lemma} \label{serio}
Let \eqref{e2} hold. Assume that $\rho\in \mathbb R$, $\theta>-2$, $\tau=0$ and $\lambda=-2\rho$. 
Let $u$ be any positive solution of \eqref{e1} and $c:=\limsup_{|x|\to 0} u(x)/\log \,(1/|x|)$.  
\begin{itemize}
\item[(i)] If $c=\infty$, then with $\beta=(\theta+2)/(q-1)>0$ and 
$f_\rho(\beta)$ as in \eqref{frol}, 
we have 
\begin{equation} \label{pao1} \lim_{|x|\to 0} |x|^\beta u(x)=\beta^{\frac{2}{q-1}}=[f_\rho(\beta)]^{\frac{1}{q-1}}=\Lambda . \end{equation}
\item[(ii)] 
If $c\in (0,\infty)$, then 
 \begin{equation} \label{pao0} \lim_{|x|\to 0} \frac{u(x)}{\log \,(1/|x|)}=c. 
 \end{equation}

\item[(iii)] If $c=0$, then 
there exists $\lim_{|x|\to 0} u(x)\in [0,\infty)$. 
\end{itemize}	
\end{lemma}

\begin{rem}
{\rm When $\Omega=\mathbb R^N$ in Lemma~\ref{serio}, we show later in Proposition~\ref{ze-ta} that the alternative (iii) cannot occur. In addition, when $c=\infty$, then $u\equiv U_{\rho,\beta}$ in $\mathbb R^N\setminus \{0\}$. }
\end{rem}

\begin{proof} 
Note that $\mathbb L_{\rho,\lambda,\tau}[\log \,(1/|x|)]=0$ for every $0<|x|<1$ since $\tau=0$ and $\lambda=-2\rho$.  

\vspace{0.2cm}
(i) Let $c=\infty$. By Corollary~\ref{pere} (see also Remark~\ref{rek44}), we get  \begin{equation} \label{fasj} \lim_{|x|\to 0} \frac{u(x)}{\log \,(1/|x|)}=\infty.
 \end{equation}	

Fix $r_0\in (0,1)$ such that $B_{r_0}(0)\subset \subset \Omega$. 
Let $d_1>1$ be a constant as in the spherical Harnack-type inequality of Proposition~\ref{aurr}. Using a similar argument to \cite{Cmem}*{Lemma 6.4}, (for every positive solution $u$ of \eqref{e1}), we construct a 
positive radial solution $V$ of \eqref{e1} in $B_{r_0}(0)\setminus \{0\}$ such that 
\begin{equation} \label{nozu} u(x)/d_1\leq V(|x|)\leq u(x) \quad \mbox{for every } 0<|x|\leq r_0. \end{equation}
For the reader's convenience, we provide the details. 
Fix $n_0>1$ large such that $1/n_0<r_0$. 
For every integer $n\geq n_0$, we consider the following boundary value problem
\begin{equation}
\left\{ \begin{aligned}
& \mathbb L_{\rho,\lambda,\tau}[v]=|x|^\theta v^q && \mbox{for every } 1/n<|x|<r_0,& \\
& v(x)=\min_{\partial B_{1/n}(0)} u && \mbox{for } |x|=1/n,&\\
& v(x)=\min_{\partial B_{r_0}(0)} u && \mbox{for } |x|=r_0.&   
 \end{aligned}
\right.	\label{nun1}
\end{equation}

Remark that $u$ (respectively, $u/d_1$) is a positive super-solution (respectively, sub-solution) of \eqref{nun1}.  
Let $V_n$ be the {\em maximal} positive solution of \eqref{nun1} with respect to the pair ($u/d_1, u$) of sub-super-solutions of \eqref{nun1}. By the radial symmetry of the domain and the data in \eqref{nun1}, as well as the invariance of the problem under rotations, we obtain that $V_n$ is radially symmetric. Moreover, by the sub-super-solutions method and the maximality property of $V_n$, we have
$$  u(x)/d_1\leq V_{n+1}(|x|)\leq V_n(|x|)\leq u(x)\quad \mbox{for every } 1/n\leq |x|\leq r_0. 
$$
For every $0<r\leq r_0$, we define 
\begin{equation} \label{geas} V(r)=\lim_{n\to \infty} V_n(r).	
\end{equation}
 
By virtue of Proposition~\ref{aurr}, we infer that $V_n\to V$ in $C^1_{\rm loc}(B_{r_0}(0)\setminus \{0\})$ and $V$ is a positive radial solution of \eqref{e1} in $B_{r_0}(0)\setminus \{0\}$ satisfying $V(r_0)=\min_{\partial B_{r_0}(0)} u$ and \eqref{nozu}. 

From \eqref{fasj} and \eqref{nozu}, we see that 
\begin{equation} \label{bion} \lim_{r\to 0^+} \frac{V(r)}{\log \,(1/r)}=\infty.\end{equation}
Moreover, there exists a small constant $r_1\in (0,r_0)$ such that $V'(r)<0$ for every $r\in (0,r_1)$. Indeed, if $V'(r_*)=0$ for some $r_*\in (0,r_0)$, then $\Delta V(r_*)=r_*^\theta [v(r_*)]^q>0$. 

We have obtained that $V$ is a positive radial solution of 
\begin{equation} \label{mira1} \Delta V=|x|^\theta V^q\quad \mbox{in } D_{r_1}(0)
\setminus \{0\}, 
\end{equation}  where by $D_{r_1}(0)$ we denote the open ball in $\mathbb R^2$ centered at $0$ and radius $r_1$. Since $V$ satisfies \eqref{bion}, from \cite{CD2010}*{Theorem 1.1}, it follows that 
\begin{equation} \label{mira2} \lim_{r\to 0^+} r^\beta V(r)= \beta^{\frac{2}{q-1}}=[f_\rho(\beta)]^{\frac{1}{q-1}}=\Lambda. 
\end{equation} 
Using \eqref{nozu}, we obtain that 
$$ \liminf_{|x|\to 0} |x|^\beta u(x)\geq \Lambda. $$

In \eqref{nun1}, we replace $\min_{\partial B_{1/n}(0)} u$ by $\max_{\partial B_{1/n}(0)}u$ and $\min_{\partial B_{r_0}(0)} u$
by $\max_{\partial B_{r_0}(0)}u $. 
With this change in \eqref{nun1}, we let $v_n$ be the {\em minimal} positive solution with respect to the pair $(u,d_1 u)$ of sub-super-solutions. Then, as before, $v_n$ is radially symmetric and $$ u(x)\leq v_{n}(x)\leq v_{n+1}(x)\leq d_1 u(x) \quad 
\mbox{for every } 1/n\leq |x|\leq r_0.$$ 
We define $v(r)=\lim_{n\to \infty} v_n(r)$ for $0<r\leq r_0$. For $r_1>0$ small enough, we obtain that $v$ is a positive radial solution of \eqref{mira1} satisfying \eqref{bion} and $u(x)\leq v(|x|)\leq d_1 u(x)$ for all $0<|x|\leq r_0$. Therefore, \eqref{mira2} holds for $v$ instead of $V$ so that 
$$\limsup_{|x|\to 0} |x|^\beta u(x)\leq \Lambda.$$ 
This concludes the proof of \eqref{pao1}. 

\vspace{0.2cm}
(ii) 
Let $c\in (0,\infty)$. 
 By Theorem~\ref{urz1} with $\Phi(r)=\log \,(1/r)$ and Remark~\ref{rek47}, we obtain \eqref{pao0}.

\vspace{0.2cm}
(iii) Let $c=0$. We first show that $\limsup_{|x|\to 0} u(x)<\infty$.  Assume the contrary. Then, by Corollary~\ref{peree}, we have
$\lim_{|x|\to 0} u(x)=\infty$. Following exactly the same ideas as in part (i), we obtain that $V$ in \eqref{geas} is a positive radial solution of \eqref{mira1} satisfying 
$\lim_{r\to 0^+} V(r)=\infty$ and \eqref{nozu}. By \cite{CD2010}*{Theorem 1.1}, we have either \eqref{mira2} or there exists $\mu\in (0,\infty)$ such that $\lim_{r\to 0^+} V(r)/\log \,(1/r)=\mu$. Either of these situations leads to a contradiction with $c=0$. 

Hence, $\limsup_{|x|\to 0} u(x)<\infty$. This, jointly with 
Theorem~\ref{rrzi} leads to $\lim_{|x|\to 0} u(x)\in [0,\infty)$. 
The proof of Lemma~\ref{serio} is now complete.     
\end{proof}


\section{Proof of Assertion ($N$) in Theorem~\ref{thi}} \label{errv}
The main result in this section is the following.

\begin{theorem} \label{axiv1} 
Let \eqref{e2} hold. Assume that either 
$\beta\in (\varpi_1,0)$ in Case $[N_1]$ or $\beta\in (\varpi_1,\varpi_2]$ in Case $[N_2]$ with $\rho>\Upsilon$. 
Then, for every positive solution $u$ of \eqref{e1}, there exists $a\in \mathbb R_+$ such that 
\begin{equation} \label{cine}
\lim_{|x|\to 0} |x|^{\varpi_1} u(x)=a\quad \mbox{and}\quad \lim_{|x|\to 0} |x|^{\varpi_1} x\cdot \nabla u(x)=-a\,\varpi_1. 	
\end{equation} 
\end{theorem} 

\begin{proof} 
Let $u$ be any positive solution of \eqref{e1}.  
By Corollary~\ref{urz}, we get \eqref{cine} by showing that 
\begin{equation} a:=\limsup_{|x|\to 0} |x|^{\varpi_1} u(x)\in (0,\infty). 
\end{equation}

In light of Theorem~\ref{zic1}, we have $a>0$. It remains to prove that $a<\infty$. This is done next. 

\begin{proposition} \label{axv1} 
In the settings of Theorem~\ref{axiv1}, 
every positive sub-solution $u$ of \eqref{e1} satisfies 
\begin{equation}  \label{hho}
\limsup_{|x|\to 0}|x|^{\varpi_1} u(x)<\infty . 
\end{equation}
\end{proposition}

Let $\varepsilon>0$ be arbitrary. 
For every $r>0$, we define $V_\varepsilon(r)$ by
\begin{equation} \label{vio} V_\varepsilon(r)=\left\{
\begin{aligned}
& \varepsilon+M r^{-\varpi_1}  && \mbox{in Case } [N_1], & \\
& \varepsilon r^{-\beta} + M r^{-\varpi_1}
&&  
 \mbox{in Case } [N_2] \ \mbox{with } \rho>\Upsilon. & 
\end{aligned} \right.
\end{equation}

\begin{lemma} \label{lema53} For arbitrary $M>0$ and $R>0$, the function $V_\varepsilon$ in \eqref{vio} is a positive radial super-solution of \eqref{e1} in $B_{R}(0)\setminus \{0\}$.
\end{lemma}

\begin{proof} We separate the proof of Case $[N_1]$ from that in Case $[N_2]$. 

\vspace{0.2cm}
$\bullet$ If Case~$[N_1](b)$ holds (that is, $\lambda<0$ and $\tau\in (0,1)$), we find that 
\begin{equation} \label{nm1}  \mathbb L_{\rho,\lambda,\tau}[V_\varepsilon(r)]\leq 
\mathbb L_{\rho,\lambda,\tau}[M\,r^{-\varpi_1}]=0<r^\theta (V_\varepsilon(r))^q\quad \mbox{for each } r>0.
\end{equation}

$\bullet$ If Case~$[N_1](a)$ holds (that is, $\tau=0$ and $\lambda<2\rho$), then
we recover \eqref{nm1} in which the first inequality becomes equality. Thus, in Case $[N_1]$, from \eqref{nm1}, we see that, for every $M>0$,  the function $V_\varepsilon$
is a positive radial super-solution of \eqref{e1} in $B_{R}(0)\setminus \{0\}$ with $R>0$ arbitrary. 

\vspace{0.2cm}
$\bullet$ In Case $[N_2]$ (when $\rho\not=\Upsilon$), by a standard computation, for every $r>0$, we arrive at 
$$ \begin{aligned} 
\mathbb L_{\rho,\lambda,\tau}[V_\varepsilon(r)]=& M \varpi_1\left( \varpi_1+2\rho\right) r^{-\varpi_1-2}
+\varepsilon \beta\left(\beta+2\rho \right) r^{-\beta-2} \\
&+
\lambda M |\varpi_1|^{1-\tau} r^{-\varpi_1-2} \left( 1+\frac{\varepsilon r^{\varpi_1-\beta}}{M}
\right)^\tau \left( 1+\frac{\varepsilon r^{\varpi_1-\beta}}{M} \,\frac{\beta}{\varpi_1}\right)^{1-\tau}.
\end{aligned} $$
Using that $\lambda>0$, $\varpi_1<\beta<0$ and $f_{\rho}(\beta)\leq 0=f_{\rho}(\varpi_1)$, we obtain that 
$$ \begin{aligned}  \mathbb L_{\rho,\lambda,\tau}[V_\varepsilon(r)] &< 
M \varpi_1( \varpi_1+2\rho) r^{-\varpi_1-2}
+\varepsilon \beta(\beta+2\rho) r^{-\beta-2}+ \lambda M |\varpi_1|^{1-\tau} r^{-\varpi_1-2} \left( 1+\frac{\varepsilon r^{\varpi_1-\beta}}{M}
\right)\\
& =M r^{-\varpi_1-2} f_{\rho}(\varpi_1) +\varepsilon r^{-\beta-2} f_{\rho}( \beta) \\
&=
\varepsilon r^{-\beta-2} f_{\rho}( \beta)\leq 0<r^\theta [V_\varepsilon(r)]^q
\end{aligned}
$$
for every $r>0$. This concludes the proof of Lemma~\ref{lema53} in Case $[N_j]$ with $j\in \{1,2\}$. 
\end{proof}

\vspace{0.2cm}
{\bf Proof of Proposition~\ref{axv1} concluded.} Let $r_0\in (0,1/e)$ be small such that 
$B_{r_0}(0)\subset\subset \Omega$. 
In the definition of $V_\varepsilon$, we fix $M>M_0$ large as needed, depending on $u$ and $r_0$, such that 
\begin{equation} \label{fim2} 
 \max_{x\in  \partial B_{r_0}(0)}u(x) \leq V_0(r_0):=Mr_0^{-\varpi_1}. 
 \end{equation}
Thus, for every $\varepsilon>0$, we have $u\leq V_\varepsilon$ on $ \partial B_{r_0}(0)$. 

From Corollary~\ref{zerop}, we have 
\begin{equation} \label{eraq} \lim_{|x|\to 0} |x|^{\beta} u(x)=0.
\end{equation} 
In view of \eqref{eraq}, we infer that for every $\varepsilon>0$, there exists $r_\varepsilon\in (0,r_0)$ small such that 
$$ u(x)\leq \varepsilon |x|^{- \beta} < V_\varepsilon(|x|)\quad \mbox{ for all } 0<|x|\leq r_\varepsilon.$$ Using that $V_\varepsilon$ is a positive radial super-solution in $B_{r_0}(0)\setminus \{0\}$ and $V'_\varepsilon(r)\not=0$ for all $r\in (0,r_0)$, by the comparison principle, we arrive at 
$u(x)\leq V_\varepsilon(x)$ for all $ 0<|x|\leq r_0$. By letting $\varepsilon\to 0$, we conclude the proof of Proposition~\ref{axv1}. 
\end{proof}



\section{Proof of Assertions~$(Z_1)$ and $(Z_2)$ in Theorem~\ref{thi}} \label{sept} 

Throughout this section, we let $\beta=\varpi_1<0$ in Case $[N_j]$ 
with $j\in \{1,2\}$. 
We define
\begin{equation} \label{pgam} \begin{aligned} 
& p=\frac{1}{q-1},\  \gamma=\left(
\frac{ |\varpi_1| (1-\lambda \tau |\varpi_1|^{-\tau-1} )}{q-1}\right)^{\frac{1}{q-1}} &&
\mbox{in Case } [N_1]\ \mbox{or if } \rho>\Upsilon\ \mbox{ in Case } [N_2], &\\
& p=\frac{2}{q-1},\  \gamma=\left(
\frac{2\left(q+\tau\right)}{(q-1)^2}\right)^{\frac{1}{q-1}} && \mbox{if } \rho=\Upsilon\ \mbox{ in Case } [N_2].& \\
\end{aligned} 
\end{equation}

\begin{theorem} \label{law2} Let \eqref{e2}  hold and $\beta=\varpi_1$ in Case $[N_j]$ with $j\in \{1,2\}$. 
Let $p$ be given by \eqref{pgam}. 
\begin{itemize}
\item[(A)] Every  positive super-solution $u$ of \eqref{e1} satisfies 
\begin{equation} \label{nodd}
\liminf_{|x|\to 0} |x|^{\varpi_1}  |\log |x||^p \, u(x)>0.
 \end{equation}
 \item[(B)] 
Every positive sub-solution $u$ of \eqref{e1} satisfies 
 \begin{equation} \label{babm} 
  \limsup_{|x|\to 0} |x|^{\varpi_1} \left| \log |x| \right| ^p u(x)<\infty .
\end{equation}
 \end{itemize} 
\end{theorem}

\begin{theorem} \label{law1} Let \eqref{e2} hold and $\beta=\varpi_1$ in Case $[N_j]$ with $j\in \{1,2\}$. 
\begin{itemize}
\item[(a)] Then, 
every positive sub-solution $u$ of \eqref{e1} satisfies 
\begin{equation} \label{bbb}
 \limsup_{|x|\to 0} |x|^{\varpi_1}  |\log |x||^p \, u(x) \leq \gamma. 
\end{equation} 

\item[(b)] Every 
positive super-solution $u$ of \eqref{e1} satisfies 
\begin{equation} \label{wow}
\liminf_{|x|\to 0} |x|^{\varpi_1}  |\log |x||^p \, u(x)\geq \gamma. 
\end{equation}
\end{itemize} 
\end{theorem}

\subsection{Proof of Theorem~\ref{law2} (A)}
\label{secc5}

We adapt the ideas in the proof of Theorem~\ref{zic1}. We divide the proof into three Steps.

\vspace{0.2cm}
{\bf Step 1.} 
{\em Let $u$ be any  
positive super-solution of \eqref{e1}. Then, }
\begin{equation} \label{simo2} \mbox{\it for every } \eta>0,\ \mbox{\it we have } \lim_{|x|\to 0} |x|^{\varpi_1-\eta} u(x)=\infty .\end{equation}
\begin{proof}[Proof of Step~1.] We distinguish two situations: 

\vspace{0.2cm}
$\bullet$ For Case~$[N_1]$, as well as Case $[N_2]$ with $\rho\geq \lambda^{1/(\tau+1)}$, we gain \eqref{simo2} 
via Proposition~\ref{nwp2}. 

$\bullet$ For Case $[N_2]$ with $\Upsilon\leq \rho<\lambda^{1/(\tau+1)}$, we have already proved \eqref{simo2} in Proposition~\ref{gain}. 

\begin{proposition}
\label{nwp2} 
Let \eqref{e2} hold and $\beta=\varpi_1$ 
in either Case $[N_1]$ or Case $[N_2]$ with $\rho\geq \lambda^{1/(\tau+1)}$.  
Then, every positive super-solution $u$ of \eqref{e1} satisfies 
\begin{eqnarray}
& \liminf_{|x|\to 0} |x|^{\varpi_1} \left|\log |x| \right|^{1/(q-1)} u(x)>0 &  \mbox{in  Case $[N_1]$,}  \label{fas1}  \\
& \liminf_{|x|\to 0} |x|^{\varpi_1} \left | \log |x|\right |^{2/(q-1)} u(x)>0 &  \mbox{in  Case $[N_2]$.}  \label{fas2}
\end{eqnarray} 
\end{proposition}

\begin{proof} We imitate the proof of Proposition~\ref{nwp1} in which we assign the following meaning to $\mu$: 
\begin{equation} \label{muk} \mu=\left\{ 
\begin{aligned}
& 1/(q-1) && \mbox{if } \beta=\varpi_1\ \mbox{in Case } [N_1],&\\
& 2/(q-1) && \mbox{if } \beta=\varpi_1\ \mbox{in Case } [N_2]\ \mbox{with } \rho\geq \lambda^{1/(\tau+1)}. &
\end{aligned}
\right.
\end{equation} 

We define $c_0=c_0(\mu)>0$ as follows
\begin{equation} \label{furo} c_0=\left[\mu\, |\varpi_1| \right]^{\frac{1}{q-1}} 
 \ \  \mbox{in Case } [N_1],\quad 
c_0= \left[\mu \left(\mu+1\right)\right]^{\frac{1}{q-1}} 
\ \  \mbox{in Case } [N_2]. \end{equation} 

Let $c\in (0,c_0)$ and $\delta\in (0,1/e)$ be arbitrary. For every $\delta\leq r\leq 1/e$, we let $u_\delta(r)$ be as in \eqref{eqo1}. 
The proof of Lemma~\ref{lema64} and of \eqref{omina} in Proposition~\ref{nwp1} remain valid if we understand that $\mu$ and $c_0$ have here the meaning in \eqref{muk} and \eqref{furo}, respectively. From \eqref{omina}, we readily conclude the proof of Proposition~\ref{nwp2}. 
\end{proof}

This completes the proof of Step~1. \end{proof}

{\bf Step 2.} {\em Construction of a suitable family of sub-solutions $\{z_{\delta,\varepsilon}\}_{\delta\in (0,1/2)}$ of \eqref{e1}.}

\begin{proof}[Proof of Step~2.]
Let $\varepsilon>0$ and $c>0$ be fixed. For every $\delta\in (0,1/2)$ and $r\in (0,1)$, we define 
\begin{equation} \label{zis0} z_{\delta,\varepsilon}(r):= c \, r^{\varepsilon-\varpi_1} 
\left(1 -e^{-r^\varepsilon \log^p \left(1/r\right)}\right)^{\delta-1}.  
\end{equation}

\begin{lemma} \label{lema83} 
Under the assumptions of Theorem~\ref{law2}, there exists $c_*>0$ such that for every $\varepsilon>0$ small, 
there exists $r_\varepsilon\in (0,1)$ such that for all $c\in (0,c_*)$ and $\delta\in (0,1/2)$, the function $z_{\delta,\varepsilon} $ in \eqref{zis0} a sub-solution of \eqref{e1} in $B_{r_\varepsilon}(0)\setminus \{0\}$ and, moreover,  
$z'_{\delta,\varepsilon}(r)>0$ for every $r\in (0,r_\varepsilon)$. 
\end{lemma}

\begin{proof} The calculations are in the spirit of Lemma~\ref{l53}. Let $a>0$ be given by \eqref{defa}.  
We set 
$$ c_*=  \left(\frac{a}{2}\right)^{\frac{1}{q-1}} \left(\frac{p}{2}\right)^{p}.$$ 

We fix $c\in (0,c_*)$ and 
$\varepsilon>0$ small as needed in both Case $[N_1]$ and Case $[N_2]$. 
Let $\delta\in (0,1/2)$ be arbitrary. 
We define $h$ and $\varphi_\delta$ as in \eqref{zis}. For every $r\in (0,1)$, we define 
$$
T_{\delta,\varepsilon}(r)=
\varepsilon+\left(\delta-1\right) h\left(r^\varepsilon \log^p(1/r) \right) 
\left(\varepsilon-\frac{p}{\log \,(1/r)}\right)>0.
$$ 

Remark that using $\varphi_\delta$ in \eqref{zis}, we have
\begin{equation} \label{nutt} 
 T_{\delta,\varepsilon}(r)=\varepsilon \varphi_\delta(r^\varepsilon \log^p (1/r)) 
+\frac{p\left(1-\delta\right) h\left(r^\varepsilon \log^p (1/r)\right)}{\log \,(1/r)}\\
<\varepsilon+
\frac{p}{\log \,(1/r)}.
\end{equation}
 
In Case $[N_2]$, we let $\nu\in (0,1)$ and $t_\nu>0$ small enough such that  \eqref{hart1} holds. Moreover, we fix $\varepsilon\in (0,t_\nu/2)$ and $r_\nu\in (0,1/e)$ small 
such that 
$$ \frac{p}{\log \,(1/r_\nu)}<t_\nu/2.  
$$ 
Hence, in Case $[N_2]$, we have $T_{\delta,\varepsilon}(r)<t_\nu$ for all $r\in (0,r_\nu)$ so that from \eqref{hart1}, we find that
\begin{equation} \label{ono} 
\left(1+\frac{T_{\delta,\varepsilon}(r)}{|\varpi_1|} \right)^{1-\tau}\geq 1+\frac{1-\tau}{|\varpi_1|} T_{\delta,\varepsilon}(r) 
-\frac{\tau \left(1-\tau\right) \left(1+\nu\right)}{2\varpi_1^2}  T^2_{\delta,\varepsilon}(r) .
\end{equation}

 Since $\lim_{t\to 0^+} h(t)=1$ (see \eqref{zis}), there exists 
$t_*>0$ small such that 
$$ \frac{t}{\left(1-e^{-t}\right) h^p(t)}<2^{1/(q-1)}\quad \mbox{for all } t\in (0,t_*).$$ 

Let 
$r_\varepsilon\in (0,1)$ be small 
enough such that $r^\varepsilon \log^p (1/r)< t_*$ for every $r\in (0,r_\varepsilon)$. In addition, in case $[N_2]$, we let 
$r_\varepsilon<r_\nu$. 
So, for every $c\in (0,c_*)$ and $r\in (0,r_\varepsilon)$, we get 
$$ c^{q-1}  \left( \frac{r^{\varepsilon} \log^p(1/r) } {1-e^{- r^\varepsilon \,\log^p(1/r)}  } \right)^{q-1}
 \leq a \left( p/2\right)^{p(q-1)} \left[h \left(r^\varepsilon \log^p(1/r) \right)\right]^{p(q-1)}. 
$$

In both Case $[N_1]$ and Case $[N_2]$, using that $\varpi_1=\beta$, for all $r\in (0,r_\varepsilon)$, we get that 
\begin{equation} \label{hagg} \begin{aligned}
r^\theta z_{\delta,\varepsilon}^{q}(r)& =c^{q-1} 
\left[r^{\varepsilon}\log^p (1/r)\right]^{q-1} (1-e^{-r^\varepsilon \log^p(1/r)})^{(\delta-1)(q-1)} \frac{z_{\delta,\varepsilon}(r)}{r^2\log^{p(q-1)} \,(1/r)} \\
& \leq c^{q-1} \left(\frac{
r^{\varepsilon} \log^p (1/r) } {1-e^{-r^\varepsilon \log^p (1/r)}} \right)^{q-1}
 \frac{z_{\delta,\varepsilon}(r)}{r^2\log^{p\left(q-1\right)} \,(1/r)}\\
 & \leq a \left(  \frac{p\, h\left(r^\varepsilon \log^p (1/r) \right)}{2\,\log \,(1/r)}\right)^{p\left(q-1\right)} \frac{z_{\delta,\varepsilon}(r)}{r^2}. 
\end{aligned}
\end{equation}
 
From \eqref{nutt}, using that $\delta\in (0,1/2)$, we have 
\begin{equation} \label{barr} T_{\delta,\varepsilon}(r) \, \log \,(1/r) \geq   p\left(1-\delta\right) h\left(r^\varepsilon \log^p (1/r)\right) >
\frac{p}{2}\, h\left(r^\varepsilon \log^p (1/r)\right).
\end{equation}

To conclude the proof of Lemma~\ref{lema83}, it suffices to show that 
\begin{equation} \label{har} \mathbb L_{\rho,\lambda,\tau}[z_{\delta,\varepsilon}(r)] \geq 
\frac{a \, z_{\delta,\varepsilon}(r) }{r^2} [T_{\delta,\varepsilon}(r)]^{p\,(q-1)}
\quad \mbox{for all } r\in (0,r_\varepsilon).
\end{equation}

Indeed, by combining
\eqref{hagg}, \eqref{barr} and \eqref{har}, we arrive at
$$ \mathbb L_{\rho,\lambda,\tau}[z_{\delta,\varepsilon}(r)] \geq r^\theta [z_{\delta,\varepsilon}(r)]^q\quad \mbox{for all } r\in (0,r_\varepsilon)\ 
\mbox{and } \delta\in (0,1/2). $$
Hence, $z_{\delta,\varepsilon} $ in \eqref{zis0} a sub-solution of \eqref{e1} in $B_{r_\varepsilon}(0)\setminus \{0\}$. 

\vspace{0.2cm}
{\bf Proof of \eqref{har}.} 
Fix $r\in (0,r_\varepsilon)$. Defining $\varphi_{\delta}$ as in \eqref{zis} and differentiating \eqref{zis0}, we get
 \begin{equation} \label{name} 
z_{\delta,\varepsilon}'(r)= |\varpi_1|\frac{z_{\delta,\varepsilon}(r)}{r} \left( 1+ T_{\delta,\varepsilon}(r) \right)>0.
\end{equation}
 
Using \eqref{name} and \eqref{rumm}, in a similar way to \eqref{revv}, we obtain that 
\begin{equation} \label{wuv} \begin{aligned}
 \mathbb L_\rho[z_{\delta,\varepsilon}(r)]& = \frac{z_{\delta,\varepsilon}(r)}{r^2} \left[
\varpi_1^2+2\rho \varpi_1
- 2\left( \varpi_1+\rho\right) T_{\delta,\varepsilon}(r) 
+T_{\delta,\varepsilon}^2(r) +\frac{\left(1-\delta\right)p \,h\left(r^\varepsilon \,\log^p(1/r) \right)}{\log^2(1/r)}
\right. \\
& \qquad \qquad \left. 
\ \ +\left(\delta-1\right) h'(r^\varepsilon \log^p (1/r) ) \,r^\varepsilon \log^p (1/r)  \left(\varepsilon-\frac{p}{\log\,(1/r)}\right)^2
\right].
\end{aligned}
\end{equation}
From the definition of $h$ in \eqref{zis}, we find that 
$$h'(t)=\frac{e^{-t} \left(1-e^{-t} -t\right)}{(1-e^{-t})^2}<0\quad \mbox{for all } t>0. $$
Hence, using that $\delta\in (0,1/2)$ and $h(t)>0$ for all $t>0$, we get
\begin{equation} \label{hunn}  \mathbb L_\rho[z_{\delta,\varepsilon}(r)]\geq  \frac{z_{\delta,\varepsilon}(r)}{r^2} \left[
\varpi_1^2+2\rho \varpi_1
- 2\,( \varpi_1+\rho)\, T_{\delta,\varepsilon}(r) 
\right].
\end{equation}
Moreover, in Case $[N_2]$, from \eqref{wuv}, we need a more precise inequality, namely 
\begin{equation} \label{wop} 
 \mathbb L_\rho[z_{\delta,\varepsilon}(r)] \geq \frac{z_{\delta,\varepsilon}(r)}{r^2} \left[
\varpi_1^2+2\rho \varpi_1
- 2\,( \varpi_1+\rho)\, T_{\delta,\varepsilon}(r) 
+T_{\delta,\varepsilon}^2(r) 
\right].
\end{equation}

$\bullet$ Let Case $[N_1]$ hold. In this case, we see that 
$$ \lambda \,G_\tau (z_{\delta,\varepsilon}(r))\geq 
\lambda\,|\varpi_1|^{1-\tau}\, \frac{z_{\delta,\varepsilon}(r)}{r^2}\left[ 1+ \frac{ \left(1-\tau\right)}{|\varpi_1|}
T_{\delta,\varepsilon}(r)
\right].
$$
This, jointly with \eqref{hunn}, \eqref{xio} and \eqref{nutt}, leads to
\begin{equation} \label{nuq} 
 \mathbb L_{\rho,\lambda,\tau} [z_{\delta,\varepsilon}(r)]
 \geq \left(q-1\right) \gamma^{q-1}  \frac{z_{\delta,\varepsilon}(r)}{r^2} 
T_{\delta,\varepsilon}(r). 
  \end{equation}
Since $a=(q-1)\, \gamma^{p-1}$ and $p\,(q-1)=1$, from \eqref{nuq}, we conclude  \eqref{har} in Case $[N_1]$. 

\vspace{0.2cm}
$\bullet$ Let Case $[N_2]$ hold. From our choice of $\nu$, $\varepsilon$ and $r_\nu$, using \eqref{ono} and \eqref{name}, we have 
 \begin{equation} \label{wov1} \lambda \,G_\tau (z_{\delta,\varepsilon}(r))
\geq \lambda\, |\varpi_1|^{1-\tau}\frac{z_{\delta,\varepsilon}(r)}{r^2} \left[
1+\frac{ \left(1-\tau\right)} {|\varpi_1|} T_{\delta,\varepsilon}(r)-\frac{\tau\left(1-\tau\right)\left(1+\nu\right)}{2\varpi_1^2} \,T^2_{\delta,\varepsilon}(r)
\right].
\end{equation}
Similar to  
Lemma~\ref{l53}, from \eqref{wop} and \eqref{wov1}, we infer that
\begin{equation} \label{xx2}
\begin{aligned}
& \mathbb L_{\rho,\lambda,\tau} [z_{\delta,\varepsilon}(r)]\geq  \left(q-1\right)\gamma^{q-1}  \frac{z_{\delta,\varepsilon}(r)}{r^2} T_{\delta,\varepsilon}(r)
&& \mbox{if } \rho>\Upsilon,&\\
& \mathbb L_{\rho,\lambda,\tau} [z_{\delta,\varepsilon}(r)]
 \geq  \left[ 1-\frac{\left(1-\tau\right)\left(1+\nu\right)}{2} \right] \frac{z_{\delta,\varepsilon}(r)}{r^2} 
  T^2_{\delta,\varepsilon}(r)
&& \mbox{if } \rho=\Upsilon . &
\end{aligned}
 \end{equation}
In view of the definition of $a$ in \eqref{defa}, from \eqref{xx2} we conclude   
the proof of \eqref{har}. \end{proof}

This finishes the proof of Step~2. \end{proof}

{\bf Step 3.} {\em Every  
positive super-solution $u$ of \eqref{e1} satisfies \eqref{nodd}. } 

\begin{proof}[Proof of Step~3.] 
Let $\varepsilon>0$ be small and $z_{\delta,\varepsilon}(r)$ for $r\in (0,r_\varepsilon)$ be given as in Lemma~\ref{lema83}. 
If necessary, we diminish $c\in (0,c_*)$ such that 
$$\min_{|x|=r_\varepsilon} u(x)\geq c\, (r_\varepsilon)^{\varepsilon-\varpi_1} 
\left(1-e^{-(r_\varepsilon)^{\varepsilon} \log^p \left(1/r_\varepsilon\right) }\right)^{-1}.$$ This ensures that
$ u(x)\geq z_{\delta,\varepsilon}(|x|)$ for all $ x\in \partial B_{r_\varepsilon}(0)$  and every $\delta\in (0,1/2)$.  
From Step~1, we have $u(x)\geq z_{\delta,\varepsilon}(|x|)$ for every $|x|>0$ small, where $\delta\in (0,1/2)$ is arbitrary. 
In view of Lemma~\ref{lema83}, we can apply the comparison principle (Remark~\ref{reka6}) to find that  
$$  u(x) \geq z_{\delta,\varepsilon} (|x|) \quad \mbox{for every } 0<|x|\leq r_\varepsilon\ \mbox{and all } \delta\in (0,1/2). 
$$
By letting $\delta\to 0$ and then $|x|\to 0$, we get 
$\liminf_{|x|\to 0} |x|^{\varpi_1} \left[ \log \left(1/|x|\right)\right]^p u(x) \geq c>0$, which proves \eqref{nodd}. This finishes the proof of Step~3. 
\end{proof}
 
 From Steps 1--3, we conclude the proof of Theorem~\ref{law2} (A). \qed

\subsection{Proof of Theorem~\ref{law2} (B)}
Let $u$ be any positive sub-solution of \eqref{e1}. 
Corollary~\ref{zerop} yields \eqref{eraq}.  
We construct a suitable family of super-solutions  $\{V_\varepsilon\}_{\varepsilon>0}$ of \eqref{e1} that dominate $u$ near zero and 
on $\partial B_{r_0}(0)$ for $r_0>0$ small (and independent of $\varepsilon$).

\vspace{0.2cm}
We fix $M>0$ a large constant. 
Let $\varepsilon>0$ be arbitrary. 
For every $r\in (0,1/e]$, we define 
\begin{equation} \label{nuv2}
V_{\varepsilon} (r):=\left\{ \begin{aligned} 
& 
\varepsilon+M r^{-\varpi_1} 
 \left| \log r \right|^{-p} && \mbox{in Case } [N_1],&\\
 &  r^{-\varpi_1} \left[
\varepsilon+ M \left| \log  r \right|^{-p} \right]  && \mbox{in Case } [N_2].&
\end{aligned} \right. 
\end{equation}

\begin{lemma} \label{lema63}
There exists $M_0>0$ large, depending only on $\lambda,\tau,\rho$ and $q$, such that  
for every $M\geq M_0$, the function $V_\varepsilon$ in \eqref{nuv2} is a positive radial super-solution of \eqref{e1} in $B_{1/e}(0)\setminus \{0\}$.
\end{lemma}

\begin{proof}
In the definition of $V_\varepsilon$ in \eqref{nuv2}, we let $M\geq M_0$, where 
\begin{equation} \label{choi}
\begin{aligned}
& M_0 := \left[\frac{1}{q-1} \left(\frac{q}{q-1}-\varpi_1 +|\lambda| |\varpi_1|^{-\tau} \right)\right]^{\frac{1}{q-1}}
\  \mbox{in Case } [N_1] \ \mbox{or Case }[N_2] \ \mbox{with }  \rho>\Upsilon,\\
& M_0=\left[ \frac{2\left(q+1\right)}{\left(q-1\right)^2}\right]^{\frac{1}{q-1}} \ \mbox{in Case } [N_2]\ \mbox{with } \rho=\Upsilon.
\end{aligned}
\end{equation}

We show that $V_\varepsilon$ satisfies
\begin{equation} \label{ceck} \mathbb L_{\rho,\lambda,\tau}[  V_\varepsilon(r)]\leq r^{\theta}  [V_\varepsilon(r)]^q
\quad  \mbox{(and,\ moreover, } V'_\varepsilon(r)\not=0)  \  \mbox{for  every }  r\in (0,1/e). 
\end{equation}

{\bf Proof of \eqref{ceck} in Case $[N_1]$.} Let $V_\varepsilon$ be as in \eqref{nuv2} corresponding to Case $[N_1]$.

Let $r\in (0,1/e)$ be arbitrary. For simplicity, we write $V_0$ to mean $V_\varepsilon$ with $\varepsilon=0$, namely, 
\begin{equation} \label{vzz} 
V_0(r)=M r^{-\varpi_1} \left[\log \left(1/r\right)\right]^{-1/(q-1)} .
\end{equation}
Then,  $V_\varepsilon(r)=\varepsilon+V_0(r)$. 
The choice of $M_0$ in \eqref{choi}, jointly with $\varpi_1=\beta=(2+\theta)/(q-1)$, yields
\begin{equation} \label{giv} 
r^\theta V_\varepsilon^q(r)\geq r^\theta V_0^q(r)\geq \frac{V_0(r)}{r^2}
\phi(r) \left( \frac{q}{q-1}-\varpi_1 +|\lambda| |\varpi_1|^{-\tau}  
\right) ,
\end{equation}
where $\phi(r)$ is defined by 
\begin{equation} \label{safi} \phi(r):=\frac{1}{\left(q-1\right) \log \left(1/r\right)}.
\end{equation}
In view of \eqref{giv} and $\phi(r)\in (0,1/(q-1))$ for all $r\in (0,1/e)$, we attain \eqref{ceck} 
by showing that 
\begin{equation} \label{nomi}
 \mathbb L_{\rho,\lambda,\tau} [V_{\varepsilon}(r)]\leq \frac{V_0(r)}{r^2}\phi(r)\left(
-\varpi_1 +|\lambda| |\varpi_1|^{-\tau} 
 +q \,\phi(r)  \right). 
\end{equation}

{\em Proof of \eqref{nomi}}.  
It is clear that
\begin{equation} \label{exa}
    V_\varepsilon'(r) =V_0'(r)= \frac{V_0(r)}{r} \left(-\varpi_1\varpi_1+\phi(r) \right)>0. 
\end{equation}
Since $r \phi'(r)=\phi(r)/[\log \left(1/r\right)]=(q-1)\, \phi^2(r)$, by differentiating $V_\varepsilon'(r)$ with respect to $r$, we have
$$ V_\varepsilon''(r)=\frac{V_0(r)}{r^2} \left[ (-\varpi_1+\phi(r))^2+\varpi_1
-\phi(r) +(q-1) \,\phi^2(r) \right]. 
$$
Consequently, using the definition of $\mathbb L_\rho[\cdot]$ in \eqref{adel}, we arrive at
\begin{equation}\label{gg}
    \mathbb L_\rho [V_{\varepsilon}(r)] = \frac{V_0(r)}{r^2}\left[ 
    \varpi_1\left(\varpi_1+2\rho\right) -2\left(\varpi_1+\rho \right)\phi(r)+q\,\phi^2(r)
    \right].    
\end{equation}
Using \eqref{exa} and the positivity of $\phi(r)$, we see that 
\begin{equation}\label{gg_1}
    G_\tau (V_{\varepsilon}(r)) \geq G_\tau (V_0(r)) = \frac{(V_0(r))^{\tau}\left[\frac{V_0(r)}{r}(-\varpi_1 + \phi(r) \right]^{1-\tau}}{r^{\tau+1}} \geq 
     \frac{V_0(r)}{r^2}(-\varpi_1 +\phi(r) )^{1-\tau}.
\end{equation}

Recall that $\varpi_1<0$ satisfies $f_{\rho}(\varpi_1) = 0$, that is, $\varpi_1(\varpi_1+2\rho)+\lambda|\varpi_1|^{1-\tau}=0$.  

\vspace{0.2cm}
$\blacklozenge$ Let Case $[N_1](b)$ hold, that is, $\lambda<0$ and $\tau\in (0,1)$. 
From \eqref{gg} and \eqref{gg_1}, it follows that
\begin{equation} \label{pott} 
    \mathbb L_{\rho,\lambda,\tau} [V_{\varepsilon}(r)] \leq \frac{V_0(r)}{r^2}\left[ 
     \varpi_1\left(\varpi_1+2\rho\right) 
     + \lambda  \left(-\varpi_1 +\phi(r) \right)^{1-\tau}
     -2\left(\varpi_1+\rho \right)\phi(r)+q\,\phi^2(r)
   \right] . \end{equation}

Then, we have $\varpi_1+2\rho=\lambda|\varpi_1|^{-\tau}<0$. Thus, \eqref{pott} yields that
    $$  \mathbb L_{\rho,\lambda,\tau} [V_{\varepsilon}(r)]\leq  \frac{V_0(r)}{r^2}\phi(r)\left[-  2\left(\varpi_1 + \rho\right)+q\,\phi(r) \right]
    = \frac{V_0(r)}{r^2}\phi(r)\left(-\varpi_1+|\lambda| |\varpi_1|^{-\tau} +q \,\phi(r)  \right).
$$

This proves \eqref{nomi} in Case $[N_1](b)$. 

\vspace{0.2cm}
$\blacklozenge$  Let Case $[N_1](a)$ hold, that is, $\tau=0$ and $\lambda <2\rho$. 
In this case, $\varpi_1=\lambda-2\rho$. Then, all inequalities in \eqref{gg_1} and \eqref{pott} become equalities, implying that 
$$  \mathbb L_{\rho,\lambda,\tau} [V_{\varepsilon}(r)]=\frac{V_0(r)}{r^2} \phi(r) \left[ -2\varpi_1 +\lambda-2\rho+q\,\phi(r)\right]
= \frac{V_0(r)}{r^2}\phi(r)\left(-\varpi_1+q \,\phi(r)  \right). 
$$

So, we recover \eqref{nomi}, completing the proof of \eqref{ceck} in Case $[N_1]$.

\vspace{0.2cm}
{\bf Proof of \eqref{ceck} in Case $[N_2]$ with $\rho>\Upsilon$.}
Let $V_\varepsilon$ be as in \eqref{nuv2} corresponding to Case~$[N_2]$ with $\rho>\Upsilon$. 
Let $r\in (0,1/e)$ be arbitrary. We define $V_0(r)=\lim_{\varepsilon\to 0} V_\varepsilon(r)$, which has the same expression 
as in \eqref{vzz}.  Let $\phi$ be given by  \eqref{safi}. 
Then, by a simple calculation, we get
$$ V_\varepsilon'(r)=
\frac{V_\varepsilon(r)}{r} \left( -\varpi_1+\frac{V_0(r)}{V_\varepsilon(r)}\phi(r)\right)>0.
$$ Moreover, we have 
\begin{equation} \label{gioo} \begin{aligned}
 \mathbb L_{\rho,\lambda,\tau}[  V_\varepsilon(r)]&= V''_\varepsilon(r)+(1-2\rho)\,\frac{V_\varepsilon'(r)}{r}
 +\lambda \frac{[V_\varepsilon(r)]^\tau |V_\varepsilon'(r)|^{1-\tau}}{r^{1+\tau}}
 \\
& =
\frac{V_\varepsilon(r)}{r^2} 
\left( T_\varepsilon(r)
 + 
\frac{V_0(r)}{V_\varepsilon(r)}
\phi(r)
\left[ 
-2 \left(\varpi_1+\rho\right)  + q\,\phi(r) \right] \right), 
\end{aligned}
\end{equation}
where $T_\varepsilon(r)$ is defined by 
$$  T_\varepsilon(r):=\varpi_1 \left( \varpi_1+2\rho\right) 
+\lambda |\varpi_1|^{1-\tau}  
 \left( 
 1+\frac{\phi(r)}{|\varpi_1| }
\frac{V_0(r)}{V_\varepsilon(r)}
\right)^{1-\tau} .
$$
Since $\tau\in (0,1)$, we have $(1+a)^{1-\tau}\leq 1+a\left(1-\tau\right)$ for every $a>0$.  
Hence, using that 
$\varpi_1<0$ and $  \varpi_1^2+2\rho \varpi_1 +\lambda |\varpi_1|^{1-\tau}  =0$, along with the assumption
$\lambda>0$ in Case $[N_2]$, 
we find that 
\begin{equation} \label{gioo2} T_\varepsilon(r) \leq 
 \lambda |\varpi_1|^{-\tau}\left(1-\tau\right)  \phi(r)\,
\frac{V_0(r)}{V_\varepsilon(r)}.
\end{equation}
From \eqref{gioo} and \eqref{gioo2}, jointly with $r\in (0,1/e)$, we infer that 
$$ \begin{aligned} 
 \mathbb L_{\rho,\lambda,\tau}[  V_\varepsilon(r)] &\leq 
\frac{V_0(r)}{r^2} \phi(r)
\left[  \lambda |\varpi_1|^{-\tau}\left(1-\tau\right) -2\left(\varpi_1+\rho\right) + q \,\phi(r)
\right] \\
& \leq 
\frac{V_0(r)}{r^2} \phi(r)
\left[  \lambda |\varpi_1|^{-\tau}\left(1-\tau\right) -2\left(\varpi_1+\rho\right) +q/(q-1)
\right] .
\end{aligned} $$
Since $\lambda>0$ and $\lambda|\varpi_1|^{-\tau}=\varpi_1+2\rho$, we find that
$$  \mathbb L_{\rho,\lambda,\tau}[  V_\varepsilon(r)]\leq 
\frac{V_0(r)}{r^2} \phi(r) \left[-\varpi_1+q/(q-1)\right]. 
$$
On the other hand, using that $\varpi_1=\beta=(\theta+2)/(q-1)$, we see that 
$$ r^\theta [V_\varepsilon(r)]^q\geq r^\theta [V_0(r)]^q=
 M^q r^{-\varpi_1-2}  \left[ \log \left(1/ r\right)\right]^{-\frac{q}{q-1}}=(q-1)\,M^{q-1}\frac{V_0(r)}{r^2} \phi(r).
$$
Clearly, \eqref{ceck} holds 
by choosing $M\geq M_0$, where $M_0$ is given by \eqref{choi}. 
Hence, $V_\varepsilon$ (defined in \eqref{nuv2} in Case [$N_2$] with $\rho>\Upsilon$) is a positive radial super-solution of \eqref{e1} in $B_{1/e}(0)\setminus \{0\}$. 

\vspace{0.2cm}
{\bf Proof of \eqref{ceck} in Case $[N_2]$ with $\rho=\Upsilon$.}
Let $V_\varepsilon$ be as in \eqref{nuv2} corresponding to Case~$[N_2]$ with $\rho=\Upsilon$. 
Let $r\in (0,1/e)$ be arbitrary. We define $V_0(r)$ and $\varphi(r)$ by 
$$ V_0(r)=\lim_{\varepsilon\to 0} V_\varepsilon(r)=M r
^{-\varpi_1}  \left[ \log \left(1/ r\right) \right]^{-2/(q-1)}\quad \mbox{and}\quad 
\varphi(r)=\frac{2}{\left(q-1\right) \log (1/r)}.$$  

With similar calculations to those in Case $[N_2]$ with $\rho>\Upsilon$, we find that 
$$ \begin{aligned}
  \mathbb L_{\rho,\lambda,\tau}[  V_\varepsilon(r)]
=
\frac{V_\varepsilon(r)}{r^2} 
\left[ T_\varepsilon(r)
 + 
\frac{V_0(r)}{V_\varepsilon(r)}
\varphi(r)
\left(
-2 \varpi_1-2\rho  + \frac{q+1}{2}\,\varphi(r) \right) \right], 
\end{aligned}
$$
where $T_\varepsilon(r)$ is defined by 
$$ \begin{aligned} 
 T_\varepsilon(r)&= \varpi_1^2+2\rho \varpi_1
+\lambda |\varpi_1|^{1-\tau}  
 \left( 
 1+\frac{\varphi(r)}{|\varpi_1| }
\frac{V_0(r)}{V_\varepsilon(r)}
\right)^{1-\tau}\\
& \leq  \varpi_1^2+2\rho \varpi_1 +\lambda |\varpi_1|^{1-\tau} 
+\lambda |\varpi_1|^{-\tau} \left(1-\tau\right) \varphi(r) \frac{V_0(r)}{V_\varepsilon(r)}
\\
& =\lambda |\varpi_1|^{-\tau}  \left(1-\tau\right) \varphi(r) \frac{V_0(r)}{V_\varepsilon(r)}. 
\end{aligned}
$$
Since $\rho=\Upsilon$, we have $ \lambda |\varpi_1|^{-\tau}\left(1-\tau\right) -2\varpi_1-2\rho=0$. 
Consequently, we arrive at 
$$ \begin{aligned} 
 \mathbb L_{\rho,\lambda,\tau}[  V_\varepsilon(r)] &\leq 
\frac{V_0(r)}{r^2} \varphi(r)
\left[  \lambda |\varpi_1|^{-\tau}\left(1-\tau\right) -2\varpi_1-2\rho + \frac{q+1}{2} \,\varphi(r)
\right] \\
& =\frac{\left(q+1\right)}{2} 
\frac{V_0(r)}{r^2} \varphi^2(r)=\frac{2\left(q+1\right)}{(q-1)^2} M r^{-\varpi_1-2} \left( \log \frac{1}{r}\right)^{-\frac{2q}{q-1}} .
\end{aligned} $$
In light of $\varpi_1=(\theta+2)/(q-1)$, we infer that 
$$ r^{\theta}  [V_\varepsilon(r)]^q \geq r^\theta [V_0(r)]^q=M^q r^{-\varpi_1-2} \left | \log r \right |^{-\frac{2q}{q-1}} .
$$
Therefore, by defining $M_0$ as in \eqref{choi} and taking $M\geq M_0$, we conclude \eqref{ceck}. 

\vspace{0.2cm}
This ends the proof of Lemma~\ref{lema63} in Case [$N_j$] with $j\in \{1,2\}$. 
\end{proof}

To complete the proof of Theorem~\ref{law2} (B), we proceed exactly as in the proof of Proposition~\ref{axv1} (after Lemma~\ref{lema53}). We skip these details. \qed

\subsection{Proof of Theorem~\ref{law1}} 
\label{secc4}
Let $u$ be a positive super-solution of \eqref{e1} (respectively, sub-solution of \eqref{e1} satisfying \eqref{nodd}). Let $r_0\in (0,1/e)$ be small enough such that $B_{r_0}(0)\subset \subset \Omega$. 
We follow the same strategy as for proving (a) and (b) in Theorem~\ref{pro1}, where $\Lambda$ and $\beta$ need to be replaced by $\gamma$ in \eqref{pgam} and $\varpi_1$, respectively. We show only the modifications regarding the construction of the family $\{P_\eta^+\}_{\eta\in (0,\eta_C)}$ 
(respectively, $\{P_\eta^-\}_{\eta\in (0,\eta_C)}$) of super-solutions (respectively, sub-solutions) of \eqref{e1} in $B_{r_C}(0)\setminus \{0\}$. Here,
$r_C\in (0,r_0)$ is a small constant depending on the fixed constant $C$  with $C>\gamma$ in the case of $P_\eta^+$ (respectively, $C\in (0,\gamma)$ for $P_\eta^-$).  The functions $\{P_\eta^\pm \}_{\eta\in (0,\eta_C)}$  will satisfy the following properties 
(for every $\eta\in (0,\eta_C) $ in (i)--(iii)):
\begin{enumerate}
\item[(i)] $(P^\pm_\eta)'(r)\not=0$ for every $r\in (0,r_C)$;

\item[(ii)] $u(x)\leq P^+_\eta(|x|)$ (respectively, $u(x)\geq P_\eta^-(|x|)$) for all $x\in \partial B_{r_C}(0)$;

\item[(iii)] $u(x)\leq P^+_\eta(|x|)$ (respectively, $u(x)\geq P^-_\eta(|x|)$) for every $|x|>0$ small enough; 

\item[(iv)] $\lim_{|x|\to 0} |x|^{\varpi_1} \left[ \log (1/|x|)\right]^{p} (\lim_{\eta\searrow 0} P^\pm_\eta(|x|))=C$. 
\end{enumerate}

Let $C,\eta,\alpha$ and $\nu$ be positive constants. 
For every $r\in (0,1/e)$, we define $P_{\eta,\nu}^\pm(r)$ as follows
\begin{equation} \label{zara}
P_{\eta,\alpha,\nu}^\pm(r)=C r^{-\varpi_1} \left| \log r \right|^{-p\pm  \eta} \left( 1+\frac{ |\log r|^{-\alpha}}{\nu} \right)^{\pm 1}.
\end{equation}

Here, instead of Lemma~\ref{seco}, we show the following. 

\begin{lemma}[Sub-super-solutions] 
\label{secoo} 
In the framework of  Theorem~\ref{law1}, 
for every $C>\gamma$ (respectively, $C\in (0,\gamma)$), 
there exist $\eta_C,r_C,\alpha_C>0$ such that for every $\eta\in (0,\eta_C)$,  $\alpha\in (0,\alpha_C)$ and arbitrary $\nu>0$, the function $P^+_{\eta,\alpha,\nu}$ (respectively, $P^-_{\eta,\alpha,\nu}$) in \eqref{zara} is a positive radial super-solution (sub-solution) of \eqref{e1} in $B_{r_C}(0)\setminus \{0\}$ and $(P^\pm_{\eta,\alpha,\nu})'(r)> 0$ for all $r\in (0,r_C)$.
\end{lemma}

\begin{proof}
Let $\eta,\alpha\in (0,p/3)$ and $\nu>0$ be arbitrary. 
For every $t>0$, we define $\psi(t)=t/(1+t)\in (0,1)$. For every $r\in (0,1/e)$, we define 
\begin{equation} \label{sdf}  S^\pm_{\eta,\alpha,\nu}(r)=p\mp 
 \eta \pm \alpha \,\psi\left( \frac{[\log\, (1/r)]^{-\alpha}}{\nu }\right)\in (p-\eta-\alpha, p+\eta+\alpha). \end{equation}
To simplify the notation, we drop the subscript $\nu$ from $P^\pm_{\eta,\alpha,\nu}$ and $S^\pm_{\eta,\alpha,\nu}$. 

Let $r\in (0,1/e)$ be arbitrary. 
Since $\beta=\varpi_1$, from \eqref{zara}, we see that 
\begin{equation} \label{comm} \begin{aligned}
\pm r^\theta (P_{\eta,\alpha}^\pm (r))^q& =\pm C^{q-1} \frac{P_{\eta,\alpha}^\pm (r)}{r^2} [\log \,(1/r)]^{(-p\pm \eta)(q-1)} 
\left( 1+\frac{ [\log\,(1/r)]^{-\alpha}}{\nu}\right)^{\pm (q-1)}\\
& \geq \pm C^{q-1} \frac{P_{\eta,\alpha}^\pm (r)}{r^2} [\log \,(1/r)]^{-p\left(q-1\right)} . 
\end{aligned}
\end{equation}

By differentiating \eqref{zara} with respect to $r$, we get
$$ 
 \left(P_{\eta,\alpha}^\pm\right)'(r)=|\varpi_1|\frac{P_{\eta,\alpha}^\pm(r)}{r} \left[
1 +\frac{S^\pm_{\eta,\alpha}(r)}{|\varpi_1|\,\log \left(1/r\right)} 
\right],
$$
which implies that
\begin{equation} \label{girr1} \begin{aligned} 
\lambda\, G_\tau(P_{\eta,\alpha}^\pm(r))& =\lambda \,|\varpi_1|^{1-\tau} \frac{P_{\eta,\alpha}^\pm(r)}{r^2} \left(1+\frac{S^\pm_{\eta,\alpha}(r)}{|\varpi_1| \log\, (1/r)}\right)^{1-\tau}\\
& =\lambda \,|\varpi_1|^{1-\tau} \frac{P_{\eta,\alpha}^\pm(r)}{r^2} 
\left[1+\frac{\left(1-\tau\right) S^\pm_{\eta,\alpha}(r) }{|\varpi_1|\,\log\,(1/r)}
-\frac{\left(1-\tau\right)\tau \,[S^\pm_{\eta,\alpha}(r)]^2 }{2\varpi_1^2 \log^2(1/r)} (1+o(1))
\right]
\end{aligned} \end{equation}  as $r\to 0^+$. Moreover, we find that 
 \begin{equation} \label{girr2} \begin{aligned}  \mathbb L_\rho(P_{\eta,\alpha}^\pm(r))= \frac{P_{\eta,\alpha}^\pm(r)}{r^2}& \left[
 \varpi_1^2+2\rho \varpi_1 -\frac{2\,(\varpi_1+\rho)}{\log \,(1/r)} S^\pm_{\eta,\alpha}(r) \right. 
\\
& \ \  \left.  +\frac{[S^\pm _{\eta,\alpha}(r)]^2+S^\pm_{\eta,\alpha}(r)+\alpha^2 \psi'\left( \frac{[\log \,(1/r)]^{-\alpha}}{\nu} \right) \frac{[\log \,(1/r)]^{-\alpha}}{\nu}}{\log^2 (1/r)}
 \right].
 \end{aligned}
 \end{equation}
 Recall that $f_{\rho}(\varpi_1)= \varpi_1^2+2\rho \varpi_1+\lambda |\varpi_1|^{1-\tau}=0$. 
 From \eqref{sdf}, we observe that 
 \begin{equation} \label{cho} 
 S_{\eta,\alpha}^\pm(r)\to p\quad \mbox{ as } (\eta,\alpha)\to (0,0)
 \end{equation} uniformly with respect to $r\in (0,1/e)$ and $\nu>0$. 
 
 \vspace{0.2cm}
 $\bullet$ Assume that 
either Case $[N_1]$ holds or $\rho>\Upsilon$ in Case $[N_2]$. Then, using \eqref{pgam}, we have 
 $$p=1/(q-1),\quad 
   -2\,(\varpi_1+\rho)+ \lambda |\varpi_1|^{-\tau}(1-\tau)=|\varpi_1|-\lambda \tau |\varpi_1|^{-\tau}=\left(q-1\right)\gamma^{q-1}.
 $$
  It follows that 
 \begin{equation} \label{chou} \mathbb L_{\rho,\lambda,\tau}[P_{\eta,\alpha}^\pm(r)]= (q-1)\,\gamma^{q-1}\,\frac{P_{\eta,\alpha}^\pm(r)}{r^2}
  \frac{S^\pm_{\eta,\alpha}(r)}{\log\,(1/r)} (1+o(1)) \quad \mbox{as } r\to 0^+. 
 \end{equation}
Since $C>\gamma$ in the definition of $P_{\eta,\alpha}^+(r)$ (respectively, $C\in (0,\gamma)$ for $P_{\eta,\alpha}^-(r)$), from \eqref{cho} and \eqref{chou}, we 
infer that there exist
$\eta_C,\alpha_C, r_C>0$ small enough such that  
\begin{equation} \label{sao2} \pm  \mathbb L_{\rho,\lambda,\tau}[P_{\eta,\alpha}^\pm(r)] \leq \pm C^{q-1} \frac{P_{\eta,\alpha}^\pm (r)}{r^2\, \log \,(1/r)} .\end{equation} 
for every $r\in (0,r_C)$, $\eta\in (0,\eta_C)$ and $\alpha\in (0,\alpha_C)$. 

Using that $p\,(q-1)=1$, from \eqref{comm} and \eqref{sao2}, we conclude the proof of Lemma~\ref{secoo}. 
 
  \vspace{0.2cm}
 $\bullet$ Assume that Case $[N_2]$ holds with $\rho=\Upsilon$. Then, 
  $$p=2/(q-1),\quad 
   -2\,(\varpi_1+\rho)+ \lambda |\varpi_1|^{-\tau}(1-\tau)=|\varpi_1|-\lambda \tau |\varpi_1|^{-\tau}=0.
 $$
 In light of \eqref{girr1} and \eqref{girr2}, we get
 \begin{equation} \label{gomm}   \mathbb L_{\rho,\lambda,\tau}[P_{\eta,\alpha}^\pm(r)]=
 \frac{P_{\eta,\alpha}^\pm(r)}{r^2\,\log^2\,(1/r)}
T_{\eta,\alpha,\nu}(r)\,  (1+o(1)) \quad \mbox{as } r\to 0^+,
 \end{equation} where $T_{\eta,\alpha,\nu}(r)$ is defined by 
 $$ T_{\eta,\alpha,\nu}(r)= \frac{(1+\tau)}{2}   [S^\pm_{\eta,\alpha}(r)]^2+ 
   S^\pm_{\eta,\alpha}(r) 
  +\alpha^2 \psi'\left( \frac{[\log \,(1/r)]^{-\alpha}}{\nu} \right) \frac{[\log \,(1/r)]^{-\alpha}}{\nu}.
 $$
 Since $t\psi'(t)=t/(1+t)^2\in (0,1)$ for every $t>0$, 
 using \eqref{cho} and \eqref{pgam}, we remark that 
 \begin{equation} \label{goma}  T_{\eta,\alpha,\nu}(r)\to \frac{(1+\tau)\,p^2}{2}+ p=\frac{2(q+\tau)}{(q-1)^2}=\gamma^{q-1}
 \quad \mbox{ as } (\eta,\alpha)\to (0,0)
 \end{equation} uniformly with respect to $r\in (0,1/e)$ and $\nu>0$. 

From \eqref{gomm} and \eqref{goma},  we 
infer that there exist
$\eta_C,\alpha_C, r_C>0$ small enough such that  
\begin{equation} \label{sar2} \pm  \mathbb L_{\rho,\lambda,\tau}[P_{\eta,\alpha}^\pm(r)] \leq \pm C^{q-1} \frac{P_{\eta,\alpha}^\pm (r)}{r^2\, \log^2 \,(1/r)} \end{equation} 
for every $r\in (0,r_C)$, $\eta\in (0,\eta_C)$ and $\alpha\in (0,\alpha_C)$. 

Using that $p\,(q-1)=2$, from \eqref{comm} and \eqref{sar2}, we end the proof of Lemma~\ref{secoo}. 
\end{proof}

{\bf Proof of Theorem~\ref{law1} concluded.} 
Let $\alpha_C$, $r_C$ and $\eta_C$ be given by Lemma~\ref{secoo}. Fix $\alpha\in (0,\min\{1,\alpha_C\})$ and let $\eta\in (0,\eta_C)$ be arbitrary. 
We can assume that $r_C>0$ is small enough such that $B_{r_C}(0)\subset \subset \Omega$. 
Let $\nu>0$ be small (depending on $u$ and $r_C$) such that 
$$  \begin{aligned}
& \max_{|x|=r_C} u(x)  \leq \gamma\, r_C^{-\varpi_1} \left[\log \left(1/r_C\right)\right]^{-p}
\left(1+\frac{[ \log\, (1/r_C)]^{-1}}{\nu} \right) \\
& \mbox{(respectively,}  \min_{|x|=r_C} u(x) \geq \gamma \,
r_C^{-\varpi_1} \left[\log \left(1/r_C\right)\right]^{-p}\left(1+\frac{[\log\,(1/r_C)]^{-1}}{\nu} \right)
^{-1}).
\end{aligned} $$

Since $\alpha$ and $\nu$ are fixed, we write $P^\pm_\eta$ instead of $P^\pm_{\eta,\alpha,\nu}$. 
Then, $P^+_{\eta}$ (respectively, $P^-_{\eta}$) in \eqref{zara} is a positive radial super-solution (respectively, sub-solution) of \eqref{e1} in $B_{r_C}(0)\setminus \{0\}$ satisfying 
the above properties (i), (ii) and (iv). 
Finally, from Theorem~\ref{law2}, 
we infer that
$$ \lim_{|x|\to 0} \frac{u(x)}{P^+_{\eta}(|x|)}=0\ \mbox{ (respectively, } \lim_{|x|\to 0} \frac{P^-_{\eta}(|x|)}{u(x)}=0). $$ 
Hence, $P_\eta^+$ (respectively $P_\eta^-$) satisfies the property (iii) as well.  

As in the proof of Theorem~\ref{pro1}, we use the comparison principle to conclude that $ u(x) \leq P^+_{\eta}(|x|)$ (respectively, $ u(x) \geq P^-_{\eta}(|x|)$) for every $ x\in B_{r_C}(0)\setminus \{0\}$. 
By letting $\eta\to 0$ and then $|x|\to 0$, we arrive at 
$$ 
 \limsup_{|x|\to 0} 
|x|^{\varpi_1} \left[ \log \,(1/|x|)\right]^{p} 
u(x)\leq C,\quad 
 \mbox{(respectively, } \liminf_{|x|\to 0} |x|^{\varpi_1} \left[ \log\, (1/|x|)\right]^{p}
u(x)\geq C). $$  
By letting $C\searrow \gamma$ (respectively, $C\nearrow \gamma$), we achieve \eqref{bbb} (respectively, \eqref{wow}). 
 \qed

\section{Proof of Assertion $(P_0)$ in Theorem~\ref{thi}} \label{sect9}

Throughout this section, we assume that 
 \eqref{e2} holds and $\beta\in (\varpi_2,0)$ in Case $[N_2]$. 

We recall that $\Phi_{\varpi_2(\rho)}(r)$ is defined  for $r\in (0,1)$ by 
\begin{equation} \label{phi22}
\Phi_{\varpi_2(\rho)}(r)=\left\{ 
\begin{aligned} 
& r^{-\varpi_2} && \mbox{if } \rho>\Upsilon,& \\
&  r^{-\varpi_2} \left| \log r  \right|^{\frac{2}{1+\tau}} 
&& \mbox{if } \rho=\Upsilon.&
\end{aligned} \right.
\end{equation}

Our main result in this section is the following.

\begin{theorem}  \label{ama1}
Let \eqref{e2} hold and $\beta\in (\varpi_2,0)$ in Case $[N_2]$. 
\begin{itemize} 
\item[(A)] Let $u$ be any positive sub-solution of \eqref{e1} such that $ \lim_{|x|\to 0} u(x)/ \Phi_{\varpi_2(\rho)}(|x|)=0$. Then,  
\begin{equation} 
\label{west1}
 \limsup_{|x|\to 0}|x|^{\varpi_1} u(x):=a<\infty. 
\end{equation}
\item[(B)] Let $u$ be any positive super-solution of \eqref{e1} with $ \lim_{|x|\to 0} u(x)/ \Phi_{\varpi_2(\rho)}(|x|)=\infty$. Then, 
\begin{equation} 
\label{nopi4}
\mbox{for every } \eta>0, \ \mbox{we have } \lim_{|x|\to 0}  |x|^{\beta-\eta} u(x)=\infty.
\end{equation}
\end{itemize}
\end{theorem}

In view of Theorem~\ref{ama1}, we conclude the assertion in $(P_0)$ of Theorem~\ref{thi}. 
Indeed, let $u$ be any positive solution of \eqref{e1} and $b:=\limsup_{|x|\to 0} u(x)/\Phi_{\varpi_2(\rho)}(|x|)$. By 
Corollary \ref{urz} (if $b\in \mathbb R_+$) and Corollary~\ref{pere} in Appendix~\ref{sectiune2} (if $b=\infty$), there exists 
$$ \lim_{|x|\to 0} \frac{u(x)}{ \Phi_{\varpi_2(\rho)}(|x|)}=b\in [0,\infty].$$ 

$\bullet$ If $b=0$, then by Theorems~\ref{ama1} (A) and \ref{zic1}, we get that 
$\limsup_{|x|\to 0} |x|^{\varpi_1} u(x):=a\in (0,\infty)$. Hence, by Corollary~\ref{urz}, we conclude 
 \eqref{mobb1}. 

$\bullet$ If $b\in (0,\infty)$, then by Corollary~\ref{urz}, we obtain \eqref{gapp}. 

$\bullet$ If $b=\infty$, then by combining Theorem~\ref{ama1} (B) with Theorem~\ref{pro1},  
we attain \eqref{tirv}.

\vspace{0.2cm}
The proof of Theorem~\ref{ama1} depends on whether 
$\rho>\Upsilon$ or $\rho=\Upsilon$ (see the definition of $ \Phi_{\varpi_2(\rho)}$ in \eqref{phi22}). Each of these situations 
is treated separately (see Sections~\ref{secla1}--\ref{secla4}).  

\subsection{Proof of Theorem~\ref{ama1} (A) when $\rho>\Upsilon$} \label{secla1}
Let \eqref{e2} hold and $\beta\in (\varpi_2,0)$ in Case $[N_2]$. 
We first assume that $\rho>\Upsilon$.  
This implies that $\varpi_1<-(\lambda \tau)^{1/(\tau+1)}<\varpi_2<0$ and 
\begin{eqnarray} 
&& -2\left(\varpi_2+\rho\right) +\lambda |\varpi_2|^{-\tau} \left(1-\tau\right)=|\varpi_2|^{-\tau} \left( |\varpi_2|^{1+\tau}-\lambda\tau\right)<0,  \label{raoo}\\
&& 2\,(\varpi_1+\rho) -\lambda |\varpi_1|^{-\tau} \left(1-\tau\right)=-|\varpi_1|^{-\tau} ( |\varpi_1|^{1+\tau}-\lambda\tau)<0.  \label{rabi} 
\end{eqnarray}

Assume that $u$ is a positive sub-solution of \eqref{e1} satisfying  
\begin{equation} \label{eda1}
\lim_{|x|\to 0} |x|^{\varpi_2} u(x)=0.
\end{equation}

We split the proof of \eqref{west1} into three Steps. 

\vspace{0,2cm}
{\bf Step 1.}  {\em For every $\varepsilon>0$ small enough, we have} 
\begin{equation} \label{eda2} \lim_{|x|\to 0} |x|^{\varpi_2-\varepsilon} u(x)=0.
\end{equation}

\begin{proof}[Proof of Step~1]
Let $c>0$ be fixed arbitrary. For every $\eta,\delta>0$ and $r>0$, we define
\begin{equation} \label{dvet} V_{\eta,\delta}(r)=c \,r^{-\varpi_2} \left(e^{r^\eta} -e^{-\delta} \right).
\end{equation}  

{\bf Claim 1:}  {\em There exists $\eta_*>0$ small such that for every $\eta\in (0,\eta_*)$ and $\delta,c>0$ arbitrary, 
$V_{\eta,\delta}$ is a positive radial super-solution of \eqref{e1} in $B_{1}(0)\setminus \{0\}$ such that $V_{\eta,\delta}'(r)>0$ for all $r\in (0,1)$.} 

\vspace{0.2cm}
{\bf Proof of Claim 1.} Using \eqref{raoo}, we can define $\eta_*>0$ small such that 
\begin{equation} \label{goo} \left(e+1\right) \eta_*< 2\left(\varpi_2+\rho\right) -\lambda |\varpi_2|^{-\tau} \left(1-\tau\right).
\end{equation}
We fix $\eta\in (0,\eta_*)$ and let $\delta,c>0$ be arbitrary.  
Let $E_\delta(t)$ be defined for every $t>0$ as follows
$$ E_\delta(t)=\frac{t e^{t}}{e^t-e^{-\delta}}\in (0,e^t) .
$$
Remark that $t E_\delta'(t)=E_\delta(t)- e^{-\delta-t} E_\delta^2(t)<E_\delta(t)$ for every $t>0$. 

Let $r\in (0,1)$ be arbitrary. By differentiating $V_{\eta,\delta}(r)$ with respect to $r$, we find that 
\begin{equation} \label{tou}  V'_{\eta,\delta}(r)=\frac{V_{\eta,\delta}(r)}{r} \left[-\varpi_2+\eta \,E_\delta(r^\eta)\right]>0. 
\end{equation}
Moreover, we have 
$$ \begin{aligned} 
\mathbb L_\rho[V_{\eta,\delta}(r)]& = 
\frac{V_{\eta,\delta}(r)}{r^2} \left[
\varpi_2^2+2\rho \varpi_2-2\eta\left(\varpi_2+\rho\right) E_\delta(r^\eta)+ \eta^2 E_\delta^2(r^\eta)+\eta^2 r^\eta E_\delta'(r^\eta)
 \right] \\
 & < 
 \frac{V_{\eta,\delta}(r)}{r^2} \left[
\varpi_2^2+2\rho \varpi_2-2\eta \left(\varpi_2+\rho\right) E_\delta(r^\eta)+ \eta^2 E_\delta^2(r^\eta)+\eta^2  E_\delta(r^\eta)
 \right] \\
 &= \frac{V_{\eta,\delta}(r)}{r^2} \left\{
\varpi_2^2+2\rho \varpi_2+\eta \, E_\delta(r^\eta) \left[-2 \left(\varpi_2+\rho\right) +  \left(e+1\right)\eta \right]
 \right\} .
 \end{aligned}
$$
Using that $\lambda>0$, $\tau\in (0,1)$ and \eqref{tou}, we see that 
$$ \lambda G_\tau (V_{\eta,\delta}(r))\leq  \lambda |\varpi_2|^{1-\tau} 
\frac{V_{\eta,\delta}(r)}{r^2} \left[ 1+\frac{\left(1-\tau\right) \eta}{|\varpi_2|} E_\delta(r^\eta)\right].
$$
Since $f_{\rho}(\varpi_2)=\varpi_2^2+2\rho \,\varpi_2+\lambda \,|\varpi_2|^{1-\tau}=0$, we arrive at 
$$ \mathbb L_{\rho,\lambda,\tau}[V_{\eta,\delta}(r)] \leq  
\frac{V_{\eta,\delta}(r)}{r^2} \eta \,E_\delta (r^\eta) \left[ -2 \left(\varpi_2+\rho\right) + \lambda |\varpi_2|^{-\tau} (1-\tau)+ \left(e+1\right)\eta 
\right]<0
$$ since $\eta\in (0,\eta_*)$ and \eqref{goo} holds. Hence, for every $r\in (0,1)$, we have proved that 
$$  \mathbb L_{\rho,\lambda,\tau}[V_{\eta,\delta}(r)] <0<r^\theta V^q_{\eta,\delta}(r). 
$$
This completes the proof of Claim~1. \qed

\vspace{0.2cm}
{\bf Proof of Step~1 concluded.} 
We show that for every $\varepsilon\in (0,\eta_*)$, we have \eqref{eda2}. 
Let $\eta\in (\varepsilon,\eta_*)$. It is enough to show that 
\begin{equation} \label{hopa} \limsup_{|x|\to 0} |x|^{\varpi_2-\eta} u(x)<\infty. 
\end{equation}
Let $r_0\in (0,1)$ be small such that
$B_{r_0}(0)\subset \subset \Omega$.  We let $c>0$ large such that 
$$ \max_{\partial B_{r_0}(0)} u\leq c \,r_0^{-\varpi_2} \left(e^{r_0^\eta}-1\right). $$

We now define $V_{\eta,\delta}(r)$ as in \eqref{dvet} for every $\delta>0$ and $r\in (0,1)$. 
The choice of $c>0$ gives that $u\leq V_{\eta,\delta}$ on $\partial B_{r_0}(0)$. 
 The assumption \eqref{eda1} implies that 
$$ u(x)\leq V_{\eta,\delta}(|x|) \quad \mbox{for }|x|\to 0.$$
In view of the Claim, we can apply the comparison principle to find that 
$$ u(x)\leq V_{\eta,\delta}(|x|) \quad \mbox{for every } 0<|x|\leq r_0.
$$
Since $\delta>0$ is arbitrary, by letting $\delta\to 0^+$, we conclude that 
$$u(x)\leq c\, |x|^{-\varpi_2} \left(e^{|x|^\eta}-1\right)\quad \mbox{for every } |x|\in (0,r_0].$$
Hence, $ \limsup_{|x|\to 0} |x|^{\varpi_2-\eta} u(x)\leq c$, which proves 
\eqref{hopa}. This finishes the proof of Step~1. 
\end{proof}

\begin{rem} In the proof of the Claim, it is essential that $\eta>0$ is small enough to ensure that $ \mathbb L_{\rho,\lambda,\tau}[V_{\eta,\delta}(r)] \leq 0$. Otherwise, we 
cannot prove that $ \mathbb L_{\rho,\lambda,\tau}[V_{\eta,\delta}(r)] \leq r^\theta V^q_{\eta,\delta}(r)$ for all $r\in (0,1)$. 
\end{rem}

{\bf Step 2.} 
{\em For every $\eta>0$, we have}
\begin{equation} \label{ef1}
\lim_{|x|\to 0} |x|^{\varpi_1+\eta} u(x)=0. 
\end{equation}

\begin{proof}[Proof of Step~2] To achieve \eqref{ef1}, we use an iterative process based on the comparison with a family of super-solutions of \eqref{e1}. The first step in this iterative scheme is based on Step~1. 

Let $\varepsilon>0$ be small as in Step~1.  
Moreover, since $\varpi_1<\varpi_2$, we can diminish $\varepsilon>0$ such that $\varepsilon<\varpi_2-\varpi_1$. 
We define a sequence $\{N_j\}_{j\geq 1}$ of positive numbers as follows
\begin{equation} \label{seqq1} N_1= -\varpi_2+\varepsilon\in (-\varpi_2, -\varpi_1) \quad \mbox{and}\quad N_{j+1}=\sqrt{2\rho \,N_j-\lambda \,N_j^{1-\tau}}\quad \mbox{for all } j\geq 1. 
\end{equation}
Remark that $x\longmapsto 2\rho x-\lambda x^{1-\tau}$ is increasing for $x>-\varpi_2$ since from \eqref{raoo}, we have
$$ 2\rho-\lambda\left(1-\tau\right) x^{-\tau}>2\rho-\lambda \left(1-\tau\right) |\varpi_2|^{-\tau}>-2\varpi_2>0.
$$
Thus, we find that $N_2\in (-\varpi_2,-\varpi_1)$ using that
$$ \varpi_2^2=-2\rho \varpi_2-\lambda |\varpi_2|^{1-\tau}<N_2^2=2\rho N_1-\lambda N_1^{1-\tau}<-2\rho \varpi_1-\lambda |\varpi_1|^{1-\tau}=\varpi_1^2.$$
Moreover, we see that
$$N_2^2=2\rho N_1-\lambda N_1^{1-\tau}=\varpi_2^2 +\left[2\rho-\lambda |\varpi_2|^{-\tau} (1-\tau) \right]\varepsilon\left(1+o(1)\right)\quad \mbox{as } \varepsilon\to 0^+.$$
Hence, for $\varepsilon>0$ small enough, using \eqref{raoo}, we get that 
$ N_2^2>(-\varpi_2+\varepsilon)^2=N_1^2$, which yields that $N_2>N_1$. 
By induction, one can check that $\{N_j\}_{j\geq 1}$ is an increasing sequence satisfying 
\begin{equation} \label{seqq2} -\varpi_2<N_j <N_{j+1} < -\varpi_1\quad\mbox{ for all } j\geq 1.
\end{equation}
Thus, $\{N_j\}_{j\geq 1}$ converges as $j\to \infty$ and $\lim_{j\to \infty} N_j=-\varpi_1$. 
 
 \vspace{0.2cm}
We only need to prove \eqref{ef1} for every $\eta>0$ small. So, we assume that $0<\eta<\varpi_2-\varpi_1-\varepsilon$. 
Then, there exists $k=k(\eta)\geq 1$ such that 
 \begin{equation} \label{kal} N_{k}< -\varpi_1-\eta\leq N_{k+1}<-\varpi_1. \end{equation}

We fix $\alpha,\nu>0$ small enough such that for every $j=1,\ldots,k$, we have
$$ \nu \left(2N_{j+1}-2\rho+\lambda N_j^{-\tau}(1-\tau)+\nu\right) +\alpha \left(N_{j+1}-N_j+\nu\right)<\left(N_{j+1}-N_j\right) 
\left(2\rho-\lambda N_j^{-\tau}(1-\tau)\right).
$$

Observe that the right-hand side of the above inequality is {\em positive} in light of \eqref{seqq1} and \eqref{seqq2}. 
For every $j\geq 1$, we define $T_1(j)$ and $T_2(j)$ as follows
\begin{equation} \label{t12} \begin{aligned}
& T_1(j)=N_j^2-2\rho N_j+\lambda N_j^{1-\tau},\\
& T_2(j)=\left(N_{j+1}-N_j+\nu\right) 
  \left( N_{j+1}+N_j-2\rho +\lambda N_j^{-\tau}(1-\tau)+\nu+\alpha\right).
\end{aligned} \end{equation} 

In view of \eqref{seqq1} and \eqref{seqq2}, we have
\begin{equation} \label{uraa}  T_1(j)< N_{j+1}^2-2\rho N_j+\lambda N_j^{1-\tau}=0\quad \mbox{for all } j\geq 1. 
\end{equation}
 
From the choice of $\nu$ and $\alpha$, jointly with \eqref{seqq1}, we find that
\begin{equation} \label{sumt} T_1(j)+T_2(j)<0 \quad \mbox{for all } j=1,\ldots, k.\end{equation}

 {\em Construction of a suitable family of super-solutions of \eqref{e1}.} 
Let $c>0$ and $\delta>0$ be arbitrary. For every $j=1,\ldots, k$, we define 
$V_{j,\delta}(r)$ for all $r>0$ by 
\begin{equation} \label{hauu} V_{j,\delta}(r)=c \, r^{N_j} \left( e^\delta-e^{-r^\alpha}\right)^{\frac{N_{j+1}-N_j+\nu}{\alpha}}.
\end{equation}

{\bf Claim 2:} {\em For every $\eta\in (0,\varpi_2-\varpi_1-\varepsilon)$, letting $k=k(\eta)\geq 1$ as in \eqref{kal}, there exist 
$\nu,\alpha>0$ small such that for every $j=1,\ldots, k$ and $c,\delta>0$ arbitrary, 
the function $V_{j,\delta}(r)$ in \eqref{hauu} is a positive radial super-solution of \eqref{e1} in $B_1(0)\setminus \{0\}$ satisfying
$V_{j,\delta}'(r)>0$ for all $r\in (0,1)$.}

\begin{proof}[Proof of Claim 2] For every $t>0$, we define $\Psi_\delta(t)$ by 
\begin{equation} \label{psio} \Psi_\delta(t):=\frac{t \,e^{-t}}{e^\delta-e^{-t}}\in (0,1).
\end{equation}

Let $j=1,\ldots, k$ and $r>0$ be arbitrary. By differentiating \eqref{hauu} with respect to $r$, we get  
$$ V_{j,\delta}'(r)=\frac{V_{j,\delta}(r)}{r} \left[ N_j+\left(N_{j+1}-N_j+\nu\right) \Psi_\delta(r^\alpha)\right] >0,
$$
which implies that 
$$ \lambda G_\tau(V_{j,\delta}(r))\leq \lambda N_j^{1-\tau} 
\frac{V_{j,\delta}(r)}{r^2} \left[1+\frac{\left(1-\tau\right) \left(N_{j+1}-N_j+\nu\right)}{N_j}\,\Psi_\delta(r^\alpha)
\right].
$$
Using that $t\Psi_\delta'(t)\leq \Psi_\delta(t)$ and $\Psi_\delta^2(t)\leq \Psi_\delta(t)<1$ for all $t>0$, we have
$$ F_j(r):=2N_j-2\rho+\left(N_{j+1}-N_j+\nu\right)\Psi_\delta(r^\alpha) +\alpha \frac{r^\alpha \Psi_\delta'(r^\alpha)}{\Psi_\delta(r^\alpha)} \leq
N_{j+1}+N_j-2\rho+\nu+\alpha .
$$
Therefore, by a simple calculation, we arrive at
$$ \begin{aligned}
\mathbb L_\rho[V_{j,\delta}(r)] &=
 \frac{V_{j,\delta}(r)}{r^2} \left[ N_j^2-2\rho N_j +\left(N_{j+1}-N_j+\nu\right) 
F_j(r)\, \Psi_{\delta} (r^\alpha) 
 \right]
\\
& \leq  \frac{V_{j,\delta}(r)}{r^2}
\left[ N_j^2-2\rho N_j +\left(N_{j+1}-N_j+\nu\right) \left( N_{j+1}+N_j-2\rho+\nu+\alpha \right) \Psi_\delta(r^\alpha) 
\right].
\end{aligned} $$
Consequently, using $T_1(j)$ and $T_2(j)$ in \eqref{t12}, we obtain that 
$$ 
\mathbb L_{\rho,\lambda,\tau}[V_{j,\delta}(r)] \leq 
 \frac{V_{j,\delta}(r)}{r^2}  \left[
T_1(j)  + T_2(j)\, \Psi_\delta(r^\alpha) 
 \right].
 $$
Since \eqref{uraa} holds and $\Psi_\delta(t)\in (0,1)$ for every $t>0$, using \eqref{sumt}, we get
$$ \mathbb L_{\rho,\lambda,\tau}[V_{j,\delta}(r)] \leq 
 \frac{V_{j,\delta}(r)}{r^2} \Psi_{\delta} (r^\alpha) \left[ 
 T_1(j)+T_2(j)
 \right]<0<r^\theta [V_{j,\delta}(r)]^q.
$$
This completes the proof of Claim~2. 
\end{proof}

{\bf Proof of Step~2 completed.} Let $r_0\in (0,1)$ be small enough such that $B_{r_0}(0)\subset\subset \Omega$. 
For every $\eta \in (0,\varpi_2-\varpi_1-\varepsilon)$, we let $k=k(\eta)$ and $\nu,\alpha>0$ be as in Claim~2. 
Let $\delta>0$ be arbitrary. 
Choose $c>0$ large enough such that for every $j=1,\ldots, k$, we have
$$ \max_{\partial B_{r_0}(0)} u\leq c \,r_0^{N_j} \left(1-e^{-r_0^\alpha}\right)^{\frac{N_{j+1}-N_j+\nu}{\alpha}}.
$$
 This ensures that $u\leq V_{j,\delta}$ on $\partial B_{r_0}(0)$ for all $\delta>0$ and $j=1,\ldots,k$. 
 
 From Step~1 and the definition of $N_1$ in \eqref{seqq1}, we have 
 $$ \lim_{|x|\to 0} \frac{u(x)}{V_{1,\delta}(|x|)}=0.
 $$
 In light of Claim~2, we can apply the comparison principle and infer that 
 $$ u(x)\leq V_{1,\delta}(|x|)\quad \mbox{for all } 0<|x|\leq r_0.
 $$ Since $\delta>0$ is arbitrary, by letting $\delta\to 0$, we obtain that 
 $$ u(x)\leq c\, r^{N_1} \left(1-e^{-r^\alpha}\right)^{\frac{N_2-N_1+\nu}{\alpha}}\quad \mbox{for all } 0<|x|\leq r_0,
 $$ which shows that 
 \begin{equation} \label{hajj} \limsup_{|x|\to 0} |x|^{-N_2-\nu} u(x)\leq c. \end{equation}
Hence, if $k=1$, using \eqref{kal} and \eqref{hajj}, we get that 
 $\lim_{|x|\to 0} |x|^{\varpi_1+\eta} u(x)=0 $. 
 
If $k>1$, then from \eqref{hajj}, we have
$$ \lim_{|x|\to 0} \frac{u(x)}{V_{2,\delta}(|x|)}=0.
$$
We proceed inductively over $j=1,\ldots, k-1$  to get that 
$ \lim_{|x|\to 0} u(x)/V_{k,\delta}(|x|)=0$ and, hence, by the comparison principle, we have
 $$ u(x)\leq V_{k,\delta}(|x|)\quad \mbox{for all } 0<|x|\leq r_0.
 $$
By letting $\delta\to 0$ in the above inequality, we  arrive at 
$$ \limsup_{|x|\to 0} |x|^{-N_{k+1}-\nu} u(x)\leq c, $$
which together with \eqref{kal}, leads to  $\lim_{|x|\to 0} |x|^{\varpi_1+\eta} u(x)=0 $. 
This ends the proof of Step~2. 
\end{proof}

{\bf Step 3.} {\em The assertion of \eqref{west1} holds.}

\begin{proof}[Proof of Step~3] We use Step~2 to conclude \eqref{west1} via the comparison principle with a suitable family of super-solutions of \eqref{e1}.  
Let $c>0$ be fixed. For every $\varepsilon,\delta>0$ and $r>0$, we define
\begin{equation} \label{dwet} W_{\varepsilon,\delta}(r)=c \,r^{-\varpi_1-\varepsilon} \left(e^{\delta} -e^{-r^\varepsilon} \right).
\end{equation}  

{\bf Claim 3:} {\em If $\varepsilon\in (0,|\varpi_1 |)$ is small, then 
for all $c>0$ and $\delta\in (0,1)$, the function 
$W_{\varepsilon,\delta}$ in \eqref{dwet} is a positive radial super-solution of \eqref{e1} in $B_{1}(0)\setminus \{0\}$ and
$W_{\varepsilon,\delta}'(r)>0$ for all $r\in (0,1)$.} 

\begin{proof}[Proof of Claim~3]  
In view of \eqref{rabi}, we can choose $\nu\in (0,1)$ small such that 
$$ 2\,(\varpi_1+\rho) -\lambda|\varpi_1|^{-\tau}\left(1-\tau\right) \left(1-\nu\right)<0. 
$$
Let $t_*=t_*(\nu)>0$ be small such that 
\begin{equation} \label{robo1} \left( 1-\frac{t}{|\varpi_1|} \right)^{1-\tau}\leq 1-\frac{\left(1-\tau\right) \left(1-\nu\right) t}{|\varpi_1|} \quad \mbox{for all } t\in (0,t_*). 
\end{equation}
We fix $\varepsilon\in (0,\min\,\{|\varpi_1|,t_*\})$ small enough to ensure that 
\begin{equation} \label{noma0}  2\,(\varpi_1+\rho) -\lambda|\varpi_1|^{-\tau}\left(1-\tau\right) \left(1-\nu\right)+\varepsilon \,(1+e^2)<0.
\end{equation}
Let $\delta\in (0,1)$ be arbitrary. 
For every $t>0$, we define $H_\delta(t)$ by 
$$ 1>H_\delta(t):=1-\frac{t \,e^{-t}}{e^\delta-e^{-t}}> 1-\frac{t \,e^{-t}}{1-e^{-t}}>0.
$$
Fix $r\in (0,1)$ arbitrary. 
Then, we see that 
\begin{equation} \label{fort1}  W_{\varepsilon,\delta}'(r)=\frac{W_{\varepsilon,\delta}(r)}{r} \left( -\varpi_1-\varepsilon H_\delta\left(r^\varepsilon\right)\right)>
\frac{W_{\varepsilon,\delta}(r)}{r} \left( -\varpi_1-\varepsilon\right)>0.
\end{equation}
Observe that for every $t>0$, we have 
$$ -t H_\delta'(t)=\left(1-H_\delta(t)\right) \left[ 1-e^{\delta+t} +e^{\delta+t} H_\delta(t)\right]
< e^{\delta+t} H_\delta(t). $$ 
Thus, $H^2_\delta(r^\varepsilon)-r^\varepsilon H_\delta'(r^\varepsilon)<(1+e^{\delta+r^\varepsilon}) \,H_\delta(r^\varepsilon) <(1+e^2)\,H_\delta(r^\varepsilon) $. 
Hence, it follows that 
\begin{equation} \label{gino1} \begin{aligned} 
\mathbb L_\rho[W_{\varepsilon,\delta}(r)]&= 
\frac{W_{\varepsilon,\delta}(r)}{r^2} \left[
\varpi_1^2 +2\rho \varpi_1 +2\,\varepsilon \,(\varpi_1+\rho)\, H_\delta(r^\varepsilon) +
\varepsilon^2  \left( H^2_\delta(r^\varepsilon)-r^\varepsilon H_\delta'(r^\varepsilon)\right)
\right]\\
&
< \frac{W_{\varepsilon,\delta}(r)}{r^2} \left\{
\varpi_1^2+ 2\rho \varpi_1
 +\varepsilon \, H_\delta(r^\varepsilon)  \left[ 
 2 \,(\varpi_1+\rho)+
\varepsilon (1+ e^{2}) \right]  
\right\}.
\end{aligned} \end{equation}
Since $0<\varepsilon H_\delta(r^\varepsilon)<\varepsilon<t_*$, 
using \eqref{robo1} and \eqref{fort1}, we arrive at 
\begin{equation} \label{gino2} \begin{aligned}
 \lambda \,G_\tau(W_{\varepsilon,\delta}(r)) & =\lambda |\varpi_1|^{1-\tau}\, \frac{W_{\varepsilon,\delta}(r)}{r^2} 
 \left( 1-\frac{\varepsilon}{|\varpi_1|} H_\delta(r^\varepsilon)\right) ^{1-\tau}\\
 & \leq 
 \lambda |\varpi_1|^{1-\tau}\, \frac{W_{\varepsilon,\delta}(r)}{r^2} \left( 1- \frac{ \left(1-\tau\right)
 \left(1-\nu\right) }{|\varpi_1|} \,\varepsilon H_\delta(r^\varepsilon)\right).
\end{aligned} \end{equation}
Using that $f_{\rho}(\varpi_1)=0$, 
from \eqref{noma0}, \eqref{gino1} and \eqref{gino2}, we infer that 
$$  \mathbb L_{\rho,\lambda,\tau}[W_{\varepsilon,\delta}(r)]\leq 
\frac{W_{\varepsilon,\delta}(r)}{r^2} 
\varepsilon \, H_\delta(r^\varepsilon) \left[  2\,(\varpi_1+\rho) -\lambda|\varpi_1|^{-\tau}\left(1-\tau\right) \left(1-\nu\right)+\varepsilon (1+e^2)
\right]<0.
$$
Thus,  $r^\theta W^q_{\varepsilon,\delta}(r)>0\geq  \mathbb L_{\rho,\lambda,\tau}[W_{\varepsilon,\delta}(r)]$ for all 
$r\in (0,1)$. This proves Claim 3. 
\end{proof}

{\bf Proof of Step~3 concluded.}
Let $r_0\in (0,1)$ be small such that
$B_{r_0}(0)\subset\subset \Omega$. 
Let $\varepsilon>0$ be small enough as in Claim 2. 
Fix $c>0$ large such that 
\begin{equation} \label{ripa} \max_{\partial B_{r_0}(0)} u\leq c \,r_0^{-\varpi_1-\varepsilon} \left(1-e^{-r_0^\varepsilon}\right). \end{equation}

For every $\delta\in (0,1)$ and all $r\in (0,1)$, we define $W_{\varepsilon,\delta}(r)$ as in \eqref{dwet}. 
By Step~2, we have $$ \lim_{|x|\to 0} \frac{u(x)}{W_{\varepsilon,\delta}(|x|)}=0.$$ 
From our choice of $c>0$ in \eqref{ripa}, we have $u(x)\leq W_{\varepsilon,\delta}(|x|)$ for all $|x|=r_0$. 
In light of our Claim 2, we can use the comparison principle to find that 
$$ u(x)\leq W_{\varepsilon,\delta}(|x|) \quad \mbox{for all } 0<|x|\leq r_0.$$
Since $\delta\in (0,1) $ is arbitrary, by letting $\delta\to 0$ in the above inequality, we get 
$$ u(x)\leq c |x|^{-\varpi_1-\varepsilon} (1-e^{-|x|^\varepsilon})\quad  \mbox{for every }   0<|x|\leq r_0.$$ 
This implies that $ \limsup_{|x|\to 0}|x|^{\varpi_1} u(x)\leq c   $, which proves \eqref{west1}.  Step~3 is now complete. 
\end{proof}

From Steps~1--3, we conclude the proof of Theorem~\ref{ama1} (A)  when $\rho>\Upsilon$. \qed

\subsection{Proof of Theorem~\ref{ama1} (A) when $\rho=\Upsilon$} \label{secla2}
Let \eqref{e2} hold and $\beta\in (\varpi_2,0)$ in Case $[N_2]$ with $\rho=\Upsilon$. This means that 
$\varpi_1=\varpi_2=-(\lambda \tau)^{1/(\tau+1)}$ and, hence, 
\begin{equation} 
\label{adi} 2\left(\varpi_2+\rho\right) -\lambda |\varpi_2|^{-\tau} \left(1-\tau\right)=|\varpi_2|^{-\tau} \left(\lambda\tau- |\varpi_2|^{1+\tau}\right)=0. 
\end{equation}

Assume that $u$ is a positive sub-solution of \eqref{e1} satisfying 
\begin{equation} \label{egg3}
\lim_{|x|\to 0} \frac{u(x)}{|x|^{-\varpi_2}  \left| \log |x| \right|^{2/(1+\tau)}} =0.
\end{equation}

We show that 
\begin{equation} \label{liopa} \limsup_{|x|\to 0} |x|^{\varpi_2} u(x)<\infty. \end{equation}

Let $C>0$ be arbitrary. For every $\varepsilon>0$, we define 
\begin{equation} \label{zaoo} Z_\varepsilon (r)=C r^{-\varpi_2} \left[ \varepsilon \, |\log r| +1 \right]^{\frac{2}{1+\tau}}  \quad \mbox{for all } r\in (0,1). 
\end{equation}

\noindent {\bf Claim:} {\em There exists $r_*\in (0,1)$ small such that
for all $\varepsilon>0$ and $C>0$, the function $Z_\varepsilon$ in \eqref{zaoo} is a positive radial super-solution of \eqref{e1} in $B_{r_*}(0)\setminus \{0\}$ and 
$Z_\varepsilon'(r)>0$ for all $r\in (0,r_*)$.}

\begin{proof}[Proof of the Claim] 
For arbitrary $\varepsilon>0$ and $t>0$, we define 
\begin{equation} \label{gal} \psi_\varepsilon(t)=\frac{2\,\varepsilon}{ \left(1+\tau\right) \left(\varepsilon t+1\right)}. 
\end{equation}
Let $r_*\in (0,1)$ be small such that $(1+\tau)\log \,(1/r_*)>2/|\varpi_2|$. 
For $r\in (0,r_*)$ arbitrary, we have
$$ Z_\varepsilon'(r)=\frac{Z_\varepsilon(r)}{r} \left[ 
-\varpi_2- \psi_\varepsilon(\log\, (1/r)) 
\right]>0,
$$
which implies that 
\begin{equation} \label{adi0}
\begin{aligned} \lambda \,G_\tau(Z_\varepsilon(r))&=
\lambda \,|\varpi_2|^{1-\tau} \,\frac{Z_\varepsilon(r)}{r^2}  \left[ 
1- \frac{1}{|\varpi_2|}\,\psi_\varepsilon(\log\, (1/r)) 
\right]^{1-\tau}\\
&\leq \lambda \,|\varpi_2|^{1-\tau} \,\frac{Z_\varepsilon(r)}{r^2} \left[
1-\frac{(1-\tau)}{|\varpi_2|}  \psi_\varepsilon( \log\,(1/r))-\frac{\left(1-\tau\right)\tau}{2 \varpi_2^2} \, \psi^2_\varepsilon( \log\,(1/r))
\right].
\end{aligned} \end{equation}
Moreover, using \eqref{gal}, we see that 
$\psi'_\varepsilon(t)=-(1+\tau)\, \psi^2_\varepsilon(t)/2$ for all $t>0$ so that 
$$ \begin{aligned} 
\mathbb L_\rho[ Z_\varepsilon(r) ] &=
\frac{Z_\varepsilon(r)}{r^2} \left[  \varpi_2^2+2\rho \varpi_2+
 2\left(\varpi_2+\rho\right) \psi_\varepsilon ( \log\, (1/r)) +\psi_\varepsilon^2 (\log\, (1/r))
+\psi'_\varepsilon(\log\,(1/r))
\right]\\
&= \frac{Z_\varepsilon(r)}{r^2} \left[ 
 \varpi_2^2+2\rho \varpi_2+
 2\left(\varpi_2+\rho\right) \psi_\varepsilon ( \log\, (1/r)) 
+\frac{1-\tau}{2}\,  \psi^2_\varepsilon(\log\,(1/r))
\right].
\end{aligned}
$$
Hence, using \eqref{adi} and \eqref{adi0}, we infer that 
$$ \mathbb L_{\rho,\lambda,\tau} [Z_\varepsilon(r)]\leq 
0\leq r^{\theta} Z_\varepsilon^q(r)\quad \mbox{for all } r\in (0,r_*). 
$$
This completes the proof of the Claim. 
\end{proof}

{\bf Proof of \eqref{liopa} concluded.} Let $r_*\in (0,1)$ be as in the above Claim and small so that $B_{r_*}(0)\subset \subset \Omega$. 
For $\varepsilon>0$ arbitrary, we define $Z_\varepsilon$ as in \eqref{zaoo}, where
$C>0$ is large such that 
$u(x) \leq C r_*^{-\varpi_2} $ for all $ |x|=r_*$.  
Hence, $u(x)\leq Z_\varepsilon (|x|)$ for all $|x|=r_*$. 
From \eqref{egg3}, we have 
$$ \lim_{|x|\to 0} \frac{u(x)}{Z_\varepsilon(|x|)}=0.$$ 
In view of the Claim, we can apply the comparison principle to conclude that 
$$ u(x)\leq Z_\varepsilon(|x|)\quad \mbox{for all } 0<|x|\leq r_*.
$$ By letting $\varepsilon\to 0$, we attain \eqref{liopa}. 
This ends the proof of Theorem~\ref{ama1} (A) when $\rho=\Upsilon$. \qed

\subsection{Proof of Theorem~\ref{ama1} (B) when $\rho>\Upsilon$} \label{secla3}

Let \eqref{e2} hold and $\beta\in (\varpi_2,0)$ in Case $[N_2]$ with $\rho>\Upsilon$.  
Assume that $u$ is a positive super-solution of \eqref{e1} satisfying 
\begin{equation} \label{egh2}
\lim_{|x|\to 0} |x|^{\varpi_2} u(x)=\infty.
\end{equation}

We divide the proof of \eqref{nopi4} into two Steps.

\vspace{0.2cm}
{\bf Step 1:} {\em For every $\varepsilon_0>0$ small, we have $\lim_{|x|\to 0} |x|^{\varpi_2+\varepsilon_0} u(x)=\infty.$}

\begin{proof}[Proof of Step~1]
Fix $c>0$ arbitrary. 
For every $\delta,\varepsilon>0$, we define $w_{\delta,\varepsilon}(r)$ for $r>0$ as follows
\begin{equation} \label{zeps}
w_{\delta,\varepsilon}(r)=c\,r^{-\varpi_2} \left(e^\delta-e^{-r^\varepsilon}\right)^{-1}.
\end{equation}

\begin{lemma} \label{lem93} Let \eqref{e2} hold and $\beta\in (\varpi_2,0)$ in Case $[N_2]$ with $\rho>\Upsilon$.  
Then, for every $\varepsilon>0$ small, there exists $c_\varepsilon>0$ such that for all $c\in (0,c_\varepsilon)$ and $\delta\in (0,1)$, the function $w_{\delta,\varepsilon}$ in \eqref{zeps} is a positive radial sub-solution of \eqref{e1} in $B_{1}(0) \setminus \{0\}$ satisfying $w_{\delta,\varepsilon}'(r)>0$ for all $r\in (0,1)$.  
\end{lemma}

\begin{proof} In view of \eqref{raoo}, we can fix $\nu\in (0,1)$ small such that 
\begin{equation} \label{chio}  
\chi_\nu:= 2\left(\varpi_2+\rho\right) -\lambda|\varpi_2|^{-\tau}\left(1-\tau\right) \left(1+\nu\right) >0. 
\end{equation}
Let $s_\nu>0$ be small such that 
\begin{equation} \label{cao1} \left( 1-\frac{s}{|\varpi_2|} \right)^{1-\tau}\geq 1-\frac{\left(1-\tau\right) \left(1+\nu\right)}{|\varpi_2|} \,s\quad \mbox{for all } s\in (0,s_\nu). 
\end{equation}
Using that $\varpi_2<\beta<0$, we let $\varepsilon>0$ be small enough such that 
\begin{equation} \label{varn} 0<\varepsilon<\varepsilon_\nu:= \min\left\{ \frac{(\beta-\varpi_2)(q-1)}{q},\frac{\chi_\nu}{2},s_\nu \right\}.  
\end{equation}
We let $c\in (0,c_\varepsilon)$ be arbitrary, where we define $c_\varepsilon$ by 
\begin{equation} \label{ceps} c_\varepsilon=\left(\frac{ \varepsilon \chi_\nu}{2e^{q+1}}
\right)^{1/(q-1)}.
\end{equation}

Let $r\in (0,1)$ be arbitrary. By differentiating \eqref{zeps} with respect to $r$, we get
\begin{equation} \label{cao2} w_{\delta,\varepsilon}'(r)=\frac{w_{\delta,\varepsilon}(r)}{r} \left( -\varpi_2-\varepsilon\,\Psi_\delta(r^\varepsilon) \right) >0,
\end{equation}
where for every $t>0$, we define $\Psi_\delta(t)\in (0,1)$ as in \eqref{psio}, that is, $\Psi_\delta(t)=te^{-t}/(e^\delta-e^{-t})$.

We conclude the proof of Lemma~\ref{lem93} by showing that $w_{\delta,\varepsilon}$ in \eqref{zeps} satisfies 
\begin{equation} \label{owo}
\mathbb L_{\rho,\lambda,\tau} [w_{\delta,\varepsilon}(r)] \geq r^\theta w^q_{\delta,\varepsilon}(r)\quad \mbox{for all } r\in (0,1). 
\end{equation}

{\em Proof of \eqref{owo}.} 
In view of \eqref{cao1} and \eqref{cao2}, we obtain that 
\begin{equation} \label{fumi1} \begin{aligned} 
\lambda\, G_\tau(w_{\delta,\varepsilon}(r))&= \lambda \, |\varpi_2|^{1-\tau} \frac{w_{\delta,\varepsilon}(r)}{r^2}
\left(1-\frac{\varepsilon\,\Psi_\delta(r^\varepsilon) }{|\varpi_2|} \right)^{1-\tau}\\
& \geq \lambda \, |\varpi_2|^{1-\tau} \frac{w_{\delta,\varepsilon}(r)}{r^2} 
\left( 1-\frac{\left(1-\tau\right) \left(1+\nu\right)}{|\varpi_2|}\, \varepsilon \,\Psi_\delta(r^\varepsilon)
\right).
\end{aligned} \end{equation}
By differentiating \eqref{cao2} and using that $t\Psi_\delta'(t)\leq \Psi_\delta(t)$ for every $t>0$, we arrive at
\begin{equation} \label{fumi2} \begin{aligned}
\mathbb L_\rho[w_{\delta,\varepsilon} (r)] & = 
\frac{w_{\delta,\varepsilon}(r)}{r^2} \left[
\varpi_2^2+2\rho \varpi_2 +2\left(\varpi_2+\rho\right) \varepsilon\, \Psi_\delta(r^\varepsilon)-\varepsilon^2 r^\varepsilon \Psi_\delta'(r^\varepsilon) 
\right]\\
& \geq 
\frac{w_{\delta,\varepsilon}(r)}{r^2} \left\{
\varpi_2^2+2\rho \varpi_2 + \varepsilon\, \Psi_\delta(r^\varepsilon) \left[2\left(\varpi_2+\rho\right) -\varepsilon \right]
\right\}.
\end{aligned}
\end{equation}
Since $f_{\rho}(\varpi_2)=\varpi_2^2+2\rho\varpi_2+\lambda |\varpi_2|^{1-\tau}=0$, 
from \eqref{fumi1} and \eqref{fumi2}, it follows that 
\begin{equation} \label{pux1} \mathbb L_{\rho,\lambda,\tau}[w_{\delta,\varepsilon} (r)] \geq 
\varepsilon \,\Psi_\delta(r^\varepsilon) (\chi_\nu-\varepsilon) \,\frac{w_{\delta,\varepsilon}(r)}{r^2} \geq 
\frac{\varepsilon \chi_\nu }{2} \, \,\Psi_\delta(r^\varepsilon)  \,\frac{w_{\delta,\varepsilon}(r)}{r^2} .
\end{equation}
In the last inequality, we have used that $\varepsilon<\chi_\nu/2$, where $\chi_\nu$ is given in \eqref{chio}. 

On the other hand, the right-hand side of \eqref{owo} satisfies
\begin{equation} \label{pux2} r^\theta w^q_{\delta,\varepsilon}(r)\leq  \frac{\varepsilon \chi_\nu }{2} \,\Psi_\delta(r^\varepsilon)  \,\frac{w_{\delta,\varepsilon}(r)}{r^2}\quad \mbox{for every } r\in (0,1). 
\end{equation}
Indeed, since $\delta,r\in (0,1)$ and $\psi_\delta(r^\varepsilon)\geq r^\varepsilon/e^2 $, using  \eqref{varn}, \eqref{ceps}, and $c\in (0,c_\varepsilon)$, we have 
$$ \begin{aligned} \frac{r^{\theta+2}\, w^{q-1}_{\delta,\varepsilon}(r)}{\Psi_\delta(r^\varepsilon)} 
& \leq e^2 c^{q-1} r^{(\beta-\varpi_2)(q-1)-\varepsilon} \left(1-e^{-r^\varepsilon}\right)^{-(q-1)} \\
&\leq e^2 c^{q-1} r^{(\beta-\varpi_2)(q-1)-\varepsilon q} e^{\left(q-1\right)r^\varepsilon}\\
& \leq e^{q+1}  c^{q-1} \leq e^{q+1} c_\varepsilon^{q-1}=\frac{\varepsilon \chi_\nu}{2} .
\end{aligned}
$$

From \eqref{pux1} and \eqref{pux2}, we conclude the proof of \eqref{owo}. This ends the proof of Lemma~\ref{lem93}. 
\end{proof}

{\bf Proof of Step~1 concluded.} Let $\nu>0$ and $\varepsilon_\nu$ be given as in the proof of Lemma~\ref{lem93}. 

We show that for every $\varepsilon_0\in (0,\varepsilon_\nu)$, we have 
\begin{equation} \label{pob} \lim_{|x|\to 0} |x|^{\varpi_2+\varepsilon_0} u(x)=\infty .
\end{equation}
Indeed, we fix $\varepsilon\in (\varepsilon_0,\varepsilon_\nu)$ and let $c_\varepsilon$ be prescribed by Lemma~\ref{lem93} (see \eqref{ceps}). Choose $r_0\in (0,1)$ small enough such that $B_{r_0}(0)\subset\subset \Omega$. Fix $c\in (0,c_\varepsilon)$ small enough such that 
$$ \min_{|x|=r_0} u(x)\geq c \,r_0^{-\varpi_2} (1-e^{-r_0^\varepsilon})^{-1}. 
$$
For every $\delta\in (0,1)$, we define $w_{\delta,\varepsilon}$ as in \eqref{zeps}.  The above choice of $c$ ensures that 
$  u(x)\geq w_{\delta,\varepsilon}(x) $ for all $ |x|=r_0$.   
From the assumption \eqref{egh2}, we obtain that 
$ \lim_{|x|\to 0} u(x)/w_{\delta,\varepsilon}(|x|)=\infty$.  
In light of Lemma~\ref{lem93}, we can apply the comparison principle to conclude that 
$$ u(x)\geq w_{\delta,\varepsilon}(|x|) \quad \mbox{for all } 0<|x|\leq r_0. $$
By letting $\delta\to 0$ and then $|x|\to 0$, we get that 
$$\liminf_{|x|\to 0} |x|^{\varpi_2+\varepsilon} u(x)\geq c>0,$$  
which implies \eqref{pob} since $\varepsilon_0<\varepsilon$. This completes the proof of Step~1.
\end{proof}

{\bf Step 2:} {\em The assertion of \eqref{nopi4} holds: For every $\eta>0$, we have $\lim_{|x|\to 0} |x|^{\beta-\eta} u(x)=\infty$. }

\begin{proof}[Proof of Step~2]
We need only consider the case of $\eta>0$ small. Without loss of generality, we assume that 
$\eta\in (0,\beta-\varpi_2)$. We fix $\varepsilon\in (0,\varepsilon_\nu)$ small as needed so that Step~1 applies, that is, 
\begin{equation} \label{emb2} \lim_{|x|\to 0} |x|^{\varpi_2+\varepsilon} u(x)=\infty. \end{equation}
We can assume that $\varepsilon<\beta-\varpi_2-\eta$. Otherwise, there  remains nothing to prove. 

To conclude Step~2, we use an iterative process similar to the one in Proposition~\ref{lad} except that the sequence $\{m_j\}_{j}$ in \eqref{vij} is constructed here differently. We let 
\begin{equation} \label{ma1} m_1=-\varpi_2-\varepsilon>0 \end{equation}
and for every $j\geq 1$, we define $m_{j+1}$ as the unique positive solution of $f_j(t)=0$, where
\begin{equation} \label{not} f_j(t)=\lambda \,t^{1-\tau} +t^2-2\rho \,m_j\quad \mbox{for } t>0.
 \end{equation}
 Now, $f_1$ is increasing on $(0,\infty)$ with $\lim_{t\to 0^+}f_1(t)=-2\rho \,m_{1}<0$ and $\lim_{t\to \infty} f_1(t)=\infty$.
Hence, $m_2>0$ is well-defined and, by induction, we obtain the sequence of positive numbers $\{m_j\}_{j\geq 1}$, which is decreasing.  
We only remark that $m_2<m_1$ since $f_1(m_1)>0$ using that
$$ \begin{aligned} 
f_1(m_1)& =\lambda|\varpi_2|^{1-\tau} \left( 1-\frac{\varepsilon}{|\varpi_2|}\right)^{1-\tau} +\varpi_2^2
+2\rho\varpi_2+2  \left(\varpi_2+\rho\right)\varepsilon +\varepsilon^2\\
& = \left[2\left(\varpi_2+\rho\right) -\lambda |\varpi_2|^{-\tau} (1-\tau)\right] \varepsilon+
 \left( 1-\frac{\lambda \tau\, |\varpi_2|^{-1-\tau}(1-\tau) }{2}\right) \varepsilon^2+o(\varepsilon^2)    \quad \mbox{as } \varepsilon\to 0.
\end{aligned} $$
By \eqref{raoo}, we have $2\left(\varpi_2+\rho\right) -\lambda |\varpi_2|^{-\tau} (1-\tau)>0$ so that $f_1(m_1)>0$ for small $\varepsilon>0$. 
(Even when $\rho=\Upsilon$, for small $\varepsilon>0$, we have $f_1(m_1)>0$ since 
$f_1(m_1)= \varepsilon^2 [(1+\tau)/2 +o(1)]$ as $\varepsilon\to 0$.) 
Thus, $m_2<m_1$ and by induction, one can show that $\{m_j\}_{j\geq 1}$ is decreasing and there exists 
$$\lim_{j\to \infty} m_j=\ell_0\geq 0, $$
where $\ell_0^2-2\rho \ell_0+\lambda \ell_0^{1-\tau}=0$. Such equation in $\ell_0$ has three non-negative roots: $0$, $-\varpi_2$ and $-\varpi_1$. 
Since $0<m_j<m_1=-\varpi_2-\varepsilon<-\varpi_2<-\varpi_1$ for all $j\geq 2$, we obtain that $\ell_0=0$, namely, 
$$\lim_{j\to \infty} m_j=0.$$ 
Hence, there exists $k_*\geq 1$ large such that for $j=1,\ldots, k_*$, 
\begin{equation} \label{lih1} -\beta<m_{k_{*}+1} \leq -\beta+\eta<m_{k_*}. 
\end{equation}
For $\nu\in (0,m_{k_*+1}/2)$ and $1\leq j\leq k_*$, we define 
\begin{equation} \label{ximm} \xi_j(\nu):=\lambda \left(m_{j+1}-2\nu\right)^{1-\tau} +m_j^2-2\rho \,m_j. \end{equation}
Since 
$f_j(m_{j+1})=0$, where $f_j$ is given by \eqref{not} and $0<m_{j+1}<m_j$ for all $1\leq j\leq k_*$, we get
$$
 \lim_{\nu\to 0} \xi_j(\nu) =
\lambda \,m_{j+1}^{1-\tau} +m_j^2-2\rho \,m_{j}
 =m_j^2-m^2_{j+1} >0.
$$ 
We can thus fix $0<\nu<\min\,\{ m_{k_*+1}+\beta,m_{k_*+1}/2\}$ 
small such that  
\begin{equation} \label{chop} \xi_j(\nu)>0\quad \mbox{for all } 1\leq j\leq k_* . 
\end{equation} 

We now choose $\alpha \in (0,\varpi_2+\rho)$ arbitrary. This implies that 
\begin{equation} \label{haaa} 0<\alpha<\rho-m_j\quad \mbox{for all } 1\leq j\leq k_*.
\end{equation}
Indeed, remark that $\rho>-\varpi_2$ since $-\varpi_2<(\lambda \tau)^{1/(\tau+1)}<\Upsilon<\rho$ by virtue of the assumptions in Case~$[N_2]$ with $\rho>\Upsilon$. (In fact, we have $\rho>-\varpi_2$ also in Case $[N_2]$ with $\rho=\Upsilon$ since $(\lambda \tau)^{1/(\tau+1)}<\Upsilon$.)
Using that $\rho+\varpi_2<\rho-m_1\leq \rho-m_j$ for all $j\geq 1$, by the choice of $\alpha\in (0,\varpi_2+\rho)$, we obtain \eqref{haaa}

\vspace{0.2cm}
{\bf Construction of a family of sub-solutions.} 

From \eqref{chop}, we have $\xi_*=\min_{1\leq j\leq k_*} \xi_j(\nu)>0$. Let $c\in (0,\xi_*^{1/(q-1)})$ be arbitrary.   

For every $1\leq j\leq k_*$ and $\delta>0$, we define 
$v_{j,\delta}$ in the same form as in \eqref{vij}, that is, 
\begin{equation} \label{vijn}
 v_{j,\delta}(r)=  c \,r^{m_{j}} \left( e^{r^\alpha} -e^{-\delta}\right)^\frac{m_{j+1}-m_j-\nu}{\alpha}\quad \mbox{for all } r\in (0,1].
\end{equation}

We show that there exists $r_*\in (0,1)$ small such that 
for every $1\leq j\leq k_*$, the function $v_{j,\delta}$ in \eqref{vijn} is a 
positive radial sub-solution of \eqref{e1} on $0<|x|<r_*^{1/\alpha}$, namely, 
\begin{equation} \label{fgg}
 \mathbb L_{\rho,\lambda,\tau} [v_{j,\delta}(r)] \geq  r^\theta v^{q}_{j,\delta}(r)\quad \mbox{for all } r\in (0,r_*^{1/\alpha}).
\end{equation}

{\bf Proof of \eqref{fgg}.} This is similar to the proof of \eqref{fass}. 
Let $r_*\in (0,1)$ be small such that $B_{r_*^{1/\alpha}}(0)\subset \subset \Omega$ and for all $t\in (0,r_*)$, we have \eqref{coppi}.  
Let $1\leq j\leq k_*$ and $r\in (0,r_*^{1/\alpha})$ be arbitrary. 
Note that 
in the expression of $\mathbb L_\rho[v_{j,\delta}(r) ]$ in \eqref{sos}, 
we again have that the coefficients of $ r^\alpha e^{r^\alpha}/(e^{r^\alpha}-e^{-\delta})$ and  
$ r^{2\alpha} e^{2r^\alpha}/(e^{r^\alpha}-e^{-\delta})^2$ are positive (see \eqref{haaa}).
Using \eqref{coppi}, we recover \eqref{ere} and \eqref{ror1}. So, using \eqref{ximm} and
$c\in (0,\xi_*^{1/(q-1)})$ with 
 $\xi_*=\min_{1\leq j\leq k_*} \xi_j(\nu)>0$, we get
$$ \mathbb L_{\rho,\lambda,\tau}[v_{j,\delta}(r)]]\geq \xi_j(\nu) \,\frac{v_{j,\delta}(r)}{r^2}  \geq \xi_* \,\frac{v_{j,\delta}(r)}{r^2}
\geq c^{q-1} \, \frac{v_{j,\delta}(r)}{r^2}.
$$
As in the proof of \eqref{fass}, 
by the choice of $\nu <m_{k_*+1}+\beta$, we obtain that 
$$ r^{\theta+2} v^{q-1}_{j,\delta}(r) \leq c^{q-1}.
$$
This completes the proof of \eqref{fgg}. 

\vspace{0.2cm}
{\bf Proof of Step~2 concluded.} We diminish $c>0$ such that 
$c<\min_{|x|=r_*^{1/\alpha}} u(x)$.  We point out that from \eqref{ma1}, we have $m_1=-\varpi_2-\varepsilon$ and \eqref{emb} holds because of \eqref{emb2}. 
As in the proof of Proposition~\ref{lad}, 
by an inductive argument on $j=1,\ldots, k_*$, we obtain that 
$$ u(x)\geq v_{k_*,\delta}(|x|) \quad \mbox{for all } 0<|x|\leq r_*^{1/\alpha}. 
$$
By letting $\delta\to 0$ and then $|x|\to 0$, we regain 
$$ \liminf_{|x|\to 0} |x|^{-m_{k_*+1}+\nu} u(x)\geq c,$$
which by virtue of $m_{k_{*}+1} \leq -\beta+\eta$
(see \eqref{lih1}), implies that $\lim_{|x|\to 0} |x|^{\beta-\eta} u(x)=\infty$. 
\end{proof}

This completes the proof of Theorem~\ref{ama1} (B) when $\rho>\Upsilon$. \qed

\subsection{Proof of Theorem~\ref{ama1} (B) when $\rho=\Upsilon$} \label{secla4}
Let \eqref{e2} hold and $\beta\in (\varpi_2,0)$ in Case $[N_2]$ with $\rho=\Upsilon$.  
Let $u$ be a positive super-solution of \eqref{e1} satisfying  
\begin{equation} \label{agh1}
\lim_{|x|\to 0} \frac{u(x)}{|x|^{-\varpi_2}  \left|\log |x| \right|^{2/(1+\tau)} }=\infty.
\end{equation}

The same reasoning as in Step~2 in Section~\ref{secla3} leads to \eqref{nopi4} provided that 
$u$ satisfies 
\begin{equation} \label{qax} \mbox{for every } \varepsilon_0>0\ \mbox{small, we have }  \lim_{|x|\to 0} |x|^{\varpi_2+\varepsilon_0} u(x)=\infty.   
\end{equation}

We point out that \eqref{qax} corresponds to Step~1 in Section~\ref{secla3} only in statement. However, the proof given in Section~\ref{secla3} 
uses essentially that $\rho>\Upsilon$ (in the form of \eqref{raoo}) and cannot be applied in our context of $\rho=\Upsilon$ when \eqref{adi} is valid.  
Moreover, instead of the assumption \eqref{egh2}, here we start with the property \eqref{agh1} (due to the change in the expression of $\Phi_{\varpi_2(\rho)}$ in \eqref{phi22}). 
 
 Compared with Section~\ref{secla3}, the main novelty here is the derivation of \eqref{qax}.  
 We would need two family of sub-solutions of \eqref{e1} to attain \eqref{qax} (see \eqref{waq} and \eqref{vaz}).  
 
 First, we improve \eqref{agh1} by showing in Lemma~\ref{ger0} that for every $s>2/(1+\tau)$, we have 
 \begin{equation} \label{qop} 
\lim_{|x|\to 0} \frac{u(x)}{ |x|^{-\varpi_2} \left|\log |x| \right|^s}=\infty.
\end{equation}
 
 Using \eqref{qop} and another family of sub-solutions, we conclude \eqref{qax} in Lemma~\ref{gera}.

\begin{lemma} \label{ger0} For every $s>2/(1+\tau)$, we have \eqref{qop}. 
\end{lemma}

\begin{proof} Let $\alpha,\eta,c>0$ be fixed arbitrarily. For every $\delta>0$, we define 
\begin{equation} \label{waq} w_{\delta}(r)=c\, r^{-\varpi_2} |\log r|^{\frac{2}{1+\tau} } \left( \delta+ |\log r|^{-\alpha}\right)^{-\eta}\quad \mbox{for all } r\in (0,1/e).
\end{equation}

{\bf Claim:} {\em For any $\alpha\in (0,1)$ and $\eta>0$, there exist $r_*\in (0,1/e)$ and $c_*>0$ such that for every $c\in (0,c_*)$ and $\delta\in (0,1)$, the function $w_\delta$ in \eqref{waq} is
a positive radial sub-solution of \eqref{e1} in $B_{r_*}(0)\setminus \{0\}$ and $w_\delta'(r)>0$ for every $r\in (0,r_*)$. }

\begin{proof}[Proof of the Claim]
Let $\alpha\in (0,1)$ and $\eta>0$ be fixed. We define $C_*>0$ as follows
\begin{equation} \label{csa} C_*=\frac{\left(1-\tau\right)\tau \left(1+\tau\right) 2^{\tau+2}}{3! \,|\varpi_2|^3}\, \left(\frac{2}{1+\tau} +\alpha \eta\right)^3.
\end{equation}
Let $r_*\in (0,1/e)$ be small such that for all $r\in (0,r_*)$, we have 
\begin{eqnarray}
& \displaystyle \frac{1}{|\log r|}\left(\frac{2}{1+\tau} +\alpha\,\eta \right)<\frac{|\varpi_2|}{2},  \label{edo1}\\
& \alpha  \left(1-\alpha\right)\eta\, |\log r|^{1-\alpha} >2 \,C_* \left(1+ |\log r|^{-\alpha} \right). \label{edo2} 
\end{eqnarray}
Let $c\in (0,c_*)$ be arbitrary, where $c_*>0$ is defined by 
\begin{equation} \label{ger5} c_*= \left( \frac{C_*}{M_*} \right)^{\frac{1}{q-1}}\ \mbox{and  } M_* = \sup_{r\in (0,r_*) } r^{(\beta-\varpi_2)(q-1) } |\log r|^{3+ (q-1) \left( \frac{2}{1+\tau} +\alpha \eta\right)} .
\end{equation}
We point out that $M_*\in (0,\infty)$ since $\beta\in (\varpi_2,0)$. 

Let $r\in (0,r_*)$ and $\delta\in (0,1)$ be arbitrary. For $t>0$, we define $F_{\delta}(t)$ by 
\begin{equation} \label{ger1} F_\delta(t)=\frac{2}{1+\tau} +\frac{ \alpha \eta\, t}{\delta+t}\in \left(0, \frac{2}{1+\tau}+\alpha \eta\right).
\end{equation}
Our choice of $r_*>0$ such that \eqref{edo1} holds gives that 
\begin{equation} \label{ger2}  \frac{F_\delta (|\log r|^{-\alpha})}{ |\varpi_2| |\log r|}\in (0,1/2)\quad \mbox{ for all } r\in (0,r_*).
\end{equation}
Thus, by differentiating \eqref{waq} with respect to $r$, we get 
\begin{equation} \label{aqs} w_\delta'(r)=\frac{w_\delta(r)}{r} \left(-\varpi_2 -\frac{F_\delta(|\log r|^{-\alpha})}{|\log r|}
\right)>\frac{|\varpi_2|}{2}\, \frac{w_\delta(r)}{r}>0.
\end{equation}
Hence, in view of \eqref{ger1} and \eqref{ger2}, 
we obtain that
$$ \begin{aligned}
 \lambda \,G_\tau(w_\delta(r))& =\lambda\,|\varpi_2|^{1-\tau} \frac{w_\delta(r)}{r^2} 
\left(1 -\frac{F_\delta(|\log r|^{-\alpha})}{|\varpi_2|\,|\log r|}
\right)^{1-\tau}\\
&\geq  \lambda\,|\varpi_2|^{1-\tau} \frac{w_\delta(r)}{r^2}  \left(1- \frac{(1-\tau) F_\delta(|\log r|^{-\alpha})}{|\varpi_2|\,|\log r|}
-\frac{(1-\tau)\tau\,F_\delta^2 (|\log r|^{-\alpha})}{2\, \varpi_2^2\,\log^2 r} - \frac{C_*}{|\log r|^3}
\right).
\end{aligned}
$$
where $C_*>0$ is defined in \eqref{csa}. 
Using \eqref{aqs}, we arrive at 
$$ \begin{aligned} 
\mathbb L_\rho(w_\delta(r)) =\frac{w_\delta(r)}{r^2} 
& \left[ \varpi_2^2+2\rho \varpi_2 +\frac{2\left(\varpi_2+\rho\right) F_\delta(|\log r|^{-\alpha})}{|\log r|} \right.
\\
& \left.+\frac{F^2_\delta(|\log r|^{-\alpha}) -F_\delta(|\log r|^{-\alpha}) -\alpha |\log r|^{-\alpha} F'_\delta(|\log r|^{-\alpha})  }{\log^2 r}
\right]. 
\end{aligned}
$$
For $t>0$, we define $J_\delta(t)$ as follows
\begin{equation} \label{ger3} J_\delta(t)=  \frac{1+\tau}{2}\,F^2_\delta(t) -F_\delta(t) -\alpha\, t F'_\delta(t).
\end{equation}
Since Case $[N_2]$ holds with $\rho=\Upsilon$, we have $\beta_1=\varpi_2=-(\lambda \tau)^{1/(1+\tau)}$ and \eqref{adi} holds. Hence, 
\begin{equation} \label{ger4} \mathbb L_{\rho,\lambda,\tau}(w_\delta(r)) \geq 
\frac{w_\delta(r)}{r^2\, |\log r|^3}  \left[ |\log r| \,J_\delta(|\log r|^{-\alpha})- C_*
\right].
\end{equation}
Recall that $\alpha,\delta\in (0,1)$. In view of \eqref{ger1} and \eqref{ger3}, we find that 
$$ J_\delta(t)=\frac{\alpha \eta \,t}{\delta+t} \left(1-\frac{\alpha \delta}{\delta+t}\right)+\frac{(1+\tau)}{2} \frac{\alpha^2\eta^2 t^2}{(\delta+t)^2}>
\frac{\alpha \left(1-\alpha\right) \eta \, t}{\delta+t}\geq \frac{\alpha \left(1-\alpha\right) \eta \, t}{1+t}.
$$
This, jointly with \eqref{edo2}, implies that 
$$  |\log r| \,J_\delta(|\log r|^{-\alpha})\geq 2\,C_*\quad \mbox{for all } r\in (0,r_*).
$$
Therefore, using \eqref{ger4}, we infer that 
$$  \mathbb L_{\rho,\lambda,\tau}(w_\delta(r)) \geq 
C_* \frac{w_\delta(r)}{r^2\, |\log r|^3} \quad \mbox{for all } r\in (0,r_*).
$$
On the other hand, from the definition of $w_\delta$ in \eqref{waq} and \eqref{ger5}, we have 
$$ r^\theta w_\delta^q(r)\leq \frac{w_\delta(r)}{r^2} \,c^{q-1} r^{\left(\beta-\varpi_2\right)\left(q-1\right)} |\log r|^{\left(q-1\right) \left(\frac{2}{1+\tau} +\alpha \eta\right)}\leq C_* \frac{w_\delta(r)}{r^2\, |\log r|^3} \quad \mbox{for all } r\in (0,r_*).
$$
This completes the proof of the Claim. 
\end{proof}

{\bf Proof of Lemma~\ref{ger0} concluded.} 
Let $s>2/(1+\tau)$ be arbitrary. Fix $\alpha\in (0,1)$ and $\eta>0$ such that $\alpha\, \eta>s-2/(1+\tau)$. 
Let $r_*\in (0,1/e) $ and $c_*>0$ be given by the Claim. 

We choose $c\in (0,c_*)$ small enough such that 
$$ \min_{|x|=r_*} u(x)\geq c r_*^{-\varpi_2} |\log r_*|^{\alpha \eta +2/(1+\tau)}. 
$$

Let $\delta\in (0,1)$ be arbitrary. 
This choice of $c$ yields that $u(x)\geq w_\delta(|x|)$ for all $|x|=r_*$. Moreover, from \eqref{agh1}, we have 
$$ \lim_{|x|\to 0} \frac{u(x)}{w_\delta(|x|)}=\infty. 
$$
Hence, by the comparison principle, we infer that $u(x)\geq w_\delta(|x|)$ for all $0<|x|\leq r_*$. By letting $\delta\to 0$ and then $|x|\to 0$, we arrive at 
$$ \liminf_{|x|\to 0} |x|^{\varpi_2} |\log |x||^{-\frac{2}{1+\tau} -\alpha \eta} u(x) \geq c>0. 
$$
Since $\alpha\, \eta>s-2/(1+\tau)$, we readily conclude \eqref{qop}. 
This ends the proof of Lemma~\ref{ger0}. 
\end{proof}

Based on Lemma~\ref{ger0}, we are ready to conclude \eqref{qax} in our next result. 

\begin{lemma} \label{gera} For every $\varepsilon_0>0$ small enough, we have $\lim_{|x|\to 0} |x|^{\varpi_2+\varepsilon_0} u(x)=\infty$. 
\end{lemma}

\begin{proof} Fix $\mu>2/(1+\tau)$ and choose $\nu>0$ small such that 
$$ \nu \in \left(0, \frac{1+\tau -2/\mu}{1-\tau}\right).   
$$
Let $s_\nu>0$ be small enough such that for all $s\in (0,s_\nu)$, 
\begin{equation} \label{oga} \left(1-\frac{s}{|\varpi_2|}\right)^{1-\tau} \geq 1-\frac{(1-\tau)}{|\varpi_2|} s-\frac{\left(1-\tau\right)\tau \left(1+\nu\right)}{2\varpi_2^2} s^2. 
\end{equation}

We show that for every $\varepsilon_0\in (0, \min\, \{s_\nu,\beta-\varpi_2\} )$, we have 
\begin{equation} \label{omn} \lim_{|x|\to 0} |x|^{\varpi_2+\varepsilon_0} u(x)=\infty.
\end{equation}

We fix $\alpha=\alpha(\mu,\nu)>0$ small such that $\varepsilon_0<\alpha\mu <\min\, \{s_\nu,\beta-\varpi_2\} $. 
Let $c>0$ be arbitrary. 

For every $\delta>0$ and $r\in (0,1)$, we define 
\begin{equation} \label{vaz}
z_\delta(r)=\frac{c}{\delta^\mu} r^{-\varpi_2} \log^\mu (1+\delta/r^\alpha).  
\end{equation}

{\bf Claim:} {\em 
There exists $c_\mu>0$ such that for every $c\in (0,c_\mu)$ and $\delta>0$, the function $z_\delta$ in \eqref{vaz} 
is a positive radial sub-solution of \eqref{e1} in $B_1(0)\setminus \{0\}$ and $z_\delta'(r)>0$ for every $r\in (0,1)$.}

\begin{proof}[Proof of the Claim] 
For every $t>0$, we define $J(t)$ as follows
\begin{equation} \label{jef} J(t)=\frac{t}{\left( t+1\right) \log\,(t+1)}\in (0,1). 
\end{equation}

We have $M_\mu\in (0,\infty)$, where we set
\begin{equation} \label{juw}  M_\mu=
\sup_{t\in (0,\infty) } (1+t)^2\left( \frac{\log \,(1+t) }{t}\right)^{\mu \left(q-1\right)+2}.
\end{equation}
Let $c\in (0,c_\mu)$ be arbitrary, where $c_\mu>0$ is defined by 
\begin{equation} \label{cmu1} c_\mu=\left[ \frac{\alpha^2\mu^2}{M_\mu} \left(\frac{1+\tau-\left(1-\tau\right)\nu}{2}-\frac{1}{\mu} \right)
\right]^{\frac{1}{q-1}}. \end{equation}
In what follows, we let $r\in (0,1)$. 
From \eqref{vaz} and using that $0<\alpha\mu <\beta-\varpi_2<|\varpi_2|$, we get
\begin{equation} \label{oga1} z'_\delta(r)=\frac{z_\delta(r)}{r} \left(-\varpi_2-\alpha \mu J(\delta/r^\alpha) \right)>0. 
\end{equation}
Since $\alpha\mu\in (0,s_\nu)$ and $J(t)\in (0,1)$, from \eqref{oga} and \eqref{oga1}, we infer that 
$$  \begin{aligned}
\lambda \,G_\tau[z_\delta(r)]& =\lambda |\varpi_2|^{1-\tau} \frac{z_\delta(r)}{r^2}  \left(1-\frac{\alpha \mu}{|\varpi_2|} J(\delta/r^\alpha) \right)^{1-\tau}\\
&\geq \lambda |\varpi_2|^{1-\tau} \frac{z_\delta(r)}{r^2}   
\left( 1-\frac{\left(1-\tau\right)\alpha\mu }{|\varpi_2|} J(\delta/r^\alpha) -\frac{\left(1-\tau\right)\tau \left(1+\nu\right) (\alpha  \mu)^2 }{2\varpi_2^2} J^2(\delta/r^\alpha)
\right).
\end{aligned}
$$
By a standard calculation, we obtain that 
$$ \mathbb L_\rho[z_\delta(r)] =\frac{z_\delta(r)}{r^2} \left[ \varpi_2^2+2\rho\varpi_2+2\,(\varpi_2+\rho) \alpha\mu J(\delta/r^\alpha) +
(\alpha\mu)^2 J^2(\delta/r^\alpha) + \alpha^2 \mu \,J'(\delta/r^\alpha) \,\delta/r^\alpha
\right].  
$$
Using that $f_{\rho}(\varpi_2)=0$ and \eqref{adi}, 
we see that
\begin{equation} \label{jer} 
 \mathbb L_{\rho,\lambda,\tau}[z_\delta(r)] \geq 
\frac{z_\delta(r)}{r^2} \left[ \alpha^2\mu^2 \left(1-\frac{\left(1-\tau\right)\left(1+\nu\right)}{2} \right) J^2(\delta/r^\alpha) +\alpha^2 \mu J'(\delta/r^\alpha)\,\delta/r^\alpha
\right].
\end{equation}
From the definition of $J$ in \eqref{jef}, we get that 
$$ t J'(t)=\frac{t}{(1+t)^2 \,\log\, (1+t)} -J^2(t)>-J^2(t)\quad \mbox{for all } t>0. 
$$ 
Using this fact in \eqref{jer}, we arrive at 
$$   \mathbb L_{\rho,\lambda,\tau}[z_\delta(r)] \geq 
\frac{z_\delta(r)}{r^2} \alpha^2\mu^2 \left(\frac{1+\tau-\left(1-\tau\right)\nu}{2}-\frac{1}{\mu} \right) J^2(\delta/r^\alpha). 
$$
On the other hand, using \eqref{juw}, $\alpha\mu <\beta-\varpi_2$ and $q>1$,  we get
$$ \begin{aligned} 
r^\theta z_\delta^q(r)& =\frac{z_\delta(r)}{r^2} c^{q-1} r^{(\beta-\varpi_2-\alpha \mu)\left(q-1\right)} \left( \frac{\log\,(1+\delta/r^\alpha)}{\delta/r^\alpha}\right)^{\mu(q-1)}\\
&
\leq  \frac{z_\delta(r)}{r^2} c^{q-1} r^{(\beta-\varpi_2-\alpha \mu)(q-1)} M_\mu \,J^2 (\delta/r^\alpha)\\
& \leq 
\frac{z_\delta(r)}{r^2} c^{q-1} M_\mu \,J^2 (\delta/r^\alpha)\quad \mbox{for all } r\in (0,1).
\end{aligned} $$
Since $c\in (0,c_\mu)$ with $c_\mu$ given in \eqref{cmu1}, we readily conclude that 
$$ \mathbb L_{\rho,\lambda,\tau}[z_\delta(r)] \geq r^\theta z_\delta^q(r)\quad \mbox{for all } r\in (0,1).  
$$
This finishes the proof of the Claim.
\end{proof}

{\bf Proof of Lemma~\ref{gera} concluded.}
Let $r_0\in (0,1)$ be  small such that $B_{r_0}(0)\subset \subset\Omega$. Let $c_\mu>0$ be as in the Claim. 
Choose $c\in (0,c_\mu)$ small such that
$ \min_{|x|=r_0}u(x)\geq c\,r_0^{-\varpi_2-\alpha \mu} $. This ensures that $u(x)\geq z_\delta(|x|)$ for all $|x|=r_0$ and $\delta>0$. 
In view of Lemma~\ref{ger0}, we find that 
$$ \lim_{|x|\to 0} \frac{u(x)}{z_\delta(|x|)}=0\quad \mbox{for every } \delta>0. 
$$
Hence, by the comparison principle, we infer that 
$$u(x)\geq z_\delta(|x|)\quad \mbox{for all } 0<|x|\leq r_0\ \mbox{and }\delta>0.$$
By letting $\delta\to 0$ and then $|x|\to 0$, we arrive at 
\begin{equation} \label{omma} \liminf_{|x|\to 0} |x|^{\varpi_2+\alpha\mu} u(x)\geq c>0.  
\end{equation}
Since $\varepsilon_0<\alpha\mu<\beta-\varpi_2$, from \eqref{omma}, we attain \eqref{omn}. 
 \end{proof}

In light of Lemma~\ref{gera}, using the same proof as in Step~2 in Section~\ref{secla3}, we conclude \eqref{nopi4}.
This finishes the proof of Theorem~\ref{ama1} (B) when $\rho=\Upsilon$. 
\qed


\part{Proof of Theorem~\ref{globo}}

\section{Optimal condition for the existence of global solutions} \label{nine}

The main result of this section is the following.

\begin{theorem} \label{ipo} Let \eqref{e2} hold. Then, \eqref{e1} has positive solutions in $\mathbb R^N\setminus \{0\}$ if and only if 
$f_\rho(\beta)>0$.
\end{theorem}

If $f_\rho(\beta)>0$, then \eqref{e1} has at least 
one positive solution in $\mathbb R^N\setminus \{0\}$,  see $U_{\rho,\beta}$ in \eqref{u0}. 
We conclude Theorem~\ref{ipo} from Proposition~\ref{haha}, which relies essentially on 
Proposition~\ref{pzerop} below.

\begin{proposition} \label{pzerop} Let \eqref{e2} hold 
and $f_{\rho}(\beta)\leq 0$. Then, every  
positive solution $u$ of \eqref{e1} satisfies $\lim_{|x|\to 0} |x|^\beta u(x)=0$. 
\end{proposition}

\begin{proof} 
By Corollary~\ref{zerop}, it remains to study the case $\beta=0$. Let $u$ be any positive solution of \eqref{e1}. 
  Let $r_0>0$ be small such that $B_{4r_0}(0)\subset\subset \Omega$. Let $d_1>1$ be as in 
  Proposition~\ref{aurr}. By the proof of 
  Lemma~\ref{serio}, there exists a positive radial solution $V$ of \eqref{e1} in $B_{r_0}(0)\setminus \{0\}$ such that
  \begin{equation} \label{woss}  u(x)/d_1\leq V(|x|) \leq u(x)\quad \mbox{for every } 0<|x|\leq r_0.
  \end{equation}

  \vspace{0.2cm}
 We show that every positive radial solution $v$ of \eqref{e1} in $B_{r_0}(0)\setminus \{0\}$ satisfies 
  $\lim_{r\to 0^+} v(r)=0$. Indeed, since $\theta=-2$, we see that $v$ is a positive solution of the following ODE
  \begin{equation} \label{froi} v''(r)+(1-2\rho)\frac{v'(r)}{r} +\lambda\, v^\tau(r) \frac{|v'(r)|^{1-\tau}}{r^{1+\tau}}=r^{-2} v^q(r)\quad \mbox{for } 0<r<r_0. 
  \end{equation} 
  Observe that there exists $r_1\in (0,r_0) $ such that  $v'(r)\not=0$ for every $r\in (0,r_1)$. 
  Indeed, if $v'(r_*)=0$ for some $r_*\in (0,r_0)$, then $v''(r_*)>0$ so that $r_*$ is a local minimum. As $v$ has no local maximum on $(0,r_0)$, it follows that $v'$ has at most one critical point on $(0,r_0)$.   
  This shows that $v'$ has constant sign on $(0,r_1)$ for some $r_1\in (0,r_0)$. Hence, in view of Theorem~\ref{thi}, there exists 
  $$ \lim_{r\to 0^+} v(r)\in [0,\infty).$$   
  
  Assume by contradiction that $\lim_{r\to 0^+} v(r)\in (0,\infty)$. We define $Y(r)$ by 
  \begin{equation} \label{yire} Y(r)=\frac{r v'(r)}{v(r)}\quad \mbox{for every }r\in (0,r_1).
  \end{equation}	
  
  By a simple calculation, we find that 
  $$ r Y'(r)=-Y^2(r) +2\rho Y(r) +v^{q-1}(r)-\lambda |
  Y(r)|^{1-\tau}.$$
  
  We show that $Y'(r)\not =0$ for any $r>0$ sufficiently small. We reason by contradiction. Suppose that
  $Y'(r_k)=0$ for all $k\geq 2$, where 
$\{r_k\}_{k\geq 2}$ is a sequence of positive numbers decreasing to zero as $k\to \infty$. We assume that $r_{k+1}<r_k\leq r_1$ for all $k\geq 2$. 
Then, we obtain that 
$$ r_k^2 \,Y''(r_k)=\left(q-1\right) [v(r_k)]^{q-1} Y(r_k) \quad \mbox{for all } k\geq 2. 
$$
Hence, since $q>1$, all the critical points $r_k$ ($k\geq 2$) of $Y$ will be local 
minima points for $Y$ if $v'>0$ on $(0,r_1)$ (respectively, local maxima points if $v'<0$ on $(0,r_1$). This is a contradiction. 

Thus, $Y'(r)\not =0$ for every $r>0$ sufficiently small. Since $\lim_{r\to 0^+} v(r)\in (0,\infty)$, we infer that there exists $\lim_{r\to 0^+} Y(r)=0$.  
We multiply \eqref{froi} by $r^2$ and taking $r\to 0^+$, we arrive at 
$$ \lim_{r\to 0^+} r^2 \,v''(r)=\lim_{r\to 0^+} [v(r)]^{q}\in (0,\infty). 
$$
This is a contradiction with $\lim_{r\to 0^+} rv'(r)=0$. The proof of $\lim_{r\to 0^+} v(r)=0$ is now complete. 
  This fact, jointly with \eqref{woss}, implies that $\lim_{|x|\to 0} u(x)=0$.  
\end{proof}

Consequently, when $\beta\leq 0$, we obtain the following non-existence result for \eqref{e1} on a bounded domain $\Omega$, subject to a homogeneous Dirichlet boundary condition on $\partial \Omega$.  

\begin{cor} \label{non-exi}
Let $\Omega\subset \mathbb R^N$ be a bounded domain such that $0\in \Omega$. Let \eqref{e2} hold, $\rho,\lambda\in \mathbb R$ and $\theta\leq -2$. Then, \eqref{e1}, subject to $u=0$ on $\partial \Omega$, has no positive $C^2(\Omega\setminus \{0\})\cap C(\overline{\Omega}\setminus\{0\})$-solutions. 
\end{cor}

\begin{proof}
Assume that $ u$ is a positive $C^2(\Omega\setminus \{0\})\cap C(\overline{\Omega}\setminus\{0\})$-solution
of \eqref{e1}, subject to $u=0$ on $\partial\Omega$. Then,   
$\lim_{|x|\to 0} u(x)=0$ by Corollary~\ref{contt} (for $\theta<-2$) and Proposition~\ref{pzerop} (for $\theta=-2$). Hence, by setting $u(0)=0$, then $\max_{\overline{\Omega}}u=u(x_*)>0$ 
for some $x_\star\in \Omega\setminus \{0\}$. Since  $\Delta u(x_\star)\leq 0$ and $\nabla u(x_\star)=0$, the left-hand side (right-hand side) of \eqref{e1} at $x=x_\star$ would be non-positive (positive). This is a contradiction. \end{proof}

\begin{cor} \label{goos}
Let \eqref{e2} hold 
and $\rho,\lambda,\theta\in \mathbb R$. Let $u$ be any positive solution of \eqref{e1} in $\Omega$, where $\Omega$ is a domain such that
$\mathbb R^N\setminus \overline {B_R(0)}\subseteq \Omega$ for some $R>0$.
\begin{itemize}	
\item[(i)] If $f_{\rho}(\beta)\leq 0$, then $\lim_{|x|\to \infty} |x|^\beta u(x)=0$. 	
\item[(ii)] If $f_\rho(\beta)>0$, then either 
$\lim_{|x|\to \infty} |x|^\beta u(x)=0$ or $\lim_{|x|\to \infty} |x|^\beta u(x)=\Lambda$. 
\end{itemize}
\end{cor}

\begin{proof} We apply the modified Kelvin transform as in the Appendix~\ref{Kel}, then use Proposition~\ref{pzerop} and Theorem~\ref{pro1} (c) to conclude the claims of (i) and (ii), respectively.  
\end{proof}

\begin{proposition} \label{haha} Let \eqref{e2} hold. If $f_\rho(\beta)\leq 0$, then \eqref{e1} in $\mathbb R^N\setminus \{0\}$ has no positive solutions. 
\end{proposition}

\begin{proof}
	Assume by contradiction that $u$ is a positive $C^1(\mathbb R^N\setminus \{0\})$-solution of \eqref{e1}. 
		Then, by Proposition~\ref{pzerop} and Corollary~\ref{goos}, we have
\begin{equation} \label{wola} \lim_{|x|\to 0} |x|^\beta u(x)=0,\quad 
\lim_{|x|\to \infty} |x|^{\beta} u(x)=0. \end{equation}

$\bullet$ Assume that $\beta\not=0$.  
Since $f_\rho(\beta)\leq 0$, we see that 
$\varepsilon |x|^{-\beta}$ is a super-solution of \eqref{e1} in $\mathbb R^N\setminus \{0\}$. 
 Hence, from \eqref{wola} and the comparison principle, it follows that 
 $$ u(x)\leq \varepsilon |x|^{-\beta}\quad \mbox{for every }
 |x|>0.$$
 By letting $\varepsilon\to 0,$ we get $u\equiv 0$ in $\mathbb R^N\setminus \{0\}$, which is a contradiction. 
 
 \vspace{0.2cm}  
$\bullet$ If $\beta=0$, then \eqref{wola} gives that 
$\lim_{|x|\to 0} u(x)=\lim_{|x|\to \infty} u(x)=0$ and hence, $u$ achieves its positive global maximum at some point $x_*\in \mathbb R^N\setminus \{0\}$. 
Let $\varepsilon>0$ be small and $R>0$ be large such that $x_*\in D_{\varepsilon,R}:=\{x\in \mathbb R^N:\ \varepsilon<|x|<R\}$ and $u|_{\partial D_{\varepsilon,R}}<u(x_*)$. 
Then, by Proposition~\ref{aurr} and the elliptic regularity theory, $u\in C^2(D_{\varepsilon,R})$ is a classical solution of \eqref{e1} on $D_{\varepsilon,R}$ and 
$$ 0<|x_*|^{-2} [u(x_*)]^q=\mathbb L_{\rho,\lambda,\tau}[u(x_*)]=\Delta u(x_*)\leq 0, $$  
 which is a contradiction.
 This completes the proof of Proposition~\ref{haha}. 
\end{proof}

\section{Classification of solutions
in $\mathbb R^N\setminus \{0\}$ (Theorem~\ref{rod1c})} \label{amod}

Throughout this section, we assume that \eqref{e2} holds, $\theta<-2$ and $f_{\rho}(\beta)>0$. 

In Theorem~\ref{rod1c} below,  we reveal 
all the behaviours at infinity that can coexist with those near zero 
for the positive solutions of \eqref{e1} in $\mathbb R^N\setminus \{0\}$. 

\begin{theorem}
\label{rod1c} 	
Let \eqref{e2} hold, $\theta<-2$ and $f_\rho(\beta)>0$. 
Let $u$ be any positive solution of \eqref{e1} in $\mathbb R^N\setminus \{0\}$ other than $U_{\rho,\beta}$ in \eqref{u0}.  

$\bullet$ Assume that $\beta\in (\varpi_2,0)$ in Case $ [N_2]$. Then, we have the following three possibilities: 
\begin{itemize}
	\item[(I)] Either $u(x)\sim b \,\Phi_{\varpi_2(\rho)}(|x|)$ as $|x|\to 0$ for some $b\in \mathbb R_+$ and $u(x)\sim U_{\rho,\beta}(|x|)$ as $|x|\to \infty$;
	\item[(II)] Or $u(x)\sim U_{\rho,\beta}(|x|)$ as $|x|\to 0$ and 
	 $\lim_{|x|\to \infty} u(x)=c$ for some $c\in \mathbb R_+$;	
	\item[(III)] Or, there exist $b,c\in \mathbb R_+$ such that $u(x)\sim b \,\Phi_{\varpi_2(\rho)}(|x|)$ 
	as $|x|\to 0$ and $\lim_{|x|\to \infty} u(x)=c$. 	
\end{itemize} 

$\bullet$ In the remaining situations, we have \begin{itemize} 
\item[(II$'$)] $u(x)\sim U_{\rho,\beta}(|x|)$ as $|x|\to 0$ 
and there exists $c\in \mathbb R_+$ such that 
\begin{eqnarray} 
& \displaystyle \lim_{|x|\to \infty} |x|^{\varpi_1} u(x)=c \ \ 
\mbox{in Case }[N_1]\ \mbox{with } \beta<\varpi_1, \label{roz1} \\
& \displaystyle  \lim_{|x|\to \infty} \frac{u(x)}{\Psi_{\varpi_1(\rho)}(|x|)}=c \ \ 
\mbox{in Case }[N_2]\ \mbox{with } \beta<\varpi_1, \label{miau} \\ 
& \displaystyle
\lim_{|x|\to \infty} \frac{u(x)}{\log |x|}=c \ \  \mbox{if } \lambda=2\rho\ \mbox{when } \tau=0, \label{must1}  \\
& \displaystyle \lim_{|x|\to \infty} u(x)=c  \ \  \mbox{in Case } [N_0]\  \mbox{with } 
\lambda\not=2\rho\ \mbox{if } \tau=0. \label{roz2}
\end{eqnarray}	
\end{itemize}	

\end{theorem}

\begin{rem} {\rm  In Theorems~\ref{exol1}, \ref{exol2} and \ref{exol3}, we conclude the non-trivial task of proving that each of the asymptotic profiles in the classification of Theorems~\ref{rod1c} (at zero and at infinity taken together) is actually realised by a unique positive radial solution. 
}
\end{rem}

\subsection{Strategy for proving Theorem~\ref{rod1c}}  \label{trap1}
Let $u$ be any positive solution of \eqref{e1} in $\mathbb R^N\setminus \{0\}$ other than $U_{\rho,\beta}$. 
In Proposition~\ref{ze-ta}, we establish the conclusion of (II)$'$ in Theorem~\ref{rod1c} when $\lambda=2\rho$ and $\tau=0$.  We next discuss the approach  for the remaining situations. 
We define  
 \begin{equation}  c:=\left\{ \begin{aligned} 
&  \limsup_{|x|\to \infty} |x|^{\varpi_1} u(x) \ \ 
\mbox{in Case }[N_1] \ \mbox{with } \beta<\varpi_1 ,\\
& \limsup_{|x|\to \infty} \frac{u(x)}{\Psi_{\varpi_1(\rho)}(|x|)} \ \ 
\mbox{if Case }[N_2]\ \mbox{with } \beta<\varpi_1,\\
&  \limsup_{|x|\to \infty} u(x) \ \   
\mbox{in Case } [N_0]\ \mbox{with } \lambda\not=2\rho\ \mbox{if } \tau=0\\
& \qquad  \qquad \qquad \mbox{or if } \beta\in (\varpi_2,0)\ \mbox{in Case } [N_2]. 
\end{aligned} \right.
\label{confi}
\end{equation}

We show that $c\not =0$ for every positive solution of \eqref{e1} in $\mathbb R^N\setminus \{0\}$ (see Proposition~\ref{gild}). Then, we have a dichotomy: either $c\in (0,\infty)$ or $c=\infty$. 

$\lozenge$ Assume that $c\in (0,\infty)$. Then, by applying Theorem~\ref{urz1} (see also Remark~\ref{rek47}) and Theorem~\ref{rrzi}, in conjunction with the modified Kelvin transform (see the Appendix~\ref{Kel}), we reach the claims in \eqref{roz1} and \eqref{roz2},  
as well as   
$ \lim_{|x|\to \infty} u(x)=c$ if  $\beta\in (\varpi_2,0)$ in Case $[N_2]$. For the proof of \eqref{miau}, see Corollary~\ref{blabla}. 

$\lozenge$ Assume that $c=\infty$. Then, we prove (see 
Theorem~\ref{infc1}) that 
\begin{equation} \label{msio} u(x)\sim U_{\rho,\beta}(x)\quad \mbox{ as } |x|\to \infty . \end{equation}
In view of \eqref{msio} and since $u$ is different from $U_{\rho,\beta}$, we cannot have $\lim_{|x|\to 0} u(x)/U_{\rho,\beta}(x)=1$. Otherwise, 
letting $\varepsilon>0$ arbitrary, by the comparison principle and \eqref{msio}, we would get
  $$ \left(1-\varepsilon\right) U_{\rho,\beta}(x)\leq u(x)\leq 
 \left(1+\varepsilon\right) U_{\rho,\beta}(x)\quad\mbox{for every } |x|>0. 
 $$ By letting $\varepsilon\to 0$, we would arrive at $u\equiv U_{\rho,\beta}$ in $\mathbb R^N\setminus \{0\}$, which is a contradiction.

$\bullet$ Suppose that $\beta\in (\varpi_2,0)$ in case $[N_2]$. Then, 
by Theorem~\ref{thi} and Lemma~\ref{gold1} in Subsection~\ref{no-eo}, we have  a dichotomy near zero: 
\begin{equation} \label{novii}
\mbox{either (a) } \lim_{|x|\to 0} \frac{u(x)}{U_{\rho,\beta}(x)}=1\quad \mbox{or (b) }  
\lim_{|x|\to 0} \frac{u(x)}{ \Phi_{\varpi_2(\rho)}(|x|)}=b\ \mbox{for some } b\in \mathbb R_+. \end{equation}
Hence, we have the alternative (I) in Theorem~\ref{rod1c} when $c=\infty$ (as only \eqref{novii} (b) is possible) and the alternatives (II) and (III) if $c\in (0,\infty)$.

$\bullet$ In the remaining three situations of \eqref{confi},  only \eqref{novii} (a) holds (see Theorem~\ref{thi}). So, we reach the conclusions in (II)$'$  in Theorem~\ref{rod1c} since 
only $c\in (0,\infty)$ is possible.

\subsection{Classification of global solutions when $\lambda=2\rho$ and $\tau=0$}

We prove the following. 

\begin{proposition}
\label{ze-ta} 
Let \eqref{e2} hold, $\theta<-2$, $\lambda=2\rho$ and $\tau=0$. 
Let $u$ be any positive solution of \eqref{e1} in $\mathbb R^N\setminus \{0\}$ other than $U_{\rho,\beta}$. Then, 
 \eqref{novii} {\rm (a)} holds and there exists $c\in \mathbb R_+$ such that  
\begin{equation}
\label{musai1}
\lim_{|x|\to \infty} \frac{u(x)}{\log |x|}=c. 
\end{equation}	 
\end{proposition}

\begin{proof} By Theorem~\ref{thi}, we know that \eqref{novii} (a) holds. We need to establish  only \eqref{musai1} for some $c\in \mathbb R_+$. 
To this end, we apply the modified Kelvin transform in the Appendix~\ref{Kel}, that is, 
$v(\widetilde x)= u(x)$ with $\widetilde x=x/|x|^2$. Hence, 
$v$ is a positive solution of 
\begin{equation} \label{mayb} \mathbb L_{\widetilde \rho,\lambda,\tau} [v(\widetilde x)]= |\widetilde x|^{\widetilde \theta} v^q (\widetilde x)\quad \mbox{in } \mathbb R^N\setminus \{0\},
\end{equation}	
where $\widetilde \rho=-\rho$, $\widetilde \theta=-\theta-4>-2$ and $\tau=0$.  
We define 
$$ c:=\limsup_{|x|\to \infty} \frac{u(x)}{\log |x|}=\limsup_{|\widetilde x|\to 0} \frac{v(\widetilde x)}{\log\, (1/|\widetilde x|)}. $$ 

We distinguish three cases: $c\in (0,\infty)$, $c=\infty$ and $c=0$ (only the former is possible).

$\bullet$ Let $c\in (0,\infty)$. Lemma~\ref{serio} applied to $v$ leads to \eqref{musai1}. 

$\bullet$ Let $c=\infty$.   
Since $\lambda=-2\widetilde \rho $ and $\tau=0$,   
by applying Lemma~\ref{serio} to $v$, we obtain that 
$$ v(\widetilde x)\sim U_{-\rho,-\beta}(\widetilde x)=U_{\rho,\beta}(x)\quad \mbox{as } |\widetilde x|\to 0.  
$$
This means that $u(x)\sim U_{\rho,\beta}(x)$ as $|x|\to \infty$. 
Hence, $u\equiv U_{\rho,\beta}$ in $\mathbb R^N\setminus \{0\}$, which is not possible.

$\bullet$ Let $c=0$. By Lemma~\ref{serio}, we conclude that 
there exists 
$$\lim_{|\widetilde x|\to 0} v(\widetilde x)=\lim_{|x|\to \infty}  u(x)=\mu\in [0,\infty). $$ 

If $\mu=0$, then using that 
$\lim_{|x|\to 0} u(x)=0$, we get a contradiction as in the proof of Theorem~\ref{haha} (with $\beta=0$ there).

If $\mu\in (0,\infty)$, then again we get a contradiction. More precisely, by the comparison principle (see Lemma~\ref{co2}), we obtain that $u$ is the only positive solution of \eqref{e1} in $\mathbb R^N\setminus \{0\}$ satisfying $u(x)\sim U_{\rho,\beta}(|x|)$ as $|x|\to 0$ and $\lim_{|x|\to \infty} u(x)=\mu$. Hence, $u$ is radially symmetric and, moreover, 
$u'(r)>0$ for every $r=|x|>0$. Consequently, $u$ 
is a positive solution of 
$$ u''(r)+\frac{u'(r)}{r}=r^\theta u^q(r)\quad \mbox{for every } r>0. 
$$
This implies that $(ru'(r))'>0$ for every $r>0$. As $\mu\in (0,\infty)$, we infer that there exists $\lim_{r\to \infty} ru'(r)=0$ and $ru'(r)<0$ for every $r>0$. This contradicts $u'(r)>0$ for every $r>0$.  

Hence, neither $c=0$ nor $c=\infty$ can occur. This finishes the proof of Proposition~\ref{ze-ta}.  
\end{proof}

\subsection{Non-existence results} \label{no-eo}
In Lemma~\ref{gold1}, we show that if $\beta\in (\varpi_2,0)$ in Case $[N_2]$, then \eqref{e1} in $\mathbb R^N\setminus \{0\}$ has no positive solutions satisfying 
the alternative {\rm $(P_0)$ (i)} of Theorem~\ref{thi}.
 Furthermore, assuming that $\theta<-2$, $f_\rho(\beta)>0$ and $\lambda\not=2\rho$ if $\tau=0$, we prove in Proposition~\ref{gild} that \eqref{e1} in 
 $\mathbb R^N\setminus \{0\}$ has no positive solutions with $c=0$, where $c$ is given by \eqref{confi}.  

\begin{lemma}[Non-existence] \label{gold1} 
 Let \eqref{e2} hold.  
If $\beta\in (\varpi_2,0)$ in Case $[N_2]$, then there are no positive solutions of \eqref{e1} in $\mathbb R^N\setminus \{0\}$ satisfying 
 \begin{equation} \label{gol2} 
 \lim_{|x|\to 0} \frac{u(x)}{\Phi_{\varpi_2(\rho)}(|x|)}=0. 
 \end{equation}
\end{lemma}

\begin{proof}
Assume by contradiction that $u$ is a positive solution of \eqref{e1} in $\mathbb R^N\setminus \{0\}$ 
satisfying \eqref{gol2}. 
Since $\beta\in (\varpi_2,0)$, by Theorem~\ref{thi}, we get that 
\begin{equation} \label{gol3} \lim_{|x|\to \infty} |x|^{\varpi_2} u(x)=0.
\end{equation}
We distinguish two cases: 

(I) Let $\rho>\Upsilon$. Then,  $\Phi_{\varpi_2(\rho)}(r)=r^{-\varpi_2}$ for every $r>0$.  Let $\varepsilon>0$ be arbitrary. Hence, from \eqref{gol2} and \eqref{gol3}, 
there exist $r_\varepsilon>0$ small and $R_\varepsilon>0$ large such that 
$ u(x)\leq \varepsilon |x|^{-\varpi_2}$ for every $ |x|\in (0,r_\varepsilon]\cup [R_\varepsilon,\infty)$.  
By the comparison principle, we obtain $ u(x)\leq \varepsilon |x|^{-\varpi_2}$ for all $|x|>0$. Letting $\varepsilon\to 0$, we get $u\equiv 0$ in $\mathbb R^N\setminus \{0\}$, which is a contradiction.

(II) Let $\rho=\Upsilon$. Then, we have $\Phi_{\varpi_2(\rho)}(r)= r^{-\varpi_2} \left[ \log \,(1/r )\right]^{\frac{2}{1+\tau}}$ for $r\in (0,1)$.

Let $\varepsilon>0$ be arbitrary. 
From \eqref{gol3}, there exists $R_\varepsilon>1$ large such that 
\begin{equation} \label{gol6} u(x) \leq \varepsilon |x|^{-\varpi_2} \quad \mbox{for every } |x|\geq R_\varepsilon. 
\end{equation}
Let $r_*\in (0,1)$ be given by Lemma~\ref{fiqa}. 
Without loss of generality, we assume that 
$$ R_\varepsilon>\max\,\{ 1/\varepsilon,1/r_*\}. $$
For every $0<r\leq R_\varepsilon$, we define $\Phi_{\varepsilon}(r)$ as follows
\begin{equation} \label{pji} \Phi_\varepsilon(r)=\frac{\varepsilon}{[\log R_\varepsilon]^{\frac{2}{1+\tau}}} \, r^{-\varpi_2} \left[\log \left(  R^2_\varepsilon/r \right)\right]^\frac{2}{1+\tau}.\end{equation}
In particular, $\Phi_\varepsilon(R_\varepsilon)=\varepsilon R_\varepsilon^{-\varpi_2}$ so that using \eqref{gol6}, we get 
\begin{equation}
\label{yave} 
u(x)\leq \Phi_\varepsilon(|x|)\quad \mbox{for every } |x|=R_\varepsilon.	
\end{equation}
By virtue of Lemma~\ref{fiqa},  we infer that 
$\Phi_\varepsilon'(r)>0$ for every $r\in (0,R_\varepsilon)$ and 
\begin{equation} \label{rulo} \mathbb L_{\rho,\lambda,\tau} [\Phi_\varepsilon(|x|)]\leq 0<|x|^\theta\, [\Phi_\varepsilon(|x|)]^q\quad \mbox{for every } |x|\in (0,R_\varepsilon). 
\end{equation}	
From \eqref{gol2} and the definition of $\Phi_\varepsilon$ in \eqref{pji}, we obtain that  
\begin{equation} \label{rulo2} \lim_{|x|\to 0} \frac{u(x)}{\Phi_\varepsilon(|x|)}=0. \end{equation} 
In light of \eqref{yave}--\eqref{rulo2}, we can use the comparison principle to conclude that
\begin{equation} \label{gol7} u(x)\leq  \Phi_\varepsilon(|x|)\quad \mbox{for every } 0<|x|\leq R_\varepsilon. 
\end{equation}

Fix $x\in \mathbb R^N\setminus \{0\}$ arbitrary. 
By letting $\varepsilon>0$ small enough, we have  
$0<|x|<1/\varepsilon<R_\varepsilon$. Then, letting $\varepsilon\to 0$ in \eqref{gol7}, we obtain that $u(x)=0$, which is a contradiction. 
\end{proof}

\begin{proposition} \label{gild}  
 Let \eqref{e2} hold, $\theta<-2$ and $f_\rho(\beta)>0$. 
 Assume that $\lambda\not =2\rho$ when $\tau=0$. Let $u$ be any positive solution of \eqref{e1} in $\mathbb R^N\setminus \{0\}$. Then, defining $c$ as in \eqref{confi}, we have 
$c\not=0$.  
\end{proposition}

\begin{proof}  
Assume by contradiction that 
there exists a positive solution $u$ of \eqref{e1} in $\mathbb R^N\setminus \{0\}$ for which $c=0$, where $c$ is given by \eqref{confi}.  
If Case $[N_0]$ holds or if $\beta\in (\varpi_2,0)$ in Case $[N_2]$, then we reach a contradiction as in the proof of Proposition~\ref{haha} (with $\beta=0$ there). 

It remains to analyse the case when  
$\beta<\varpi_1$ in Case $[N_j]$ with $j=1,2$. We show that we arrive at a contradiction by distinguishing two cases:   

\vspace{0.2cm}
(I) Assume either Case $[N_1]$ or $\rho>\Upsilon$ in Case $[N_2]$.  
In either of these situations, our assumption that $c=0$ 
(see also  the definition of $\Psi_{\varpi_1(\rho)}$ in \eqref{psia}) implies that $u$ satisfies 
\begin{equation} 
\label{prio}
\lim_{|x|\to \infty} |x|^{\varpi_1} u(x)=0.	
\end{equation}
By Theorem~\ref{thi} and $\beta<\varpi_1$, we get
\begin{equation} \label{trin1} \lim_{|x|\to 0}|x|^{\varpi_1} u(x)
=0. \end{equation} 
We reach a contradiction as in the proof of Lemma~\ref{gold1} (Case (I)). 
Let $\varepsilon>0$ be arbitrary. Thus, using \eqref{prio}, 
we can find $r_\varepsilon>0$ small and  
$ R_\varepsilon>1 $ large such that 
\begin{equation} \label{soll1} u(x)\leq \varepsilon  |x|^{-\varpi_1}\quad \mbox{for every } |x|\in (0,r_\varepsilon]\cup  [R_\varepsilon,\infty).
\end{equation}
Recalling that $\mathbb L_{\rho,\lambda,\tau}[|x|^{-\varpi_1}]=0$ for every $|x|>0$, by the comparison principle, we conclude that
\eqref{soll1} is true for every $|x|>0$. By letting $\varepsilon\to 0$, we get $u\equiv 0$, which is a contradiction. 

\vspace{0.2cm}
(II) We now let $\rho=\Upsilon$ in Case $[N_2]$. Then,  $\Psi_{\varpi_1(\rho)}(r)=r^{-\varpi_1} \left(\log r \right)^{2/(1+\tau)}$ for every $r>1$. 

We adapt the proof of Lemma~\ref{gold1} pertaining to $\rho=\Upsilon$. However, we need to pay attention to the fact that here 
$\Psi_{\varpi_1(\rho)}$ is {\em not} a super-solution of 
\begin{equation} \label{fyre} \mathbb L_{\rho,\lambda,\tau}[\Psi]=0\quad  \mbox{for } |x|>1
\ \mbox{ large.}
\end{equation} 
On the contrary, $\Psi_{\varpi_1(\rho)}$ is a {\em sub-solution} of \eqref{e1} for $|x|>1$ large (see Lemma~\ref{siww}). 
Hence, instead of $\Psi_{\varpi_1(\rho)}$, 
we will work with a positive radial {\em super-solution} $\widetilde \Psi(|x|)$ of \eqref{fyre} such that  $\widetilde \Psi (|x|)\sim \Psi_{\varpi_1(\rho)}(|x|)$ as $|x|\to \infty$. For every $r>1$, we define $\widetilde \Psi(r)$ as in \eqref{stay}, that is, 
$$ \widetilde \Psi(r) =\widetilde \Psi_{\varpi_1}(r)=r^{-\varpi_1} \left(\log r\right)^{\frac{2}{1+\tau}}	
+r^{-\varpi_1} \log r .$$
Then, by Lemma~\ref{feed}, $\widetilde \Psi$ is a positive super-solution of \eqref{fyre} for $|x|>R_1$ with $R_1>1$ large. 

Let $\varepsilon>0$ be arbitrary. 
As in Case (I), we have \eqref{trin1}
so that for $r_\varepsilon>0$ small, it holds 
\begin{equation} \label{minb} u(x)\leq \varepsilon\,|x|^{-\varpi_1}\quad \mbox{for every } 0<|x|\leq r_\varepsilon.  
\end{equation} 
If necessary, we diminish $r_\varepsilon$ so that
$ 0< r_\varepsilon <\min\, \{\varepsilon,1/R_1\}$.  
For every $r\geq r_\varepsilon$, we define $\Xi_{\varepsilon}(r)$ by 
$$  \Xi_{\varepsilon}(r)= 
\frac{\varepsilon}{[\log \,(1/r_\varepsilon)]^{\frac{2}{1+\tau}}} \left\{ r^{-\varpi_1} \left[ \log \left(\frac{r}{r_\varepsilon^2}\right)\right]^\frac{2}{1+\tau} +r^{-\varpi_1}\log \left(\frac{r}{r_\varepsilon^2}\right) \right\} .
$$
Hence, from \eqref{minb}, we have 
\begin{equation} \label{myt0}  u(x)\leq \Xi_{\varepsilon}(|x|)\quad \mbox{for every } |x|=r_\varepsilon.
\end{equation} 
Moreover, since $r/r_\varepsilon^2>1/r_\varepsilon>R_1$ for every $r>r_\varepsilon$, by Lemma~\ref{feed}, it follows that 
\begin{equation} \label{myt2}  \mathbb L_{\rho,\lambda,\tau} \,[\Xi_{\varepsilon}(|x|)]
\leq 0 <|x|^\theta [\Xi_{\varepsilon}(|x|)]^q \quad \mbox{for every } |x|>r_\varepsilon .\end{equation} 
In view of our assumption that $c=0$, we have 
\begin{equation} \label{myt1} \lim_{|x|\to \infty} \frac{u(x)}{ \Xi_{\varepsilon}(|x|)}=0. 
\end{equation}
From \eqref{myt0}--\eqref{myt1} and the comparison principle, we conclude that 
\begin{equation} \label{myt3} u(x)\leq   \Xi_{\varepsilon}(|x|)\quad \mbox{for every } |x|\geq r_\varepsilon.
\end{equation}
Fix $x\in \mathbb R^N\setminus \{0\}$ arbitrary. 
Remark that $\lim_{\varepsilon\to 0} \Xi_{\varepsilon}(|x|)=0$ using that  $0<r_\varepsilon<\varepsilon$.   Letting $\varepsilon\to 0$ in \eqref{myt3}, we get $u(x)=0$, which is a contradiction. 
This ends the proof of Proposition~\ref{gild}.  
\end{proof}


\subsection{The case $c=\infty$}

Here, our main result is the following.

\begin{theorem} \label{infc1}
 Let \eqref{e2} hold, $\theta<-2$ and $f_\rho(\beta)>0$. 
 Assume that $\lambda\not =2\rho$ when $\tau=0$. Let $u$ be any 
 positive solution of \eqref{e1} in a domain of $\mathbb R^N$ containing $\mathbb R^N\setminus \overline{B_R(0)}$ for some $R>0$.
  Then, if $c$ in \eqref{confi} satisfies $c=\infty$, then we have 
  $u(x)\sim U_{\rho,\beta}(|x|)$ as $|x|\to \infty$. 
\end{theorem}

The claim of Theorem~\ref{infc1} follows by combining Propositions~\ref{ze-ta2} and \ref{ze-ta3}. 

\begin{proposition}
\label{ze-ta2} In the framework of Theorem~\ref{infc1}, we assume that $\beta<\varpi_1$ in Case $[N_j]$ with $j=1,2$. If $c$ in \eqref{confi} satisfies $c=\infty$, then $u(x)\sim U_{\rho,\beta}(|x|)$ as $|x|\to \infty$.
\end{proposition}

\begin{proof} 
By Corollary~\ref{goos}, to conclude that $u(x)\sim U_{\rho,\beta}(|x|)$ as $|x|\to \infty$, it suffices to show that
\begin{equation} \label{lups}
	\liminf_{|x|\to \infty} |x|^\beta u(x)>0. 
\end{equation}

The idea for proving \eqref{lups} is to construct a family 
$\{z_\delta\}_{\delta\in (0,1) }$ of positive radial 
sub-solutions of \eqref{e1} for $|x|>1$ with the following properties (for every $\delta\in (0,1)$ except for (4)): 
\begin{equation} \label{gumaa}
\begin{aligned} 
& \mbox{(1) } z_\delta(R)\leq \min_{|x|=R} u(x); &&
 \mbox{(2) } \lim_{|x|\to \infty} \frac{u(x)}{z_\delta (|x|)}=\infty;&\\
& \mbox{(3) } z_\delta'(r)\not=0\ \mbox{for every }r>1; &&
\mbox{(4) }\lim_{r\to \infty} r^{\beta} \left(\lim_{\delta\to 0^+} z_\delta (r)\right)=c_0\in (0,\infty). 
\end{aligned} \end{equation}

By the first three properties of $\{z_\delta\}_{\delta>0}$,  we can apply the comparison principle to obtain 
$$ z_\delta (|x|)\leq u(x)\quad \mbox{for all } R\leq |x|<\infty \ \mbox{and every } \delta\in (0,1).
$$ Then, letting $\delta\to 0^+$ and using the property (4) of $z_\delta$, we arrive at \eqref{lups}.

\vspace{0.2cm}
In Case $[N_2]$, we need to treat the case $\rho=\Upsilon$ separately due to the different definition of $\Psi_{\varpi_1(\rho)}$ in \eqref{psia}.
Hence, we divide the proof of \eqref{lups} into two Steps. 

\vspace{0.2cm}
{\bf Step 1.}  {\em Proof of \eqref{lups}, assuming either Case $[N_1]$ or $\rho\not=\Upsilon$ in Case $[N_2]$}.  

\vspace{0,2cm}
By our assumptions (see \eqref{confi}), the modified Kelvin transform and
Corollary~\ref{pere}, we get 
\begin{equation} \label{trand} \lim_{|x|\to \infty} |x|^{\varpi_1} u(x)=\infty.  
\end{equation}

{\bf Construction of the family of sub-solutions $\{z_\delta\}_{\delta>0}$.}

Since we assumed $\rho\not= \Upsilon$ when Case $[N_2]$ occurs, we see that 
\begin{equation} \label{huer} \beta<\varpi_1<-(\lambda \tau)^{1/(\tau+1)}<\varpi_2<0.\end{equation}

In Case $[N_1]$, we have $\lambda^+ \tau=0$ as either $\tau=0$ or $\lambda<0$ in Case $[N_1]$. 

We fix $\alpha>0$ small such that 
\begin{equation} \label{miuc} 0<\alpha< \min\,\{  |\varpi_1|
 \left(1-\lambda^+ \tau |\varpi_1|^{-\tau-1} \right),
 (\varpi_1-\beta)(q-1)
 \} .\end{equation}

We define two positive constants $b_\alpha$ and $c_\alpha$ as follows 
\begin{equation} \label{calf1} \begin{aligned}
& b_\alpha:=\sup_{0<t<1} e^t \left( \frac{1-e^{-t}}{t}\right)^{1-\frac{(\varpi_1-\beta)(q-1)}{\alpha}} ,\\
 & c_\alpha=\left( \frac{(\varpi_1-\beta) (|\varpi_1|-\lambda^+\tau |\varpi_1|^{-\tau} -\alpha)}{b_\alpha} \right)^{\frac{1}{q-1}}. 
\end{aligned}
\end{equation} 
Fix $c_0\in (0,c_\alpha)$ arbitrary. 
For every $\delta,r>0$, we define 
\begin{equation} \label{duzi}  
Q_\delta(r)=\frac{ r^{-\alpha}\, e^{-r^{-\alpha}}
}{e^\delta-e^{-r^{-\alpha}}}
\quad \mbox{and}\quad z_\delta(r)=c_0\, r^{-\varpi_1} \left(e^\delta-e^{-r^{-\alpha}}\right)^{- \frac{\varpi_1-\beta}{\alpha}}
. \end{equation}

The properties (2)--(4) in \eqref{gumaa} are easily checked for $z_\delta$ in \eqref{duzi}. Let $R>1$ be large such that $\mathbb R^N\setminus \overline{B_R(0)}\subset \subset  \Omega_1$. By choosing $c_0\in (0,c_\alpha)$ small (depending on $u$ and $R$) such that 
$$  c_0 R^{-\varpi_1} \left(1-e^{-R^{-\alpha}}\right)^{-\frac{\varpi_1-\beta}{\alpha}}\leq \min_{|x|=R} u(x),
$$ we ensure that the first property in \eqref{gumaa} is satisfied by $z_\delta$ for every $\delta>0$. 

We next show that for arbitrary $c_0\in (0,c_\alpha)$ and every $\delta>0$, the function $z_\delta$ in \eqref{duzi} is a positive radial sub-solution of 
\eqref{e1} in $\mathbb R^N\setminus \overline{B_1(0)}$, that is, 
\begin{equation} \label{surf1} \mathbb L_{\rho,\lambda,\tau} [z_\delta(|x|)] \geq |x|^\theta 
[z_\delta (|x|)]^q \quad \mbox{for every } |x|>1. 
\end{equation} 

{\bf Proof of \eqref{surf1}.}
Let $r>1$ be arbitrary. 
By the choice of $\alpha$ in \eqref{miuc}, \eqref{duzi} and the definition of $b_\alpha$ in \eqref{calf1},  we see that  
\begin{equation} \label{geza} b_\alpha\geq e^{r^{-\alpha}} \left(\frac{1-e^{-r^{-\alpha}}}{r^{-\alpha}}\right)^{1-\frac{(\varpi_1-\beta)(q-1)}{\alpha}} \geq \frac{z^{q-1}_\delta(r)}{c_0^{q-1}} \frac{r^{\theta+2}}{Q_\delta (r)}.
\end{equation}
By a simple calculation, we find that 
$$ \begin{aligned}
 & z_\delta'(r)=\frac{z_\delta(r)}{r} \left[-\varpi_1+(\varpi_1-\beta)\,Q_\delta (r)\right]>0,\\
&  Q_\delta'(r)=\alpha \frac{Q_\delta(r)}{r} \left[ r^{-\alpha}-1+Q_\delta(r)\right],\\
& \lambda\, G_\tau[z_\delta(r)]=\lambda |\varpi_1|^{1-\tau} \frac{z_\delta (r)}{r^2} \left(1+\frac{(\varpi_1-\beta)}{|\varpi_1|}\,Q_\delta (r)\right)^{1-\tau}. 
\end{aligned} $$
It follows that 
\begin{equation} \label{ingf}
\begin{aligned}  
\mathbb L_\rho[z_\delta(r)]= \frac{z_\delta(r)}{r^2} & \left
\{ 
\varpi_1^2+2\rho \varpi_1+\left(\varpi_1-\beta\right)\left [-2\left(\varpi_1+\rho\right)-\alpha+\alpha\, r^{-\alpha} \right] Q_\delta(r) \right.\\
& \left. +\left(\varpi_1-\beta\right) \left(\varpi_1-\beta+\alpha\right) Q_\delta^2(r)
\right\}. 
\end{aligned} 
\end{equation}

We recall that $f_\rho(\varpi_1)=0$, that is, $\varpi_1^2+2\rho \varpi_1+\lambda |\varpi_1|^{1-\tau}=0$. 

$\bullet$ If Case $[N_1]$ occurs, 
then either $\lambda\leq 0$ or $\tau=0$, which implies that 
$$   \lambda\, G_\tau[z_\delta(r)]\geq \lambda |\varpi_1|^{1-\tau} \frac{z_\delta (r)}{r^2} \left(1+\frac{(\varpi_1-\beta)}{|\varpi_1|}\,Q_\delta (r)\right).
$$
Therefore, using \eqref{ingf} and $\beta<\varpi_1<0$, we arrive at 
$$ \mathbb L_{\rho,\lambda,\tau} [z_\delta(r)] \geq 
\left( \varpi_1-\beta\right) \left( -\varpi_1-\alpha\right)  
\frac{z_\delta(r)}{r^2} Q_\delta(r) . 
$$

$\bullet$ In Case $[N_2]$, using that $\beta<\varpi_1$ and $\lambda>0$, we obtain that 
$$ 	
 \lambda\, G_\tau[z_\delta(r)] \geq
\lambda |\varpi_1|^{1-\tau} \frac{z_\delta (r)}{r^2} \left[
1+\frac{\left(1-\tau\right)\left(\varpi_1-\beta\right)}{|\varpi_1|} Q_\delta(r)
-\frac{\tau\left(1-\tau\right)\left(\varpi_1-\beta\right)^2}{2\,\varpi_1^2} \,Q_\delta^2(r)
\right].
$$
This, jointly with \eqref{ingf}, implies that 
$$  \mathbb L_{\rho,\lambda,\tau} [z_\delta(r)] \geq 
\left(\varpi_1-\beta\right) \left(-\varpi_1-\lambda \tau |\varpi_1|^{-\tau} -\alpha \right) \frac{z_\delta(r)}{r^2} Q_\delta(r) .
$$ 

In both Case $[N_1]$ and Case $[N_2]$, using \eqref{miuc} and \eqref{calf1}, we find that 
$$  \mathbb L_{\rho,\lambda,\tau} [z_\delta(r)] \geq c_\alpha^{q-1} b_\alpha \frac{z_\delta(r)}{r^2} Q_\delta(r),
$$
which, together with $c_0\in (0,c_\alpha)$ and \eqref{geza}, proves 
the claim of \eqref{surf1}. \qed 

\vspace{0.2cm}
This completes the proof of Step~1.  \qed

\vspace{0.2cm}
{\bf Step~2:} {\em Proof of \eqref{lups}, assuming that $\rho=\Upsilon$ in Case $[N_2]$}. 

\vspace{0.2cm}
In this case, by our assumption that $c=\infty$ and Corollary~\ref{pere}, we obtain that 
 \begin{equation} \label{guma}
 \lim_{|x|\to \infty} |x|^{\varpi_1} \left( \log |x|\right)^{-\frac{2}{1+\tau}}
u(x)=\infty.	
 \end{equation}

To reach \eqref{lups}, we need to update the construction of the family $\{z_\delta\}_{\delta>0}$ of positive radial sub-solutions of 
\eqref{e1} for $|x|>1$ to satisfy the properties 
(1)--(4). The construction is done in two stages (as explained below) because it seems difficult to obtain property (4) directly without first improving  \eqref{guma} as follows 
\begin{equation} \label{furn}
\lim_{|x|\to \infty} |x|^{\varpi_1} (\log |x|)^{-\eta} u(x)=\infty\quad \mbox{for every } \eta>2/(1+\tau). 	
\end{equation}

We fix  $s\geq 2/(1+\tau)$ and let $c_0,\alpha>0$ be as small as needed (depending on our choice of $s$). For every $\delta\in (0,1)$, we define 
$z_{\delta,s}$ of the form
\begin{equation} \label{gara} z_{\delta,s}(r)=\frac{c_0}{\delta^s} r^{-\varpi_1} \log^s (1+\delta \,r^\alpha)\quad \mbox{for every } r>1.  
\end{equation}

We next find the right conditions that ensure 
\eqref{surf1} for   
 $z_\delta=z_{\delta,s}$ and every $\delta\in (0,1)$. 
 
Since $\beta=(\theta+2)/(q-1)$, from \eqref{gara}, we see that
\begin{equation} \label{riu1} r^{\theta+2} z_{\delta,s}^{q-1}(s)=c_0^{q-1} r^{(\alpha s-\varpi_1+\beta)(q-1)} \left(
\frac{\log \,(1+\delta\,r^\alpha)}{\delta\,r^\alpha}\right)^{s\left(q-1\right)}. 
\end{equation}  
 
For every $t>0$, we define $J(t)$ as in \eqref{jef}. Then, for arbitrary $r>1$, we obtain that 
$$  \begin{aligned} & z_{\delta,s}'(r)=\frac{z_{\delta,s}(r)}{r} \left[-\varpi_1+\alpha s J(\delta\,r^\alpha)\right]>0,\\
  & \mathbb L_\rho[z_{\delta,s}(r)]=	
  \frac{z_{\delta,s}(r)}{r^2} \left[
  \varpi_1^2+2\rho \varpi_1-2\alpha s \left(\varpi_1+\rho\right) J(\delta\,r^\alpha)
  +\alpha^2 (s^2-s) J^2(\delta\,r^\alpha) +
  \alpha^2 s \frac{J(\delta\,r^\alpha)}{1+\delta r^\alpha}\right].
  \end{aligned}
$$ 
Moreover, since $\lambda>0$, we find that 
$$ \begin{aligned} \lambda \,G_\tau[z_{\delta,s}(r)]& = \lambda |\varpi_1|^{1-\tau} \frac{z_{\delta,s}(r)}{r^2}
\left(1+\frac{\alpha s}{|\varpi_1|} J(\delta\,r^\alpha) \right)^{1-\tau}\\
& \geq \lambda |\varpi_1|^{1-\tau} \frac{z_{\delta,s}(r)}{r^2} \left[1+(1-\tau) \frac{\alpha s}{|\varpi_1|} J(\delta\,r^\alpha) -\frac{\tau\left(1-\tau\right)\alpha^2 s^2 }{2\, \varpi_1^2} J^2(\delta\,r^\alpha)  \right].
\end{aligned}
$$
Using that $\varpi_1^2+2 \rho\varpi_1 +\lambda |\varpi_1|^{1-\tau}=0$ and $\varpi_1=\varpi_2=-(\lambda \tau)^{1/(\tau+1)}$, we arrive at 
\begin{equation} \label{furi} \mathbb L_{\rho,\lambda,\tau} [z_{\delta,s}(r)]\geq  \alpha^2 s\, \frac{z_{\delta,s}(r)}{r^2} 
 \left[ \left( \frac{s\left(1+\tau\right)}{2}-1\right) J^2(\delta\, r^\alpha)+
\frac{J(\delta\,r^\alpha)}{1+\delta\,r^\alpha}
\right].
\end{equation}

{\bf Construction of the family $\{z_{\delta,s}\}_{\delta\in (0,1)}$ of positive  sub-solutions of \eqref{e1} for $|x|>1$.}

We next prove \eqref{lups} in two stages by a suitable choice of $s$, $\alpha$ and $c_0$ in  \eqref{gara}: 

$\bullet$ First, by taking $s=2/(1+\tau)$, $\alpha$ satisfying \eqref{alfs} and $0<c_0 \leq R^{\varpi_1-\alpha s} \min_{|x|=R} u(x)$ such that \eqref{cosa} holds, we establish \eqref{furn} based on the comparison principle. 

$\bullet$ Second, in \eqref{gara}, we fix $s>2/(1+\tau)$,  $\alpha=(\varpi_1-\beta)/s$ and $0<c_0\leq R^{\beta} \min_{|x|=R} u(x)$ such that \eqref{cooo} holds. Then, for every $\delta>0$, we show that $z_{\delta,s}$ in \eqref{gara} is a positive radial sub-solution of \eqref{e1} for $|x|>1$ and satisfies all the desired properties (1)--(4) in \eqref{gumaa}.

\vspace{0.2cm}
{\bf Proof of \eqref{furn}.} Our starting point is \eqref{guma} so that to guarantee the second property in \eqref{gumaa}, we cannot take $s>2/(1+\tau)$ at this stage (only later after establishing \eqref{furn}). Hence, in the definition of $z_{\delta,s}$ in \eqref{gara}, we must start with $ s=2/(1+\tau)$, which means that the coefficient of $J^2(\delta r^\alpha)$ in the right-hand side of \eqref{furi} becomes zero. 
Consequently, to obtain \eqref{surf1} for $z_{\delta}=z_{\delta,s}$, we shall see that a restriction arises on the choice of $\alpha$ (the second inequality in \eqref{alfs}), which prevents us from taking $\alpha=(\varpi_1-\beta)/s$. Thus, by the comparison principle, we don't get \eqref{lups} in the first place, only \eqref{furn}. We give the details below.

From \eqref{furi} with $s=2/(1+\tau)$, we get    
  $$  \mathbb L_{\rho,\lambda,\tau} [z_{\delta,s}(r)]\geq  \alpha^2 s \, \frac{z_{\delta,s}(r)}{r^2}\frac{J(\delta\,r^\alpha)}{1+\delta\,r^\alpha}=\frac{z_{\delta,s}(r)}{r^2}
  \frac{\alpha^2 s}{(1+\delta\,r^\alpha)^2} \,
  \frac{\delta \,r^\alpha}{\log \,(1+\delta \,r^\alpha)}.  
  $$
Thus, by taking $s=2/(1+\tau)$ in \eqref{gara}, the function $z_{\delta,s}$ becomes a positive radial sub-solution of 
\eqref{e1} in $\mathbb R^N\setminus \overline{B_1(0)}$ whenever we have
\begin{equation} \label{riuu}   \frac{\alpha^2 s}{(1+\delta\,r^\alpha)^2} \,
  \frac{\delta \,r^\alpha}{\log \,(1+\delta \,r^\alpha)}\geq r^{\theta+2} z_{\delta,s}^{q-1}(r)\quad \mbox{for all } r>1.
\end{equation}
For every $t>0$, we define 
$$ {\mathfrak M}_s(t)=t^{(q-1)\left(s-\frac{\varpi_1-\beta}{\alpha}\right)} (1+t)^2 \left( 
\frac{\log\, (1+t)}{t}\right)^{s\left(q-1\right)+1}.
$$
In view of \eqref{riu1}, the inequality in \eqref{riuu} is equivalent to 
\begin{equation} \label{riu2}  \delta^{(q-1) (\frac{\varpi_1-\beta}{\alpha}-s)}
c_0^{q-1} {\mathfrak M}_s(\delta\,r^\alpha)\leq \alpha^2 s\quad \mbox{for every } r>1\ \mbox{and }\delta\in (0,1). 
\end{equation}
We therefore need to choose $\alpha>0$ small to ensure that the power of $\delta$ (in the left-hand of \eqref{riu2}) is non-negative and $\sup_{t>0} {\mathfrak M}_s(t)<\infty$, namely,
\begin{equation} \label{alfs} 0<\alpha\leq \frac{\varpi_1-\beta}{s}  \quad \mbox{and} \quad \alpha< \left(q-1\right)
\left(\varpi_1-\beta\right). 
\end{equation}
Finally, in \eqref{gara}, we take $c_0>0$ satisfying 
\begin{equation} \label{cosa} 0<c_0\leq c_{\alpha,s},\quad \mbox{where }
c_{\alpha,s}:=\left(
\frac{\alpha^2 s}{\sup_{t>0} {\mathfrak M}_s(t)}
\right)^\frac{1}{q-1}. 
\end{equation}
With such choices of $\alpha$ and $c_0$, we regain 
\eqref{surf1} for $z_\delta=z_{\delta,s}$ and arbitrary $\delta\in (0,1)$. 

Since $\varpi_1<0$, we see that $z'_{\delta,s}(r)>0$ for every $r>1$. By diminishing $c_0=c_0(R,u)>0$ so that 
$ c_0 R^{\alpha s-\varpi_1} \leq \min_{|x|=R} u(x) 
$, we also obtain that 
$$ z_{\delta,s}\leq u\quad \mbox{on } \partial B_R(0)\quad \mbox{ for all } \delta\in (0,1).$$ Hence, $z_{\delta,s}$ satisfies the properties (1)--(3) in \eqref{gumaa}. 

Using \eqref{guma} and the comparison principle, we infer that 
$$ u(x)\geq z_{\delta,s}(|x|) \quad \mbox{for all } |x|\geq 1\quad \mbox{and }\delta\in (0,1). 
$$
For fixed $|x|\geq 1$, by passing to the limit $\delta\to 0^+$, we arrive at 
$ u(x)\geq   c_0 |x|^{\alpha s-\varpi_1} $. Thus, we have
\begin{equation} \label{riri} \liminf_{|x|\to \infty} |x|^{\varpi_1-\alpha s} u(x)\geq c_0. 
\end{equation}
However, unless it happens that $1+\tau<2\left(q-1\right)$, we cannot take $\alpha s=\varpi_1-\beta$ because of the second inequality in \eqref{alfs}. In any case, from \eqref{riri}, we conclude \eqref{furn}.  \qed

\vspace{0.2cm}
{\bf Proof of \eqref{lups} completed.} 
In the definition of $z_{\delta,s}$ in \eqref{gara}, we fix $s>2/(1+\tau)$ and take $\alpha=(\varpi_1-\beta)/s$. 
For every $t>0$, we define 
$$  {\mathfrak M}_s(t)=(1+t)^2 \left( \frac{\log\, (1+t)}{t}\right)^{s\left(q-1\right)+2}.
$$
Clearly, we have $\sup_{t>0} {\mathfrak M}_s(t)\in (0,\infty)$. 
We choose $c_0>0$ such that   
\begin{equation} \label{cooo}  c_0\in (0,c_{\alpha,s}), \quad \mbox{where }c_{\alpha,s}:=
\left( \frac{
 \alpha^2 s\left(\frac{s\left(1+\tau\right)}{2}-1\right)} {\sup_{t>0} {\mathfrak M}_s (t)}
 \right)^{\frac{1}{q-1}}.
\end{equation} 
 
We can fix  {\em any} $s>2/(1+\tau)$ and we still guarantee the second property in \eqref{gumaa} with $z_\delta=z_{\delta,s}$ because we have proved \eqref{furn}. The effect of this choice of $s$ is that the coefficient of $J^2(\delta \,r^\alpha)$ in the right-hand side of \eqref{furi} is now {\em positive} and, hence,
\begin{equation} \label{furi2} \mathbb L_{\rho,\lambda,\tau} [z_{\delta,s}(r)]\geq  \alpha^2 s
 \left( \frac{s\left(1+\tau\right)}{2}-1\right)  \frac{z_{\delta,s}(r)}{r^2}\,J^2(\delta\, r^\alpha).
\end{equation}
Since $\alpha s=\varpi_1-\beta$, from \eqref{riu1} and our choice of $c_0$ in \eqref{cooo}, we find that 
$$   \alpha^2 s
 \left( \frac{s\left(1+\tau\right)}{2}-1\right)  J^2(\delta\, r^\alpha)\geq r^{\theta+2} z^{q-1}_{\delta,s}(r)\quad \mbox{for all } r>1\ \mbox{and } \delta>0.
$$
We thus obtain that $z_{\delta,s}$ is a positive radial sub-solution of 
\eqref{e1} for $|x|>1$ and every $\delta>0$.
The properties (1)--(4) in \eqref{gumaa} are easily verified by $z_\delta=z_{\delta,s}$. This ends the proof of Step~2. \qed

\vspace{0.2cm}
From Step~1 and Step~2, we conclude the proof of Proposition~\ref{ze-ta2}. 	
\end{proof}

\begin{proposition} \label{ze-ta3}
In the framework of Theorem~\ref{infc1}, we assume that either Case $[N_0]$ holds with $\lambda\not=2\rho$ if $\tau=0$ or $\beta\in (\varpi_2,0)$ in Case $[N_2]$. 
If $c$ in \eqref{confi} satisfies $c=\infty$, that is, \begin{equation}
\label{zein}
c=\limsup_{|x|\to \infty} u(x)=\infty,
\end{equation}
 then  $u(x)\sim U_{\rho,\beta}(|x|)$ as $|x|\to \infty$.  	
\end{proposition}

\begin{proof}
By our assumptions, we have  
\begin{equation} \label{mubb} \lambda>2\rho\quad \mbox{when } \tau=0.
\end{equation}

In view of Theorem~\ref{pro1} and the modified Kelvin transform, it suffices to prove that if the positive solution $u$ of \eqref{e1} in $\mathbb R^N\setminus \overline{B_R(0)}$ satisfies \eqref{zein}, then 
\begin{equation}
\label{univ} 
\mbox{for every } \eta>0,\quad \lim_{|x|\to \infty} |x|^{\beta+\eta} u(x)=\infty. 	
\end{equation}

(I) In addition, if $\tau\not=0$,  
then we suppose  that 
\begin{equation} \label{giob} \lambda |\beta|^{-\tau}-2\rho \geq 0. 
\end{equation} 

(When $\tau=0$, we have strict inequality in \eqref{giob} because of \eqref{mubb}.)  
Then, we can directly prove \eqref{univ} by means of the comparison principle with a suitable family of positive radial sub-solutions $\{v_\delta\}_{\delta>0}$ of \eqref{e1} in $\mathbb R^N\setminus \overline{B_R(0)}$. 
For every $\varepsilon>0$ small, we have  
\begin{equation} \label{onon} 
\mathfrak E_\varepsilon:=
2\rho \left(\beta+\varepsilon\right)+ \lambda \left|\beta+\varepsilon\right|^{1-\tau} >0. 
\end{equation}
Let $\varepsilon\in (0,|\beta|)$ and $\alpha=\alpha(\varepsilon)>0$ be fixed as small as needed. In particular, let $\alpha>0$ satisfy
\begin{equation} \label{mup1} 0<  \alpha<
\min \left\{ \frac{\mathfrak E_\varepsilon}{\left(|\beta|-\varepsilon\right)}, (|\beta|-\varepsilon) (q-1)
\right\}. 
\end{equation}
We define $c_\varepsilon>0$ as follows
\begin{equation} \label{mup2} c_\varepsilon:=\left[ \mathfrak E_\varepsilon- \left(|\beta|-\varepsilon\right) \alpha\right]^\frac{1}{q-1}.  
\end{equation} 
For arbitrary $\delta>0$ and $c_0\in (0,c_\varepsilon]$, we define 
$$ v_\delta(r)=c_0 \left(\delta+r^{-\alpha} \right)^{\frac{\beta+\varepsilon}{\alpha}} \quad \mbox{for every } r>0. 
$$
For every $\delta>0$, we show that $v_\delta$ is a positive radial sub-solution of \eqref{e1} in $\mathbb R^N\setminus \overline{B_1(0)}$, that is,
\begin{equation} \label{muzi} \mathbb L_{\rho,\lambda,\tau}[v_\delta(|x|)]\geq |x|^\theta [v_\delta(|x|)]^q\quad \mbox{for every } |x|>1. 
\end{equation}
 
Indeed, let $|x|=r>1$ be arbitrary. Using \eqref{mup1} and \eqref{mup2}, we see that
$$ \begin{aligned} 
\mathbb L_{\rho,\lambda,\tau} [v_\delta(r)]& =\frac{v_\delta(r)}{r^2} \left[
(\beta+\varepsilon)^2 \left(\frac{r^{-\alpha}}{\delta+r^{-\alpha}}\right)^2+\mathfrak E_\varepsilon \frac{r^{-\alpha}}{\delta+r^{-\alpha}} +\frac{\left(\beta+\varepsilon\right)\alpha \,\delta \,r^{-\alpha}}{(\delta+r^{-\alpha})^2}
\right] \\
& \geq \left[
\mathfrak E_\varepsilon -\left(|\beta|-\varepsilon\right) \alpha 
\right]
 \frac{v_\delta(r)}{r^2} \frac{r^{-\alpha}}{\delta+r^{-\alpha}}\geq r^\theta v_\delta^q(r),
 \end{aligned}
$$
which proves the claim of \eqref{muzi}. 

\vspace{0.2cm}
{\bf Proof of \eqref{univ} concluded.} 
We diminish $c_0>0$ to ensure that 
$ c_0 R^{|\beta|-\varepsilon}\leq \min_{|x|=R} u(x)  $. 
Hence, for any $\delta>0$, we get $v_\delta(|x|)\leq 
u(x)$ for all $|x|=R$. By our assumption \eqref{zein}, we have
$$ \lim_{|x|\to \infty} \frac{u(x)}{v_\delta(|x|)}=\infty. 
$$
Clearly, $v_\delta'(r)>0$ for every $r>0$. 
Hence, by the comparison principle, we infer that 
$$ u(x)\geq v_\delta(|x|)\quad \mbox{for all } |x|\geq R\ \mbox{and } \delta>0.  
$$
By letting $\delta\to 0^+$, we find that 
$\liminf_{|x|\to \infty} |x|^{\beta+\varepsilon} u(x)\geq c_0>0$. For $\eta>0$ arbitrary, by letting $\varepsilon\in (0,\eta)$ be small enough, we conclude \eqref{univ}. \qed

\vspace{0.2cm}
(II) We assume that \eqref{giob} is not satisfied, that is,   
\begin{equation} \label{neby} 2\rho \beta+\lambda |\beta|^{1-\tau}<0.  
\end{equation}
It remains to prove \eqref{univ} in Case $[N_0]$ (b), as well as when $\beta\in (\varpi_2,0)$ in Case $[N_2]$.  
Since $\beta<0$ and $\lambda>0$, from \eqref{neby} 
we have $\rho>0$ and 
$$ -\beta>\left(2\rho/\lambda\right)^{-\frac{1}{\tau}}.  
$$

As in Part (I) in which we replace 
$\beta$ by $-[2\rho/\lambda]^{-1/\tau}$, we find that  
\begin{equation}
\label{muio1}
\mbox{for every } \eta>0, 
\quad \lim_{|x|\to \infty} |x|^{\eta-
\left(\frac{2\rho}{\lambda}\right)^{-1/\tau}} u(x)=\infty. \end{equation}

To reach \eqref{univ} from \eqref{muio1}, we need to apply an iterative process of 
comparing the solution $u$ with suitable families of sub-solutions. 

 \vspace{0.2cm}
Our assumption gives the positivity of $\widetilde f_\rho(\beta)$, namely, 
$$ \widetilde f_\rho(\beta)=-\beta-2\rho  +\lambda |\beta|^{-\tau}>0.$$ 

If $\beta\in (\varpi_2,0)$ in Case $[N_2]$, then since $t\longmapsto \widetilde f_\rho(t)$ is increasing on $[\varpi_2,0)$ and $\varpi_2<\beta<0$, we have  
$$ \widetilde f_{\rho}(t)\geq \widetilde f_\rho(\beta)\quad \mbox{for every } \beta\leq t<0. 
$$

If Case $[N_0]$ (b) holds, then we always have  
$$ \inf_{t<0} \widetilde f_\rho(t)= \widetilde f_{\rho} (-(\lambda \tau)^{1/(\tau+1)})=2\left( \Upsilon-\rho\right)>0.$$
We define $c_*>0$ as follows
\begin{equation} c_*=\left\{ 
\begin{aligned}
& \widetilde f_\rho(\beta) && \mbox{if } \beta\in (\varpi_2,0)\ \mbox{in Case }[N_2], &\\
& 2\left(\Upsilon-\rho\right) && \mbox{if Case } [N_0]\ {\rm (b)}\ \mbox{holds}. & 	
\end{aligned}
\right.  \label{cst0}
\end{equation}

Fix $ \varepsilon>0$ and 
$\alpha=\alpha(\varepsilon)>0$ small enough such that 
\begin{equation}
\label{jiff}
0< \varepsilon <
\frac{(2\rho/\lambda)^{-1/\tau} }{2}
\min  \left\{ \frac{c_*}{4\rho}, 1 \right\}
\quad \mbox{and}\quad 
0<\alpha < \varepsilon \min\,\{ q-1,1 \}. 	
\end{equation}

We define an increasing sequence of positive numbers $\{m_j\}_{j\geq 1}$  with $\lim_{j\to \infty} m_j=
\infty$, where 
\begin{equation} \label{locc} m_1=\left( \frac{2\rho}{\lambda}\right)^{-\frac{1}{\tau}}-\varepsilon>
\frac{1}{2} \left( \frac{2\rho}{\lambda}\right)^{-\frac{1}{\tau}}
\quad \mbox{and}\quad 
m_{j+1}=m_1+j\varepsilon \quad \mbox{for every } j\geq 1.
\end{equation}  
 
The assumption in \eqref{neby} and the choice of $m_1$ gives that $m_1<m_2<-\beta$.  

Let $k\geq 1$ be such that $m_{1}+k\varepsilon<-\beta\leq m_{1}+\left(k+1\right)\varepsilon$, that is, 
\begin{equation} \label{hilg} k=\left[\frac{-m_1-\beta}{\varepsilon} \right] -1,
\end{equation} where $[\cdot]$ in \eqref{hilg} stands for the integer part. 
Thus, using also 
\eqref{cst0}, we have 
$$ \widetilde f_\rho(m_j)\geq c_*\quad \mbox{for every } 1\leq j\leq k. 
$$

We fix $c_0>0$ such that 
\begin{equation} \label{hone2} 0<c_0<\left( 
\frac{c_*}{4} \left( \frac{2\rho}{\lambda}\right)^{-\frac{1}{\tau}}\right)^{\frac{1}{q-1}}. 
\end{equation}

For $\delta>0$ fixed arbitrary and every $r>0$, 
we define $v_{j,\delta}(r)$ as follows
$$\begin{aligned}
 & v_{j,\delta}(r)=c_0 \,r^{m_j} \left(\delta+r^{-\alpha}\right)^{-\frac{m_{j+1}-m_j+\varepsilon}{\alpha}}\quad \mbox{for }1\leq j\leq k-1,\\
&  v_{k,\delta}(r)=c_0 \,r^{m_k} 
\left( \delta+r^{-\alpha}\right)^
{\frac{\beta+m_k}{\alpha}}.
\end{aligned} $$

{\bf Claim:}  For every $1\leq j\leq  k$, the function $v_{j,\delta}$ is a radial sub-solution of \eqref{e1} in $\mathbb R^N\setminus \overline{B_1(0)}$.

\begin{proof}[Proof of the Claim]    
Only for the purpose of unifying the calculations, we will redefine $m_{k+1}=-\beta-\varepsilon$ (if $m_{k+2}>-\beta$). Then, $v_{k,\delta}$ has the same form as $v_{j,k}$ with $j=k$. 
Observe that with this new definition of $m_{k+1}$, we have 
$$ \beta+m_{j+1}+\varepsilon<0\quad \mbox{for all }1\leq j\leq k-1
\quad \mbox{and}\quad \beta+m_{k+1}+\varepsilon=0.
$$
Hence, by our choice of $\alpha$ in \eqref{jiff}, we see that 
\begin{equation} \label{ahu} c_0^{q-1}   \frac{ v_{j,\delta}(r)}{r^2} \frac{r^{-\alpha}}{\delta+r^{-\alpha}} \geq 
r^{\theta} [v_{j,\delta}(r)]^q\quad \mbox{for every } r>1\ \mbox{and } 1\leq j\leq k.
\end{equation}

Fix $1\leq j\leq  k$. Let $r>1$ be arbitrary.
It follows that 
$$\begin{aligned} 
& 	
v'_{j,\delta}(r) =\frac{ v_{j,\delta}(r)}{r} 
\left[m_j+\frac{\left(m_{j+1}-m_j+\varepsilon
 \right) r^{-\alpha}} {\delta+
r^{-\alpha} }
\right]>0,\\
& 
\mathbb L_{\rho} [v_{j,\delta}(r)]=
\frac{ v_{j,\delta}(r)}{r^2} 
\left[ m_j^2 -2\rho m_j +
\left( 2m_j-2\rho -\frac{\alpha \delta}{\delta+r^{-\alpha}}
\right)
\frac{ 
\left(m_{j+1}-m_j +\varepsilon\right) r^{-\alpha}}{ \delta+r^{-\alpha}} 
\right.\\
& 
\qquad \qquad \qquad \qquad \quad \left. + 
\frac{ \left(m_{j+1}-m_j+\varepsilon\right) r^{-2\alpha}}{\left(\delta+r^{-\alpha}\right)^2}
\right].
\end{aligned}
$$
Under the situations of (II), we have $\lambda>0$. Hence,
$$ \lambda \,G_\tau[v_{j,\delta}(r)]=
  \lambda \frac{ v_{j,\delta}(r)}{r^2} 
  \left[m_j+\frac{\left(m_{j+1}-m_j+\varepsilon
 \right) r^{-\alpha}} {\delta+
r^{-\alpha} }
\right]^{1-\tau}\geq 
\lambda \, m_j^{1-\tau} 
\frac{ v_{j,\delta}(r)}{r^2}.
$$
Consequently, 
we arrive at 
$$ \mathbb L_{\rho,\lambda,\tau}[v_{j,\delta}(r)]\geq 
 \frac{ v_{j,\delta}(r)}{r^2} \left[ 
 m_j^2 -2\rho m_j+ \lambda \,m_j^{1-\tau} 
  +\frac{ \left( 2m_j-2\rho -\alpha \right)
\left(m_{j+1}-m_j +\varepsilon\right) r^{-\alpha}}{ \delta+r^{-\alpha}} 
 \right].
$$
Since $ m_j^2 -2\rho m_j+ \lambda \,m_j^{1-\tau}=m_j\,\widetilde f_\rho(-m_j)>0$,  
we find that
\begin{equation} \label{hone3} \mathbb L_{\rho,\lambda,\tau}[v_{j,\delta}(r)]\geq 
\mathfrak H_{j,\varepsilon,\alpha} 
 \frac{ v_{j,\delta}(r)}{r^2} \frac{r^{-\alpha}}{\delta+r^{-\alpha}},
\end{equation}
where we define $\mathfrak H_{j,\varepsilon,\alpha}$ as follows
\begin{equation} \label{cee2}  \begin{aligned}
\mathfrak H_{j,\varepsilon,\alpha} &= m_j^2 -2\rho m_j+ \lambda \,m_j^{1-\tau} +\left( 2m_j-2\rho -\alpha \right)
\left(m_{j+1}-m_j +\varepsilon\right)\\
&   \geq \left(2m_{j+1}m_j-m_j^2 -2\rho m_{j+1} +
\lambda m_j^{1-\tau}\right) +2 \left(m_j-\rho-\varepsilon\right)\varepsilon.
\end{aligned} 
\end{equation}
Thus, since $m_{j+1}>m_j>\frac{1}{2} \left( \frac{2\rho}{\lambda}\right)^{-\frac{1}{\tau}}$, we get 
\begin{equation} \label{cee1}  \begin{aligned} 
  	2m_{j+1}m_j-m_j^2 -2\rho m_{j+1} +
\lambda m_{j}^{1-\tau}& >m_j^2 -2\rho m_{j+1}+\lambda m_j
m_{j+1}^{-\tau}\\ 
&=m_j \,\widetilde f_\rho(m_j)
\geq \frac{c_*}{2} \left( \frac{2\rho}{\lambda}\right)^{-\frac{1}{\tau}}.
  \end{aligned}
 \end{equation}
Using \eqref{jiff} and \eqref{cee1} in \eqref{cee2}, jointly with \eqref{hone2}, we arrive at 
\begin{equation} \label{hone4} \mathfrak H_{j,\varepsilon,\alpha}>
\frac{c_*}{2} \left( \frac{2\rho}{\lambda}\right)^{-\frac{1}{\tau}}-2\rho \varepsilon
>
\frac{c_*}{4} \left( \frac{2\rho}{\lambda}\right)^{-\frac{1}{\tau}} > c_0^{q-1}\quad \mbox{for every } 
1\leq j\leq k.
\end{equation}

In view of \eqref{ahu}, \eqref{hone3} and \eqref{hone4}, we conclude the proof of the Claim. 
\end{proof}

{\bf Proof of \eqref{univ} concluded.}
We diminish $c_0>0$ such that 
$ c_0 R^{-\beta}\leq \min_{|x|=R} u(x)$. 
Hence,
$$ u\geq v_{j,\delta}\quad \mbox{on } \partial B_{R}(0)\quad \mbox{for every } 1\leq j\leq k\ \mbox{and } \delta>0.$$ In view of \eqref{muio1} and the definition of $m_1$ in \eqref{locc}, we get
$$ \lim_{|x|\to \infty} \frac{u(x)}{v_{1,\delta}(|x|)}=\infty.$$  
Thus, using the Claim and the comparison principle, we infer that 
$$ u(x)\geq v_{1,\delta}(|x|) \quad \mbox{for all }
|x|\geq R\ \mbox{and } \delta>0. $$
By letting $\delta\to 0^+$, we find that $\liminf_{|x|\to \infty} |x|^{-m_2-\varepsilon}
u(x)\geq c_0>0 $, leading to 
$$ \lim_{|x|\to \infty} \frac{u(x)}{v_{2,\delta}(|x|)}=\infty.$$
By induction over $j=1,\ldots, k$, we get 
$u(x)\geq v_{k,\delta}(|x|)$ for all $|x|\geq R$ and $\delta>0$ so that 
\begin{equation} \label{fam0} \liminf_{|x|\to \infty} |x|^{\beta} u(x)\geq c_0>0. 
\end{equation} 
This ends the proof of \eqref{univ} under the assumption \eqref{neby}.  \qed

\vspace{0.2cm}
The proof of Proposition~\ref{ze-ta3} is now complete. 
\end{proof}

\section{Existence; Alternative (I) in Theorem~\ref{rod1c}} \label{alt1}

\begin{theorem}[Existence, Alternative (I) in Theorem~\ref{rod1c}] \label{exol1} 
Let \eqref{e2} hold and $\beta\in (\varpi_2,0)$ in Case $[N_2]$. 
 Then, for each $b\in \mathbb R_+$, there is a unique positive solution $u_{b}$ of \eqref{e1} in $\mathbb R^N\setminus \{0\}$ such that 
	\begin{equation} \label{ahga} 
	\lim_{|x|\to 0} \frac{u(x)}{ \Phi_{\varpi_2(\rho)}(|x|)}=b\quad \mbox{and}\quad \lim_{|x|\to \infty} \frac{u(x)}{U_{\rho,\beta}(|x|)}=1. 	
	\end{equation} 
	Moreover, $u_{b}$ is radially symmetric and 
$u_{b}'(r)>0$ for every $r>0$. 
\end{theorem}

\begin{rem} {\rm In Theorem~\ref{exol1}, for every $b>0$, we have $u_{b}=T_\sigma[u_1]$
with $\sigma= b^{1/(\beta-\varpi_2)}$. }	
\end{rem}

\begin{proof}
Let \eqref{e2} hold and $\beta\in (\varpi_2,0)$ in Case $[N_2]$. 
 	Fix $a>0$ arbitrary. For every integer $k>a$, we consider the boundary value problem
  \begin{equation}
  \label{ehh}
 \left\{ \begin{aligned}
 	& \mathbb L_{\rho,\lambda,\theta} [v]=|x|^{\theta} v^q\quad \mbox{in } D_{a,k}:=\{x\in \mathbb R^N:\ 
 	a<|x|<k\},\\
 	& v(x)=(1/2) \,U_{\rho,\beta}(|x|)\quad \mbox{on } 
 	\partial D_{a,k},
 \end{aligned}
	\right.
  \end{equation}
  Then, $(1/2)\, U_{\rho,\beta}$ (respectively, $U_{\rho,\beta}$) is a positive sub-solution (respectively, super-solution) of \eqref{ehh}. By the method of sub-super-solutions, we obtain the existence of the 
  {\em minimal} positive solution of \eqref{ehh}, say $v_{a,k}$, where the minimality is understood with respect to the above-mentioned pair of sub-super-solutions. By the invariance of \eqref{ehh} under rotations and the symmetry of the domain and the boundary conditions, we infer that $v_{a,k}$ is radially symmetric. Moreover, by the sub-super-solutions method, we conclude that 
  \begin{equation} \label{hype}  (1/2)\,  U_{\rho,\beta}(|x|)\leq 
  v_{a,k}(|x|)
  \leq v_{a,k+1}(|x|)\leq U_{\rho,\beta}(|x|)\quad \mbox{for all } x\in D_{a,k}. 
  \end{equation} 
  For every $r\geq a$, we define $v_{a}(r)=\lim_{k\to \infty} v_{a,k}(r)$.
  Then, by Proposition~\ref{aurr}, we obtain that
  $ v_{a,k}\to v_{a}$ in $C^{1}_{\rm loc}(\mathbb R^N\setminus \overline{B_a(0)})$ and $v_{a}$ is a positive radial solution of \eqref{e1} in  $\mathbb R^N\setminus \overline{B_a(0)}$. 
  
  Using \eqref{hype}, we see that 
  $$  
 (1/2)\, U_{\rho,\beta}(r)\leq 
  v_{a}(r)\leq 
   U_{\rho,\beta}(r)\quad \mbox{for all } r\in (a,\infty).
  $$
  This, jointly with Corollary~\ref{goos}, implies that 
  that \begin{equation} \label{tigg} \lim_{r\to \infty} \frac{v_{a}(r)}{U_{\rho,\beta}(r)}=1.
 \end{equation}

  Let $\underline r=\inf \{s>0: \ v_{a}\ \mbox{is a positive radial solution of } \eqref{e1} 
  \ \mbox{for } |x|>s \}.$
  
  \vspace{0.2cm}
  {\bf Claim:} 
We have $\underline r=0$, that is, $v_{a}$ is a positive radial solution of \eqref{e1} in $\mathbb R^N\setminus \{0\}$. 
  
  \vspace{0.2cm}
  {\bf Proof of the Claim.} 
 Assume by contradiction that $\underline r>0$. 
 As $v_{a}$ is a radial solution of \eqref{e1} 
 for $|x|>\underline r$, we have 
\begin{equation} \label{ioni} v_{a}''(r)+(1-2\rho) \frac{v_{a}'(r)}{r} +\lambda\, [v_{a}(r)]^\tau \frac{|v_{a}'(r)|^{1-\tau}}{r^{1+\tau}} =r^\theta [v_{a}(r)]^q\quad \mbox{for all } r>\underline r. 
\end{equation}

$\bullet$ If $v'_{a}(r_*)=0$ for some $r_*>\underline r$, then $r_*$ would be a local minimum point for $v_{a}$ (since $v_{a}''(r_*)>0$) and, hence, $v_{a}'(r)<0$ for every $r\in (\underline r, r_*)$ and $v_{a}'(r)>0$ for every $r>r_*$. 
 
$\bullet$ Alternatively, if $v'_{a}(r)\not =0$ for every $r>\underline r$, then   $v'_{a}(r)>0$ for all $r>\underline r$ in view of \eqref{tigg}. 
 
 In either of these cases, we have that there exists 
 $\lim_{r\searrow \underline r} v_{a}(r)\in [0,\infty]$.  
 We next show that 
 \begin{equation} \label{gifd} \lim_{r\searrow \underline r}v_{a}(r)\in (0,\infty).	
 \end{equation}

{\bf Proof of \eqref{gifd}.} 
If $\lim_{r\searrow \underline r}v_{a}(r)=0$, then from \eqref{tigg}, we get  
$\lim_{|x|\to \infty} |x|^{\varpi_2}\,v_a(x)=0$. 
Then, for arbitrary $\delta>0$, there exist 
$\underline r<R_1(\delta)<R_2(\delta)$ such that 
$$ 
v_{a}(r)\leq \delta\, r^{-\varpi_2}\quad \mbox{for every } 
r\in (\underline r,R_1(\delta)] \cup [R_2(\delta),\infty).$$  
By applying the comparison principle, we conclude that 
$v_{a}(r)\leq \delta \, r^{-\varpi_2}$ for every $r>\underline r$. By letting $\delta\to 0$, we arrive at $v_{a}\equiv 0$ on $(\underline r,\infty)$, reaching a contradiction.  

On the other hand, for every $r\in (\underline r,a]$, the graphs of 
$v_{a}(r)$ and $U_{\rho,\beta}(r)$ have no points of intersection. Otherwise, from \eqref{tigg} and the comparison principle, we would obtain that 
$v_{a}(r)= U_{\rho,\beta}(r)$ for every $r>r_1$ for some $r_1\in (\underline r,a)$. This is impossible since 
\begin{equation} \label{fuio} v_{a}(a)=(1/2)\, U_{\rho,\beta}(a)<U_{\rho,\beta}(a). 
\end{equation}	

Consequently, by the comparison principle, we have 
\begin{equation} \label{forbx} 0< v_{a}(r)\leq U_{\rho,\beta}(r)\quad \mbox{for every } r>\underline r. 
\end{equation}

This concludes the proof of \eqref{gifd}. 

\vspace{0.2cm}
We define $Y(r)=r v'_{a}(r)/v_{a}(r)$ for every $r>\underline r$. With the same argument as in the proof of 
Proposition~\ref{pzerop}, 
we derive that $Y'(r)\not =0$ for every $r>\underline r$ close enough to $\underline r$ (using that for such $r$, we have $v'_{a}(r)$ does not vanish and thus $Y(r)$ has the same sign as $v'_{a}(r)$). 

Therefore, there exists $\lim_{r\to \underline r} Y(r)$, which together with \eqref{gifd}, yields that 
there exists $\lim_{r\to \underline r} v'_{a}(r)$ (potentially infinite). We next show that  
\begin{equation} \label{fugi}
\lim_{r\to \underline r} v'_{a}(r)\in \mathbb R.
\end{equation}
Assume by contradiction that $\lim_{r\searrow \underline r} v'_{a}(r)=-\infty$ or $+\infty$. 
Recall that $\lambda>0$ in Case $[N_2]$. 
Now, 
we multiply \eqref{ioni} by $r^{1-2\rho}$ and let $r\searrow \underline r$ to obtain that 
$$ \lim_{r\searrow \underline r }  \frac{d}{dr} \left[ r^{1-2\rho} v'_{a}(r)\right]=-\infty. 
$$  
Hence, there exists $\lim_{r\searrow \underline r} r^{1-2\rho} v'_{a}(r)$, which would necessarily imply that $\lim_{r\searrow \underline r} v'_{\varepsilon,a}(r)=\infty$. 
We next divide \eqref{ioni} by $v'_{a}(r)$ and let $r\searrow \underline r$ to find that 
$$ \lim_{r\searrow \underline r} \frac{v''_{a}(r)}{v'_{a}(r)}=\frac{2\rho-1}{\underline r}.  
$$
 
 Let $L>0$ be large such that $L>(1-2\rho)/\underline r$. Then, $\lim_{r\searrow \underline r} [v''_{a}(r)+L\,v'_{a}(r)]=\infty$, which implies that 
 $r\longmapsto v'_{a}(r)+L \,v_{a}(r)$ is increasing for $r\in (\underline r,\underline r+\delta)$ if $\delta>0$ is small enough. However, this is impossible since $\lim_{r\searrow \underline r} v'_{a}(r)=\infty$.
 This finishes the proof of \eqref{fugi}.   
 
 In light of \eqref{gifd} and \eqref{fugi}, we infer that 
 $\underline r=0$. This completes the proof of the Claim. \qed

\vspace{0.2cm}
The proof of the Claim (see \eqref{forbx}) gives that 
$ 0<v_{a}(r)\leq U_{\rho,\beta}(r)$ 
for every $ r>0 $ and thus 
$$v_{a}'(r)>0\quad \mbox{ for all } r>0. $$ Moreover, by Theorem~\ref{pro1}, we have \eqref{dual}, which in our case leads to 
\begin{equation} \label{alora} \lim_{r\to 0^+} r^\beta v_{a}(r)=0. 
\end{equation}
Indeed, if we were to have $\lim_{r\to 0^+}  v_{a}(r)/U_{\rho,\beta}(r)=1$, then from \eqref{tigg} and the comparison principle,  we would get
$v_{a}(r)=U_{\rho,\beta} (r)$ for all $r>0$. This is a contradiction with \eqref{fuio}.

Hence, \eqref{alora} holds. By Theorem~\ref{thi} and Lemma~\ref{gold1}, there exists $b\in (0,\infty)$ such that 
$$ \lim_{r\to 0^+} \frac{v_{a}(r)}{\Phi_{\varpi_2(\rho)}(r)}=b.  
$$

Hence, for {\em some} $b\in (0,\infty)$, problem  \eqref{e1} in $\mathbb R^N\setminus \{0\}$, subject to  
\eqref{ahga}, 
has a positive solution, relabelled $u_{b}$. Remark that, in addition, $u_b$ is {\em radially symmetric} and satisfies $u_{b}'(r)>0$ for every $r>0$.   
If $\widetilde u_{b}$ were to be another positive solution (not necessarily radial) of \eqref{e1} in $\mathbb R^N\setminus \{0\}$, subject to \eqref{ahga},  then 
we would have $$ \lim_{|x|\to 0} \frac{u_{b}(x)}{\widetilde u_{b}(x)}=1\quad \mbox{and }\quad \lim_{|x|\to \infty} \frac{u_{b}(x)}{\widetilde u_{b}(x)}=1. $$ 
Then, by the comparison principle, we conclude that 
$u_{b}\equiv \widetilde u_{b}$ in $\mathbb R^N\setminus \{0\}$. 

Letting $\sigma>0$ arbitrary and $T_\sigma$ as in \eqref{ssig}, we obtain that $T_\sigma[u_{b}]$ is the unique positive solution of \eqref{e1} in $\mathbb R^N\setminus \{0\}$, subject to  
\begin{equation} \label{nova2} \lim_{|x|\to 0} \frac{u(x)}{ \Phi_{\varpi_2(\rho)}(|x|)}=b \,\sigma^{\beta-\varpi_2}\quad \mbox{and}\quad 
\lim_{|x|\to \infty} \frac{u(x)}{U_{\rho,\beta}(||x)}=1. \end{equation}
Since $\beta>\varpi_2$, it is clear that by varying $\sigma>0$, the  first limit in \eqref{nova2} can be any positive number. This completes the proof of Theorem~\ref{exol1}.
\end{proof}

\section{Existence; Alternatives (II) and (II$'$) in Theorem~\ref{rod1c}} \label{alt2}

In this section, in the framework of Theorem~\ref{globo}, we prove that for every $c\in \mathbb R_+$, there exists  a unique positive solution $u_{\infty,c}$ of \eqref{e1} in $\mathbb R^N\setminus \{0\}$ with the properties stated in Theorem~\ref{globo}~(i). These properties correspond to the alternative (II) of Theorem~\ref{rod1c} if $\beta\in (\varpi_2,0)$ in Case $[N_2]$. Otherwise, the properties in (II)$'$ of Theorem~\ref{rod1c} will apply for $u_{\infty,c}$.   

\begin{theorem} \label{exol2} 
Let \eqref{e2} hold, $\theta<-2$ and $f_\rho(\beta)>0$. Then, for every $c\in \mathbb R_+$, there exists a unique positive solution $u_{\infty,c}$ of \eqref{e1} in $\mathbb R^N\setminus \{0\}$ satisfying 
$ u(x)\sim U_{\rho,\beta}(x)$ as $|x|\to 0$ 
and  \eqref{fbio} as $ |x|\to \infty$. 
Moreover, $u_{\infty,c}$ is radially symmetric and  
$u_{\infty,c}'>0$ on $(0,\infty)$. 
\end{theorem}

\begin{rem} {\rm In Theorem~\ref{exol2}, for every $c>0$, we have $u_{\infty,c}=T_\sigma[u_{\infty,1}]$
with $\sigma=c^{1/(\beta-\varpi_1)}$ if $\beta<\varpi_1$ in Case $[N_j]$ with $j=1,2$ and  
$\sigma=c^{1/\beta}$ in the remaining situations.}
\end{rem}
  
\begin{proof}
As in the proof of Theorem~\ref{exol1}, we fix 
$k>a>0$. Let $v_{a,k}$ be the minimal positive (radial) solution of \eqref{ehh} with respect to the pair $( U_{\rho,\beta}/2, U_{\rho,\beta})$ of sub-super-solutions of \eqref{ehh}. We let $a=1/n$ for every large integer 
$n\geq 1$ satisfying $1/n<k$. 
 
 By the sub-super-solutions method, we infer that 
  \begin{equation} \label{hype22}  (1/2)\, U_{\rho,\beta}(|x|)\leq 
  v_{1/n,k}(|x|)
  \leq v_{1/(n+1),k}(|x|)\leq U_{\rho,\beta}(|x|) \end{equation}
for all $x\in D_{1/n,k}=\{x\in \mathbb R^N:\ 1/n<|x|<k\}$. 
We keep $k$ fixed. For every $r\in (0,k)$, we define 
$V_{k}(r)=\lim_{n\to \infty} v_{1/n,k}(r)$. 
 As before, by Proposition~\ref{aurr}, we obtain that
  $$ v_{1/n,k}\to V_{k}\quad \mbox{in }  C^1_{\rm loc} (B_k(0)\setminus \{0\} ) \ \mbox{ as } n\to \infty$$ and $V_{k}$ is a positive radial solution of \eqref{e1} in $B_{k}(0)\setminus \{0\}$.   
  
  In view of \eqref{hype22}, we find that 
  \begin{equation} \label{mope}  (1/2)\, U_{\rho,\beta}(r)\leq 
  V_{k}(r)\leq U_{\rho,\beta}(r)
  \quad \mbox{for all } r\in (0,k). 
  \end{equation}
  
  From \eqref{mope} and Theorem~\ref{pro1} (c), 
  we conclude that 
  \begin{equation} \label{fami} \lim_{r\to 0^+} 
  \frac{ V_{k}(r)}{U_{\rho,\beta}(r)}=1.
  \end{equation}
    	
  Let $\overline r=\sup \,\{s>0: \ V_{k}\ \mbox{is a positive radial solution of } \eqref{e1} 
  \ \mbox{for } 0<|x|<s \}.$
  
  \vspace{0.2cm}
  {\bf Claim:} 
We have $\overline r=\infty$, that is, $V_{k}$ is a positive radial solution of \eqref{e1} in $\mathbb R^N\setminus \{0\}$. 

 \vspace{0.2cm}
 {\bf Proof of the Claim.} Assume by contradiction that $\overline r<\infty$. Here, since $\lim_{r\to 0^+} V_{k}(r)=0$, the only possibility is that $V_{k}'(r)>0$ for every $r\in (0,\overline r)$ (using that $V_{k}(r)$ cannot have any local minimum point on $(0,\overline r)$). Thus, there exists $\lim_{r\nearrow \overline r} V_{k}(r)\in (0,\infty)$. This limit is finite from \eqref{mope} when $k=\overline r$. If   
 $k<\overline r$, then for every $r\in (k,\overline r)$, the graphs of 
$V_{k}(r)$ and $U_{\rho,\beta}(r)$ have no points of intersection. Otherwise, from \eqref{fami} and the comparison principle, we would find that $V_{k}(r)= U_{\rho,\beta}(r)$ for every $r\in (0,k]$, which is a contradiction with $$ V_{k}(k)=(1/2)\, U_{\rho,\beta}(k)<U_{\rho,\beta}(k).$$ We must have 
$$  0< V_{k}(r)\leq U_{\rho,\beta}(r)
  \quad \mbox{for all } r\in (0,\overline r). 
$$

We define $Y(r)=r V'_{k}(r)/V_{k}(r)$ for every $r\in (0,\overline r)$. Following the reasoning in Proposition~\ref{pzerop}, 
we derive that $Y'(r)\not =0$ for every $0<r<\overline r$ close enough to $\overline r$. 

So, there exists $\lim_{r\nearrow \overline r} Y(r)$ (possibly, $\infty$). Thus, 
there exists $\lim_{r\nearrow \overline r} V'_{k}(r)\in [0,\infty]$.  
To conclude the Claim, it remains to rule out the case $\lim_{r\nearrow \overline r} V'_{k}(r)=\infty$. 

 Similar to the proof of \eqref{fugi}, 
 we assume by contradiction that $\lim_{r\nearrow \overline r} V'_{k}(r)=\infty$.

$\bullet$ If Case $[N_0]$ (b) holds or if  $\beta\in (\varpi_2,0)$ in Case $[N_2]$, then $\lambda>0$ so that we  
get  
 $$ \lim_{r\nearrow \overline r }  \frac{d}{dr} \left[ r^{1-2\rho} V'_{k}(r)\right]=-\infty. 
$$  
So, $ r\longmapsto r^{1-2\rho} V'_{k}(r)$ is decreasing for $r\in (\overline r-\delta,\overline r)$ for $\delta>0$ small, which is a contradiction with 
$\lim_{r\nearrow \overline r} r^{1-2\rho} V'_{k}(r)=\infty$. Hence, we have 
$\lim_{r\nearrow \overline r} V'_{k}(r)\in [0,\infty)$. 

$\bullet$ If Case $[N_0]$ (a) or Case $[N_1]$ (a) holds, then $\tau=0$ and $V_{k}$ is a positive solution of
\begin{equation} \label{iona} V_{k}''(r)+(1-2\rho+\lambda) \frac{V_{k}'(r)}{r}  =r^\theta [V_{k}(r)]^q\quad \mbox{for all } r>(0,\overline r). 
\end{equation}

$\diamond$ If Case $[N_0]$ (a) holds, then $\lambda\geq 2\rho$. So, \eqref{iona} yields that $ \lim_{r\nearrow \overline r} V_{k}''(r)=-\infty $. 
Thus, $ r\longmapsto V'_{k}(r)$ is decreasing on $(\overline r-\delta,\overline r)$ for $\delta>0$ small, which contradicts 
$\lim_{r\nearrow \overline r} V'_{k}(r)=\infty$. 

$\diamond$ If Case $[N_1]$ (a) holds, then $\lambda< 2\rho$. Multiplying \eqref{iona} by $r^{\lambda-2\rho}$ and taking $r\nearrow \overline r$, we find that
$\lim_{r\nearrow \overline r} \frac{d}{dr} [r^{\lambda-2\rho} V'_{k}(r)]=-\infty$. Hence, 
$ r\longmapsto r^{\lambda-2\rho} V'_{k}(r)$ is decreasing on $(\overline r-\delta,\overline r)$ for $\delta>0$ small, which is a contradiction with 
$\lim_{r\nearrow \overline r} r^{\lambda-2\rho} V'_{k}(r)=\infty$. 

$\bullet$ If Case $[N_1]$ (b) holds, then $\lambda<0$ and $\tau\in (0,1)$. It follows that  
$$ \lim_{r\nearrow \overline r} \frac{(r^{1-2\rho}\,V_{k}'(r) )'}{(r^{1-2\rho}V_{k}'(r)  )^{1-\tau}}
=-\lambda  \lim_{r\searrow \overline r }  r^{-2\rho \tau-1} V_{k}(r):=\nu
\in (0,\infty).
$$
In particular, $r\longmapsto [ r^{1-2\rho} V_{k}'(r)]^\tau -2 \tau \nu r$ is decreasing on $(\overline r-\delta,\overline r)$ for $\delta>0$ small, which a contradiction with 
$\lim_{r\nearrow \overline r} r^{1-2\rho} V'_{k}(r)=\infty$. 

\vspace{0.2cm}
We have now exhausted all the possible situations. Hence, we have $\lim_{r\nearrow \overline r} V'_{k}(r)\in [0,\infty). $
This completes the proof of the Claim. \qed

\vspace{0.2cm}
We cannot have $\lim_{r\to \infty}  
   V_{k}(r)/U_{\rho,\beta}(r)=1$ (as this would imply $V_{k}(r)= U_{\rho,\beta}(r)$ for all $r>0$, which is impossible). Hence, using Theorem~\ref{rod1c}, we find that 
 for some $c\in (0,\infty)$, problem  \eqref{e1} in $\mathbb R^N\setminus \{0\}$, subject to $ u(x)\sim U_{\rho,\beta}(x)$ as $|x|\to 0$ 
and  \eqref{fbio} as $ |x|\to \infty$, 
 has a positive solution $V_k$, relabelled $u_{\infty,c}$.  
  Moreover, $u_{\infty,c}$ is 
radially symmetric and 
$$u_{\infty,c}'(r)>0\quad \mbox{for every } r>0.$$ 

Similar to the proof of Theorem~\ref{exol1}, by the comparison principle, we conclude that $u_{\infty,c}$ is the only positive solution of \eqref{e1} in $\mathbb R^N\setminus \{0\}$
with the same properties near zero and at infinity. 

By letting $\sigma>0$ arbitrary and applying the transformation $T_\sigma$ to $u_{\infty,c}$, we can extend the conclusion to every constant $c>0$. This finishes the proof of Theorem~\ref{exol2}. 
\end{proof}


\section{Existence; Alternative (III) in Theorem~\ref{rod1c}} \label{alt3}	

\begin{theorem}[Existence, Alternative (III) in Theorem~\ref{rod1c}] 
\label{exol3} 
Let \eqref{e2} hold and $\beta\in (\varpi_2,0)$ in Case $[N_2]$.   
	For every $b,c\in \mathbb R_+$, there is a unique positive solution $u_{b,\infty,c}$ of \eqref{e1} in $\mathbb R^N\setminus \{0\}$ satisfying 
	\begin{equation} \label{bono}  \lim_{|x|\to 0} \frac{u(x)}{ \Phi_{\varpi_2(\rho)}(|x|)}=b \quad \mbox{and}
\quad  \lim_{|x|\to \infty} u(x)=c. \end{equation} 
	In addition, $u_{b,\infty,c}$ is radially symmetric and 
$u_{b,\infty,c}'(r)>0$ for all $r=|x|>0$. 
\end{theorem}

\begin{proof}
We assume that \eqref{e2} holds and $\beta\in (\varpi_2,0)$ in Case $[N_2]$. 

\begin{lemma} \label{138} For every constant $b>0$, there exists a sequence $\{c_j\}_{j\geq 1} $ decreasing to zero as $j\to \infty$ such that \eqref{e1} in $\mathbb R^N\setminus \{0\} $, subject to  $\lim_{|x|\to 0} \frac{u(x)}{ \Phi_{\varpi_2(\rho)}(|x|)}=b$ and $\lim_{|x|\to \infty} u(x)=c_j$, has a unique positive solution $u_{b,\infty,c_j}$ for every $j\geq 1$. Moreover, for every $j\geq 1$, the solution $u_{b,\infty,c_j}$ is radially symmetric and 
$u'_{b,\infty,c_j}(r)>0$ for all $r>0$. 	
\end{lemma}

\begin{proof}
 We fix $b>0$ arbitrary. By Theorem~\ref{exol1}, there exists a unique positive solution
 $u_b$ of \eqref{e1} in $\mathbb R^N\setminus \{0\}$ satisfying \eqref{ahga}.  
In addition, $u_b$ is radially symmetric and $u_b'(r)>0$ for each $r>0$. 
Thus, we have 
\begin{equation} \label{honn} \lim_{r\to 0} \frac{u_b(r)}{ \Phi_{\varpi_2(\rho)}(r)}=b \quad \mbox{and}
\quad  \lim_{r\to \infty} \frac{u_b(r)}{U_{\rho,\beta}(r)}=1. \end{equation}
Let $\sigma>1$ be arbitrary and $T_\sigma$ be the scaling transformation as in \eqref{ssig}. We now define 
$$  u_{b,\sigma} (r)= \sigma^{\varpi_2-\beta} T_\sigma[u_b](r)=\sigma^{\varpi_2} u_b(\sigma r)\quad \mbox{for every } r>0. 
$$
Since $T_\sigma[u_b]$ is a positive radial solution of \eqref{e1} in $\mathbb R^N\setminus \{0\}$ and $\varpi_2<\beta<0$, using \eqref{honn} and $\sigma>1$, we get that $u_{b,\sigma}$ is a positive radial sub-solution of \eqref{e1} in $\mathbb R^N\setminus \{0\}$ satisfying 
\begin{equation} \label{honn2} \lim_{r\to 0} \frac{u_{b,\sigma}(r)}{ \Phi_{\varpi_2(\rho)}(r)}=b \quad \mbox{and}
\quad  \lim_{r\to \infty} \frac{u_{b,\sigma}(r)}{U_{\rho,\beta}(r)}=\sigma^{\varpi_2-\beta}<1.  \end{equation} 
In light of \eqref{honn} and \eqref{honn2}, we infer from the comparison principle that 
$u_{b,\sigma}(r)\leq u_b(r)$ for all $r>0$. Moreover, from the strong maximum principle (see Lemma~\ref{adol}), we have 
$$u_{b,\sigma}(r)< u_b(r)\quad \mbox{ for all }r>0. $$

We now use a similar argument to the one in Theorem~\ref{exol2}. More precisely, let $k>0$ and $n\geq 1$ be a large integer so that $1/n>k$. We consider the boundary value problem 
 \begin{equation}
  \label{eahh}
 \left\{ \begin{aligned}
 	& \mathbb L_{\rho,\lambda,\theta} [v]=|x|^{\theta} v^q && \mbox{in } D_{1/n,k}:=\{x\in \mathbb R^N:\ 
 	1/n<|x|<k\},&\\
 	& v(x)= u_{b,\sigma}(x) &&\mbox{on } 
 	\partial D_{1/n,k}.&
 \end{aligned}
	\right.
  \end{equation}

Here, $u_{b,\sigma}$ (respectively, $u_b$) is a positive radial sub-solution (respectively, super-solution) of \eqref{eahh}. Let $v_{n,k,\sigma}$ be the minimal positive solution of \eqref{eahh}
with respect to the above pair of sub-super-solutions. Then, as in the proof of Theorem~\ref{exol1},  
$v_{n,k,\sigma}$ is radially symmetric and 
\begin{equation} \label{tige} u_{b,\sigma}(r)\leq v_{n,k,\sigma}(r)\leq v_{n+1,k,\sigma} (r)\leq u_b(r)\quad \mbox{for all } 1/n\leq r\leq k.   
\end{equation}
For every $r\in (0,k)$, we define $V_{k,\sigma}(r)=\lim_{n\to \infty} v_{n,k,\sigma}(r)$. By repeating the arguments in the proof of Theorem~\ref{exol2} (with the super-solution $U_{\rho,\beta}$ there replaced by $u_b$ here), we see that $V_{k,\sigma}$ is a positive radial solution of \eqref{e1} in
$\mathbb R^N\setminus \{0\}$ satisfying 
$$ V_{k,\sigma}(r)\leq u_b(r) \quad \mbox{for all } r\in (0,\infty).$$

We let $n\to \infty$ in \eqref{tige}, then use \eqref{honn} and \eqref{honn2}
to gain that   
\begin{equation} \label{rigo} \lim_{r\to 0} \frac{V_{k,\sigma}(r)}{ \Phi_{\varpi_2(\rho)}(r)}=b. 
\end{equation}   
Now, $\lim_{r\to \infty}  
   V_{k,\sigma}(r)/U_{\rho,\beta}(r)=1$ is not possible as the comparison principle and \eqref{honn} would lead to $V_{k,\sigma}(r)\equiv u_b(r)$ for all $r>0$, which is a contradiction with $V_{k,\sigma}(k)=u_{b,\sigma}(k)<u_b(k)$. Hence, Theorem~\ref{rod1c} implies that $$ \lim_{r\to \infty} V_{k,\sigma}(r)=M_\sigma\quad \mbox{ for some constant } M_\sigma>0.$$  
   As before, since $\lim_{r\to 0} V_{k,\sigma}(r)=0$, we infer that $V'_{k,\sigma}(r)>0$ for every
   $r>0$. 
   
   We have thus obtained that $V_{k,\sigma}$ is a positive radial solution of \eqref{e1} in $\mathbb R^N\setminus \{0\}$, subject to 
   $  \lim_{|x|\to 0} \frac{u(x)}{ \Phi_{\varpi_2(\rho)}(|x|)}=b$ and $\lim_{|x|\to \infty} u(x)=M_\sigma$. 
   By the comparison principle, we obtain that any positive solution for such a problem must coincide with $V_{k,\sigma}$. 
   
   \vspace{0.2cm} {\bf Proof of Lemma~\ref{138} concluded.} 
   We point out that we don't know the value of the positive constant $M_\sigma$. However, using \eqref{honn}, we can choose $k>0$ large such that 
   $$ (1/2) \, U_{\rho,\beta} (r) \leq  u_b(r) \leq U_{\rho,\beta} (r)\quad \mbox{for all }r\geq k.   
   $$
   This gives that 
   $$ (1/2)\, \sigma^{\varpi_2-\beta} U_{\rho,\beta} (k) \leq V_{k,\sigma}(k)=u_{b,\sigma}(k)=\sigma^{\varpi_2} u_b(\sigma k) \leq \sigma^{\varpi_2-\beta} U_{\rho,\beta} (k).  
   $$
   We now choose an increasing sequence $\{\sigma_j\}_{j\geq 1}$ with $\lim_{j\to \infty} \sigma_j=\infty$  
   such that $\sigma_1>1$ and 
   $$\sigma_{j+1} > 2^{1/(\beta-\varpi_j)} \sigma_j\quad \mbox{for every } j\geq 1.$$
   With this choice, we establish that $V_{k,\sigma_j}(k)$ decreases as $j$ increases. More precisely, we have  
   \begin{equation} \label{yeh} V_{k,\sigma_{j+1}}(k) \leq \sigma_{j+1}^{\varpi_2-\beta} U_{\rho,\beta} (k)  < \frac{1}{2} \sigma_j^{\varpi_2-\beta} U_{\rho,\beta} (k)  \leq V_{k,\sigma_j} (k)\quad \mbox{for all } j\geq 1. 
   \end{equation} 
   On the other hand, $V_{k,\sigma_{j+1}}$ and $V_{k,\sigma_j}$ are both positive radial solutions of 
   \eqref{e1} in $\mathbb R^N\setminus \{0\}$ and share the same asymptotic behaviour near zero, in view of \eqref{rigo}. Moreover, 
   $$V'_{k,\sigma_{j}}(r)>0\quad \mbox{ for every }r>0\ \mbox{ and } j\geq 1.$$
   Therefore, by the comparison principle and the strong maximum principle, we must have 
   $$ V_{k,\sigma_{j+1}}(r)<V_{k,\sigma_{j}}(r)\quad \mbox{for all } r>0. $$   
   Consequently, $M_{\sigma_{j+1}}=\lim_{r\to \infty} V_{k,\sigma_{j+1}}(r)<\lim_{r\to \infty} V_{k,\sigma_{j}}(r)=M_{\sigma_j}$ for every $j\geq 1$.  
      
   Observe from \eqref{yeh} that $\lim_{j\to \infty} V_{k,\sigma_j}(k)=0$ using that $\varpi_2<\beta$ and $\lim_{j\to \infty} \sigma_j=\infty$. 
   
   For every $r>0$, we define $V_{k,\infty}(r)=\lim_{j\to \infty} V_{k,\sigma_j} (r)$. Then, $V_{k,\infty}$ is a non-negative radial solution of \eqref{e1} in $\mathbb R^N\setminus \{0\}$, which satisfies
   $V_{k,\infty} (k)=0$. By the strong maximum principle (see Lemma~\ref{add}), it follows that 
   $$ \lim_{j\to \infty} M_{\sigma_j} =0.$$
By relabelling $M_{\sigma_j}=c_j$ and $V_{k,\sigma_j}=u_{b,\infty,c_j}$ for $j\geq 1$, we conclude the proof of Lemma~\ref{138}. 
\end{proof}

In view of the invariance of \eqref{e1} under the transformation $T_\sigma$ ($\sigma>0$) in \eqref{ssig}, Lemma~\ref{138} gives that for every $c>0$, there exists a sequence $\{b_j\}_{j\geq 1}$ decreasing to $0$ as $j\to \infty$ such that for every $j\geq 1$, equation \eqref{e1} in $\mathbb R^N\setminus \{0\}$, subject to 
\begin{equation} \label{bolo} \lim_{|x|\to 0} \frac{u(x)}{ \Phi_{\varpi_2(\rho)}(|x|)}=b_j \quad \mbox{and}
\quad  \lim_{|x|\to \infty} u(x)=c
\end{equation}
has a unique positive solution $u_{b_j,\infty,c}$. Moreover, for every $j\geq 1$, the solution
$u_{b_j,\infty,c}$ is symmetric and $u'_{b_j,\infty,c}(r)>0$ for all $r>0$.

\begin{lemma} \label{139} Let $j\geq 1$ and $c>0$ be arbitrary. 
Then, for every $b\in (b_j,\infty)$, there exists a unique positive solution $u_{b,\infty,c}$ of 
 \eqref{e1} in $\mathbb R^N\setminus \{0\}$, subject to \eqref{bono}. 
Moreover, $u_{b,\infty,c}$ is radially symmetric and $u_{b,\infty,c}'(r)>0$ for every $r>0$. 
\end{lemma}

\begin{proof} For $j\geq 1$ fixed, let $b\in (b_j,\infty)$ be arbitrary. It is enough to construct a positive radial solution $u_{b,\infty,c}$ of 
 \eqref{e1} in $\mathbb R^N\setminus \{0\}$, subject to \eqref{bono}. Then, 
 $ u_{b,\infty,c}'(r)>0$ for every $ r>0$. Hence, by the comparison principle, 
 $u_{b,\infty,c}$ will be the only positive solution for that problem. 
 
 By Lemma~\ref{138}, there exists $c_0\in (0,c)$ such that \eqref{e1} in $\mathbb R^N\setminus \{0\}$, subject to 
 	\begin{equation} \label{bobo}  \lim_{|x|\to 0} \frac{u(x)}{ \Phi_{\varpi_2(\rho)}(|x|)}=b \quad \mbox{and}
\quad  \lim_{|x|\to \infty} u(x)=c_0 \end{equation} 
 has a unique positive solution $u_{b,\infty,c_0}$. In addition, $u_{b,\infty,c_0}$ is radially symmetric and 
 $$ u'_{b,\infty,c_0}(r)>0\quad \mbox{ for all } r>0. $$
Let $u_b$ be as in Theorem~\ref{exol1}, namely, the unique positive solution of \eqref{e1} in $\mathbb R^N\setminus \{0\}$ satisfying
\begin{equation} \label{mion} \lim_{|x|\to 0} \frac{u(x)}{ \Phi_{\varpi_2(\rho)}(|x|)}=b \quad \mbox{and}
\quad \lim_{|x|\to \infty} \frac{u(x)}{U_{\rho,\beta}(x)}=1.
\end{equation}
By Theorem~\ref{exol2}, there exists a unique positive solution $u_{\infty,c}$ of \eqref{e1} in $\mathbb R^N\setminus \{0\}$ satisfying 
\begin{equation} \label{har1}
\lim_{|x|\to 0} \frac{u(x)}{U_{\rho,\beta}(x)}=1\quad \mbox{and}\quad \lim_{|x|\to \infty} u(x)=c.
\end{equation} 
Moreover, $u_b$  and $u_{\infty,c}$ are radially symmetric and $\min\{ u_b'(r), u_{\infty,c}'(r)\}>0$ for all $r>0$.  

By the comparison principle, we have 
\begin{equation} \label{hpm} \max\{ u_{b,\infty,c_0}(r), u_{b_j,\infty,c}(r)\} \leq \min\{ u_b(r),u_{\infty,c}(r)\}\quad \mbox{for every } r>0.
\end{equation} 
By the strong maximum principle (see Lemma~\ref{adol}), the above inequality is strict for every $r>0$. 
Since $b>b_j>0$ and $c_0\in (0,c)$, from \eqref{bolo} and \eqref{bobo}, we get
$$ \lim_{r\to 0^+} \frac{u_{b,\infty,c_0}(r)}{u_{b_j,\infty,c}(r)}=\frac{b}{b_j}>1\quad \mbox{and}
\quad \lim_{r\to \infty} \frac{u_{b,\infty,c_0}(r)}{u_{b_j,\infty,c}(r)}=\frac{c_0}{c}\in (0,1).
$$
Thus, there exists $r_1>0$ such that $u_{b,\infty,c_0}(r_1)=u_{b_j,\infty,c}(r_1)$.  
Moreover, by the comparison principle, jointly with the strong maximum principle, 
it follows that $$u_{b,\infty,c_0}(r)> u_{b_j,\infty,c}(r)\  \mbox{for every }r\in (0,r_1)\ 
\mbox{and }  
u_{b,\infty,c_0}(r)< u_{b_j,\infty,c}(r)\ \mbox{for all } r>r_1.$$ 

Let $n,k\geq 1$ be large integers such that $1/n<r_1<k$. 
Our aim is to construct a positive radial solution $W_{b,k}$ of \eqref{e1} in $B_{k}(0)\setminus \{0\}$
satisfying 
\begin{equation} \label{iob}
 \lim_{|x|\to 0} \frac{u(x)}{ \Phi_{\varpi_2(\rho)}(|x|)}=b \quad \mbox{and} \quad 
 u(x)=u_{b_j,\infty,c} (|x|)\ \mbox{on } |x|=k.
\end{equation}

To this end, we consider the boundary value problem
 \begin{equation}
  \label{yau}
 \left\{ \begin{aligned}
 	& \mathbb L_{\rho,\lambda,\theta} [w]=|x|^{\theta} w^q && \mbox{in } D_{1/n,k}:=\{x\in \mathbb R^N:\ 
 	1/n<|x|<k\},&\\
 	& w(x)= u_{b,\infty,c_0}(|x|) &&\mbox{on } |x|=1/n,&\\
	& w(x)=u_{b_j,\infty,c}(|x|) && \mbox{on } |x|=k.&
 \end{aligned}
	\right.
  \end{equation}
Observe that $u_{b,\infty,c_0}$ (respectively, $u_b$) is a positive radial sub-solution (respectively, super-solution) of \eqref{yau}.
Using \eqref{hpm} and the method of sub-super-solutions, we obtain that \eqref{yau} has a minimal positive solution $w_{b,n,k}$ with respect to 
the pair $(u_{b,\infty,c_0},u_b)$ of sub-super-solutions. 
Moreover, similar to \eqref{ehh}, we infer that $w_{b,n.k}$ is radially symmetric and 
\begin{equation} \label{hara} u_{b,\infty,c_0}(r)\leq w_{b,n,k}(r)\leq w_{b,n+1,k}(r)\leq u_b(r)\quad \mbox{for every } 1/n<r<k.
\end{equation}
For every $r\in (0,k)$, we define $W_{b.k}(r)=\lim_{n\to \infty} w_{b,n,k}(r)$. 
By letting $n\to \infty$ in \eqref{hara}, we get
\begin{equation}
\label{hara2}
u_{b,\infty,c_0}(r)\leq W_{b,k}(r) \leq u_b(r)\quad \mbox{for every } r\in (0,k). 
\end{equation} 
Since $u_b$ satisfies the same asymptotic condition 
near zero as $u_{b,\infty,c_0}$ (see \eqref{bobo} and \eqref{mion}), it follows that
$W_{b,k}$ is a positive radial solution of \eqref{e1} in $B_{k}(0)\setminus \{0\}$ satisfying \eqref{iob}. 
Moreover, since $\lim_{r\to 0^+} W_{b,k}=0$, we must have $W_{b,k}'(r)>0$ for all $r\in (0,k)$. Hence, 
$W_{b,k}$ is the only positive solution of \eqref{e1} in $B_{k}(0)\setminus \{0\}$ satisfying \eqref{iob}. 

\vspace{0.2cm}
By the comparison principle and  and \eqref{hpm},
we find that 
\begin{equation} \label{har2} u_{b_j,\infty,c} (r)\leq W_{b,k}(r)\leq W_{b,k+1}(r)\leq u_{\infty,c}(r)\quad \mbox{for every } r\in (0,k].
\end{equation}
For every $r>0$, we now define $u_{b,\infty,c}(r)$ by
\begin{equation} \label{opab} u_{b,\infty,c}(r)=\lim_{k\to \infty} W_{b,k}(r).
\end{equation}
Since $u_{b_j,\infty,c}(r)$ and $u_{\infty,c}(r)$ converge to $c$ as $r\to \infty$, 
by letting $k\to \infty$ in \eqref{har2}, then $r\to \infty$, we arrive at 
$\lim_{r\to \infty} u_{b,\infty,c}(r)=c.
$ Moreover, passing to the limit $k\to \infty$ in \eqref{hara2}, we find that 
$$ \lim_{r\to 0} \frac{u_{b,\infty,c}(r)}{\Phi_{\varpi_2(\rho)}(r)}=b. 
$$
Hence, $u_{b,\infty,c}$ (given by \eqref{opab}) is a positive radial solution of \eqref{e1} in $\mathbb R^N\setminus \{0\}$, 
subject to \eqref{bono}. Such a problem cannot have any other positive 
solutions (as explained at the beginning of the proof). The proof of Lemma~\ref{139} is now complete.  
\end{proof}

Since $b_j\searrow 0$ as $j\to \infty$, from Lemma~\ref{139}, we conclude
the proof of Theorem~\ref{exol3}. \end{proof}

\appendices 
\section{Auxiliary results}  \label{Kel}

\subsection{The modified Kelvin transform}

Let \eqref{e2} hold and $\rho,\lambda,\theta\in \mathbb R$. 
Assume that $u$ is a positive solution of \eqref{e1} in $B_1(0)\setminus \{0\}$.    
For $0<|x|<1$, we define $$ v(\widetilde x)=u(x)\quad \mbox{ with } \widetilde x=x/|x|^2.$$ 
Then, by letting $\widetilde \theta=-\theta-4$, we see that $v(\widetilde x)$ is a positive solution of 
\begin{equation} \label{ela1}
 \mathbb L_{-\rho,\lambda,\tau}[v(\widetilde x)]=\Delta v(\widetilde x)-\left(N-2-2\rho\right) \frac{\widetilde x\cdot \nabla v(\widetilde x)}{|\widetilde x|^2} +
 \lambda \frac{v^\tau(\widetilde x)\,|\nabla v(\widetilde x)|^{1-\tau}}{|\widetilde x|^{1+\tau}}
 =|\widetilde x|^{\widetilde \theta} \,v^q(\widetilde x)
\end{equation}
for $ \widetilde x\in \mathbb R^N \setminus \overline{B}_1(0)$. Indeed, for every $1\leq j\leq N$, we have 
 $$\frac{\partial u}{\partial x_j}
 =|\widetilde x|^2 \frac{\partial v}{\partial \widetilde x_j}-2\, \widetilde x_j \,  \widetilde x \cdot \nabla v(\widetilde x),
 $$ which implies that $x\cdot \nabla u(x)=-\widetilde x\cdot \nabla v(\widetilde x)$ and 
 $|\nabla u(x)|=|\widetilde x|^2 |\nabla v(\widetilde x)|$, as well as
 $$ \Delta u(x)+2\left(2-N\right) \frac{x\cdot \nabla u(x)}{|x|^2} =|\widetilde x|^4 \Delta v(\widetilde x).$$ 
We define $\widetilde \beta=(\widetilde \theta+2)/(q-1)$. Then, we have 
 $$ \widetilde \beta=-\beta\quad \mbox{and}\quad f_{-\rho}(\widetilde \beta)=f_\rho(\beta).$$ 
Hence, if $f_{-\rho}(\widetilde \beta)>0$, then a positive radial solution $U_{-\rho,\widetilde \beta}(\widetilde x)$ of \eqref{ela1} in $\mathbb R^N\setminus \{0\}$ is given by  
$$ U_{-\rho,\widetilde \beta}(\widetilde x)=[f_{-\rho}(\widetilde \beta)]^{\frac{1}{q-1}} |\widetilde x|^{-\widetilde \beta}=[f_{\rho} (\beta)]^{\frac{1}{q-1}}\, | x|^{-\beta}=U_{\rho,\beta}(x).  
$$

\subsection{Strong maximum principles and comparison principle}

Using \eqref{e2}, we first show in Lemma~\ref{add} that 
the strong maximum principle applies for \eqref{e1}. 

\begin{lemma}[Strong Maximum Principle, I] \label{add}
Let $N\geq 2$  and $\omega$ be a domain with $\overline{\omega}\subset \mathbb R^N\setminus \{0\} $. Let \eqref{e2} hold and 
$\rho,\lambda,\theta\in \mathbb R$ be arbitrary. 
 Assume that 
 $u$ (respectively, $v $) is any non-negative $C^1(\omega)$ (distributional) solution of 
\begin{equation} \label{hoia} \mathbb L_{\rho,\lambda,\tau}[u]\leq  |x|^\theta u^q\quad \mbox{in } \omega\quad (\mbox{respectively, }\quad  \mathbb L_{\rho,\lambda,\tau}[v]\leq  0\quad \mbox{in } \omega).
\end{equation}
Then, $u$ (respectively, $v$) is either positive in $\omega$ or identically zero in $\omega$. 
\end{lemma}

\begin{proof} For every $x\in \mathbb R^N\setminus \{0\}$, $z\geq  0$ and $\boldsymbol{\xi}\in \mathbb R^N$, we define  
$$ \widetilde B_1(x,z,\boldsymbol{\xi})=\widetilde B_2(x,z,\boldsymbol{\xi}) -|x|^\theta z^q\quad \mbox{and }
\widetilde B_2(x,z,\boldsymbol{\xi})=-\left(N-2+2\rho\right) \frac{x\cdot \boldsymbol{\xi}}{|x|^2} +\lambda\, G_\tau(x,z,\boldsymbol{\xi}). 
$$
Since $q>1$ and $\tau\in [0,1)$, by Young's inequality yields for $j=1,2$, we find that 
$$ |\widetilde B_j(x,z,\boldsymbol \xi)|\leq |N-2+2\rho| \frac{|\boldsymbol{\xi}|}{|x|}+|\lambda | \left( 
\tau \frac{z}{|x|^2} +(1-\tau) \frac{|\boldsymbol{\xi}|}{|x|}
\right) +|x|^\theta z
$$
for every $x\in \mathbb R^N\setminus \{0\}$, $0<z\leq 1$ and $|\boldsymbol{\xi}|\leq 1$. 
Hence, for every compact $K$ in $\mathbb R^N\setminus \{0\}$, there exist non-negative constants $b_1$ and $b_2$ such that
$$ \widetilde B_j(x,z,\boldsymbol{\xi})\geq -b_1 |\boldsymbol{\xi}| -b_2 z\quad \mbox{for all } x\in K,\ 0<z\leq 1,\  |\boldsymbol{\xi}|\leq 1,
 $$ where $j=1,2$. 
 From \eqref{hoia}, $u$ (respectively, $v$) is a non-negative $C^1(\omega)$ solution of 
\begin{equation} \label{hag} \Delta u+\widetilde B_1(x,u,\nabla u)\leq 0\quad \mbox{in } \omega 
\quad (\mbox{respectively, }\quad  \Delta v+\widetilde B_2(x,v,\nabla v) \leq 0\quad \mbox{in } \omega).
\end{equation}
The conditions in the strong maximum principle of Theorem~2.5.1 in \cite{PS2007}*{p. 34} are satisfied for \eqref{hag}  in every compact $K$ in $\mathbb R^N\setminus \{0\}$. This implies that if there exists $x_0\in \omega$ such that $u(x_0)=0$ (respectively, $v(x_0)=0$), then necessarily $u\equiv 0$ in $\omega$ (respectively, $v\equiv 0$ in $\omega$). 
\end{proof}

We will use (see Theorem~\ref{urz}) the following strong maximum principle.  

\begin{lemma}[Strong Maximum Principle, II] \label{adol} Let $N\geq 2$, \eqref{e2} hold, $\rho,\lambda, \theta\in \mathbb R$ and $\omega\subset \mathbb R^N$ be a domain such that $\overline \omega\subset \mathbb R^N\setminus 
\{0\}$. Assume that $u_1,u_2\in C^1(\omega)$
are positive distributional solutions of 
\begin{equation} \label{vezi}
\mathbb L_{\rho,\lambda,\tau} [u_1]\geq 0\ \  \mbox{in }\omega\quad \mbox{and}\quad
\mathbb L_{\rho,\lambda,\tau} [u_2]\leq 0 \ \  \mbox{in }\omega.
\end{equation}
(Alternatively, instead of \eqref{vezi}, we may also assume that
\begin{equation} \label{vemmi}
\mathbb L_{\rho,\lambda,\tau} [u_1]\geq |x|^\theta u_1^q\ \  \mbox{in }\omega\quad \mbox{and}\quad
\mathbb L_{\rho,\lambda,\tau} [u_2]\leq |x|^\theta u_2^q  \ \  \mbox{in }\omega.)
\end{equation}
If $u_1\leq u_2$ in $\omega$ and $|\nabla u_1|+|\nabla u_2|>0$ in $\omega$, then either $u_1\equiv u_2$ in $\omega$ or $u_1<u_2$ in $\omega$. 
\end{lemma}

\begin{proof} We point out that we {\em cannot} conclude the proof by applying Lemma~\ref{add} to $u_2-u_1$ since 
the operator $\mathbb L_{\rho,\lambda,\tau}[\cdot]$ is nonlinear (when $\lambda\not=0$) by the assumption $\tau\in [0,1)$. 

Let $\omega'=\{x\in \omega: \ u_1(x)=u_2(x)\}$. Then, $\omega'$ is relatively closed with respect to $\omega$. 
We assume that $\omega'\not=\emptyset$ and let 
$x_0\in  \omega'$ be arbitrary. Since $\omega$ is open, there exists $\varepsilon>0$ small such that $B_\varepsilon(x_0)\subset\subset \omega$.
Recall that $u_2-u_1\geq 0$ in $\omega$ and thus
$u_2-u_1$ achieves its minimum $0$ at $x_0\in \omega$. This, jointly with the assumption that
$|\nabla u_1(x_0)|+|\nabla u_2(x_0)|>0$, 
implies that $|\nabla u_1(x_0)|=|\nabla u_2(x_0)|>0$. If necessary, we diminish $\varepsilon>0$ to ensure that for some $m>0$ small, we have
$\min_{\overline B_{\varepsilon}(x_0)}u_j\geq m$ and 
$\min_{\overline B_{\varepsilon}(x_0)} |\nabla u_j|\geq m$ for $j=1,2$.   

Let $\widehat B(x,z,\boldsymbol{\xi})$ be a scalar function that is continuously differentiable in the variables $z\in \mathbb R$ and $\boldsymbol{\xi}\in \mathbb R^N$ such that when $z\geq m/2$ and $|\boldsymbol{\xi}|\geq m/2$, we have
$$ \widehat B(x,z,\boldsymbol{\xi})=(N-2+2\rho) \frac{x\cdot \boldsymbol{\xi}}{|x|^2}-\lambda z^\tau \frac{|\boldsymbol{\xi}|^{1-\tau}}{|x|^{1+\tau}}.
$$
(If \eqref{vemmi} is assumed instead of \eqref{vezi}, then we also include $-|x|^\theta z^q$ in the definition of $\widehat B(x,z,\boldsymbol{\xi})$.)
Hence, using \eqref{vezi}, we see that $u_1$ and $u_2$ are positive $C^1(B_\varepsilon(x_0))$ distributional solutions of 
$$ \begin{aligned}  
& \Delta u_1-\widehat B(x,u_1,\nabla u_1)\geq 0\quad \mbox{in } B_\varepsilon(x_0),\\
& \Delta u_2-\widehat B(x,u_2,\nabla u_2)\leq 0\quad \mbox{in } B_\varepsilon(x_0).
\end{aligned}
$$
Since $u_1\leq u_2$ in  $B_\varepsilon(x_0)$ and $u_1(x_0)=u_2(x_0)$, by \cite{Ser}*{Theorem 1}, we get $u_1\equiv u_2$ in $B_\varepsilon(x_0)$. 
This shows that $\omega'$ is relatively open in $\omega$. So, $\omega'=\omega$, which yields that $u_1\equiv u_2$ in $\omega$. 
\end{proof}

We next recall two comparison principles that are used frequently in this paper. 

\begin{lemma}[Comparison principle, I] \label{co1}
	Let $D$ be a bounded domain in $\mathbb{R}^N$  with $N\geq 2$. Let
	$\widehat{B}(x,z,\boldsymbol{\xi}): D\times \mathbb{R}\times \mathbb{R}^N\to \mathbb{R}$ be continuous in $D\times \mathbb{R}\times \mathbb{R}^N$ and
	continuously differentiable with respect to $\boldsymbol{\xi}$ for $|\boldsymbol{\xi}|>0$ in $\mathbb{R}^N$. Assume that
	$\widehat{B}(x,z,\boldsymbol{\xi})$ is non-decreasing in $z$ for fixed $(x,\boldsymbol{\xi}) \in D\times  \mathbb{R}^N$. Let $u_1$ and $u_2$ be
	non-negative $C^1(D)\cap C(\overline{D})$ (distributional) solutions of
	\begin{equation}\label{asd}
	\left\{\begin{aligned}
	& \Delta u_1-\widehat{B}(x,u_1,\nabla u_1)\geq 0 && \text{in } D,&\\
	&  \Delta u_2-\widehat{B}(x,u_2,\nabla u_2) \leq 0 && \text{in } D.&
	\end{aligned}
	\right.
	\end{equation}
	Suppose that
	$|\nabla u_1|+|\nabla u_2|>0$ in $D$ and either $\sup_{D} |\nabla u_1|<\infty$ or $\sup_{D} |\nabla u_2|<\infty$. If $u_1\leq u_2$ on $\partial D$, then $u_1\leq u_2$ in $D$.
\end{lemma}

\begin{proof} The conclusion follows from 
\cite{PS2004}*{Theorem 10.1} by taking $\widehat{A}(x,u,\nabla u)=\nabla u$. 
\end{proof}

\begin{lemma}[Comparison principle, II] \label{co2}
	Let $D$ be a bounded domain in $\mathbb{R}^N$  with $N\geq 2$. Assume that
	$\widehat{B}(x,z,\boldsymbol{\xi}): D\times \mathbb{R}\times \mathbb{R}^N\to \mathbb{R}$ is locally Lipschitz continuous with respect to
	$\boldsymbol{\xi}$ in $D\times \mathbb{R}\times \mathbb{R}^N$ and is non-decreasing in $z$ for fixed $(x,\boldsymbol{\xi}) \in D\times  \mathbb{R}^N$.
	
	Let $u_1$ and $u_2$ be (distribution) solutions in $C^1(D)$ of \eqref{asd}. If 
	$$u_1\leq u_2+M\quad \mbox{on }\partial D,$$ where $M$ is a positive constant, then $u_1\leq u_2+M$ in $D$.
\end{lemma}

\begin{proof} We apply Corollary~3.5.2 to Theorem~3.5.1 in \cite{PS2007}. 
\end{proof}

\begin{rem} {\rm Theorem~3.5.1 in \cite{PS2007}, like Theorem~10.1 in \cite{PS2004}, is essentially 
Theorem 10.7(i) in 
\cite{GT1983} except that the functions $\widehat{A}$ and $\widehat{B}$ in \cites{PS2004,PS2007} are allowed to be singular at $\boldsymbol{\xi}=0$. This is compensated 
in \cite{PS2004}*{Theorem 10.1} by the additional condition that $|\nabla u_1|+|\nabla u_2|>0$ in $D$. }
\end{rem}

\begin{rem} \label{reka6}
{\rm Let $\rho,\lambda,\theta\in \mathbb R$, $\tau\in [0,1)$ and $q>1$. 
Let $D$ be a bounded domain in $\mathbb R^N$ ($N\geq 2$) such that 
$\overline D\subset \mathbb R^N\setminus \{0\}$. 
Assume that $u_1,u_2\in C^1(D)\cap C(\overline D)$ are positive distributional solutions of 
$\mathbb L_{\rho,\lambda,\tau} [u_1]\geq |x|^\theta u_{1}^{q} $ in $D$ and 
$\mathbb L_{\rho,\lambda,\tau} [u_2]\leq |x|^\theta u_{2}^{q} $ in $D$ 
such that $u_1\leq u_2$ on $\partial D$. 
In addition, when $\tau\in (0,1)$ and $\lambda\not=0$, we also assume that
$|\nabla u_1|+|\nabla u_2|>0$ in $D$ and either $\sup_D|\nabla u_1|<\infty$ or $\sup_D|\nabla u_2|<\infty$.   
Then, we have $u_1\leq u_2$ in $D$. This follows from
Lemma~\ref{co2} if $\tau=0$ or if $\lambda=0$. Otherwise, we apply Lemma~\ref{co1} to $u_1$ and $u_2$ if $\lambda<0$, respectively to 
 $w_{1}=\log \,(Cu_1)$ and $w_{2}=\log \,(Cu_2)$ if $\lambda>0$, where $C>0$ is a large constant so that $Cu_j>1$ on $\overline{D}$ for $j=1,2$. Indeed, if $\lambda>0$, then the conditions in Lemma~\ref{co1} will be satisfied by $w_{1}$ and $w_{2}$ since 
 $|\nabla w_1|+|\nabla w_2|>0$ in $D$, $w_1\leq w_2$ on $\partial D$ and 
 $$ \begin{aligned} 
 & 	
\Delta w_1+ |\nabla w_1|^2-(N-2+2\rho) \frac{x\cdot \nabla w_1}{|x|^2}+\lambda 
 \frac{|\nabla w_1|^{1-\tau}}{|x|^{1+\tau}}\geq C^{1-q} |x|^\theta e^{(q-1) w_1}\quad \mbox{in } D,\\
 &  \Delta w_2+ |\nabla w_2|^2-(N-2+2\rho) \frac{x\cdot \nabla  w_2}{|x|^2}+\lambda 
 \frac{|\nabla w_2|^{1-\tau}}{|x|^{1+\tau}}\leq C^{1-q} |x|^\theta e^{(q-1) w_2}\quad \mbox{in } D.
  \end{aligned}
 $$  
 Thus, we get $w_1\leq w_2$ in $D$, which yields that $u_1\leq u_2 $ in $D$.

}	
\end{rem}

\section{A priori estimates and consequences} \label{sectiune2}

Assuming \eqref{e2} and letting $\lambda^+$ be the positive part of $\lambda$, we define 
$C_0=C_0(N,q,\theta,\rho,\lambda,\tau)$ by
\begin{equation} \label{c0} C_0:=2^{\frac{|\theta|}{q-1}}\left[ \frac{16}{q-1} \left( \frac{2\left(q+1\right)}{q-1} +N+|N-2+2\rho|+\left(1-\tau\right) \lambda^+\right)+4\tau \lambda^+ \right]^{\frac{1}{q-1}}.
 \end{equation}

\begin{lemma}[{{\em A priori} estimates}] \label{api} Let \eqref{e2} hold and $\rho,\lambda,\theta\in \mathbb R$. Let $r_0>0$ satisfy $B_{2r_0}(0)\subset\subset \Omega$. Then,  letting $C_0$ be as in \eqref{c0}, 
every positive sub-solution $u$ of \eqref{e1} satisfies
\begin{equation} 
\label{robo} u(x)\leq C_0 \,|x|^{-\beta} \quad \ \ \ \mbox{for all } 0<|x|\leq r_0.
\end{equation}
\end{lemma}

\begin{proof} 
We adapt the ideas in \cite{Cmem}*{Lemma 4.1} (see also \cite{MF1}*{Lemma 4.2}), where $\rho=(2-N)/2$ and $\tau=1$. 
Fix $x_0\in \mathbb{R}^N$ with $0<|x_0|\leq r_0$. 
	For every $x\in B_{|x_0|/2}(x_0)$, we define 
	\begin{equation}\label{pp}
	\mathcal{P}(x):=C_0\,|x_0|^{-\beta} \left [\zeta(x)\right]^{-\frac{2}{q-1}},\quad\mbox{where } 
	\zeta(x):=1-\left (\frac{2\,|x-x_0|}{|x_0|}\right )^2. 
	\end{equation}
Let $u$ be any positive sub-solution of \eqref{e1}. We would like to show that 	
	\begin{equation} \label{he}
	u(x)\leq \mathcal{P}(x)\quad\mbox{for every } x\in B_{|x_0|/2}(x_0).
	\end{equation}	
Then, since $C_0$ is independent of $x_0$, we obtain \eqref{robo} by taking $x=x_0$ in \eqref{he}.

We remark that in proving \eqref{he}, we need to make some adjustments compared with \cites{Cmem,MF1}. For the equations studied in the latter, $\mathcal P$ was a super-solution in $B_{|x_0|/2} (x_0)$
for a suitable constant $C_0$ and then \eqref{he} was derived from 
 the comparison principle in Lemma~\ref{co2} (using that $\mathcal P= \infty$ on $\partial B_{|x_0|/2}(x_0)$). 
However, because of the introduction of 
$G_\tau(x,u,\nabla u)$, we would not be able to apply Lemma~\ref{co2} directly to \eqref{e1} when $0<\tau<1$. Instead, we would need to apply Lemma~\ref{co1}. This is not feasible since 
$|\nabla \mathcal P(x_0)|=0$ and we have no information on the gradient of $u$ at $x_0$. Hence, we cannot verify the condition $|\nabla \mathcal P(x)|+|\nabla u(x)|>0$ at $x=x_0$. 

We overcome the above issue by applying Lemma~\ref{co2} to the modified equation 
\begin{equation} \label{me} 
\mathbb L_\rho [v]+\lambda^+ \left(\tau\frac{v}{|x|^2} + \left(1-\tau\right) \frac{|\nabla v|}{|x|} \right)= |x|^\theta v^q \quad \mbox{in } B_{|x_0|/2}(x_0).
\end{equation}
We observe that $u$ is a sub-solution of \eqref{me}, using Young's inequality when $\lambda>0$. 

We next show that 
$\mathcal P$ is a super-solution of \eqref{me}, namely, 
 \begin{equation} \label{xz2}
	\mathbb L_\rho [\mathcal P] +\lambda^+ \left(\tau\frac{\mathcal P(x)}{|x|^2} + (1-\tau) \frac{|\nabla \mathcal P(x)|}{|x|} \right)\leq |x|^\theta \mathcal P^q(x)
	\quad \mbox{in } B_{|x_0|/2}(x_0). 
	\end{equation}
Indeed, for every $x\in B_{|x_0|/2}(x_0)$, we have
	\begin{equation} \label{bv} 
	\begin{aligned}
	&\tau\frac{\mathcal P}{|x|^2} + (1-\tau) \frac{|\nabla \mathcal P|}{|x|} \leq 	
	C_0 \left( 4\tau+ (1-\tau)\frac{16}{q-1}\right) \frac{[\zeta(x)]^{-\frac{2q}{q-1}}}{|x_0|^{\beta+2} },\\
	& \Delta \mathcal P(x)=\frac{16\,C_0}{q-1} \left[
	\frac{8(q+1)}{q-1}\,\frac{|x-x_0|^2}{|x_0|^2} +N\zeta(x)\right]\frac{[\zeta(x)]^{-\frac{2q}{q-1}}}{|x_0|^{\beta+2}},
	\end{aligned}
	\end{equation}
	which yields that the left-hand side of \eqref{xz2} is bounded from above by 
	\begin{equation} \label{b1}
  C_0\left[  \frac{16}{q-1} \left(
	\frac{2(q+1)}{q-1}+N+|N-2+2\rho| +(1-\tau) \lambda^+\right) +4\tau \lambda^+\right]\frac{[\zeta(x)]^{-\frac{2q}{q-1}}}{|x_0|^{\beta+2}}.  
	\end{equation}
On the other hand, since $1/2\leq |x|/|x_0|<2$ for each $x\in B_{|x_0|/2}(x_0)$, we obtain that 	
	\begin{equation} \label{b2}
	|x|^\theta \mathcal P^q(x)=
	 (C_0)^q \left(\frac{|x|}{|x_0|}\right)^\theta \frac{[\zeta(x)]^{-\frac{2q}{q-1}}}{|x_0|^{\beta+2} }\geq 2^{-|\theta|} 
	 (C_0)^q \frac{[\zeta(x)]^{-\frac{2q}{q-1}}}{|x_0|^{\beta+2} }.
	\end{equation}	
	 Using the definition of $C_0$ in \eqref{c0}, we see that the quantity in \eqref{b1} is equal to the right-hand side of the inequality in \eqref{b2}. 
	 Hence, we conclude \eqref{xz2}. 
	 	 
		 \vspace{0.2cm}
		 {\em Proof of \eqref{he}.} 
		 If $\lambda\leq 0$, then we immediately get \eqref{he} by applying Lemma~\ref{co2} in relation to $u$ and $\mathcal P$ being, respectively, a positive sub-solution and super-solution of \eqref{me} such that 
		 $\mathcal P(x)\to \infty $ as $|x-x_0|\nearrow |x_0|/2 $.
		 However, when $\lambda>0$ we need a small trick to apply the comparison principle in Lemma~\ref{co2} given that the function 
		 $$ t\longmapsto |x|^\theta t^q- \lambda \tau\frac{t}{|x|^2}$$
		fails to be non-decreasing in $t$ on $[0,\infty)$ for fixed $x\in B_{|x_0|/2}(x_0)$.
		 We let $C_*>0$ be a positive large constant, depending on $u$ and $x_0$, such that
		 we have $C_* u>1$ and $C_* \mathcal P>1$ in $B_{|x_0|/2}(x_0)$. Then, $\log \,(C_* u)$ and $\log \,(C_* \mathcal P)$ are, respectively, a positive sub-solution and super-solution of  
		 $$ \mathbb L_\rho [w] +|\nabla w|^2+\lambda \frac{\tau}{|x|^2} +\lambda \left(1-\tau\right) \frac{|\nabla w|}{|x|}=(C_*)^{1-q} |x|^\theta e^{\left(q-1\right)w} \quad \mbox{in }   B_{|x_0|/2}(x_0).
		 $$
		 The assumptions in Lemma~\ref{co2} are now satisfied so that $\log (C_* u)\leq \log (C_*\mathcal P)$ in $B_{|x_0|/2}(x_0)$. This yields \eqref{he} also when $\lambda>0$,  
completing the proof of Lemma~\ref{api}. 
	\end{proof}
	
	Since the constant $C_0$ in \eqref{c0} is independent of the domain, 
	from Lemma~\ref{api} we obtain global {\em a priori} estimates for all positive solutions of \eqref{e1}. 

\begin{cor}[Global {\em a priori} estimates] \label{lem1} Let $\Omega=\mathbb R^N$ and \eqref{e2} hold. Then, given $C_0$ as in \eqref{c0}, every positive sub-solution $u$ of \eqref{e1} satisfies
	\begin{equation}\label{in123}
	u(x)\leq C_0\, |x|^{-\beta} 
	\quad\mbox{for all }x\in \mathbb R^N\setminus \{0\}. 
	\end{equation}
\end{cor} 

\begin{proposition} \label{aurr} 
Let \eqref{e2} hold and $\rho,\lambda,\theta\in \mathbb R$. 
 Then, there exist constants $d_j>0$ with $j\in \{0,1,2,3\}$ and $\alpha\in (0,1)$, all of which depending only on $N,q,\theta,\rho,\lambda$ and $\tau$, 
such that for every positive solution $u$ of \eqref{e1} with $B_{4r_0}(0)\subset\subset \Omega$ and every 
$x,x'\in \mathbb R^N$ with $0<|x|\leq |x'|<r_0$, 
\begin{eqnarray} 
& \displaystyle |\nabla u(x)|\leq d_0\,|x|^{-\beta-1},  \label{apbb}\\
& \displaystyle \max_{|x|=r} u(x)\leq d_1 \min_{|x|=r} u(x)\quad \mbox{for all } r\in (0,r_0) \quad \mbox{[spherical Harnack inequality]} \label{omas}\\
& \displaystyle |\nabla u(x)|\leq d_2\,\frac{u(x)}{|x|},  \label{noi2}\\
&	\displaystyle|\nabla u(x)-\nabla u(x')|\leq d_3\, \frac{u(x)+u(x')}{|x|^{1+\alpha}}|x-x'|^\alpha.   \label{noi3} 
\end{eqnarray}

\end{proposition}

\begin{proof} Let $r_0>0$ be such that $B_{4r_0}(0)\subset\subset \Omega$. 
We prove that 
there exists a constant $d_0>0$, depending only on $N,q,\theta,\rho,\lambda$ and $\tau$, such that every positive solution $u$ of \eqref{e1} 
satisfies \eqref{apbb}.

We follow the argument in \cite{CC2015}*{Lemma 3.8}, taking $g(x)=C_0\,|x|^{-\beta}$. 

Fix $x_0\in \mathbb R^N$ such that $0<|x_0|< r_0<1/2$. 
We define $v_{x_0}:B_1(0)\to (0,\infty)$ by 
\begin{equation} \label{Fioo}
v_{x_0}(y):=|x_0|^{\beta} u(x_0+ |x_0|\, y/2)  \quad \mbox{for every } y\in B_1(0).
\end{equation}
It follows that $v_{x_0}$ satisfies the following equation 
\begin{equation} \label{foxi} -\Delta v(y)+\widetilde B(y,v(y),\nabla v(y))=0\quad \mbox{for every } y\in B_1(0),
\end{equation} where $\widetilde B(y,z,\boldsymbol{\xi} )$ is defined for $y\in B_1(0)$, $z\geq 0$ and $\boldsymbol{\xi}\in \mathbb R^N$ by 
\begin{equation} \label{romy} \begin{aligned}
 \widetilde B(y,z,\boldsymbol{\xi})= &\frac{ (N-2+2\rho)\,|x_0| \left(x_0+\frac{1}{2} |x_0|\, y\right)\cdot \boldsymbol{\xi}}{ 2\left| x_0+\frac{1}{2} |x_0| \,y\right|^2}-
 \frac{ \lambda\, z^\tau |\boldsymbol{\xi}|^{1-\tau}\,|x_0|^{\tau+1} }{2^{\tau+1}\left|x_0+\frac{1}{2} |x_0|\, y\right|^{\tau+1}} \\
 &+\frac{ \left|x_0+\frac{1}{2} |x_0| \,y\right|^{\theta} z^q }{4\,|x_0|^\theta}.
\end{aligned}
\end{equation}
Since $|x_0|/2\leq |x_0+|x_0|\,y/2|\leq 3|x_0|/2$ for all $y\in B_1(0)$, 
from Lemma~\ref{api} and \eqref{Fioo}, we get
\begin{equation} \label{nia}
 v_{x_0}(y)\leq C_0 \left(\frac{ |x_0|}{ |x_0+\frac{1}{2} |x_0|\, y|}\right)^{\beta}\leq 2^{|\beta|}\,C_0.
\end{equation} 
For all $y\in B_1(0)$ and $\boldsymbol{\xi}\in \mathbb R^N$, we have
\begin{equation} \label{nia2} 
|\widetilde B(y,v_{x_0}(y),\boldsymbol{\xi})|\leq  |N-2+2\rho| |\boldsymbol{\xi}|+|\lambda| |\boldsymbol{\xi}|^{1-\tau} 
(v_{x_0}(y))^\tau 
 + 2^{|\theta|-2}  (v_{x_0}(y))^{q}.
 \end{equation}
Hence, from \eqref{nia} and \eqref{nia2}, 
there exists a positive constant $A_1$, depending only on $N,q,\theta,\rho,\lambda$ and $\tau$
such that 
\begin{equation} \label{foii}
\left |\widetilde B(y,v_{x_0}(y),\boldsymbol{\xi}) \right|\leq A_1(1+|\boldsymbol{\xi}|)^2\quad \mbox{for all } y\in B_1(0)\  \mbox{and }  \boldsymbol{\xi}\in \mathbb R^N.
\end{equation} Using \eqref{nia} and applying Theorem 1 in Tolksdorf \cite{T1984}, 
we obtain constants $\alpha\in (0,1)$ and $C>0$, both depending only on  $N,q,\theta,\rho,\lambda$ and $\tau$ 
such that 
$\|\nabla v_{x_0}\|_{C^{0,\alpha}(\overline {B_{1/2}(0)})}\leq C$. In  particular, we have $|\nabla v_{x_{0}}(0)|\leq C$, which proves the inequality in \eqref{apbb} with $d_0=2C$.

\vspace{0.2cm} {\em Proof of \eqref{omas}.} 
By Harnack's inequality (see \cite{Trudinger}*{Theorem 1.1}) applied to $v_{x_0}$, we obtain a positive constant $c_1$, depending only on 
$N,q,\theta,\rho,\lambda$ and $\tau$ such that 
$$ \sup_{B_{1/3}(0)} v_{x_0}\leq c_1 \inf_{B_{1/3}(0)} v_{x_0}\quad \mbox{or, equivalently, }  \sup_{B_{|x_0|/6}(x_0)} u\leq c_1 \inf_{B_{|x_0|/6}(x_0)} u.
$$
By a standard covering argument (see \cite{FV}), we conclude \eqref{omas} with $d_1=c_1^{10}$. 

Moreover, using \eqref{omas} and reasoning exactly as in \cite{CC2015}*{Corollary~3.10}, we infer that for every $0<a<b\leq 3/2$,
there exists a constant $C_{a,b}>0$, depending only on $N,q,\theta,\rho,\lambda$ and $\tau$ such that 
\begin{equation} \label{ha1jj} \max_{ar\leq |x|\leq br} u(x)\leq C_{a,b} \min_{ar\leq |x|\leq br} u(x) \quad \mbox{for every } r\in (0,r_0). 
\end{equation}
We emphasise that $C_{a,b}$ here is independent of $r_0$. 

\vspace{0.2cm}
{\em Proof of \eqref{noi2}.} We adapt the definition of $v_{x_0}$ in \eqref{Fioo} replacing $|x_0|^\beta$ by $1/u(x_0)$, i.e.,
\begin{equation} \label{vibb} v_{x_0}(y)=\frac{u(x_0+|x_0|\,y/2 )}{u(x_0)}\quad \mbox{for every } y\in B_1(0).
\end{equation}
In light of \eqref{ha1jj} with $a=1/2$, $b=3/2$ and $r=|x_0|$, we obtain a constant $A_0>0$, depending only on $N,q,\theta,\rho,\lambda$ and $\tau$ such that
$$ \| v_{x_0}\|_{L^\infty(B_1(0))}\leq A_0.
$$
Observe that $v_{x_0}$ in \eqref{vibb} satisfies \eqref{foxi} with the definition of $\widetilde B(y,z,\boldsymbol{\xi})$ slightly adjusted. More precisely,
the last term in the right-hand side of \eqref{romy} must be here multiplied by $ |x_0|^{2+\theta}u^{q-1}(x_0)$. 
Because of Lemma~\ref{api}, we have $ |x_0|^{2+\theta}u^{q-1}(x_0)\leq C_0^{q-1}$ and hence, we recover \eqref{foii}, where $A_1>0$ is a constant
depending only on $N,q,\theta,\rho,\lambda$ and $\tau$. 
With the same argument as for \eqref{apbb}, we conclude that $|\nabla v_{x_0}(0)|\leq A_2$ for a constant $A_2>0$, depending only on $N,q,\theta,\rho,\lambda$ and $\tau$. This completes the proof of \eqref{noi2}. 

\vspace{0.2cm}
{\em Proof of \eqref{noi3}.} We can now modify a standard argument in \cite{FV}*{Lemma 1.1} (see also \cite{CD2010}). 
For every fixed $\mu \in (0,r_0/3)$, we define 
$$ \Psi_{\mu}(\boldsymbol{\xi}):=\frac{u(\mu\, \boldsymbol{\xi})}{\max_{\partial B_\mu(0)} u}\quad \mbox{for all } \boldsymbol \xi\in \Gamma:=\{y\in \mathbb R^N:\ 1<|y|<4\}. 
$$
Then, $\Psi_{\mu}\in L^\infty(\Gamma)\cap H^1(\Gamma)$ is a weak solution of 
$$ -\Delta \Psi_{\mu}+ B_{\mu}=0\quad \mbox{in } \Gamma,
$$ where $B_{\mu}(\boldsymbol{\xi})$ is defined for $\boldsymbol{\xi}\in \Gamma$ by
$$ B_{\mu}(\boldsymbol{\xi})=\mu^{\theta+2} (\max_{\partial B_\mu(0)} u)^{q-1} |\boldsymbol{\xi}|^\theta \Psi_{\mu}^q(\boldsymbol{\xi}) 
+(N-2+2\rho) \frac{\boldsymbol{\xi} \cdot \nabla \Psi_{\mu}(\boldsymbol{\xi})}{|\boldsymbol{\xi}|^2}-\lambda \,\frac{ \Psi^\tau_{\mu}(\boldsymbol{\xi})
|\nabla \Psi_{\mu}(\boldsymbol \xi)|^{1-\tau}}{|\boldsymbol{\xi}|^{1+\tau}}.
$$
In view of \eqref{noi2} and \eqref{ha1jj}, it is easy to see that 
\begin{equation} \label{kio} \|\Psi_{\mu}\|_{L^\infty(\Gamma)}\leq C_1\quad \mbox{and}\quad  \|B_{\mu}\|_{L^\infty(\Gamma)}\leq C_2,\end{equation} where 
$C_1$ and $C_2$ are positive constants that depends only on $N,q,\theta,\rho,\lambda$ and $\tau$. 

Hence, using \eqref{noi2} and \eqref{kio},  by the $C^{1,\alpha}$-regularity result of Tolksdorf \cite{T1984}, there exist constants $\alpha\in (0,1)$ and $\widetilde C>0$, both 
depending only on $N,q,\theta,\rho,\lambda$ and $\tau$,
such that 
$$ \|\nabla \Psi_{\mu}\|_{C^{0,\alpha}(\Gamma^*)}\leq \widetilde C,\quad \mbox{where } \Gamma^*=\{y\in \mathbb R^N: \ 2<|y|<3.75\}. 
$$

For $x\in B_{r_0}(0)\setminus\{0\}$, there exists $\mu\in (0,r_0/3)$ such that $2\mu \leq |x|\leq 3\mu$.

Let $x,x'\in B_{r_0}(0)\setminus\{0\}$ be such that $|x|\leq |x'|$. 

If $|x|\leq |x'|< 5|x|/4$, then for $\boldsymbol{\xi}=x/\mu$ and 
$\boldsymbol{\xi}'=x'/\mu$, by using \eqref{ha1jj}, 
we find that
$$ |\nabla u(x)-\nabla u(x')|=\frac{\max_{\partial B_\mu(0)} u}{\mu} |\nabla \Psi_\mu(\boldsymbol{\xi})-\nabla \Psi_\mu(\boldsymbol{\xi}')|\leq \widetilde C
\frac{\max_{\partial B_\mu(0)} u}{\mu^{\alpha+1}}\,|x-x'|^\alpha\leq  \frac{C_3\,u(x)}{|x|^{\alpha+1}} |x-x'|^\alpha
$$ for some constant $C_3>0$, depending only on $N,q,\theta,\rho,\lambda,\tau,\zeta,\eta$ and $d_1$. 

If $|x'|\geq 5|x|/4$, then $|x-x'|^\alpha\geq (|x'|-|x|)^\alpha\geq (|x|/4)^\alpha$ so that from \eqref{noi2}, we get
$$  |\nabla u(x)-\nabla u(x')|\leq d_2\left(\frac{u(x)}{|x|}+\frac{u(x')}{|x'|}\right)\leq 
4 d_2\, \frac{u(x)+u(x')}{|x|^{\alpha+1}} |x-x'|^\alpha.
$$
This concludes the proof of \eqref{noi3} and of Proposition~\ref{aurr}.  
\end{proof}

If $\beta<0$, then Proposition~\ref{aurr} implies that the origin is a removable singularity for every positive solution $u$ of \eqref{e1}. More precisely, we have. 

\begin{cor}[Removable singularities] \label{contt}
Let \eqref{e2} hold, $\rho,\lambda\in \mathbb R$ 
and $\theta<-2$. Then, every positive solution $u$ of \eqref{e1} satisfies
$\lim_{|x|\to 0} u(x)=0$, $u\in H^1_{\rm loc}(\Omega)$ and 
\begin{equation} \label{e1aa}
 \mathbb L_{\rho,\lambda,\tau}[u]=
\Delta u-\left(N-2+2\rho\right) \frac{x\cdot \nabla u}{|x|^2} +\lambda \,G_\tau(x,u,\nabla u)=
|x|^\theta \,u^q\quad \mbox{in  }\mathcal D'(\Omega).
\end{equation}
\end{cor}

\begin{proof} If $u$ is a positive solution of \eqref{e1}, then by \eqref{robo}, we get 
$\lim_{|x|\to 0} u(x)=0$ since $\beta<0$. Hence, by  defining $u(0)=0$, we have $u\in C(\Omega)$. Now, $u\in C^1(\Omega\setminus \{0\})$ and by Proposition~\ref{aurr}, we have $|\nabla u(x)|^2\leq C_1^2 |x|^{-2\beta-2}\in L^1_{\rm loc}(\Omega)$ using that 
$-2\beta-2>-2\geq -N$. Consequently, $u\in H^1_{\rm loc}(\Omega)$. 
It remains to show \eqref{e1aa}. 
Let $\varphi \in C^1_c(\Omega)$ be arbitrary. We prove that \begin{equation}\label{GOALLLLL}
\int_\Omega \nabla u\cdot \nabla \varphi\,dx +(N-2+2\rho) \int_\Omega \frac{x\cdot \nabla u}{|x|^2}\varphi\,dx
-\lambda\int_\Omega G_\tau(x,u,\nabla u)\,\varphi\,dx+\int_\Omega |x|^{\theta}u^q\varphi\,dx=0.  
\end{equation}

{\em Proof of \eqref{GOALLLLL}.} 
    For every $\varepsilon > 0$ small, let $w_\varepsilon$ be a non-decreasing and smooth function on $[0,\infty)$ such that $w_\varepsilon (r)\in (0,1)$ for $r \in (\varepsilon, 2\varepsilon)$, $ w_\varepsilon(r)=1 $ for  $r \geq 2\varepsilon$ and $w_\varepsilon(r)=0$ for $r \in [ 0,\varepsilon]$. 

Since $\varphi w_\varepsilon\in C^1_c(\Omega\setminus \{0\}$, by using 
$\varphi w_\varepsilon$ as a test function for \eqref{e1}, we get 
\begin{equation}\label{check1}
\begin{aligned}
\int_\Omega \nabla u\cdot \nabla(\varphi w_\varepsilon)\,dx &+(N-2+2\rho) \int_\Omega \frac{x\cdot \nabla u}{|x|^2}\varphi w_\varepsilon\,dx\\
& -\lambda\int_\Omega G_\tau(x,u,\nabla u)\,\varphi w_\varepsilon\,dx+\int_\Omega |x|^{\theta}u^q\varphi w_\varepsilon\,dx=0.
\end{aligned} \end{equation}

As 
$\theta-\beta q=-\beta-2>-N $, we have $|x|^{-\beta-2}\in L^1_{\rm loc}(\mathbb R^N)$. 
So, by Proposition~\ref{aurr} and 
$\varphi\in C^1_c(\Omega)$, we infer that $\left(x\cdot \nabla u\right)\varphi/|x|^2$, $G_\tau(x,u,\nabla u) \varphi$, $|x|^{\theta}u^q\varphi $ and $\nabla u\cdot \nabla \varphi$ are in $L^1(\Omega)$.

Since $w_\varepsilon$ is radially symmetric ($w(x)=w(r)$ with $r=|x|$), we see that
\begin{equation} \label{c11}
    \int_\Omega \nabla u\cdot \nabla(\varphi w_\varepsilon)\,dx = 
    \int_\Omega w_\varepsilon \nabla u\cdot \nabla\varphi \,dx + J_\varepsilon,\quad J_\varepsilon:=
     \int_\Omega \varphi(x)\, w_\varepsilon'(|x|)  \frac{\nabla u(x)\cdot x}{|x|}  \,dx.
\end{equation}

We next show that $\lim_{\varepsilon\to 0} J_\varepsilon=0$. Indeed, since $w_\varepsilon'(r)=0$ for every $r\in (0,\varepsilon]\cup [2\varepsilon,\infty)$, by Proposition~\ref{aurr}, there exist positive constants $C_1,C_2$ (independent of $\varepsilon$) such that
\begin{equation} \label{hy1} \begin{aligned}
    \left| J_\varepsilon \right| &\leq C_1\int_{\{\varepsilon<|x|<2\varepsilon\}} \,|x|^{-\beta-1} w_\varepsilon'(|x|) \,dx
    = C_1\int_{\varepsilon}^{2\varepsilon} \int_{\partial B_\eta(0)}  |x|^{-\beta-1} w_\varepsilon'(|x|) \,dS(x)\,d\eta  \\
    &\leq C_2 \int_\varepsilon^{2\varepsilon} \eta^{N-\beta-2} w_\varepsilon'(\eta)\,d\eta \leq C_2\eta^{N-\beta-2} w_\varepsilon(\eta)\Big|^{2\varepsilon}_\varepsilon
   =C_2(2\varepsilon)^{-\beta + N -2}\to 0\ \mbox{as } \varepsilon\to 0.
\end{aligned} \end{equation}
Thus, in view of \eqref{c11} and \eqref{hy1}, by letting $\varepsilon\to 0$ in \eqref{check1}, 
we arrive at \eqref{GOALLLLL}.  
\end{proof}

We next give consequences of the spherical Harnack-type inequality in Proposition~\ref{aurr}. 

\begin{cor}  \label{peree} Let \eqref{e2} hold and $\rho,\lambda,\theta\in \mathbb R$. Let $u$ be any positive solution of \eqref{e1}.  
\begin{itemize}
\item[(a)] If $\limsup_{|x|\to 0} u(x)>0$, then $\liminf_{|x|\to 0} u(x)>0$. 
\item[(b)] If $\limsup_{|x|\to 0} u(x)=\infty$, then 
$\lim_{|x|\to 0} u(x)=\infty$. 
\item[(c)] Let $w$ be a positive radial sub-solution of \eqref{e1} in $B_{r_0}(0)\setminus \{0\}$ for some small $r_0>0$. When $\tau\not=0$, we also assume that $|\nabla u(x)|+|w'(|x|)|>0$ for all $0<|x|<r_0$. Then, 
$$ \limsup_{|x|\to 0} \frac{u(x)}{w(|x|)}>0\quad \mbox{implies that } \liminf_{|x|\to 0} \frac{u(x)}{w(|x|)}>0. 
$$
\end{itemize}	
\end{cor}

\begin{proof} 
	(a) Assume that $\limsup_{|x|\to 0} u(x)>0$. 	 
	Let $\{x_n\}_{n\geq 1}$ be a sequence in $\mathbb R^N\setminus \{0\}$ with $\{|x_n|\}_{n\geq 1}$ decreasing to $0$ as $n\to \infty$ and $\lim_{n\to \infty} u(x_n)=\liminf_{|x|\to 0} u(x) $. 
	 
	 We reason by contradiction and suppose that $\liminf_{|x|\to 0} u(x)=0$. 
	 Let $r_0>0$ be small such that $B_{4 r_0}(0)\subset\subset \Omega$. Let  $n_0\geq 1$ be large enough so that 
	 $0<|x_{n+1}|<|x_n|<r_0$ for every $n\geq n_0$. Then, by the spherical Harnack-type inequality in Proposition~\ref{aurr}, we have   
\begin{equation} \label{jip} u(x_n)\geq \min_{|x|=|x_n|} u(x) \geq \frac{1}{d_1} \max_{|x|=|x_n|} u(x)\quad \mbox{for all } n\geq n_0.
\end{equation}
For each $n\geq 1$, we define $D_n=\{x\in \mathbb R^N:\ |x_{n+1}|<|x|<|x_n| \}$. Since $\lim_{n\to \infty} \max_{\partial D_n} u(x)= 0$ and $\limsup_{|x|\to 0} u(x)>0$, we infer that 
for large enough $n$, the maximum of 
$u$ over $\overline{D_n}$ is reached at some point $x_*(n)\in D_n$. It follows that $\Delta u(x_*(n))\leq 0$ and $\nabla u(x_*(n))=0$, implying that $0\geq \mathbb L_{\rho,\lambda,\tau}[u(x_*(n))]=|x_*(n)|^\theta u(x_*(n))>0$, which is a contradiction.

	(b) Assume that $\limsup_{|x|\to 0} u(x)=\infty$. Suppose by contradiction that $\liminf_{|x|\to 0} u(x)<\infty$. We take $\{x_n\}_{n\geq 1}$ and $D_n$ as in part (a). From \eqref{jip}, we get that 
	$\limsup_{n\to \infty} \max_{\partial D_n} u(x)<\infty$. Consequently, for large enough $n$, there exists 
	$x_*(n)\in D_n$ such that 
	$\max_{\overline D_n} u(x)=u(x_*(n))$, leading to a contradiction as in (a). 
	
	(c) Assume that $\limsup_{|x|\to 0} u(x)/w(|x|)=L>0$. We diminish $r_0>0$ to ensure that $B_{4 r_0}(0)\subset \subset \Omega$. 
	We choose a sequence  
$\{x_n\}_{n\geq 1}$ in $\mathbb R^N$ such that $\{|x_n|\}_{n\geq 1}$ decreases to $0$ with $0<|x_n|\leq r_0$ for all 
$n\geq 1$ and $\lim_{n\to \infty} u(x_n)/w(|x_n|)=L$.  Then, for every $\delta\in (0,L)$, 
there exists $n_1=n_1(\delta)\geq 1$ such that 
$$ u(x_n)\geq  \delta \,w(|x_n|)\quad \mbox{for all } n\geq n_1.$$
By \eqref{omas} in Proposition~\ref{aurr}, we find that 
$$ \min_{|x|=|x_n|} u(x)\geq (1/d_1) \max_{|x|=|x_n|} u(x)\geq (1/d_1)\,
u(x_n)\geq (\delta/d_1)\, w(|x_n|) \quad \mbox{for all } n\geq n_1. 
$$
Fix $\delta\in (0, \min\{L, d_1\})$. Observe that for every $\mu\in (0,1)$, we have that $\mu \,w$ is also a positive radial sub-solution of 
\eqref{e1} in $B_{r_0}(0)\setminus \{0\}$. Then, 
for each $n>n_1$, 
by the comparison principle, we get 
$u\geq ( \delta/d_1)\, w$ on $\mathcal D_n:=\{x\in \mathbb R^N:\ |x_n|<x<|x_{n_1}|\}$. By letting $n\to \infty$, we see that
\begin{equation} \label{eza1} u(x)\geq ( \delta/d_1)\, w(|x|)\quad \mbox{for every } 0<x<|x_{n_1}|.
\end{equation}
This proves that $\liminf_{|x|\to 0} u(x)/w(|x|)>0$.
\end{proof}

\begin{cor} \label{pere} Let \eqref{e2} hold and $\rho,\lambda,\theta\in \mathbb R$. Let $u$ be any positive solution of \eqref{e1}. 
Assume that $S$ is a positive $C^1$-function on $(0,r_0)$ for some $r_0>0$ small such that 
\begin{equation}  
 \mathbb L_{\rho,\lambda,\tau} [S(|x|)]\leq 0  \quad 
\mbox{and}\quad 
 |S'(|x|)|+|\nabla u(x)|>0\quad \mbox{ for all }|x|\in (0,r_0).
 \label{mar12}
\end{equation} 
We have the following: 
\begin{itemize}
\item[(a)] If $\liminf_{|x|\to 0}  u(x)/S(|x|)=0$, then $\lim_{|x|\to 0} u(x)/S(|x|)=0$; 
\item[(b)] If $\limsup_{|x|\to 0} u(x)/S(|x|)=\infty$, then $\lim_{|x|\to 0} u(x)/S(|x|)=\infty$. 
\end{itemize} 
\end{cor}

\begin{proof}
Without loss of generality, by diminishing $r_0>0$, 
we can assume that $B_{4 r_0}(0)\subset \subset \Omega$. Let $u$ be any positive solution of \eqref{e1}. 
We prove the assertions of (a) and (b) by adapting the method in \cite{Cmem}*{Corollary 4.5}. 
Let $\liminf_{|x|\to 0} u(x)/S(|x|):=\ell$. We choose a sequence  
$\{x_n\}_{n\geq 1}$ in $\mathbb R^N$ such that $\{|x_n|\}_{n\geq 1}$ decreases to $0$ with $0<|x_n|\leq r_0$ for all 
$n\geq 1$ and $$\lim_{n\to \infty} u(x_n)/S(|x_n|)=\ell.$$ 
Assume that $\ell\not=\infty$. Then, for every $\delta>\ell$, 
there exists $n_1=n_1(\delta)\geq 1$ such that 
\begin{equation} \label{niv1} u(x_n)\leq \delta S(|x_n|)\quad \mbox{for all } n\geq n_1.\end{equation}
Let $d_1>0$ be given by Proposition~\ref{aurr}. Then, we obtain that 
$$ \max_{|x|=|x_n|} u(x)\leq d_1 u(x_n)\leq d_1 \delta S(|x_n|)\quad \mbox{for all } n\geq n_1. 
$$ 
By our assumption of $S$, we find that $\mathbb L_{\rho,\lambda,\tau}  [d_1 \delta S]=d_1 \delta \,\mathbb L_{\rho,\lambda,\tau}  [ S]\leq 0$ in 
$B_{r_0}(0)\setminus \{0\}$. Hence, for each $n>n_1$, 
by the comparison principle on $\mathcal D_n:=\{x\in \mathbb R^N:\ |x_n|<x<|x_{n_1}|\}$, we get $u\leq d_1 \delta S$ on $\mathcal D_n$. As $|x_n|\to 0$ as $n\to \infty$, we arrive at 
\begin{equation} \label{erz1} u(x)\leq d_1 \delta S(|x|)\quad \mbox{for every } 0<x<|x_{n_1}|.
\end{equation}
Hence, if $0\leq \ell<\infty$, then $\limsup_{|x|\to 0} u(x)/S(|x|)\leq d_1\delta$ for every $\delta\in (\ell,\infty)$.

(a) If $\ell=0$, then by taking arbitrary $\delta>0$, we get $\lim_{|x|\to 0} u(x)/S(|x|)=0$.

(b) The above argument shows that $\ell<\infty$ implies that $\limsup_{|x|\to 0} u(x)/S(|x|)<\infty$. 

This completes the proof of Corollary~\ref{pere}. 
 \end{proof}

\begin{lemma} \label{mar11}
Let \eqref{e2} hold and $\rho,\lambda,\theta\in \mathbb R$. Let $u$ be any positive solution of \eqref{e1}. 
Assume that for $r_0>0$ small enough, $S$ is a positive $C^1$-function on $(0,r_0)$ such that \eqref{mar12} holds. 
   Then, $$\limsup_{|x|\to 0} \frac{u(x)}{S(|x|)}=\mu \in (0,\infty)
   \ \mbox{ implies that }\ 
\lim_{r\to 0^+} \sup_{|x|=r} \frac{u(x)}{S(|x|)}=
\mu.$$
\end{lemma} 

\begin{proof}
Assume that $\limsup_{|x|\to 0} u(x)/S(|x|)=\mu\in (0,\infty)$. 
Let $r_0>0$ be small such that $B_{4 r_0}(0)\subset \subset \Omega$. 
Define 
$v(r)=\sup_{|x|=r} u(x)/S(|x|)$ for every $r\in (0,r_0)$. We show that $\lim_{r\to 0^+}v(r)=\mu$. 
Since $\limsup_{r\to 0^+} v(r)=\mu$, it remains to show that 
\begin{equation} \label{mar13} \liminf_{r\to 0^+} v(r)=\mu.\end{equation} We reason by contradiction and 
	assume that there exists a sequence $\{r_n\}_{n\geq 1}$ of positive numbers decreasing to $0$ such that $\lim_{n\to \infty} 
	v(r_n)= \mu_0\in [0,\mu)$. Fix $\varepsilon>0$ small such that $\mu_0+\varepsilon<\mu$. 
	
	Let $n_0\geq 1$ be large such that 
	\begin{equation} \label{mipp} r_n< r_0 \quad \mbox{and}\quad v(r_n)\leq \mu_0+\varepsilon\quad \mbox{for all } n\geq n_0.\end{equation}
	For every $n>n_0$, we set $D_n=\{x\in \mathbb R^N:\ r_n<|x|<r_{n_0}\}$. Since $\limsup_{r\to 0^+} v(r)=\mu>\mu_0+\varepsilon$,
	we infer that there exist $r_*\in (0,r_{n_0})$ and $x_*\in \mathbb R^N$ with $|x_*|=r_*$ such that 
	\begin{equation} \label{allo} v(r_*)=u(x_*)/S(|x_*|) >\mu_0+\varepsilon. \end{equation}
	Let $n_*>n_0$ be large such that $x_*\in D_{n_*}$. From the definition of $v$ and \eqref{mipp}, we get 
	$$ u|_{\partial B_{r_n}(0)} \leq (\mu_0+\varepsilon) \,S(r_n)\quad \mbox{for all } n\geq n_0. 
	$$ Observe that $\left(\mu_0+\varepsilon\right) S$ is a positive radial super-solution of \eqref{e1} in $B_{r_0}(0)\setminus \{0\}$ since 
	$$ \mathbb L_{\rho,\lambda,\tau}[\left(\mu_0+\varepsilon\right) S(|x|)]=\left(\mu_0+\varepsilon\right) \mathbb L_{\rho,\lambda,\tau}[S(|x|)]\leq 0\leq 
	|x|^\theta S^q(|x|)\quad \mbox{for all } 0<|x|<r_0. $$
	Moreover, in view of \eqref{mar12}, 
	we can apply the comparison principle and find that 
	$$ u(x)\leq \left(\mu_0+\varepsilon\right) S(|x|)\quad \mbox{ for every } x\in D_{n_*}.$$ 
	When $x=x_*\in D_{n_*}$, we obtain a contradiction with \eqref{allo}. This proves \eqref{mar13}. 
\end{proof} 	

\begin{rem} \label{rek44}
{\rm 
(a) When $\tau=0$ in Corollary~\ref{pere} and Lemma~\ref{mar11}, we don't need the second inequality in \eqref{mar12} since we can apply the comparison principle in Lemma~\ref{co2}. 

(b) We will apply Lemma~\ref{mar11} for an arbitrary positive solution $u$ of \eqref{e1} by taking a suitable positive $C^1$-function $S$ on $(0,r_0)$ such that the first inequality in \eqref{mar12} holds and, in addition, $|S'(r)|>0$ for every $r\in (0,r_0)$ (the latter inequality is needed when $\tau\in (0,1)$).}	
\end{rem}

\section{Summary of our classification results} \label{sec-sum}
In Table~\ref{tabel}, we present the findings of Theorem~\ref{thi} for each Case $[N_j]$ with $j=0,1,2$.

\begin{table}[h!]
\small
	\begin{tabular}{|l|l|}
		\hline

		 \textbf{Assume Case $[N_0]$}, that is, $ \left\{ 
		 \begin{aligned}
		 & \mbox{Case } [N_0] \ {\rm (a)}\ {\rm let}\  \tau=0,\ \lambda\geq 2\rho \\

		 & \mbox{or Case } [N_0]\ {\rm (b)}\ \mbox{let } \lambda>0, \ \tau\in (0,1),\ \rho<\Upsilon.   	
		 \end{aligned}
		 \right. $ \\

	$\bullet$ If $\beta<0$, then $u(x)\sim U_{\rho,\beta}(x)$ as $|x|\to 0$.
		 
		 \\ \hline
		 \hline

			 \textbf{Assume Case $[N_1]$}, that is, $ \left\{ 
		 \begin{aligned}
		 & [N_1]\ {\rm (a)}\ \mbox{let } \tau=0,\ \lambda< 2\rho \\
		 & {\rm or}\ [N_1]\, {\rm (b)}\ \mbox{let } \lambda<0, \ \tau\in (0,1),\ \rho\in \mathbb R.   	
		 \end{aligned}
		 \right. $\\
		 
$\bullet$ If $\beta\in (-\infty,\varpi_1(\rho))$, then $u(x)\sim U_{\rho,\beta}(x)$ as $|x|\to 0$.\\
		 
$\bullet$ If $\beta=\varpi_1(\rho)$, then $\displaystyle u(x)\sim \left(
\frac{ |\varpi_1(\rho)| \left(1-\lambda \tau |\varpi_1(\rho)|^{-\tau-1} \right)}{q-1}\right)^{1/(q-1)} 
\frac{|x|^{-\varpi_1(\rho)} }{|\log |x||^{1/(q-1)}}$ as $|x|\to 0$. \\

$\bullet$ If $\beta\in (\varpi_1(\rho),0)$, then $u(x)\sim a\,|x|^{-\varpi_1(\rho)}$ as $|x|\to 0$ for some $a\in \mathbb R_+$.		                                                              
	         
	    \\ \hline \hline

		 \textbf{Assume Case $[N_2]$}, that is, $\lambda>0$, $\tau\in (0,1)$ and $\rho\geq \Upsilon$.\\
		 
		$\bullet$ If $\beta\in (-\infty,\varpi_1(\rho))$, then $u(x)\sim U_{\rho,\beta}(x)$ as $|x|\to 0$.\\
		$\bullet$ If $\beta=\varpi_1(\rho)$ and $\rho\not=\Upsilon$, then \\
\ \ \ \ $\displaystyle u(x)\sim \left(
\frac{ |\varpi_1(\rho)| \left(1-\lambda \tau |\varpi_1(\rho)|^{-\tau-1} \right)}{q-1}\right)^{1/(q-1)} \frac{|x|^{-\varpi_1(\rho)} }{|\log |x||^{1/(q-1)}}$ as $|x|\to 0$. \\
$\bullet$ If $\beta=\varpi_1(\rho)$ and $\rho=\Upsilon$, then 
$\displaystyle u(x)\sim  \left(
\frac{2\left(q+\tau\right)}{(q-1)^2}\right)^{1/(q-1)} \frac{|x|^{-\varpi_1(\rho)}}{ |\log |x||^{2/(q-1)}}$ as $|x|\to 0$.\\
	$\bullet$  If  $\beta\in (\varpi_1(\rho),\varpi_2(\rho)]$ (for $\rho\not=\Upsilon$), then $u(x)\sim a\,|x|^{-\varpi_1(\rho)}$ 
	as $|x|\to 0$ for some $a\in \mathbb R_+$.    \\
	$\bullet$  If $\beta\in (\varpi_2(\rho),0)$, then as $|x|\to 0$, we have the following trichotomy:\\
	     (i)\ Either $u(x)\sim a\,|x|^{-\varpi_1(\rho)}$ for some $a\in \mathbb R_+$;\\
	     (ii) Or $u(x)\sim b\, \Phi_{\varpi_2(\rho)} (|x|)=
	     \left\{ \begin{aligned}
 	& b \,|x|^{-\varpi_2(\rho)} && \mbox{if } \rho\not=\Upsilon&\\
 	& b \,|x|^{-\varpi_2(\rho)} |\log |x||^{2/(1+\tau)} && \mbox{if } \rho=\Upsilon &
 \end{aligned}
	     \right.$ for some $b\in \mathbb R_+$;\\
                               (iii) Or $u(x)\sim U_{\rho,\beta}(|x|)$. 
		\\
		\hline \hline
		
	\end{tabular}
	\vspace{0.2cm}
	\caption{Complete classification of the behaviour near zero for an arbitrary positive solution $u$ of \eqref{e1}, where $\rho,\lambda\in \mathbb R$, $\tau\in [0,1)$, $\theta<-2$ and $q>1$}
		\label{tabel}
\end{table}

The contents of Theorem~\ref{globo} and Corollary \ref{glob22} are captured in Table~\ref{table11}.

\vspace{0.1cm}
\begin{table}[h!]
\small
	\begin{tabular}{|l|l|}
		\hline
		
		Eq. \eqref{e1} in $ \mathbb R^N\setminus \{0\}$ has positive solutions if and only if $ f_\rho(\beta)>0$.
		In this case,\\ every positive solution $ u$ of 
		\eqref{e1} in $ \mathbb R^N\setminus \{0\}$ 
		is radial. 
		\\ \hline \hline

{\tt Let $\beta<0$ in Case $[N_0]$.} Then, 
for every $ c\in \mathbb R_+$, there exists a unique positive\\ radial
solution $ u_{\infty,c}$ of \eqref{e1} in $ \mathbb R^N\setminus \{0\}$
such that $u(x)\sim U_{\rho,\beta}(x)$ as $|x|\to 0 $ \\
and $u(x)\sim \left\{  
\begin{aligned} 
& c\log |x| && \mbox{if } \lambda=2\rho \ \mbox{and } \tau=0 &\\
& c && \mbox{otherwise } &
\end{aligned} 
\right.
\ \mbox{as } |x|\to \infty.$                 		                        
\\  \hline \hline

{\tt Let $\beta<\varpi_1(\rho) $ in Case $[N_1]$.}
Then, for every $c\in \mathbb R_+$, 
there is a unique positive \\
radial solution $u_{\infty,c}$
of \eqref{e1} in $ \mathbb R^N\setminus \{0\}$
such that 
$ u(x)\sim U_{\rho,\beta}(x)\ \mbox{as }|x|\to 0$ \\
and $u(x)\sim c |x|^{-\varpi_1(\rho)} \ \mbox{as } |x|\to \infty $.\\         
		  \hline \hline                                                                                                                  		
		     		
{\tt Let $\beta<\varpi_1(\rho) $ in Case $[N_2]$.} 
Then, for every $ c\in \mathbb R_+$, there exists a unique positive \\
radial solution $u_{\infty,c}$ of \eqref{e1} in $ \mathbb R^N\setminus \{0\}$
such that $u(x)\sim U_{\rho,\beta}(x)$ as $|x|\to 0$ \\
and $u(x)\sim c \Psi_{\varpi_1(\rho)}(x)$  as  
		$|x|\to \infty $. 
 \\ \hline \hline

{\tt Let $ \beta\in (\varpi_2(\rho), 0)$
	in Case $[N_2]$.} \\
	
	$\bullet$ For every $c\in \mathbb R_+$, there exists a unique positive  radial solution 
	$u_{\infty,c}$ of \eqref{e1} \\
	in $ \mathbb R^N\setminus \{0\}$ such that 
$u(x)\sim U_{\rho,\beta}(x)$ as $|x|\to 0$ and 
$u(x)\sim c$ as $|x|\to \infty$. \\

$\bullet$ For every $b\in \mathbb R_+$, there exists a unique positive radial solution $u_{b}$ of \eqref{e1} in \\ 
$ \mathbb R^N\setminus \{0\}$
such that $ u(x)\sim b\,\Phi_{\varpi_2(\rho)}(|x|)$ as $|x|\to 0$ and 
$u(x) \sim U_{\rho,\beta}(x)$ as $|x|\to \infty$. \\

$\bullet$ For every $b,c\in \mathbb R_+$, there exists a unique positive
radial solution $u_{b,\infty,c}$ of \eqref{e1}  \\
in $ \mathbb R^N\setminus \{0\}$
such that $u(x)\sim b\,\Phi_{\varpi_2(\rho)}(x)$ as 
$|x|\to 0$ and $u(x) \sim c$ as $|x|\to \infty$. 
	\\ \hline \hline

{\tt Let $\beta>0$ in Case $[N_0](-\rho)$.} 
Then, for every $ c\in \mathbb R_+$, there exists a unique positive\\ 
radial solution $ u_{c,0}$ of \eqref{e1} in $ \mathbb R^N\setminus \{0\}$
such that $u(x)\sim U_{\rho,\beta}(x)$ as $|x|\to \infty $ \\
and $u(x)\sim \left\{  
\begin{aligned} 
& c\log\,(1/ |x|) && \mbox{if } \lambda=-2\rho \ \mbox{and } \tau=0 &\\
& c && \mbox{otherwise } &
\end{aligned} \right. 
\ \mbox{as } |x|\to 0 $.                
 \\   \hline \hline

{\tt Let $\beta>-\varpi_1(-\rho)$  in
Case $[N_1](-\rho)$.} 
Then, for every $c\in \mathbb R_+$, 
there is a unique positive \\
radial solution $u_{c,0}$
of \eqref{e1} in $ \mathbb R^N\setminus \{0\}$
such that 
$ u(x)\sim U_{\rho,\beta}(x)\ \mbox{as }|x|\to \infty$ \\
and $u(x)\sim c |x|^{\varpi_1(-\rho)} \ \mbox{as } |x|\to 0 $.
 \\ \hline\hline

{\tt Let $\beta>-\varpi_1(-\rho)$ in 
Case $[N_2](-\rho)$.}  Then, for every 
$c\in \mathbb R_+$, there exists a unique \\
positive radial solution $u_{c,0}$ of 
\eqref{e1} in $ \mathbb R^N\setminus \{0\}$
such that 
$u(x)\sim U_{\rho,\beta}(x)$
as $|x|\to \infty$ \\
and  $u(x)\sim c\, \Psi_{\varpi_1(-\rho)}(1/|x|)$ as  
$|x|\to 0  $. 
 \\ \hline \hline
	  
{\tt Let $\beta\in (0,-\varpi_2(-\rho))$
in Case $ [N_2](-\rho)$.} \\

$\bullet$ For every $c\in \mathbb R_+$, there exists a unique positive
radial solution $u_{c,0}$ if \eqref{e1} \\
in $\mathbb R^N\setminus \{0\}$ such that 
$ u(x)\sim U_{\rho,\beta}(x)$ as $|x|\to \infty$ and 
$u(x)\sim c$ as $|x|\to 0$. \\

$\bullet$ For every $b\in \mathbb R_+$, 
there exists a unique positive  solution $ u_{b,\infty}$ of \eqref{e1} in 
$\mathbb R^N\setminus \{0\}$ \\
such that $u(x)\sim b\,\Phi_{\varpi_2(-\rho)} (1/|x|)$  as $|x|\to \infty$ and 
$u(x) \sim U_{\rho,\beta}(x)$ as $|x|\to 0$. \\

$\bullet$ For every $b, c\in \mathbb R_+$, 
there exists a unique positive radial solution $ u_{c,0,b}$ of \eqref{e1} \\
in $\mathbb R^N\setminus \{0\}$ such that 
$u(x)\sim b\,\Phi_{\varpi_2(-\rho)}(1/|x|) $ as 
$|x|\to \infty$  and $u(x) \sim c$ as $ |x|\to 0$. 
\\ \hline \hline

	Otherwise, there are {\bf no} positive solutions of \eqref{e1} in $\mathbb R^N\setminus \{0\}$ (since $f_\rho(\beta)\leq 0$).
		
		\\ \hline
		
	\end{tabular}
	\vspace{0.2cm}
	\caption{Complete classification and existence of all positive solutions $u$ of \eqref{e1} in $\mathbb R^N\setminus \{0\}$ other than $U_{\rho,\beta}$, where $\rho,\lambda, \theta\in \mathbb R$, $\tau\in [0,1)$ and $q>1$}
		\label{table11}
\end{table}

\end{document}